\documentstyle[amssymb,amsfonts]{amsart}

\newcommand{\Lie}{{\mathcal{L}}}
\newcommand{\rLie}{\mbox{${\mathcal{L}} \mkern-09mu /$\,}}
\newcommand{\der}{\nabla}

\newcommand{\rder}{\mbox{$\nabla \mkern-13mu /$\,}}

\newcommand{\les}{\lesssim}
\newcommand{\bea}{\begin{eqnarray}}
\newcommand{\eea}{\end{eqnarray}}

\newenvironment{proof}{\noindent {\bf Proof} }{\endprf\par}

\def\pa{\partial}
\def\e{\mbox{exp}}
\def \endprf{\hfill  {\vrule height6pt width6pt depth0pt}\medskip}
\def\beaa{\begin{eqnarray*}}
\def\eeaa{\end{eqnarray*}}
\def\pa{\partial}
\def\a{{\alpha}}
\def\b{{\beta}}
\def\ga{\gamma}
\def\Ga{\Gamma}
\def\de{\delta}

\def\eps{\epsilon}

\def\si{\sigma}

\def\om{\omega}
\def\Om{\Omega}

\def\th{\theta}

\def\S{{\bf S}}

\def\g{{\bf g}}
\def\pr{\partial}

\begin{document}
\theoremstyle{plain}
\newtheorem{theorem}{Theorem}
\newtheorem{proposition}[subsection]{Proposition}
\newtheorem{lemma}[subsection]{Lemma}

\theoremstyle{remark}
\newtheorem{remark}[subsection]{Remark}
\newtheorem{remarks}[subsection]{Remarks}

\theoremstyle{definition}
\newtheorem{definition}[subsection]{Definition}

\include{psfig}
\title[Maxwell on Black Holes]{On Uniform Decay of the Maxwell Fields \\ on Black Hole Space-Times}
\author{Sari Ghanem}
\address{Institut de Math\'ematiques de Jussieu, Universit\'e Paris Diderot - Paris VII, 75205 Paris Cedex 13, France}
\email{ ghanem@@math.jussieu.fr}
\maketitle

\begin{abstract} 
This is the second in a series of papers in which we take a systematic study of gauge field theories such as the Maxwell equations and the Yang-Mills equations, on curved space-times. In this paper, we study the Maxwell equations in the domain of outer-communication of the Schwarzschild black hole. We show that if we assume that the middle components of the non-stationary solutions of the Maxwell equations verify a Morawetz type estimate supported around the trapped surface, then we can prove uniform decay properties for components of the Maxwell fields in the entire exterior of the Schwarzschild black hole, including the event horizon, by making only use of Sobolev inequalities combined with energy estimates using the Maxwell equations directly. This proof is entirely gauge independent, and does not pass through the scalar wave equation on the Schwarzschild black hole, and does not need to separate the middle components for the Maxwell fields. However, proving a Morawetz estimate directly using the Maxwell equations, without refering to the scalar wave equation, seems to be out of reach of the mathematical community as of today; which I was not able to solve yet in this work. If one is able to prove the Morawetz estimate directly using the Maxwell equations, this combined with the present work would give full conceptual proof of decay of the Maxwell fields on the Schwarzschild black hole, and would then be in particular useful for the non-abelian case of the Yang-Mills equations where the separation of the middle components cannot occur. The whole manuscript is written in an expository way where we detail all the calculations.

\end{abstract}

\setcounter{page}{1}
\pagenumbering{arabic}

\section{Introduction}

In this paper, we study the Maxwell equations on the Schwarzschild black hole. In a recent paper, [DR1]-[DR2], Dafermos and Rodnianski proved decay for solutions of the free scalar wave equation $\Box_{\g} \phi = 0$ in the exterior of the Schwarzschild black hole, up to points on the event horizon. We do not know how to make these methods work for $\Box_{\g} \phi = \phi$ or for $\Box_{\g} \phi = \phi^{2}$. Thus, this rules out the possibility of using these methods for the Maxwell equations in a hyperbolic formulation where the source terms would be $(\Box_{\g} F )_{\mu\nu}$, where $F_{\mu\nu}$ is the Maxwell field. In a recent paper, [Bl], Blue proved decay for the Maxwell fields on the exterior of the Schwarzschild background. The proof of Blue required a study of a wave equation on the Schwarzschild space-time for the middle components, which can be separated from the other components in the abelian case of the Maxwell equations. This was later extended by Andersson and Blue to Kerr metrics, [AB]. However, in the non-linear case of the Yang-Mills equations, one cannot decouple the middle components from the others. Yet, it seems difficult to generalize the results of Dafermos and Rodnianski for the free scalar wave equation to the Maxwell equations using the Maxwell energy-momentum tensor directly, without referring to the scalar wave equation, combined with suitable Sobolev inequalities. A key step to achieve this would be to bound the conformal energy without separating the so-called middle components of the Maxwell fields. This would provide a new independent proof and improves the result of Blue, and would be in particular useful for the non abelian case of the Yang-Mills equations where such separation cannot occur. I tried to do this in the goal of proving uniform boundedness for Yang-Mills fields on the exterior of the Schwarzschild black hole and the Kerr metric. However, as Klainerman pointed out to me later, many people tried to get a more conceptual proof of decay for the Maxwell equations, without passing through the scalar wave equation, and achieving this would be very significant. While I was not able to solve this yet, in this paper I write a proof of decay for the Maxwell fields on the exterior of the Schwarzschild space-time, directly without separating the middle components or refering to the wave equation, assuming that we have a Morawetz estimate at the $0$-derivative level. However, as Andersson and Blue wrote in [AB], proving a Morawetz estimate for the Maxwell field directly would be an important advance in the field, which I did not achieve yet in the present work. To explain, let us recall (see Appendix) that in the exterior, the Schwarzschild metric can be written as,
\beaa
ds^{2} = -(1 - \frac{2m}{r})dt^{2} + \frac{1}{(1 - \frac{2m}{r})}dr^{2} + r^{2}(d\th^{2} + \sin^{2}(\th) d\phi^{2}) \label{eq:Schw}
\eeaa
If we define,
\beaa
r^{*} &=& r + 2m\log(r - 2m) \\
v &=& t + r^{*}  \\
w &=& t - r^{*} 
\eeaa
then, we have,
\beaa
\notag
ds^{2} &=&  - (1 - \frac{2m}{r})dv dw + r^{2}d\sigma^{2} \\
&=&  - \frac{(1 - \frac{2m}{r})}{2} dv\bigotimes dw - \frac{(1 - \frac{2m}{r})}{2} dw \bigotimes dv + r^{2}d\sigma^{2}
\eeaa

Let,
\bea
\hat{\frac{\pa}{\pa w}}  &=&  \frac{1}{(1-\frac{2m}{r})} \frac{\pa}{\pa w} \\
\hat{\frac{\pa}{\pa v}}  &=&   \frac{\pa}{\pa v} \\
\hat{\frac{\pa}{\pa \th}}  &=&  \frac{1}{r} \frac{\pa}{\pa \th} \\
\hat{\frac{\pa}{\pa \phi}}  &=& \frac{1}{r \sin\th }  \frac{\pa}{\pa \phi}
\eea

and at a point in the Schwarzschild space-time, let $e_{1}$, $e_{2}$ be a normalized basis of $S^{2}$, which verifies for all $A, B \in \{1, 2 \}$,
\bea
\g(e_{A}, \hat{\frac{\pa}{\pa w}} ) &=& \g (e_{A}, \hat{\frac{\pa}{\pa v}} ) =   0 \\
\g (e_{A}, e_{B} ) &=& \de_{AB}
\eea

In this context, by "assuming a Morawetz estimate", we mean exactly that for $r_{0}$, $R_{0}$ as in the proof of \eqref{KcontolledbytG}, and for all $t_{i}= (1.1)^{i} t_{0}$, where $t_{0}$ is any real positive number and $i$ is any natural number, if we define:
\bea
\notag
J_{ F }^{(G)}  (  t_{i} \leq t \leq t_{i+1} )  (r_{0} < r < R_{0} )  &=&   \int_{t = t_{i}}^{ t= t_{i+1}}  \int_{r^{*} =  r_{0}^{*} }^{r^{*}= R_{0}^{*} } \int_{\S^{2}}  [    | F_{\hat{v}\hat{w}}|^{2}  +  \frac{1}{4 } | F_{\hat{\phi}\hat{\th}}|^{2} ] .  dr^{*} d\sigma^{2} dt \\
\eea
then, we assume that the non-stationary solutions verify
\bea
&& J_{ F }^{(G)}  (  t_{i} \leq t \leq t_{i+1} ) (r_{0} < r < R_{0} ) \label{Assumption1} \\
\notag
 &\les&  | \hat{E}^{(\frac{\pa}{\pa t})}_{F} (t_{i}) | + | \hat{E}^{(\frac{\pa}{\pa t})}_{F} (t_{i+1})|  + \sum_{j=1}^{3} ( | \hat{E}^{(\frac{\pa}{\pa t})}_{\Lie_{\Om_{j} } F} (t_{i})| +  |\hat{E}^{(\frac{\pa}{\pa t})}_{\Lie_{\Om_{j} } F} (t_{i+1})|  ) 
\eea

where $\Om_{j}$, $j \in \{ 1, 2, 3 \}$, is a basis of angular momentum operators, and where
\bea
\notag
\hat{E}^{(\frac{\pa}{\pa t})}_{F} (t) &=&  \int_{r^{*} = - \infty}^{r^{*} =  \infty} \int_{\S^{2}}   [    \frac{1}{r^{2} (1-\frac{2m}{r}) }  |F_{t\th}|^{2} +  \frac{1}{r^{2}(1-\frac{2m}{r}) \sin^{2} \th}   |F_{t\phi}|^{2} \\
\notag
&&+  \frac{1}{r^{2}(1-\frac{2m}{r}) } |{F_{r^{*}\th}}^{} |^{2}    +  \frac{1}{r^{2} (1-\frac{2m}{r}) \sin^{2} \th}  | {F_{r^{*}\phi}}^{}|^{2}          ]   r^{2} (1-\frac{2m}{r}) d\sigma^{2} dr^{*} \\ \label{definitionofenergyhatieenergywithoutzerocomponents}
\eea
that is the negative of the energy without the middle components $F_{t r^{*}}$, $F_{\th\phi}$. The only stationary solutions of the Maxwell equations on the exterior of the Schwarzschild black hole, are the so-called Coulomb solutions. Hence, the assumption above is assumed for the non-Coulomb solutions. Our proof would still work with any product of Lie derivatives on the right hand side of \eqref{Assumption1}, with adjusting accordingly the quantities in the theorem that depend on the initial data.
 
We will prove that if the middle components of the non-stationary solutions verify a Morzwetz type estimate at the zero-derivative level, \eqref{Assumption1}, then we can prove uniform decay properties of solutions to the Maxwell equations in the domain of outer-communication of the Schwarzschild black hole space-time, including the event horizon, by making use of suitable Sobolev inequalities combined with energy estimates using the energy momentum tensor of the Maxwell fields. We do not make any use of decomposition into spherical harmonics. We start with a Cauchy hypersurface prescribed by $t = t_{0}$ where the initial data has to verify certain regularity conditions (there is no vanishing condition on the bifurcate sphere for $F$). Away from the horizon (in the region $ r \geq R > 2m$, for a fixed $R$), we will prove that
\beaa
 |F_{\hat{\mu}\hat{\nu}}|(w, v, \om) & \leq & \frac{ C }{ (1 + |v|) } \\
 |F_{\hat{\mu}\hat{\nu}}|(w, v, \om) & \leq & \frac{   C  }{ (1 + |w|) } 
\eeaa 
for all $\hat{\mu}, \hat{\nu} \in \{\hat{\frac{\pa}{\pa w}}, \hat{\frac{\pa}{\pa v}}, \hat{\frac{\pa}{\pa \th}}, \hat{\frac{\pa}{\pa \phi}} \}$, where $ F_{\mu \nu}$ is the Maxwell field. Near the horizon, and in the entire exterior region $r \geq 2m$, up to points on the event horizon, we will prove that
\beaa
 |F_{\hat{v} \hat{w}} (v, w, \om) | &\leq& \frac{ C  }{  \max\{1, v \}  }  , \quad \quad |F_{e_{1} e_{2}} (v, w, \om) | \leq  \frac{ C }{ \max\{1, v \} }\\
|F_{ \hat{v}e_{a} } (v, w, \om ) |  &\leq&  \frac{  C }{  \max\{1, v \} }, \quad \quad  | \sqrt{1-\frac{2m}{r}} F_{ \hat{w}e_{a} } (v, w, \om ) | \leq   \frac{ C }{ \max\{1, v \}  } 
\eeaa

To explain really thoroughly:

If one tries to generalize the proof of Dafermos and Rodnianski for the free scalar wave equation to the Maxwell equations, on the Schwarzschild space-time, using the Maxwell energy-momentum tensor instead, some of the difficulties that appear are:

\begin{enumerate}
\item
In order to bound the conformal energy of the Maxwell fields and their derivatives in the direction of Killing vector fields, one needs to control a space-time integral near the trapped surface $r = 3m$, that involves the so-called middle components of the Maxwell field  (see \eqref{zerocomponontsconformalenergy}). In fact, the terms in the space-time integral $J^{(K)}$, obtained by applying the divergence theorem on the Morawetz vector field $K = - w^2 \frac{\pa}{\pa w} - v^2 \frac{\pa}{\pa v} $ contracted with the energy momentum tensor of the Maxwell fields, are non-negative in a region $r_{0} \leq r \leq R_{0}$ that contains $r=3m$. It seems that this cannot be controlled using these methods due to the presence of the other components with the "wrong" sign in the space-time integral generated from a space type vector field of control (see \eqref{contracteddeformationforG}). Indeed, using the divergence theorem with the vector field $G =  f(r^{*})  \frac{\pa}{\pa r^{*}} $, where $f$ is a bounded function, we will obtain a space-time integral and boundary terms. However, the space-time integral obtained from $G$ has independent terms that do not appear in $J^{(K)}$ that enter with the wrong sign, and hence it cannot be made positive. We will prove that if we get past this, see assumption \eqref{Assumption1}, we can can then write a gauge independent proof of uniform decay of the Maxwell fields in the exterior of the Schwarzschild black hole up to points on the horizon, using the Maxwell equations directly. 

\item
One needs to construct a new field which verifies the Maxwell equations and the Bianchi identities, that coincides with the original field in some region and vanishes identically outside another specific region (see the proof of \eqref{EKovertsquare}). In the case of the wave equation $\Box_{\g} \phi = 0$, one can multiply the initial data in the Cauchy problem by a cut-off function, and consider the evolution of such data to obtain a solution that verifies the wave equation and the properties stated previously. In the case of the Maxwell and the Yang-Mills equations, if one multiplies the initial data by a cut-off function then the constraint equations would not be satisfied anymore. It seems at first sight that one cannot get a new field that verifies the needed properties. While this is true if one wanted to do this for all components, nevertheless, one can do this for all the components except to the $F_{rt}$ and $F_{\th\phi}$ components, where the multiplication should be at the level of the space derivative of the components. Hence, one can construct a new field that can be made to coincide with $F$ in a certain region, and vanish outside another except to these last components. The somewhat good news is that the calculations show that these "bad" terms do not appear in the boundary terms generated from the divergence theorem applied to a space-like vector field contracted with the energy-momentum tensor (see \eqref{definitionoftheboundarytermgeneratedfromG}). Hence, in our assumption \eqref{Assumption1}, we suppose that the space-time integral near the the trapped surface $r=3m$ of the middle components, can be bounded by the energy without the middle components, which would be the case if this estimate was obtained by controlling the space-time integral by boundary terms generated from space type vector fields multiplied by a bounded functionr, as shown in estimate \eqref{controlenergyforG}. This is crucial to establish \eqref{EKovertsquare} that is the main estimate to bound the conformal energy in \eqref{boundingEK}.

\item In order to prove decay for a generalized energy that would control the $L^{2}$ norm of the fields near the horizon, one is confronted to a situation where it seems crucial to control a space-time integral supported on a bounded region in space near the event horizon, that contains all the components. We overcome this by our assumption \eqref{Assumption1}; it can also be used to bound the space-time integral containing the other components as in \eqref{estimate2H}.

Also, in addition to the above:
 
\item As opposed to the case of the wave equation, the flux of a generalized energy that controls the $L^{2}$ norm of the fields near the horizon do not contain all of their components. On $v = constant$ hypersurfaces it contains only $F_{\hat{w}\hat{\th}}$, $F_{\hat{w}\hat{\phi}}$, $F_{\hat{v}\hat{w}}$ and $F_{\hat{\th}\hat{\phi}}$, and on $w = constant$ hypersurfaces, it contains only the components $F_{\hat{v}\hat{w}}$, $F_{\hat{\th}\hat{\phi}}$, $F_{\hat{v}\hat{\th}}$ and $F_{\hat{v}\hat{\phi}}$. In addition, while using Sobolev inequalities near the horizon, since $\frac{\pa }{\pa v}$ and $\frac{\pa}{\pa w}$ are not Killing this would add an additional difficulty, while in the case of the wave equation, the squares of these derivatives appear in the fluxes which are easily controlled. We get around these problems by using suitable Sobolev type inequalities for each component, combined with the Bianchi identities and the field equations, in a way that the derivatives in the direction of $\frac{\pa }{\pa v}$ and $\frac{\pa}{\pa w}$ can be controlled by Killing derivatives, $\pa_t$ and $\pa_{\Om_{j}}$, of the components that appear in the flux to which we would have proved decay.
 
\end{enumerate}

In the case of the Yang-Mills equations there is the additional impediment that is the equations are non-linear. If one gets an energy identity for the Yang-Mills fields, one cannot write directly the same energy identity for the derivatives of the field in the direction of Killing vector fields, as opposed to the Maxwell fields, due to the non-linearity of the equations.

More precisely, we will prove the following theorem,

\subsection{The statement}\

\begin{theorem} \label{thetheorem}

Let $(\overline{M}, \overline{\g} )$ be a maximally extended Schwarzschild space-time. We know by then that the exterior of the black hole space-time, $(\overline{M}, \overline{\g} )$, is isometric to $(M, \g)$ where,

\beaa
\g &=& -(1 - \frac{2m}{r})dt^{2} + \frac{1}{(1 - \frac{2m}{r})}dr^{2} + r^{2}(d\th^{2} + \sin^{2}(\th) d\phi^{2}) \\
&=&    - \frac{(1 - \frac{2m}{r})}{2} dv\bigotimes dw - \frac{(1 - \frac{2m}{r})}{2} dw \bigotimes dv + r^{2}d\sigma^{2} 
\eeaa 

Let, $\Sigma_{t=t_{0}} $ be a Cauchy hypersurface prescribed by $t=t_{0}$. In a system of coordinates, let $ F_{\mu \nu}(w, v, \om) $ be the components of the Maxwell field defined as the solution of the Cauchy problem of the Maxwell equations:
\bea
\der^{\a} F_{\a\b} =  0 \label{Maxwellequations1}
\eea
\bea
\der_{\a} F_{\b\ga} + \der_{\b} F_{\ga\a} + \der_{\ga} F_{\a\b} = 0 \label{Maxwellequations2}
\eea
where the initial data prescribed on the Cauchy hypersurface $\Sigma_{t=t_{0}}$ verifies the Maxwell constraint equations:
$$ {\der}^{\nu} F_{t\nu} (t_{0}, r, \om)  = {\der}^{\nu} F_{\mu\nu} (\frac{\pa}{\pa t})^{\mu} (t_{0}, r, \om)  = 0 $$
$$[ {\der}_{r} F_{\th\phi} + {\der}_{\th} F_{\phi r} + {\der}_{\phi} F_{r \th} ] (t_{0}, r, \om)  = 0 $$

We assume that for $r_{0}$ and $R_{0}$ as in the proof of \eqref{KcontolledbytG}, ($r_{0} \leq 3m \leq R_{0}$), and for all $t_{i}= (1.1)^{i} t_{0}$, where $t_{0}$ is any real positive number and $i$ is any natural number, the non-stationary solutions verify
\bea
&& \int_{t = t_{i}}^{ t= t_{i+1}}  \int_{r^{*} =  r_{0}^{*} }^{r^{*}= R_{0}^{*} } \int_{\S^{2}}  [    | F_{\hat{v}\hat{w}}|^{2}  +  \frac{1}{4 } | F_{\hat{\phi}\hat{\th}}|^{2} ] .  dr^{*} d\sigma^{2} dt  \label{Assumption1writtenusingJG} \\
\notag
&\les&  | \hat{E}^{(\frac{\pa}{\pa t})}_{F} (t_{i}) | + | \hat{E}^{(\frac{\pa}{\pa t})}_{F} (t_{i+1})|  + \sum_{j=1}^{3} ( | \hat{E}^{(\frac{\pa}{\pa t})}_{\Lie_{\Om_{j} } F} (t_{i})| +  |\hat{E}^{(\frac{\pa}{\pa t})}_{\Lie_{\Om_{j} } F} (t_{i+1})|  )  
\eea

where $\Om_{j}$, $j \in \{ 1, 2, 3 \}$, is a basis of angular momentum operators, and where

\bea
\notag
\hat{E}^{(\frac{\pa}{\pa t})}_{F} (t)  &=& \int_{r^{*} = - \infty }^{r^{*} = \infty }  \int_{\phi = 0}^{ 2\pi} \int_{\th= 0 }^{\pi }   2 [  (1- \frac{2m}{r} ) ^{2} (   | F_{\hat{w}\hat{\th}} |^{2} +  | F_{\hat{w}\hat{\phi}} |^{2} ) +   | F_{\hat{v}\hat{\th}} |^{2} +  |F_{\hat{v}\hat{\phi}} |^{2}      ] . r^{2}  \sin(\th) d\th d\phi dr^{*} 
\eea

that is the energy without the middle components $F_{\hat{v}\hat{w}}$ and $F_{\hat{\th}\hat{\phi}}$.

\begin{remark}
Our proof would work with any arbitrary product of Lie derivatives of $F$ on the right hand side of assumption \eqref{Assumption1writtenusingJG}, with an adjustment on the initial data accordingly.
\end{remark}
 
Then, we have,

\beaa
 |F_{\hat{v} \hat{w}} (v, w, \om) | &\leq& \frac{ C  }{  \max\{1, v \}  }  , \quad \quad |F_{e_{1} e_{2}} (v, w, \om) | \leq  \frac{ C }{ \max\{1, v \} }\\
|F_{ \hat{v}e_{a} } (v, w, \om ) |  &\leq&  \frac{  C }{  \max\{1, v \} }, \quad \quad | \sqrt{1-\frac{2m}{r}} F_{ \hat{w}e_{a} } (v, w, \om ) | \leq   \frac{ C }{ \max\{1, v \}  } 
\eeaa

in the entire exterior region, up to points on the horizon, under certain regularity conditions on the initial data prescribed in what follows, to which we also add the assumption that their limit goes to zero at spatial infinity on the initial slice $\Sigma_{t=t_{0}}$.

More precisely, away from the horizon (in the region $ r \geq R > 2m$), we have,

\beaa
 |F_{\hat{\mu}\hat{\nu}}|(w, v, \om) &\lesssim& \frac{  E_{F} }{ (1 + |v|) } 
\eeaa

and,

\beaa
 |F_{\hat{\mu}\hat{\nu}}|(w, v, \om)  &\lesssim& \frac{   E_{F}  }{ (1 + |w|) } 
\eeaa

for all normalized components $\hat{\mu}, \hat{\nu} \in \{\hat{\frac{\pa}{\pa w}}, \hat{\frac{\pa}{\pa v}}, \hat{\frac{\pa}{\pa \th}}, \hat{\frac{\pa}{\pa \phi}} \}$, and where $E_{F}$ is defined by,

\beaa
E_{F} &= & [ \sum_{i=0}^{1}  \sum_{j=0}^{5} E_{ r^{j} (\rLie)^{j} (\Lie_{t})^{i}  F }^{(\frac{\pa}{\pa t})} (t=t_{0})     +  E_{ r^{6} (\rLie)^{6}  F }^{(\frac{\pa}{\pa t})} (t=t_{0})   \\
&& + \sum_{i=0}^{1}  \sum_{j=0}^{4} E_{ r^{j} (\rLie)^{j} (\Lie_{t})^{i}  F  }^{(K)} (t=t_{0})   +  E_{ r^{5} (\rLie)^{5}  F  }^{(K)} (t=t_{0}) ]^{\frac{1}{2}} \\
\eeaa

where $\rLie$ is the Lie derivative restricted on the 2-spheres, and where,

\beaa
E^{(K)}_{F} (t_{i}) &=& \int_{r^{*} = - \infty}^{r^{*}= \infty}    \int_{\phi = 0}^{ 2\pi}  \int_{\th= 0 }^{\pi }   (   w^{2} (1-\frac{2m}{r})^{2} [ | F_{\hat{w}\hat{\th} } |^{2} +  | F_{\hat{w} \hat{\phi} } |^{2} ]  +  v^{2}  [  | F_{\hat{v} \hat{\th} } |^{2} +  | F_{\hat{v} \hat{\phi} } |^{2} ]  \\
&& \quad \quad \quad \quad \quad +  (\om^{2} + v^{2} )  (1- \frac{2m}{r} ) [  |F_{\hat{v} \hat{w} }|^{2}   +   | F_{\hat{\phi}\hat{\th} }|^{2}]  )   r^{2}  \sin(\th)d\th d\phi dr^{*} 
\eeaa

and where,

\beaa
 E_{F}^{  (\frac{\pa}{\pa t} )} ( t= t_{0} ) &=& \int_{r^{*} = - \infty }^{r^{*} = \infty }  \int_{\phi = 0}^{ 2\pi} \int_{\th= 0 }^{\pi }   2 [   (1- \frac{2m}{r} )^{2} ( | F_{\hat{w}\hat{\th}} |^{2} +  | F_{\hat{w}\hat{\phi}} |^{2} ) +   | F_{\hat{v}\hat{\th}} |^{2} +  | F_{\hat{v}\hat{\phi}} |^{2}  \\
&& +   (1- \frac{2m}{r} ) (  |F_{\hat{v}\hat{w}}|^{2}    +  | F_{\hat{\phi}\hat{\th}}|^{2}  )  ] . r^{2}  \sin(\th) d\th d\phi dr^{*} 
\eeaa

We expect the angular momentum derivatives, or any other Killing derivatives, in the assumption \eqref{Assumption1} to come out due to the presence of the trapped surface $r=3m$. Hence, we also assume - although we can make the proof without the following assumption by using only \eqref{Assumption1} with the price of losing more derivatives on the initial data - that the solutions we are looking at verify
\bea
\notag
\int_{t = t_{i}}^{ t= t_{i+1}}  \int_{r^{*} =  r_{0}^{*} }^{r^{*}= R_{0}^{*} } \int_{\S^{2}}  [    | F_{\hat{v}\hat{w}}|^{2}  +  \frac{1}{4 } | F_{\hat{\phi}\hat{\th}}|^{2} ] .|r^{*} - (3m)^{*}|.  dr^{*} d\sigma^{2} dt &\les&  | \hat{E}^{(\frac{\pa}{\pa t})}_{F} (t_{i}) | + | \hat{E}^{(\frac{\pa}{\pa t})}_{F} (t_{i+1})| \\ \label{Assumption2} 
\eea

Then, near the horizon (in the region $2m \leq r \leq R $), we have,

\beaa
 |F_{\hat{v} \hat{w}} (v, w, \om) | &\lesssim & \frac{ E_{1}  }{  \max\{1, v \}  }  , \quad \quad |F_{e_{1} e_{2}} (v, w, \om) | \lesssim  \frac{ E_{1}  }{ \max\{1, v \} }\\
|F_{ \hat{v}e_{a} } (v, w, \om ) |  &\les&  \frac{  E_{2} }{  \max\{1, v \} }, \quad \quad  | \sqrt{1-\frac{2m}{r}}  F_{ \hat{w}e_{a} } (v, w, \om ) | \lesssim   \frac{ E_{2}  }{ \max\{1, v \}  } \\
\eeaa
for $a \in \{1, 2 \}$, and where,

\beaa
E_{1} &= & [ \sum_{j = 0}^{6}    E_{   r^{j}(\rLie)^{j}   F }^{(\frac{\pa}{\pa t}) } (t=t_{0})  + \sum_{j = 0}^{5}    E_{   r^{j}(\rLie)^{j}   F }^{( K ) } (t=t_{0})  +  \sum_{j=1}^{3}   E_{r^{j}(\rLie)^{j} F }^{ \# (\frac{\pa}{\pa t} )} ( t= t_{0} ) ]^{\frac{1}{2}}  \\
\eeaa

where,

\beaa
 E_{F}^{ \# (\frac{\pa}{\pa t} )} ( t= t_{0} ) &=& \int_{r^{*} = - \infty }^{r^{*} = \infty } \int_{\S^{2}}    [  (1-\frac{2m}{r} ) (  | F_{\hat{w}\hat{\th}} |^{2} +  | F_{\hat{w}\hat{\phi}} |^{2}   ) +(  | F_{\hat{v}\hat{\th}} |^{2} +  | F_{\hat{v}\hat{\phi}} |^{2} ) \\
&& + (  |F_{\hat{v}\hat{w}}|^{2}   +   | F_{\hat{\phi}\hat{\th}}|^{2} )  ] . r^{2} d\sigma^{2} dr^{*} ( t = t_{0} )
\eeaa

and where,

\beaa
E_{2} &=& [ E_{F}^{2} + \sum_{i=0}^{1}  \sum_{j=1}^{2}    E_{ r^{j}(\rLie)^{j} (\Lie_{t})^{i} F }^{ \# (\frac{\pa}{\pa t} )} ( t= t_{0} )   +    E_{r^{3}(\rLie)^{3} F }^{ \# (\frac{\pa}{\pa t} )} ( t= t_{0} )  ]^{\frac{1}{2}} \\
&= & [ \sum_{i=0}^{1} \sum_{j = 0}^{5}  E_{   r^{j}(\rLie)^{j}  (\Lie_{t})^{i}  F }^{(\frac{\pa}{\pa t}) } (t=t_{0})    +  E_{  r^{6}(\rLie)^{6}   F}^{(\frac{\pa}{\pa t}) } (t=t_{0}) \\
&& + \sum_{i=0}^{1}  \sum_{j = 0}^{4}   E_{   r^{j}(\rLie)^{j}  (\Lie_{t})^{i}  F}^{( K ) } (t=t_{0})  +  E_{     r^{5}(\rLie)^{5} F }^{( K ) } (t=t_{0}) \\
&& + \sum_{i=0}^{1} \sum_{j=1}^{2}      E_{ r^{j}(\rLie)^{j} (\Lie_{t})^{i} F }^{ \# (\frac{\pa}{\pa t} )} ( t= t_{0} )   +    E_{r^{3}(\rLie)^{3} F }^{ \# (\frac{\pa}{\pa t} )} ( t= t_{0} )  ]^{\frac{1}{2}}\\
\eeaa

\end{theorem}

\subsection{Strategy of the proof}\

We decompose our proof of decay for the Maxwell fields, on the Schwarzschild black hole, into two parts. The first proves decay away from the horizon (in the region $r \geq R$ where $R > 2m$, arbitrarily fixed), and the second part deals with the region near the horizon ($2m \leq r \leq R$).

\subsubsection{In the first part} All integrations will be done on spacelike hypersurfaces prescribed by $t = constant$, and space-time integrals will be understood as integrals on a region bounded by those hypersurfaces. The starting point of our proof is a suitable use of Sobolev inequalities. Sobolev inequalities permit one to bound the $L^{\infty}$ norm of the square of the Maxwell field (the square is taken with respect to the scalar product on the Lie algebra $<\,,\,>$), by the $L^{2}$ norm of the Maxwell field and its derivatives up to some order. And yet, the  $L^{2}$ norm of the Maxwell fields and their derivatives in the direction of Killing vector fields can be controlled away from the horizon by the energy $E^{(\frac{\pa}{\pa t})}$ of those (energy obtained from the vector field $\frac{\pa}{\pa t}$). We lose control on the $L^{2}$ norm of the components $F_{\hat{v}\hat{w}}$, $F_{\hat{\th}\hat{\phi}}$, $F_{\hat{w}\hat{\th}}$, $F_{\hat{\th}\hat{\phi}}$ of those fields near the horizon due to the presence of the $(1-\frac{2m}{r})$ term (that vanishes at the horizon) that appears in the expression of the energy. And so, since the covariant derivatives of the Maxwell fields in the direction of non-Killing vector fields can be transformed into covariant derivatives in the direction of Killing vector fields by using the field equations and the Bianchi identities, one can bound the $L^{2}$ norm of those fields by their energy which is conserved, because we have conservation of energy for the Maxwell equations and therefore for their Killing derivatives because the Maxwell equations are linear. This way, we can bound the Maxwell fields away from horizon.

However, if we prove decay of the local $L^{2}$ norms of the Maxwell fields and their derivatives in the direction of Killing vector fields, we can prove decay of the Maxwell fields. The key point is that the $L^{2}$ norms here can be taken to be space integrals on only a bounded region. This way, away from the horizon, they can be controlled by a piece of the energy integral, that is the energy as a space integral without integrating on the whole space, but only on the bounded region (this is because the terms that appear in the space integral of the energy are exactly the squares of the Maxwell fields, multiplied by the term $(1-\frac{2m}{r})$ for some components). These energies taken on a bounded region of space, can be bounded by the conformal energy (obtained by taking the Morawetz vector field $K = - \om^{2} \frac{\pa}{\pa \om} - v^{2} \frac{\pa}{\pa v}$) divided by the minimum on that region of $v^{2}$ and $w^{2}$ (see  \eqref{EKoverminvsquaredandwsquared}). Consequently, if we bound the conformal energy, we have shown so far how one can possibly obtain decay from this of solutions to the Maxwell equations away from the horizon.

To bound the conformal energy of the Maxwell fields and their derivatives in the direction of Killing vector fields, we proceed as follows:

\begin{enumerate}

\item We will use the space-time integral $J_{F}^{(G)}$ in our assumption \eqref{Assumption1}, of which the terms are positive, to control $J_{F}^{(K)}$ (the space-time integral obtained by using the divergence theorem on the Morawetz vector field $K$ contracted with the energy momentum tensor of the Maxwell fields) in the following sense, see estimate \eqref{KcontolledbytG}:

\bea
J_{F}^{(K)} ( t_i \leq t \leq t_{i+1} ) &\les& t_{i+1} J_{F}^{(G)} ( t_i \leq t \leq t_{i+1} ) (r_{0} < r < R_{0} )
\eea
and from our assumption \eqref{Assumption1}, we have
\beaa
&& J_{F}^{(G)} ( t_i \leq t \leq t_{i+1} ) \\
&\les&  | \hat{E}^{(\frac{\pa}{\pa t})}_{F} (t_{i}) | + | \hat{E}^{(\frac{\pa}{\pa t})}_{F} (t_{i+1})|  + \sum_{j=1}^{3} ( | \hat{E}^{(\frac{\pa}{\pa t})}_{\Lie_{\Om_{j} } F} (t_{i})| +  |\hat{E}^{(\frac{\pa}{\pa t})}_{\Lie_{\Om_{j} } F} (t_{i+1})|  ) 
\eeaa

where $\Om_{j}$, $j \in \{1, 2, 3 \}$, is a basis of angular momentum derivatives, and $\hat{E}^{(\frac{\pa}{\pa t})}_{F} (t)$ is the energy without the middle components, see \eqref{definitionofenergyhatieenergywithoutzerocomponents}.

\item We construct a new field, $\hat{F}$, such that it verifies the Maxwell equations and the Bianchi identities, and it coincides with the original field in some region and vanishes identically outside another specific region for the components which appear in the boundary terms $\hat{E}^{(\frac{\pa}{\pa t})}_{F}(t)$, so that we could write:

\bea
| \hat{E}^{(\frac{\pa}{\pa t})}_{\hat{F}}(t) | \les \frac{E^{(K)}_{F}(t) }{t^{2}}
\eea

and consequently (see \eqref{EKovertsquare}):

\bea
\notag
&& J_{F}^{(G)} ( t_i \leq t \leq t_{i+1} ) (r_{0} < r < R_{0} )  \\
&\lesssim&   \frac{1}{t_{i}^{2}} E_{F}^{(K)} (t=t_{i}) + \frac{1}{t_{i}^{2}} \sum_{j} E_{\Omega_{j} F}^{(K)} (t=t_{i}) \\
\notag
&& + \frac{1}{t_{i+1}^{2}} E_{F}^{(K)} (t=t_{i+1}) + \frac{1}{t_{i+1}^{2}} \sum_{j} E_{\Omega_{j} F}^{(K)} (t=t_{i+1})   
\eea

For $t_{i}$ such that $ t_{i+1} = (1.1) t_{i} $ (where $i$ is an integer), and $| r^{*}(r_{0}) | +| r^{*}(R) | \leq 0.4 t_{i}$, we can make use of the divergence theorem to properly commutate these inequalities, and use the fact that the series $\sum \frac{ i}{t_{i}}$ converges, to establish a uniform bound on the conformal energy that depends on the initial data and its Killing Lie derivatives derivatives. We obtain \eqref{boundingEK}:
\bea
\notag
&& E_{F}^{(K)} (t) \\
\notag
&\les&  E_{F, r \rLie F, r^{2} (\rLie)^{2} F, r^{3} (\rLie)^{3} F }^{(\frac{\pa}{\pa t})}   (t=t_{0})  +  E_{F, r \rLie  F, r^{2} (\rLie)^{2} F  }^{(K)} (t=t_{0})  = E_{F}^{M}  \\
\eea

\subsubsection{In the second part} To obtain decay near the horizon, we are going to integrate in rectangles in the Penrose diagram representing the exterior of the Schwarzschild black hole, of which one side is included in the horizon. We will apply the divergence theorem with the vector field $  H = - \frac{ h(r^{*})}{(1-\frac{2m}{r})} \frac{\pa}{\pa w}   -  h(r^{*}) \frac{\pa}{\pa v}  $ contracted with the energy momentum tensor of the Maxwell fields, where $h \geq 0$ is supported in the region  $ 2m \leq r \leq (1.2) r_{1} $ for $r_{1}$ chosen such that, $ 2m < r_{0} \leq r_{1} < (1.2) r_{1} < 3m $, and where $h$ is such that $ \lim_{r^{*} \to - \infty} h (r^{*}) =  1 $, and for $r \leq r_{1}$, we have $h > 0$, $h' \geq 0$, $ h' \leq \frac{2m}{r^{2}}  h$, $(1-\frac{2m}{r}) \frac{3}{r} h \leq h'$, and $ (1+\frac{6m}{r^{2}}) h \leq \frac{2m h'}{(r-2m)}$. By applying the divergence theorem in the rectangles described previously with $H$ defined as such, we get that the flux through the hypersurfaces prescribed by $v = constant$ is roughly speaking the $L^{2}$ norms of $\sqrt{1-\frac{2m}{r}} F_{\hat{w}\hat{\th}}$ and $\sqrt{1-\frac{2m}{r}} F_{\hat{w}\hat{\phi}}$, and the $L^{2}$ norms of $\sqrt{1-\frac{2m}{r}} F_{\hat{v}\hat{w}}$ and $\sqrt{1-\frac{2m}{r}} F_{\hat{\th}\hat{\phi}}$. In addition, we get a space-time integral supported near the event horizon $-I_{F}^{(H)}$, of which the terms are roughly speaking the squares of $F_{\hat{w}\hat{\th}}$, and $F_{\hat{w}\hat{\phi}}$, and roughly a factor that goes to zero when $r$ goes to $2m$ multiplied by the squares of $F_{\hat{v}\hat{w}}$, $F_{\hat{\th}\hat{\phi}}$, $F_{\hat{v}\hat{\th}}$, and $F_{\hat{v}\hat{\phi}}$.\\

\item In order to prove decay for the flux of $H$, that is a generalized energy that would control the $L^{2}$ norm of some components of the fields near the horizon, one is confronted to a situation where it seems crucial to control a space-time integral supported on a bounded region in space near the event horizon, that contains the non-middle components. This could be overcome by the assumed Morawetz estimate, \eqref{Assumption1}, that could be used to bound the space-time integral of the non-middle components as well. Just to simplify the calculations we assume that we have \eqref{Assumption2}, since we expect the derivatives in assumption \eqref{Assumption1} to come out due to the trapped surface $r=3m$, i.e.
\bea
\notag
&& \int_{t = t_{i}}^{ t= t_{i+1}}  \int_{r^{*} =  r_{0}^{*} }^{r^{*}= R_{0}^{*} } \int_{\S^{2}}  [    | F_{\hat{v}\hat{w}}|^{2}  +  \frac{1}{4 } | F_{\hat{\phi}\hat{\th}}|^{2} ] .|r^{*} - (3m)^{*}|.  dr^{*} d\sigma^{2} dt \\
&\les&  | \hat{E}^{(\frac{\pa}{\pa t})}_{F} (t_{i}) | + | \hat{E}^{(\frac{\pa}{\pa t})}_{F} (t_{i+1})| 
\eea

Hence, away from the horizon, in $r \geq r_{1}$, the space-time integral, $|I_{F}^{(H)} |$, can be bounded by the standard energy supported on a bounded region in space, to which one can prove decay due to the boundedness of the conformal energy that we would have already established in the first part.

Near the horizon, in $r \leq r_{1}$, the choices $h' \geq 0$, $ h' \leq \frac{2m}{r^{2}}  h$, $(1-\frac{2m}{r}) \frac{3}{r} h \leq h'$ were constructed on purpose to obtain $0 \leq - I_{F}^{(H)} (r \leq r_{1}) $. This last fact will lead to an inequality on $- I_{F}^{(H)} $, that involve the flux of the standard energy that one can bound, and that from the vector field $H$ on $v = constant$ hypersurface, \eqref{estimate3H1}, for $v_{i} = t_{i} + r_{1}^{*} $, and $w_{i} = t_{i} - r_{1}^{*} $, where $t_{i}$ is defined as in the first part, $t_{i} = (1.1)^{i} t_{0}$:

\bea
\notag
&& - I_{F}^{(H)} (  v_{i} \leq v \leq v_{i+1} ) ( w_{i} \leq w \leq \infty) ( r \leq r_{1} ) \\
\notag
&& - F_{F}^{(H)} ( v = v_{i+1} ) ( w_{i} \leq w \leq \infty )  - F_{F}^{(H)} ( w = \infty ) ( v_{i} \leq v \leq v_{i+1} )  \\
\notag
&\lesssim& F_{F}^{(\frac{\pa}{\pa t} )} ( w = w_{i} ) ( v_{i} \leq v \leq v_{i+1} )  - F_{F}^{(H)} ( v = v_{i} ) ( w_{i} \leq w \leq \infty ) \\
\notag
&&  +  | E^{(\frac{\pa}{\pa t})}_{F}  (  -(0.85)t_{i} \leq r^{*} \leq (0.85)t_{i}  ) (t= t_{i}) | | \\
\eea

Using the Cauchy stability one can bound the flux from $H$ by the initial data prescribed on the initial Cauchy hypersurface, and hence bound the space-time integral $- I_{F}^{(H)} ( r \leq r_{1} ) $. In addition, we can prove inequality \eqref{estimate4H}:
\bea
\notag
&& \inf_{ v_{i} \leq v \leq v_{i+1} }  - F_{F}^{(H)} ( v  ) ( w_{i} \leq w \leq \infty )  \\
\notag
&\lesssim& \frac{ - I_{F}^{(H)}(  v_{i} \leq v \leq v_{i+1} ) ( w_{i} \leq w \leq \infty) ( r \leq r_{1}) }{  (v_{i+1} - v_{i} ) } \\
&&+ \sup_{ v_{i} \leq v \leq v_{i+1} } F_{F}^{(\frac{\pa}{\pa t} )} ( v  ) ( w_{i} \leq w \leq \infty ) ( r \geq r_{1} ) 
\eea

To prove decay to the flux from $H$ on $v = constant$ hypersurfaces near the horizon and on $w = constant$ hypersurfaces on a segment  of fixed length in $v$, we apply the last two inequalities above in the rectangle prescribed by $[v_{i}, v_{i+1}].[w_{i}, \infty]$, and we commutate them properly. This will lead to decay in $v_{+} = \max\{1, v \}$ of the flux from $H$ on $v = constant$ near the horizon, and on $w = constant$ hypersurfaces with a fixed length in $v$, as shown in estimates \eqref{estimate6Hflux} and \eqref{estimate6Hwequalconstantflux}):

\bea
 - F_{F}^{(H)} ( v  )(2m \leq r \leq R) &\lesssim &  \frac{ [ |E_{F}^{  (\frac{\pa}{\pa t} )} | + E_{F}^{ \# (\frac{\pa}{\pa t} )} ( t= t_{0} ) + E_{F}^{M} ] }{  v_{+}^{2}  }
\eea
and,
\bea
 - F_{F}^{(H)} ( w ) (  v-1 \leq \overline{v} \leq v )  &\lesssim &  \frac{ [ |E_{F}^{  (\frac{\pa}{\pa t} )} | + E_{F}^{ \# (\frac{\pa}{\pa t} )} ( t= t_{0} ) + E_{F}^{M} ] }{  v_{+}^{2}  } 
\eea

Finally:
 
\item To prove decay for the normalized components $F_{\hat{v}\hat{w}}$ and $F_{e_{1}e_{2}}$ we make use of a Sobolev inequality restricted on $w= constant$ hypersurfaces with a fixed length in $v$. Using the field equations, the $L^{2}$ norms of derivatives in the direction of $\frac{\pa}{\pa v}$ can be controlled by the $L^{2}$ norms of $F_{\hat{v}e_{1}}$, $F_{\hat{v}e_{2}}$ and of their angular derivatives, and of $F_{\hat{v}\hat{w}}$ and $F_{e_{1}e_{2}}$. This leads to a bound by the flux obtained from $H$ on $w = constant$ hypersurfaces, of $F$ and its angular momentum derivatives. Since, those are Killing derivatives, and since the Maxwell equations are linear, and we proved decay of for the flux, this leads to the desired result.

We do the same to prove decay for the components $F_{\hat{v}e_{1}}$ and $F_{\hat{v}e_{2}}$ except that this time, the $L^{2}$ norms of the derivatives in the direction of $\frac{\pa}{\pa v}$ can be bounded by the $L^{2}$ norms of the time derivatives of $F_{\hat{v}e_{1}}$, $F_{\hat{v}e_{2}}$, and of the angular derivatives of $F_{\hat{v}\hat{w}}$, $F_{e_{1}e_{2}}$, and of $F_{\hat{v}e_{1}}$, $F_{\hat{v}e_{2}}$, using the field equations and the Bianchi identities.

The components $\sqrt{1-\frac{2m}{r}} F_{\hat{w}e_{1}}$ and $\sqrt{1-\frac{2m}{r}}  F_{\hat{w}e_{2}}$ can be controlled by using a Sobolev inequality where we integrate on the hypersurfaces $v= constant$. As a result of direct computation using the field equations and the Bianchi identities, the $L^{2}$ norms of derivatives in the direction of $\frac{\pa}{\pa w}$ can be controlled by the $L^{2}$ norms of time derivatives of $\sqrt{1-\frac{2m}{r}}  F_{\hat{w}e_{1}}$, $\sqrt{1-\frac{2m}{r}}  F_{\hat{w}e_{2}}$, of angular derivatives of $\sqrt{1-\frac{2m}{r}}  F_{\hat{v}\hat{w}}$, $\sqrt{1-\frac{2m}{r}}  F_{e_{1}e_{2}}$, and of $\sqrt{1-\frac{2m}{r}}  F_{\hat{w}e_{1}}$, $\sqrt{1-\frac{2m}{r}} F_{\hat{w}e_{2}}$. Making use of the decay of the flux from the vector field $H$ on $v =constant$, we obtain decay of these local $L^{2}$ norms, and hence we prove pointwise decay for $\sqrt{1-\frac{2m}{r}} F_{\hat{w}e_{a}}$, $a \in \{1, 2 \}$.\\

\end{enumerate}

In order to cover the whole exterior region in theorem \eqref{thetheorem}, we chose in the first part $R=r_{1}$, and we can choose $r_{1} = r_{0}$.
 
\begin{remark}
The whole manuscript is written in an expository way, where we detail all the calculations, and we show standard material in the Appendix.
\end{remark}

\textbf{Acknowledgments.} The author would like to thank his PhD thesis advisors, Fr\'ed\'eric H\'elein and Vincent Moncrief, for their advice and support, Sergiu Klainerman for suggesting the problem in part of a research proposal for the author's doctoral dissertation, and Pieter Blue for pointing out mistakes in an earlier version by making helpful remarks. This work was supported by a full tuition fellowship from Universit\'e Paris VII - Institut de Math\'ematiques de Jussieu, and from the Mathematics Department funds of Yale University. The author would like to thank the Mathematics Department of Yale University for their kindness and hospitality while writing this work in Spring 2012. The manuscript was edited by the author while receiving financial support from the Albert Einstein Institute, Max-Planck Institute for Gravitational Physics, under an invitation from Lars Andersson, and the author would like to thank him for his kind invitation and for his interest in this work, and the Albert Einstein Institute for their kindness and hospitality.\\

\section{Conservation Laws}

Let $\Psi_{\mu\nu}$ be a an anti-symmetric two tensor valued in the Lie algebra. Consider the energy-momentum tensor
\bea
T_{\mu\nu}^{}(\Psi) = <\Psi_{\mu\b}, {\Psi_{\nu}}^{\b}> - \frac{1}{4} \g_{\mu\nu} < \Psi_{\a\b}, \Psi^{\a\b}>
\eea
Considering the Schwarzschild time $t^{'}$ (see \eqref{Schwarzschildtime} in Appendix), and considering two spacelike hypersurfaces $\Sigma_{t^{'}_{1}}$, $\Sigma_{t^{'}_{2}}$, $t^{'}_{2} > t^{'}_{1}$. We consider the region $B = J^{+}(\Sigma_{t^{'}_{1}})\cap J^{-}(\Sigma_{t^{'}_{2}}) \cap D$, where $D$ is the closure of the exterior region of the black hole, known as the domain of outer-communication of the black hole.\ 

Considering a vector field $V^{\nu}$ we let $$J_{\mu}(V) = V^{\nu}T_{\mu\nu} = T_{\mu V}$$

We have,

\begin{eqnarray*}
\der^{\mu} J_{\mu}(V) &=& \pa^{\mu} T_{\mu V} -  T ( \der^{\mu} e_{\mu}, V)  = \der^{\mu} T_{\mu V} + T (e_{\mu}, \der^{\mu} V) 
\end{eqnarray*}

Hence,

\begin{eqnarray*}
\der^{\mu} J_{\mu}(V) &=& \der^{\mu} T_{\mu V}  + T (e_{\mu}, \der^{\mu} V) \\
&=&  V^{\nu} ( \der^{\mu}T_{\mu\nu} ) + (\der^{\mu}V^{\nu} ) T_{\mu\nu}   \\
&=& V^{\nu} ( \der^{\mu}T_{\mu\nu} ) + \pi^{\mu\nu}  T_{\mu\nu}  
\end{eqnarray*}

(by symmetry of $T$)

Applying the divergence theorem on $ J_{\mu}(V)$ in the region $B$ bounded to the past by $\Sigma_{t^{'}_{1}}$ and to the future by $\Sigma_{t^{'}_{2}}$, and by a null hypersurface $N$, we obtain: \
\bea
\notag
 \int_{B} V^{\nu} ( \der^{\mu}T_{\mu\nu} )   dV_{B} + \int_{B}  \pi^{\mu\nu}(V) T_{\mu\nu} dV_{B}  &=& \int_{\Sigma_{t^{'}_{1}}} J_{\mu}(V) n^{\mu} dV_{\Sigma_{t^{'}_{1}}} -  \int_{\Sigma_{t^{'}_{2}}} J_{\mu}(V) n^{\mu} dV_{\Sigma_{t^{'}_{2}}} \\
&& - \int_{N}J_{\mu}(V) n_{N}^{\mu} dV_{N} \label{conservationlawdivergncetheorem}
\eea
where $n^{\mu}$ are the unit normal to the hypersurfaces $\Sigma$, $n_{N}^{\mu}$ is any null generator of $N$, $dV_{\Sigma}$ are the induced volume forms and $dV_{N}$ is defined such that the divergence theorem applies. \

Considering the Maxwell field $F$, as it verifies \eqref{Maxwellequations1} and \eqref{Maxwellequations2}, we have
\bea
&&\der^{\nu} T_{\mu\nu} (F) =  0    
\eea
Also, considering the spherical symmetry of the Schwarzschild black hole, the angular momentum operators are Killing and therefore $\Lie_{\Om_{j}} F$ verifies the Maxwell equations\eqref{Maxwellequations1} and \eqref{Maxwellequations2}, and thus
\beaa
&&\der^{\nu} T_{\mu\nu} (\Lie_{\Om_{j}} F) =  0    
\eeaa
where $\Om_{j}$ is a basis of angular momentum operators, $j \in \{ 1 , 2, 3 \}$. \\

And given that $\frac{\pa}{\pa t }$ is Killing, we have
\beaa
&&\der^{\nu} T_{\mu\nu} (\Lie_{t} F) =  0
\eeaa
Taking $\Psi$ any product of $\Lie_{t}$, $\Lie_{\Om_{i}}$, and $F$, where $\Om_{i}$, $i \in \{1, 2, 3\}$, is a basis of angular momentum operators, we get that $\Psi$ satisfies the Maxwell equations \eqref{Maxwellequations1} and \eqref{Maxwellequations2} since $\frac{\pa}{\pa t}$ and $\Om_{j}$, $j \in \{ 1 , 2, 3 \}$ are Killing vector fields and therefore. Hence,
\bea
\der^{\nu} T_{\mu\nu} ( \Psi) =  0 \label{theenrgymomuntumtensorofpsiisdivergencefree}
\eea 

Let,
\bea
\mu = \frac{2m}{r}
\eea
Now, let's compute,
\beaa
&&  < \Psi_{\a\b}, \Psi^{\a\b}> \\
&=&   < \Psi_{w\a}, \Psi^{w\a}> + < \Psi_{v\a}, \Psi^{v\a}> + < \Psi_{\th\a}, \Psi^{\th\a}> + < \Psi_{\phi\a}, \Psi^{\phi\a}>   \\
&=&    \frac{-2}{(1-\mu)} < \Psi_{w\a}, {\Psi_{v}}^{\a}> +  \frac{-2}{(1-\mu)} < \Psi_{v\a}, {\Psi_{w}}^{\a}> +  \frac{1}{r^{2}}  < \Psi_{\th\a}, {\Psi_{\th}}^{\a}> \\
&& +  \frac{1}{r^{2}\sin^{2}\th} < \Psi_{\phi\a}, {\Psi_{\phi}}^{\a}>  \\
&=&      \frac{-2}{(1-\mu)} < \Psi_{w\th}, {\Psi_{v}}^{\th}>   -   \frac{2}{(1-\mu)} < \Psi_{w\phi}, {\Psi_{v}}^{\phi}>   - \frac{2}{(1-\mu)} < \Psi_{wv}, {\Psi_{v}}^{v}>  \\
&&   - \frac{2}{(1-\mu)} < \Psi_{v\th}, {\Psi_{w}}^{\th}> -  \frac{2}{(1-\mu)} < \Psi_{v\phi}, {\Psi_{w}}^{\phi}>  -  \frac{2}{(1-\mu)} < \Psi_{vw}, {\Psi_{w}}^{w}> \\
&& +  \frac{1}{r^{2}}  < \Psi_{\th v}, {\Psi_{\th}}^{v}> +  \frac{1}{r^{2}}  < \Psi_{\th w}, {\Psi_{\th}}^{w}> +  \frac{1}{r^{2}}  < \Psi_{\th\phi}, {\Psi_{\th}}^{\phi}>\\
&& +  \frac{1}{r^{2}\sin^{2}\th} < \Psi_{\phi v}, {\Psi_{\phi}}^{v}>   +  \frac{1}{r^{2}\sin^{2}\th} < \Psi_{\phi w}, {\Psi_{\phi}}^{w}>   +  \frac{1}{r^{2}\sin^{2}\th} < \Psi_{\phi\th}, {\Psi_{\phi}}^{\th}>  \\
&=&      \frac{-2}{r^{2} (1-\mu)} < \Psi_{w\th}, {\Psi_{v\th}}^{}>   -   \frac{2}{r^{2} \sin^{2} \th (1-\mu)} < \Psi_{w\phi}, {\Psi_{v\phi}}^{}>   +  \frac{4}{(1-\mu)^{2}}  < \Psi_{wv}, {\Psi_{vw}}^{}>  \\
&&   - \frac{2}{r^{2} (1-\mu)} < \Psi_{v\th}, {\Psi_{w\th}}> -  \frac{2}{r^{2} \sin^{2}\th (1-\mu)} < \Psi_{v\phi}, {\Psi_{w\phi}}^{}>  +  \frac{4}{(1-\mu)^{2}} < \Psi_{vw}, {\Psi_{w v}}^{}> \\
&& -  \frac{2}{r^{2}(1-\mu)}  < \Psi_{\th v}, \Psi_{\th v}>   -  \frac{2}{r^{2}(1-\mu)}  < \Psi_{\th w}, \Psi_{\th w}> +  \frac{1}{r^{4} \sin^{2} \th }  | \Psi_{\th\phi} |^{2} \\
&& -  \frac{2}{r^{2}\sin^{2}\th (1-\mu) } < \Psi_{\phi v}, {\Psi_{\phi w}}^{}>   -  \frac{2}{r^{2}\sin^{2}\th (1-\mu)} < \Psi_{\phi w}, {\Psi_{\phi v}}^{}>   +  \frac{1}{r^{4}\sin^{2}\th} | \Psi_{\phi\th}|^{2}  \\
&=&      \frac{-8}{r^{2} (1-\mu)} < \Psi_{w\th}, {\Psi_{v\th}}^{}>   -   \frac{8}{r^{2} \sin^{2} \th (1-\mu)} < \Psi_{w\phi}, {\Psi_{v\phi}}^{}>   -  \frac{8}{(1-\mu)^{2}}  | \Psi_{vw}|^{2}  +  \frac{2}{r^{4}\sin^{2}\th} | \Psi_{\phi\th}|^{2}  
\eeaa
Computing,
\beaa
T_{ww} &=&  <\Psi_{w\b}, {\Psi_{w}}^{\b}> - \frac{1}{2} \g_{w w} < \Psi_{\a\b}, \Psi^{\a\b}> \\
&=& < \Psi_{wv}, {\Psi_{w}}^{v} > +  < \Psi_{w\th}, {\Psi_{w}}^{\th} >  + < \Psi_{w\phi}, {\Psi_{w}}^{\phi} >  \\
&=& \frac{1}{r^{2}} | \Psi_{w\th} |^{2} + \frac{1}{r^{2}\sin^{2}\th} | \Psi_{w\phi} |^{2} 
\eeaa
\beaa
 T_{vv} &=&  <\Psi_{v\a} , {\Psi_{v}}^{\a} >  \\
&=&  <\Psi_{vw} , {\Psi_{v}}^{w} >  +  <\Psi_{v\th} , {\Psi_{v}}^{\th} >  + <\Psi_{v\phi} , {\Psi_{v}}^{\phi} >  \\
&=& \frac{1}{r^{2}} | \Psi_{v\th} |^{2} + \frac{1}{r^{2}\sin^{2}\th} | \Psi_{v\phi} |^{2} 
\eeaa
\beaa
&& T_{v w} \\
&=&  <\Psi_{v\a}, {\Psi_{w}}^{\a} > - \frac{1}{4} \g_{vw} < \Psi_{\a\b}, \Psi^{\a\b}> \\
&=& <\Psi_{vw}, {\Psi_{w}}^{w} > + <\Psi_{v\th}, {\Psi_{w}}^{\th} > + <\Psi_{v\phi}, {\Psi_{w}}^{\phi} >  - \frac{1}{4} \g_{vw}  < \Psi_{\a\b}, \Psi^{\a\b}>  \\
&=&   \frac{-2}{(1-\mu)} <\Psi_{vw}, \Psi_{wv} > + \frac{1}{r^{2}} <\Psi_{v\th}, \Psi_{w\th} > + \frac{1}{r^{2}\sin^{2}\th}  <\Psi_{v\phi}, \Psi_{w\phi} >\\
&& + \frac{ (1-\mu) }{8}  [  \frac{-8}{r^{2} (1-\mu)} < \Psi_{w\th}, {\Psi_{v\th}}^{}>   -   \frac{8}{r^{2} \sin^{2} \th (1-\mu)} < \Psi_{w\phi}, {\Psi_{v\phi}}^{}>  \\
&&  -  \frac{8}{(1-\mu)^{2}}  | \Psi_{vw}|^{2}  +  \frac{2}{r^{4}\sin^{2}\th} | \Psi_{\phi\th}|^{2}   ] \\
&=&   \frac{2}{(1-\mu)} |\Psi_{vw}|^{2}  + \frac{1}{r^{2}} <\Psi_{v\th}, \Psi_{w\th} > + \frac{1}{r^{2}\sin^{2}\th}  <\Psi_{v\phi}, \Psi_{w\phi} >\\
&&   - \frac{1}{r^{2} } < \Psi_{w\th}, {\Psi_{v\th}}^{}>   -   \frac{1}{r^{2} \sin^{2} \th } < \Psi_{w\phi}, {\Psi_{v\phi}}^{}>  \\
&&  -  \frac{1}{(1-\mu)}  | \Psi_{vw}|^{2}  +  \frac{(1-\mu)}{4r^{4}\sin^{2}\th} | \Psi_{\phi\th}|^{2}    \\ 
&=&     \frac{1}{(1-\mu)} |\Psi_{vw}|^{2}   +  \frac{(1-\mu)}{4r^{4}\sin^{2}\th} | \Psi_{\phi\th}|^{2}  
\eeaa
\beaa
&& T_{\th\th} \\
&=&  < \Psi_{\th\a},  {\Psi_{\th}}^{\a}> - \frac{1}{4} \g_{\th\th}  < \Psi_{\a\b}, \Psi^{\a\b}> \\
&=& < \Psi_{\th v},  {\Psi_{\th}}^{v}> + < \Psi_{\th w},  {\Psi_{\th}}^{w}>  +< \Psi_{\th\phi},  {\Psi_{\th}}^{\phi}> \\
&& - \frac{r^{2}}{4}  [    \frac{-8}{r^{2} (1-\mu)} < \Psi_{w\th}, {\Psi_{v\th}}^{}>   -   \frac{8}{r^{2} \sin^{2} \th (1-\mu)} < \Psi_{w\phi}, {\Psi_{v\phi}}^{}>  \\
&& -  \frac{8}{(1-\mu)^{2}}  | \Psi_{vw}|^{2}  +  \frac{2}{r^{4}\sin^{2}\th} | \Psi_{\phi\th}|^{2} ] \\
&=&   \frac{-2}{(1-\mu)}  < \Psi_{\th v},  {\Psi_{\th w}}^{}> - \frac{2}{(1-\mu)}  < \Psi_{\th w},  {\Psi_{\th v}}^{}>  + \frac{1}{r^{2}\sin^{2}\th}  |  \Psi_{\th\phi} |^{2} \\
&&   +   \frac{2}{ (1-\mu)} < \Psi_{w\th}, {\Psi_{v\th}}^{}>   +   \frac{2}{ \sin^{2} \th (1-\mu)} < \Psi_{w\phi}, {\Psi_{v\phi}}^{}>  \\
&& + \frac{2r^{2}}{(1-\mu)^{2}}  | \Psi_{vw}|^{2}  -  \frac{1}{2r^{2}\sin^{2}\th} | \Psi_{\phi\th}|^{2} \\
&=&   \frac{-2}{(1-\mu)}  < \Psi_{\th v},  {\Psi_{\th w}}^{}>  +   \frac{2}{ \sin^{2} \th (1-\mu)} < \Psi_{w\phi}, {\Psi_{v\phi}}^{}>  \\
&& + \frac{2r^{2}}{(1-\mu)^{2}}  | \Psi_{vw}|^{2}  +  \frac{1}{2r^{2}\sin^{2}\th} | \Psi_{\phi\th}|^{2} 
\eeaa
\beaa
&& T_{\phi\phi} \\
&=&  < \Psi_{\phi\a},  {\Psi_{\phi}}^{\a}> - \frac{1}{4} \g_{\phi\phi}  < \Psi_{\a\b}, \Psi^{\a\b}> \\
&=& < \Psi_{\phi v},  {\Psi_{\phi}}^{v}> + < \Psi_{\phi w},  {\Psi_{\phi}}^{w}>  +< \Psi_{\phi\th},  {\Psi_{\phi}}^{\th}> \\
&& - \frac{r^{2}\sin^{2}\th }{4}  [    \frac{-8}{r^{2} (1-\mu)} < \Psi_{w\th}, {\Psi_{v\th}}^{}>   -   \frac{8}{r^{2} \sin^{2} \th (1-\mu)} < \Psi_{w\phi}, {\Psi_{v\phi}}^{}>  \\
&& -  \frac{8}{(1-\mu)^{2}}  | \Psi_{vw}|^{2}  +  \frac{2}{r^{4}\sin^{2}\th} | \Psi_{\phi\th}|^{2} ] \\
&=&   \frac{-2}{(1-\mu)}  < \Psi_{\phi v},  {\Psi_{\phi w}}^{}> - \frac{2}{(1-\mu)}  < \Psi_{\phi w},  {\Psi_{\phi v}}^{}>  + \frac{1}{r^{2}}  |  \Psi_{\th\phi} |^{2} \\
&&   +   \frac{2\sin^{2}\th}{ (1-\mu)} < \Psi_{w\th}, {\Psi_{v\th}}^{}>   +   \frac{2}{  (1-\mu)} < \Psi_{w\phi}, {\Psi_{v\phi}}^{}>  \\
&& + \frac{2r^{2}\sin^{2}\th}{(1-\mu)^{2}}  | \Psi_{vw}|^{2}  -  \frac{1}{2r^{2}} | \Psi_{\phi\th}|^{2} \\
&=&   \frac{-2}{(1-\mu)}  < \Psi_{\phi v},  {\Psi_{\phi w}}^{}>  +   \frac{2 \sin^{2} \th}{ (1-\mu)} < \Psi_{w\th}, {\Psi_{v\th}}^{}>  \\
&& + \frac{2r^{2}\sin^{2}\th}{(1-\mu)^{2}}  | \Psi_{vw}|^{2}  +  \frac{1}{2r^{2}} | \Psi_{\phi\th}|^{2} 
\eeaa
This gives,
\beaa
&& \pi^{\a\b}T_{\a\b} \\
&=& T_{\om\om}\pi^{\om\om} +  T_{vv}\pi^{vv} +  2T_{v\om}\pi^{v\om} +  T_{\th\th}\pi^{\th\th} + T_{\phi\phi}\pi^{\phi\phi} \\
&=& [\frac{1}{r^{2}} | \Psi_{w\th} |^{2} + \frac{1}{r^{2}\sin^{2}\th} | \Psi_{w\phi} |^{2}]( \frac{-2}{(1 - \mu)} \pa_{v}V^{w}) \\
&& + [ \frac{1}{r^{2}} | \Psi_{v\th} |^{2} + \frac{1}{r^{2}\sin^{2}\th} | \Psi_{v\phi} |^{2} ] (\frac{-2}{(1 - \mu)} \pa_{w}V^{v})\\
&& + 2 [ \frac{1}{(1-\mu)} |\Psi_{vw}|^{2}   +  \frac{(1-\mu)}{4r^{4}\sin^{2}\th} | \Psi_{\phi\th}|^{2}  ] ( \frac{-1}{(1-\mu)} [    \pa_{v} V^{v}  + \pa_{w} V^{w}   +      \frac{m}{r^{2}} ( V^{v}  -   V^{w} )  ] ) \\
&&+ [  \frac{-2}{(1-\mu)}  < \Psi_{\th v},  {\Psi_{\th w}}^{}>  +   \frac{2}{ \sin^{2} \th (1-\mu)} < \Psi_{w\phi}, {\Psi_{v\phi}}^{}>  \\
&& + \frac{2r^{2}}{(1-\mu)^{2}}  | \Psi_{vw}|^{2}  +  \frac{1}{2r^{2}\sin^{2}\th} | \Psi_{\phi\th}|^{2}  ] ( \frac{(1 - \mu)}{2r^{3}}(V^{v} - V^{w}) ) \\
&&+ [   \frac{-2}{(1-\mu)}  < \Psi_{\phi v},  {\Psi_{\phi w}}^{}>  +   \frac{2 \sin^{2} \th}{ (1-\mu)} < \Psi_{w\th}, {\Psi_{v\th}}^{}>  \\
&& + \frac{2r^{2}\sin^{2}\th}{(1-\mu)^{2}}  | \Psi_{vw}|^{2}  +  \frac{1}{2r^{2}} | \Psi_{\phi\th}|^{2} ] (\frac{(1 - \mu)}{2r^{3}\sin^{2}\th } (V^{v} - V^{w}) )
\eeaa
\beaa
&=& [\frac{1}{r^{2}} | \Psi_{w\th} |^{2} + \frac{1}{r^{2}\sin^{2}\th} | \Psi_{w\phi} |^{2}]( \frac{-2}{(1 - \mu)} \pa_{v}V^{w}) \\
&& + [ \frac{1}{r^{2}} | \Psi_{v\th} |^{2} + \frac{1}{r^{2}\sin^{2}\th} | \Psi_{v\phi} |^{2} ] (\frac{-2}{(1 - \mu)} \pa_{w}V^{v})\\
&& + [ \frac{1}{(1-\mu)} |\Psi_{vw}|^{2}   +  \frac{(1-\mu)}{4r^{4}\sin^{2}\th} | \Psi_{\phi\th}|^{2}  ] ( \frac{-2}{(1-\mu)} [    \pa_{v} V^{v}  + \pa_{w} V^{w}   +      \frac{m}{r^{2}} ( V^{v}  -   V^{w} ) ]  ) \\
&&+ [  \frac{4r^{2} }{(1-\mu)^{2}}  | \Psi_{vw}|^{2}  +  \frac{1}{r^{2}\sin^{2}\th } | \Psi_{\phi\th}|^{2} ] (\frac{(1 - \mu)}{2r^{3}}(V^{v} - V^{w}) )
\eeaa
Thus,
\bea
\notag
&& \pi^{\a\b}(V)T_{\a\b}(\Psi) \\
&=& [\frac{1}{r^{2}} | \Psi_{w\th} |^{2} + \frac{1}{r^{2}\sin^{2}\th} | \Psi_{w\phi} |^{2}]( \frac{-2}{(1 - \mu)} \pa_{v}V^{w}) \\
\notag
&& + [ \frac{1}{r^{2}} | \Psi_{v\th} |^{2} + \frac{1}{r^{2}\sin^{2}\th} | \Psi_{v\phi} |^{2} ] (\frac{-2}{(1 - \mu)} \pa_{w}V^{v})\\
\notag
&& + [ \frac{1}{(1-\mu)^{2}} |\Psi_{vw}|^{2}   +  \frac{1}{4r^{4}\sin^{2}\th} | \Psi_{\phi\th}|^{2}  ] ( -2 [    \pa_{v} V^{v}  + \pa_{w} V^{w}   +      \frac{(3\mu-2)}{2r} ( V^{v}  -   V^{w} ) ]  ) 
\eea

\subsection{The vector field $\frac{\pa}{\pa t}$}\

Let,
\beaa
t^{\ga} = (\frac{\pa}{\pa t})^{\ga}
\eeaa
We have,

\bea
\notag
\frac{\pa}{\pa v} &=& \frac{\pa t}{\pa v} \frac{\pa }{\pa t} + \frac{\pa r^{*}}{\pa v} \frac{\pa }{\pa r^{*}} \\
&=& \frac{1}{2} \frac{\pa }{\pa t} + \frac{1}{2} \frac{\pa }{\pa r^{*}} \label{vtr*}\\
\notag
\frac{\pa}{\pa w} &=& \frac{\pa t}{\pa w} \frac{\pa }{\pa t} + \frac{\pa r^{*}}{\pa w} \frac{\pa }{\pa r^{*}} \\
&=& \frac{1}{2} \frac{\pa }{\pa t} - \frac{1}{2} \frac{\pa }{\pa r^{*}}  \label{wtr*}
\eea
Hence,
\beaa
\frac{\pa}{\pa t} &=& \frac{\pa}{\pa v} + \frac{\pa}{\pa w} \\
&=&  t^{v} \frac{\pa}{\pa v} + t^{w} \frac{\pa}{\pa w}
\eeaa

where $t^{v} = 1$ and $t^{w} = 1$. Thus,

\beaa
&& \pi^{\a\b}(\frac{\pa}{\pa t})T_{\a\b}(\Psi) \\
&=& [\frac{1}{r^{2}} | \Psi_{w\th} |^{2} + \frac{1}{r^{2}\sin^{2}\th} | \Psi_{w\phi} |^{2}]( \frac{-2}{(1 - \mu)} \pa_{v}t^{w}) \\
&& + [ \frac{1}{r^{2}} | \Psi_{v\th} |^{2} + \frac{1}{r^{2}\sin^{2}\th} | \Psi_{v\phi} |^{2} ] (\frac{-2}{(1 - \mu)} \pa_{w}t^{v})\\
&& + [ \frac{1}{(1-\mu)^{2}} |\Psi_{vw}|^{2}   +  \frac{1}{4r^{4}\sin^{2}\th} | \Psi_{\phi\th}|^{2}  ] ( -2 [    \pa_{v} t^{v}  + \pa_{w} t^{w}   +      \frac{(3\mu-2)}{2r} ( t^{v}  -   t^{w} ) ]  ) \\
&=& 0
\eeaa
In other words, since $\frac{\pa}{\pa t}$ is Killing, it's deformation tensor vanishes, i.e. $\pi^{\a\b}(\frac{\pa}{\pa t}) = 0$, therefore
\bea
 \pi^{\a\b}(\frac{\pa}{\pa t})T_{\a\b}(\Psi) = 0
\eea

Let,
\bea
\hat{\frac{\pa}{\pa t}}  &=&  \frac{1}{\sqrt{(1-\mu)}} \frac{\pa}{\pa t}\\
\hat{\frac{\pa}{\pa r^{*}}}  &=&  \frac{1}{\sqrt{(1-\mu)}}  \frac{\pa}{\pa r^{*}}
\eea

We also assume that,
\bea
\lim_{r \to \infty} \Psi_{\hat{\mu}\hat{\nu}} (t=t_{0})= 0 \label{limwhenrgoestoinfinityofthefieldcomponentsattequaltzero}
\eea
From a local existence result that ensures that a certain regularity will be conserved, one can prove that the condition above will be satisfied for all time, i.e.
\bea
\lim_{r \to \infty} \Psi_{\hat{\mu}\hat{\nu}} (t) = 0 \label{limwhenrgoestoinfinityofthefieldcomponentsatalltime}
\eea
 Thus, we will have no integrals on spatial infinity.\\

Applying the divergence theorem to the vector $ t^{\mu} T_{\mu\nu}$ in the region $B$ bounded by two hypersurfaces $\Sigma_{t_{1}}$ and $\Sigma_{t_{2}}$  defined by $t= constant$, where $t_{2} \geq t_{1}$, we get,

\beaa
&& \int_{B} \pi^{\a\b}(\frac{\pa}{\pa t})T_{\a\b}(\Psi)dV_{B}  +   \int_{B} t^{\nu} ( \der^{\mu}T_{\mu\nu}(\Psi) )   dV_{B} \\
&=& \int_{B} t^{\nu} ( \der^{\mu}T_{\mu\nu}(\Psi) )   dV_{B} \\
&=& \int_{\Sigma_{t_{1}}} J_{\mu}(\frac{\pa}{\pa t}) (\frac{\pa}{\pa \hat{t} })^{\mu} dV_{\Sigma_{t_{1}}} -  \int_{\Sigma_{t_{2}}} J_{\mu}(\frac{\pa}{\pa t}) (\frac{\pa}{\pa \hat{t} })^{\mu} dV_{\Sigma_{t_{2}}} \\
&=& \int_{\Sigma_{t_{1}}} J_{\hat{t}}(\frac{\pa}{\pa t})  dV_{\Sigma_{t_{1}}} -  \int_{\Sigma_{t_{2}}} J_{\hat{t}}(\frac{\pa}{\pa t})  dV_{\Sigma_{t_{2}}} \\
&=& \int_{\Sigma_{t_{1}}} T_{\hat{t} t}(\Psi)  dV_{\Sigma_{t_{1}}} -  \int_{\Sigma_{t_{2}}} T_{\hat{t} t}(\Psi)  dV_{\Sigma_{t_{2}}} \\
&=& \int_{\Sigma_{t_{1}}} T_{\hat{t} \hat{t}}(\Psi)  (\sqrt{1-\mu}) dV_{\Sigma_{t_{1}}} -  \int_{\Sigma_{t_{2}}} T_{\hat{t} \hat{t}}(\Psi)  (\sqrt{1-\mu}) dV_{\Sigma_{t_{2}}} \\
\eeaa

We have,
\beaa
dV_{\Sigma_{t}} = (\sqrt{(1-\mu) r^{4}\sin^{2}(\th) } d{r^{*}} d\th d\phi  = (\sqrt{1-\mu}) r^{2} d{r^{*}} \sin(\th) d\th d\phi 
\eeaa
(we have $\sin(\th) \geq 0$ because $ 0 \leq \th \leq \pi$). Thus,
\beaa
dV_{\Sigma_{t}}   = (\sqrt{1-\mu}) r^{2} d{r^{*}} d\si^{2} 
\eeaa

On the other hand,
\beaa
&& T_{\hat{t}\hat{t}}^{}(\Psi) \\
&=& <\Psi_{\hat{t}\b}, {\Psi_{\hat{t}}}^{\b}> - \frac{1}{4} \g_{\hat{t}\hat{t}} < \Psi_{\a\b}, \Psi^{\a\b}> \\
&=&   <\Psi_{\hat{t}\b}, {\Psi_{\hat{t}}}^{\b}> + \frac{1}{4} [   \frac{-8}{r^{2} (1-\mu)} < \Psi_{w\th}, {\Psi_{v\th}}^{}>   +   \frac{-8}{r^{2} \sin^{2} \th (1-\mu)} < \Psi_{w\phi}, {\Psi_{v\phi}}^{}>  \\
&& +  \frac{-8}{(1-\mu)^{2}}  | \Psi_{vw}|^{2}  +  \frac{2}{ r^{4}\sin^{2}\th} | \Psi_{\phi\th}|^{2} ] \\
&=&  \g^{\hat{\b}\hat{\ga}} <\Psi_{\hat{t}\b}, \Psi_{\hat{t} \hat{\ga}}> + [   -2 < \Psi_{\hat{w}\hat{\th}}, {\Psi_{\hat{v}\hat{\th}} }^{}>   -   2 < \Psi_{\hat{w}\hat{\phi}}, {\Psi_{\hat{v}\hat{\phi}}}^{}>   -  2  | \Psi_{\hat{v}\hat{w}}|^{2}  +  \frac{1}{2 } | \Psi_{\hat{\phi}\hat{\th}}|^{2} ] \\
\eeaa

\beaa
 <\Psi_{\hat{t}\b}, {\Psi_{\hat{t}}}^{\b}>   &=& -  <\Psi_{\hat{t}\hat{t} }, \Psi_{\hat{t}\hat{t} }> + <\Psi_{\hat{t}\hat{r^{*}} }, \Psi_{\hat{t}\hat{r^{*}} }> + <\Psi_{\hat{t}\hat{\th} }, \Psi_{\hat{t}\hat{\th} }> + <\Psi_{\hat{t}\hat{\phi} }, \Psi_{\hat{t}\hat{\phi} }> \\
&=& |\Psi_{\hat{t}\hat{r^{*}} }|^{2}  + |\Psi_{\hat{t}\hat{\th} }|^{2}  + |\Psi_{\hat{t}\hat{\phi} }|^{2} 
\eeaa

From \eqref{vtr*} and \eqref{wtr*}, we get,

\bea
\notag
\Psi_{v\th} &=& \Psi_{\mu\nu}(\frac{\pa}{\pa v})^{\mu} (\frac{\pa}{\pa \th})^{\nu} = \Psi_{\mu\nu}(\frac{1}{2} \frac{\pa }{\pa t} + \frac{1}{2} \frac{\pa }{\pa r^{*}} )^{\mu} ( \frac{\pa }{\pa \th} )^{\nu} \\
\notag
&=& \frac{1}{2} \Psi_{\mu\nu}( \frac{\pa }{\pa t}  )^{\mu} (   \frac{\pa }{\pa \th} )^{\nu} + \frac{1}{2} \Psi_{\mu\nu}(  \frac{\pa }{\pa r^{*}} )^{\mu} (  \frac{\pa }{\pa \th } )^{\nu}  \\
&=& \frac{1}{2} \Psi_{t\th} + \frac{1}{2} \Psi_{r^{*} \th}  \label{v th}
\eea

\bea
\notag
\Psi_{w\th} &=& \Psi_{\mu\nu}(\frac{\pa}{\pa w})^{\mu} (\frac{\pa}{\pa \th})^{\nu} = \Psi_{\mu\nu}(\frac{1}{2} \frac{\pa }{\pa t} - \frac{1}{2} \frac{\pa }{\pa r^{*}} )^{\mu} ( \frac{\pa }{\pa \th} )^{\nu} \\
\notag
&=& \frac{1}{2} \Psi_{\mu\nu}( \frac{\pa }{\pa t}  )^{\mu} (   \frac{\pa }{\pa \th} )^{\nu} - \frac{1}{2} \Psi_{\mu\nu}(  \frac{\pa }{\pa r^{*}} )^{\mu} (  \frac{\pa }{\pa \th } )^{\nu}  \\
&=& \frac{1}{2} \Psi_{t\th} - \frac{1}{2} \Psi_{r^{*} \th}  \label{w th}
\eea

and similarly,

\bea
\Psi_{v\phi} &=& \frac{1}{2} \Psi_{t\phi} + \frac{1}{2} \Psi_{r^{*} \phi}  \label{v phi}
\eea

\bea
\Psi_{w\phi} &=& \frac{1}{2} \Psi_{t\phi} - \frac{1}{2} \Psi_{r^{*} \phi}  \label{w phi}
\eea

Thus,

\beaa
 < \Psi_{\hat{w}\hat{\th} }, \Psi_{\hat{v}\hat{\th} } > &=&  < \frac{\Psi_{\hat{t}\hat{\th} } - \Psi_{\hat{r^{*}}\hat{\th} }}{2}, \frac{\Psi_{\hat{t}\hat{\th} }  + \Psi_{\hat{r^{*}}\hat{\th} } }{2} > \\
&=&  \frac{1}{4} [ |\Psi_{\hat{t}\hat{\th} }|^{2} - | \Psi_{\hat{r^{*}}\hat{\th} }|^{2} ]
\eeaa
and,
\beaa
 < \Psi_{\hat{w}\hat{\phi} }, \Psi_{\hat{v}\hat{\phi} } > &=&  \frac{1}{4} [ |\Psi_{\hat{t}\hat{\phi} }|^{2} - | \Psi_{\hat{r^{*}}\hat{\phi} }|^{2} ]
\eeaa

We have,

\beaa
&& \Psi_{vw} \\
&=& \Psi_{\mu\nu}(\frac{\pa}{\pa v})^{\mu} (\frac{\pa}{\pa w})^{\nu} = \Psi_{\mu\nu}(\frac{1}{2} \frac{\pa }{\pa t} + \frac{1}{2} \frac{\pa }{\pa r^{*}} )^{\mu} (\frac{1}{2} \frac{\pa }{\pa t} - \frac{1}{2} \frac{\pa }{\pa r^{*}} )^{\nu} \\
&=& \frac{1}{4} \Psi_{\mu\nu}( \frac{\pa }{\pa t}  )^{\mu} ( \frac{\pa }{\pa t} -  \frac{\pa }{\pa r^{*}} )^{\nu} + \frac{1}{4} \Psi_{\mu\nu}(  \frac{\pa }{\pa r^{*}} )^{\mu} ( \frac{\pa }{\pa t} -  \frac{\pa }{\pa r^{*}} )^{\nu}  \\
&=& \frac{1}{4} \Psi_{\mu\nu}( \frac{\pa }{\pa t}  )^{\mu} ( \frac{\pa }{\pa t}  )^{\nu} - \frac{1}{4} \Psi_{\mu\nu}(  \frac{\pa }{\pa t} )^{\mu} (   \frac{\pa }{\pa r^{*}} )^{\nu}  + \frac{1}{4} \Psi_{\mu\nu}( \frac{\pa }{\pa r^{*}}  )^{\mu} ( \frac{\pa }{\pa t} )^{\nu} - \frac{1}{4} \Psi_{\mu\nu}(  \frac{\pa }{\pa r^{*}} )^{\mu} (   \frac{\pa }{\pa r^{*}} )^{\nu}  \\
\eeaa

Since $\Psi_{\mu\nu}$ is anti-symmetric two tensor, we get,

\bea
\notag
\Psi_{vw} &=&   \frac{1}{2} \Psi_{\mu\nu}( \frac{\pa }{\pa r^{*}}  )^{\mu} ( \frac{\pa }{\pa t} )^{\nu}   \\
&=& \frac{1}{2} \Psi_{r^{*} t} \label{v w}
\eea

Therefore,

\beaa
&& T_{\hat{t}\hat{t}}^{}(\Psi) \\
&=& |\Psi_{\hat{t}\hat{r^{*}} }|^{2}  + |\Psi_{\hat{t}\hat{\th} }|^{2}  + |\Psi_{\hat{t}\hat{\phi} }|^{2} +  \frac{1}{2} [  | \Psi_{\hat{r^{*}}\hat{\th} }|^{2}  - |\Psi_{\hat{t}\hat{\th} }|^{2} ] +  \frac{1}{2} [ | \Psi_{\hat{r^{*}}\hat{\phi} } - |\Psi_{\hat{t}\hat{\phi} }|^{2} |^{2} ] - \frac{1}{2} | \Psi_{\hat{r^{*}} \hat{t} } |^{2} +  \frac{1}{2 } | \Psi_{\hat{\phi}\hat{\th}}|^{2} \\
&=& \frac{1}{2} [ |\Psi_{\hat{t}\hat{r^{*}} }|^{2}  + |\Psi_{\hat{t}\hat{\th} }|^{2}  + |\Psi_{\hat{t}\hat{\phi} }|^{2} + |\Psi_{\hat{r^{*}}\hat{\th}}|^{2} + | \Psi_{\hat{r^{*}}\hat{\phi}}|^{2} + | \Psi_{\hat{\phi}\hat{\th}}|^{2} ]
\eeaa

Thus,

\bea
\notag
&& \int_{B} T^{\nu} ( \der^{\mu}T_{\mu\nu}(\Psi) )   dV_{B}  \\
\notag
&=& \int_{\Sigma_{t_{1}}} T_{\hat{t} \hat{t}}(T)  (\sqrt{1-\mu}) dV_{\Sigma_{t_{1}}} -  \int_{\Sigma_{t_{2}}} J_{\hat{t} \hat{t}}(T)  (\sqrt{1-\mu}) dV_{\Sigma_{t_{2}}} \\
\notag
&=& \int_{\Sigma_{t_{1}}}  \frac{1}{2} [ |\Psi_{\hat{t}\hat{r^{*}} }|^{2}  + |\Psi_{\hat{t}\hat{\th} }|^{2}  + |\Psi_{\hat{t}\hat{\phi} }|^{2} + |\Psi_{\hat{r^{*}}\hat{\th}}|^{2} + | \Psi_{\hat{r^{*}}\hat{\phi}}|^{2} + | \Psi_{\hat{\phi}\hat{\th}}|^{2} ] (1-\mu) r^{2} d{r^{*}} d\si^{2} \\
\notag
&&- \int_{\Sigma_{t_{2}}}  \frac{1}{2} [ |\Psi_{\hat{t}\hat{r^{*}} }|^{2}  + |\Psi_{\hat{t}\hat{\th} }|^{2}  + |\Psi_{\hat{t}\hat{\phi} }|^{2} + |\Psi_{\hat{r^{*}}\hat{\th}}|^{2} + | \Psi_{\hat{r^{*}}\hat{\phi}}|^{2} + | \Psi_{\hat{\phi}\hat{\th}}|^{2} ] (1-\mu) r^{2} d{r^{*}} d\si^{2} \\
&=& E_{\Psi}^{(\frac{\pa}{\pa t})} (t=t_{2}) - E_{\Psi}^{(\frac{\pa}{\pa t})} (t=t_{1})
\eea

where,
\bea
&& E_{\Psi}^{(\frac{\pa}{\pa t})} (t=t_{i}) \\
\notag
&=& \int_{\Sigma_{t_{i}}}  - \frac{1}{2} [ |\Psi_{\hat{t}\hat{r^{*}} }|^{2}  + |\Psi_{\hat{t}\hat{\th} }|^{2}  + |\Psi_{\hat{t}\hat{\phi} }|^{2} + |\Psi_{\hat{r^{*}}\hat{\th}}|^{2} + | \Psi_{\hat{r^{*}}\hat{\phi}}|^{2} + | \Psi_{\hat{\phi}\hat{\th}}|^{2} ] (1-\mu) r^{2} d{r^{*}} d\si^{2}
\eea

Taking $\Psi$ any product of $\Lie_{t}$, $\Lie_{\Om_{i}}$, and $F$, where $\Om_{i}$, $i \in \{1, 2, 3\}$, is a basis of angular momentum operators, since $
\int_{B} T^{\nu} ( \der^{\mu}T_{\mu\nu}(\Psi) )   dV_{B}  = 0 $, we have conservation of the energy generated from the vector field $\frac{\pa}{\pa t}$.\\

\subsection{The vector field $K$}\

Let
\bea
K &=& - \om^{2} \frac{\pa}{\pa \om} - v^{2} \frac{\pa}{\pa v}\\
\notag
&=& K^{\om}  \frac{\pa}{\pa \om} +    K^{v}  \frac{\pa}{\pa v}
\eea
We have,
\beaa
\pa_{v} K^{\om} &=& - \pa_{v} \om^{2}\\
&=& 0\\
\pa_{\om} K^{v} &=& - \pa_{\om} v^{2}\\
&=& 0
\eeaa
Computing,

\beaa
&& \pi^{\a\b}(K)T_{\a\b}(\Psi)\\
 &=& [\frac{1}{r^{2}} | \Psi_{w\th} |^{2} + \frac{1}{r^{2}\sin^{2}\th} | \Psi_{w\phi} |^{2}]( \frac{-2}{(1 - \mu)} \pa_{v}K^{w}) \\
&& + [ \frac{1}{r^{2}} | \Psi_{v\th} |^{2} + \frac{1}{r^{2}\sin^{2}\th} | \Psi_{v\phi} |^{2} ] (\frac{-2}{(1 - \mu)} \pa_{w}K^{v})\\
&& + [ \frac{1}{(1-\mu)^{2}} |\Psi_{vw}|^{2}   +  \frac{1}{4r^{4}\sin^{2}\th} | \Psi_{\phi\th}|^{2}  ] ( -2 [    \pa_{v} K^{v}  + \pa_{w} K^{w}   +      \frac{(3\mu-2)}{2r} ( K^{v}  -   K^{w} ) ]  ) \\
&=&   [ \frac{1}{(1-\mu)^{2}} |\Psi_{vw}|^{2}   +  \frac{1}{4r^{4}\sin^{2}\th} | \Psi_{\phi\th}|^{2}  ]  ( -2 [  -2v -2w +      \frac{(3\mu-2)}{2r}  ( -v^{2}  +w^{2} ) ]  ) \\
&=&  [ \frac{1}{(1-\mu)^{2}} |\Psi_{vw}|^{2}   +  \frac{1}{4r^{4}\sin^{2}\th} | \Psi_{\phi\th}|^{2}  ]    [  4(v + w) +     \frac{(3\mu-2)}{r} ( w^{2}  - v^{2} ) ]  
\eeaa
Thus,
\beaa
\pi^{\a\b}(K)T_{\a\b}(\Psi) &=&    (v+w) [ 4 + \frac{ ( 3\mu - 2 )}{r} ( v - w ) ] .[  \frac{1 }{(1-\mu)^{2}}  | \Psi_{vw}|^{2}  +  \frac{1}{4r^{4}\sin^{2}\th } | \Psi_{\phi\th}|^{2} ]
\eeaa
Recall that $v$ and $w$ are defined as in \eqref{v} and \eqref{w}, thus,
\beaa
 v + \om &=&  2 t \\
v - \om &=& 2 r^{*}
\eeaa
Therefore, we also have,
\bea
\notag
\pi^{\a\b}(K)T_{\a\b}(\Psi) &=&   4t  [ 2 + \frac{ ( 3\mu - 2 )r^{*}}{r} ] .[  \frac{1 }{(1-\mu)^{2}}  | \Psi_{vw}|^{2}  +  \frac{1}{4r^{4}\sin^{2}\th } | \Psi_{\phi\th}|^{2} ]  \\ \label{zerocomponontsconformalenergy}
\eea
We define,
\bea
J_{\Psi}^{(K)} ( t_{i} \leq t \leq t_{i+1} )  = \int_{t = t_{i} }^{ t = t_{i+1} }  \int_{r^{*} = - \infty}^{r^{*}= \infty} \int_{\S^{2}} \pi^{\a\b}(K)T_{\a\b}(\Psi) dVol
\eea

Computing,
\bea
&& E^{(K)}(t_{i}) =  \int_{r^{*} = - \infty}^{r^{*}= \infty} \int_{\S^{2}} J_{\a}(K) n^{\a} dVol_{t=t_{i} } (t=t_{i} )
\eea
where $$n^{\a} = - \frac {\frac{\pa}{\pa t} }{\sqrt{(1-\mu)}}$$ and $$dVol_{t=t_{i}} = r^{2} \sqrt{(1-\mu)} d\sigma^{2} dr^{*}$$

\bea
\notag
&&E_{\Psi}^{(K)}(t_{i}) \\
\notag
&=& \int_{r^{*} = - \infty}^{r^{*}= \infty} \int_{\S^{2}} -  \frac {1}{\sqrt{(1-\mu)}} [ (\frac{\pa}{\pa v})^{\a}  + ({\frac{\pa}{\pa \om}})^{\a} ] J_{\a}(K)  r^{2} \sqrt{(1-\mu)} d\sigma^{2} dr^{*}  \\
\notag
&=& \int_{r^{*} = - \infty}^{r^{*}= \infty} \int_{\S^{2}}  -\frac{1}{\sqrt{(1-\mu)}} [ - v^{2} T_{vv} -\om^{2}T_{v\om}   - v^{2}  T_{\om v} - \om^{2} T_{\om\om} ]  r^{2} \sqrt{(1-\mu)} d\sigma^{2} dr^{*} \\
\notag
&=& \int_{r^{*} = - \infty}^{r^{*}= \infty} \int_{\S^{2}}   (   w^{2} [\frac{1}{r^{2}(1-\mu)} | \Psi_{w\th} |^{2} + \frac{1}{r^{2}\sin^{2}\th(1-\mu)} | \Psi_{w\phi} |^{2} ]  \\
\notag
&& +  v^{2} [ \frac{1}{r^{2}(1-\mu)} | \Psi_{v\th} |^{2} + \frac{1}{r^{2}\sin^{2}\th(1-\mu)} | \Psi_{v\phi} |^{2} ]  \\
\notag
&&+  (\om^{2} + v^{2} ) [ \frac{1}{(1-\mu)^{2} } |\Psi_{vw}|^{2}   +  \frac{1}{4r^{4}\sin^{2}\th} | \Psi_{\phi\th}|^{2}]  )   r^{2} (1-\mu) d\sigma^{2} dr^{*} \\
\eea

\subsection{The vector field $G$}\

Let
\bea
G &=& -f(r^{*}) \frac{\pa}{\pa \om} + f(r^{*}) \frac{\pa}{\pa v}\\
\notag
&=& G^{\om} \frac{\pa}{\pa \om} +     G^{v} \frac{\pa}{\pa v}
\eea
where $f(r^{*})$ depends only on $r^{*}$.\\ 
Computing,
\beaa
\frac{\pa}{\pa \om} G^{\om} &=& \frac{\pa r^{*}}{\pa \om} \frac{\pa}{\pa r^{*}} G^{\om}   +    \frac{\pa t}{\pa \om} \frac{\pa}{\pa t} G^{\om} \\
&=& -\frac{1}{2} \frac{\pa}{\pa r^{*}} G^{\om}   + 0 \\
&=& \frac{1}{2} f^{'}
\eeaa
where $f^{'} = \frac{\pa}{\pa r^{*}} f$\\

Similarly,
\beaa
\frac{\pa}{\pa v} G^{v} &=& \frac{\pa r^{*}}{\pa v} \frac{\pa}{\pa r^{*}} G^{v}   +    \frac{\pa t}{\pa v} \frac{\pa}{\pa t} G^{v} \\
&=& \frac{1}{2} \frac{\pa}{\pa r^{*}} G^{v}   + 0 \\
&=& \frac{1}{2} f^{'} \\
\frac{\pa}{\pa v} G^{\om} &=& \frac{\pa r^{*}}{\pa v} \frac{\pa}{\pa r^{*}} G^{\om}   +    \frac{\pa t}{\pa v} \frac{\pa}{\pa t} G^{\om} \\
&=& \frac{1}{2} \frac{\pa}{\pa r^{*}} G^{\om}   + 0 \\
&=&  - \frac{1}{2} f^{'} \\
\frac{\pa}{\pa \om} G^{v} &=& \frac{\pa r^{*}}{\pa \om} \frac{\pa}{\pa r^{*}} G^{v}   +    \frac{\pa t}{\pa \om} \frac{\pa}{\pa t} G^{v} \\
&=& - \frac{1}{2} \frac{\pa}{\pa r^{*}} G^{v}   + 0 \\
&=& - \frac{1}{2} f^{'}
\eeaa

\beaa
&& \pi^{\a\b}(G)T_{\a\b}(\Psi) \\
&=& [\frac{1}{r^{2}} | \Psi_{w\th} |^{2} + \frac{1}{r^{2}\sin^{2}\th} | \Psi_{w\phi} |^{2}]( \frac{-2}{(1 - \mu)} \pa_{v}G^{w}) \\
&& + [ \frac{1}{r^{2}} | \Psi_{v\th} |^{2} + \frac{1}{r^{2}\sin^{2}\th} | \Psi_{v\phi} |^{2} ] (\frac{-2}{(1 - \mu)} \pa_{w}G^{v})\\
&& + [ \frac{1}{(1-\mu)^{2}} |\Psi_{vw}|^{2}   +  \frac{1}{4r^{4}\sin^{2}\th} | \Psi_{\phi\th}|^{2}  ] ( -2 [    \pa_{v} G^{v}  + \pa_{w} G^{w}   +      \frac{(3\mu-2)}{2r} ( G^{v}  -   G^{w} ) ]  ) \\
&=& [\frac{1}{r^{2}} | \Psi_{w\th} |^{2} + \frac{1}{r^{2}\sin^{2}\th} | \Psi_{w\phi} |^{2}]( \frac{-2}{(1 - \mu)} (-\frac{1}{2} f') ) \\
&& + [ \frac{1}{r^{2}} | \Psi_{v\th} |^{2} + \frac{1}{r^{2}\sin^{2}\th} | \Psi_{v\phi} |^{2} ] (\frac{-2}{(1 - \mu)} (-\frac{1}{2} f') )\\
&& +[ \frac{1}{(1-\mu)^{2}} |\Psi_{vw}|^{2}   +  \frac{1}{4r^{4}\sin^{2}\th} | \Psi_{\phi\th}|^{2}  ] ( -2 [  \frac{1}{2} f'  +  \frac{1}{2} f'    +       \frac{(3\mu-2)}{2r} ( f + f)  ]  ) \\
 &=& [\frac{1}{r^{2}} | \Psi_{w\th} |^{2} + \frac{1}{r^{2}\sin^{2}\th} | \Psi_{w\phi} |^{2}]( \frac{f'}{(1 - \mu)} )\\
&& + [ \frac{1}{r^{2}} | \Psi_{v\th} |^{2} + \frac{1}{r^{2}\sin^{2}\th} | \Psi_{v\phi} |^{2} ] (\frac{f'}{(1 - \mu)}  )\\
&& + [  \frac{1 }{(1-\mu)^{2}}  | \Psi_{vw}|^{2}  +  \frac{1}{4r^{4}\sin^{2}\th } | \Psi_{\phi\th}|^{2} ] . ( -2 [  f' +   \frac{(3\mu-2)}{r} f  ] )
\eeaa
Finally, we obtain,
\bea
\notag
&& T^{\a\b}(\Psi_{\mu\nu})\pi_{\a\b}(G) \\
\notag
 &=&  [\frac{1}{r^{2}} | \Psi_{w\th} |^{2} + \frac{1}{r^{2}\sin^{2}\th} | \Psi_{w\phi} |^{2} + \frac{1}{r^{2}} | \Psi_{v\th} |^{2} + \frac{1}{r^{2}\sin^{2}\th} | \Psi_{v\phi} |^{2} ] \frac{f'}{(1 - \mu)}  \\
&&-2 [  \frac{1 }{(1-\mu)^{2}}  | \Psi_{vw}|^{2}  +  \frac{1}{4r^{4}\sin^{2}\th } | \Psi_{\phi\th}|^{2} ] (  f' +  \frac{f}{r}(3\mu -2)  )  \label{contracteddeformationforG}
\eea

Computing,

\bea
&& E^{(G)}_{\Psi} (t_{i}) =  \int_{r^{*} = - \infty}^{r^{*}= \infty} \int_{\S^{2}} J_{\a}(G)(\Psi_{\mu\nu}) n^{\a} dVol_{t=t_{i} } (t=t_{i} )
\eea
where $n^{\a} = - \frac {\frac{\pa}{\pa t} }{\sqrt{(1-\mu)}}$ and $dVol_{t=t_{i}} = r^{2} \sqrt{(1-\mu)} d\sigma^{2} dr^{*}$. Thus,
\beaa
E^{(G)}_{\Psi}(t_{i}) = \int_{r^{*} = - \infty}^{r^{*}= \infty} \int_{\S^{2}} -  \frac {1}{\sqrt{(1-\mu)}}  (\frac{\pa}{\pa t})^{\a}     J_{\a}(G)(\Psi_{\mu\nu})  r^{2} \sqrt{(1-\mu)} d\sigma^{2} dr^{*}  
\eeaa

Recall \eqref{vtr*} and \eqref{wtr*}, thus,
\bea
G = f(r^{*}) \frac{\pa}{\pa r^{*}}  \label{thevectorfieldGintermsoffanddrstar}
\eea

Therefore,

\beaa
E^{(G)}_{\Psi} (t_{i}) &=& \int_{r^{*} = - \infty}^{r^{*}= \infty} \int_{\S^{2}} - f    T_{t r^{*} } (\Psi_{\mu\nu})  r^{2}  d\sigma^{2} dr^{*}  \\
\eeaa
\beaa
 T_{t r^{*} } &=& <\Psi_{t\a}, {\Psi_{r^{*}}}^{\a}>\\
&=& <\Psi_{t r^{*}}, {\Psi_{r^{*}}}^{r^{*}}> +  <\Psi_{t\th}, {\Psi_{r^{*}}}^{\th}>+  <\Psi_{t\phi}, {\Psi_{r^{*}}}^{\phi}> \\
&=& \frac{1}{r^{2}} <\Psi_{t\th}, {\Psi_{r^{*}\th}}^{}> +  \frac{1}{r^{2} \sin^{2} \th}  <\Psi_{t\phi}, {\Psi_{r^{*}\phi}}^{}> 
\eeaa

Thus,
\bea
&&  E^{(G)}_{\Psi} (t_{i}) \label{definitionoftheboundarytermgeneratedfromG} \\
\notag
&=& \int_{r^{*} = - \infty}^{r^{*}= \infty} \int_{\S^{2}} - f   [    \frac{1}{r^{2}} <\Psi_{t\th}, {\Psi_{r^{*}\th}}^{}> +  \frac{1}{r^{2} \sin^{2} \th}  <\Psi_{t\phi}, {\Psi_{r^{*}\phi}}^{}>          ]   r^{2}  d\sigma^{2} dr^{*}  \\
\notag
\eea

\section{Bounding the Conformal Energy on $t= constant$ Hypersurfaces}

Let $\Psi$ be any product of $\Lie_{t}$, $\Lie_{\Om_{i}}$, and $F$, where $\Om_{i}$, $i \in \{1, 2, 3\}$, is a basis of angular momentum operators. We know by then that we have \eqref{theenrgymomuntumtensorofpsiisdivergencefree}.\\

\subsection{Estimate for $E_{\Psi}^{(G)}$}\

Let $f$ in \eqref{thevectorfieldGintermsoffanddrstar} be a bounded function of $r^{*}$. Then, we have
\bea
 | E_{\Psi}^{(G)} (t=t_{i})  | \lesssim | \hat{E}_{\Psi}^{(\frac{\pa}{\pa t})} (t=t_{i}) | \lesssim | E_{\Psi}^{(\frac{\pa}{\pa t})} (t=t_{i})|  \label{controlenergyforG}
\eea

where,

\bea
\notag
\hat{E}^{(\frac{\pa}{\pa t})}_{\Psi} (t) &=&  \int_{r^{*} = - \infty}^{r^{*} =  \infty} \int_{\S^{2}}   [    \frac{1}{r^{2} (1-\mu) }  |\Psi_{t\th}|^{2} +  \frac{1}{r^{2}(1-\mu) \sin^{2} \th}   |\Psi_{t\phi}|^{2} +  \frac{1}{r^{2}(1-\mu) } |{\Psi_{r^{*}\th}}^{} |^{2}   \\
&& +  \frac{1}{r^{2} (1-\mu) \sin^{2} \th}  | {\Psi_{r^{*}\phi}}^{}|^{2}          ]   r^{2} (1-\mu) d\sigma^{2} dr^{*} 
\eea

\begin{proof}\

We have,

\beaa
&& |  E^{(G)}_{\Psi} (t_{i}) | \\
&=& | \int_{r^{*}= -\infty}^{r^{*} = \infty} \int_{\S^{2}}   - f [    \frac{1}{r^{2}} <\Psi_{t\th}, {\Psi_{r^{*}\th}}^{}> +  \frac{1}{r^{2} \sin^{2} \th}  <\Psi_{t\phi}, {\Psi_{r^{*}\phi}}^{}>          ]   r^{2}  d\sigma^{2} dr^{*}  |\\
& \lesssim & \int_{r^{*} = - \infty}^{r^{*}= \infty} \int_{\S^{2}}   | f| [    \frac{1}{r^{2}}  |\Psi_{t\th}|^{2} +  \frac{1}{r^{2}} |{\Psi_{r^{*}\th}}^{} |^{2}  +  \frac{1}{r^{2} \sin^{2} \th}   |\Psi_{t\phi}|^{2} +  \frac{1}{r^{2} \sin^{2} \th}  | {\Psi_{r^{*}\phi}}^{}|^{2}          ]   r^{2}  d\sigma^{2} dr^{*} \\
&& \text{(by using $a.b \les a^{2} + b^{2}$)} \\
& \lesssim & \int_{r^{*} = - \infty}^{r^{*}= \infty} \int_{\S^{2}}   | f| [    \frac{1}{r^{2} (1-\mu) }  |\Psi_{t\th}|^{2} +  \frac{1}{r^{2}(1-\mu) } |{\Psi_{r^{*}\th}}^{} |^{2}  +  \frac{1}{r^{2}(1-\mu) \sin^{2} \th}   |\Psi_{t\phi}|^{2} \\
&& +  \frac{1}{r^{2} (1-\mu) \sin^{2} \th}  | {\Psi_{r^{*}\phi}}^{}|^{2}          ]   r^{2} (1-\mu) d\sigma^{2} dr^{*} \\
& \lesssim & | \hat{E}^{(\frac{\pa}{\pa t})}_{\Psi} (t=t_{i}) |
\eeaa
(because $f$ is bounded).\\

Then, we have,

\beaa
&& E^{(\frac{\pa}{\pa t})}_{\Psi} (t=t_{i}) \\
&& =  - \int_{r^{*}= -\infty}^{r^{*} = \infty} \int_{\S^{2}}  (\frac{1}{2}   | \Psi_{\hat{t} \hat{r^{*}} }|^{2} + \frac{1}{2} |\Psi_{\hat{t} \hat{\th}}|^{2} + \frac{1}{2} | \Psi_{\hat{t} \hat{\phi} }|^{2} + \frac{1}{2} |\Psi_{\hat{r^{*}} \hat{\th} }|^{2} +  \frac{1}{2} |\Psi_{ \hat{r^{*}} \hat{\phi} }|^{2} + \frac{1}{2} | \Psi_{ \hat{\th} \hat{\phi} }|^{2}   ) .r^{2} (1-\mu) d\sigma^{2} dr^{*}  \\
\eeaa
and therefore,
\beaa
| \hat{E}^{(\frac{\pa}{\pa t})}_{\Psi} (t=t_{i}) | &=&  \int_{r^{*} = - \infty}^{r^{*} =  \infty} \int_{\S^{2}}   [   |\Psi_{\hat{t}\hat{\th}}|^{2} +     |\Psi_{\hat{t}\hat{\phi}}|^{2} +   |{\Psi_{\hat{r^{*}}\hat{\th}}}^{} |^{2}    +    | {\Psi_{\hat{r^{*}}\hat{\phi}}}^{}|^{2}          ]   r^{2} (1-\mu) d\sigma^{2} dr^{*} \\ 
&\les& | E^{(\frac{\pa}{\pa t})}_{\Psi} (t=t_{i}) |
\eeaa

\end{proof}

\subsection{Controlling $J_{\Psi}^{(K)}$ in terms of $J_{\Psi}^{(G)}$}\

\bea
J_{\Psi}^{(K)}( t_{i} \leq t \leq t_{i+1})    &\lesssim&  t_{i+1}     J_{\Psi}^{(G)} ( t_i \leq t \leq t_{i+1} ) (r_{0} \leq r \leq  R_{0} )   \label{KcontolledbytG}
\eea

where $r_{0} \leq 3m \leq R_{0}$. And we have,

\bea
 E_{\Psi}^{(K)} (t=t_{i+1})  &\leq& J_{\Psi}^{(K)}( t_{i} \leq t \leq t_{i+1})  + E_{\Psi}^{(K)} (t=t_{i}) \label{EKcontolledbyJK}
\eea

\begin{proof}\

We have,

\beaa
&& J_{\Psi}^{(K)}( t_{i} \leq t \leq t_{i+1}) \\
&=&   \int_{t=t_{i}}^{t_{i+1}} \int_{r^{*}=-\infty}^{r^{*}=\infty} \int_{\S^{2}}  4t  [ 2 + \frac{ ( 3\mu - 2 )r^{*}}{r} ] .[  \frac{1 }{(1-\mu)^{2}}  | \Psi_{vw}|^{2}  +  \frac{1}{4r^{4}\sin^{2}\th } | \Psi_{\phi\th}|^{2} ]  r^{2}(1-\mu)d\sigma^{2}dr^{*}dt \\
\eeaa

We remark that $[ 2 + \frac{ ( 3\mu - 2 )r^{*}}{r} ] $ is positive only in a bounded interval $ [r_{0}, R_{0} ] $ where $r_{0} > 2m $ and $ R_{0} > 3m $. To see this, notice that,

$$ \lim_{r^{*} \to - \infty } [ 2 + \frac{ ( 3\mu - 2 )r^{*}}{r} ] = - \infty $$

and,

$$ \lim_{r^{*} \to  \infty } [ 2 + \frac{ ( 3\mu - 2 )r^{*}}{r} ] = 2 - 2.(1) =  0 $$

and $$ \lim_{r \to 3m  } [ 2 + \frac{ ( 3\mu - 2 )r^{*}}{r} ] = 2 + 0 = 2 > 0 $$

More precisely, let's look for the region where $J_{\Psi}^{(K)}( t_{i} \leq t \leq t_{i+1})$ is negative: 

\beaa
 2 +    \frac{r^{*}}{r}(-2 +3\mu)  \leq 0
\eeaa
when 
\beaa
r^{*} (3\mu-2)  \leq  - 2 r
\eeaa
Choosing $r_{0}$ small enough such that $(3\mu_{0}-2) \geq 0$, then we need $r_{0}$ such that for $r \leq r_{0}$ ,
\beaa
r^{*}   \leq  - \frac{2r}{(3\mu-2)} 
\eeaa
so choose $r_{0}$ such that
\beaa
r_{0}^{*}   \leq  - \frac{2 r_{0}}{(3\mu_{0} -2)} < 0
\eeaa
or choose $\hat{R_{0}}$ large such that $(3\mu(\hat{R_{0}})-2) \leq 0$, and such that for $r \geq \hat{R_{0}}$, we get,
\beaa
r^{*}   \geq  - \frac{2r}{(3\mu -2)}  > 0
\eeaa
then choose $\hat{R}$ such that 
\beaa
\hat{R_{0}}^{*}   \geq  - \frac{2\hat{R_{0}}}{(3\mu(\hat{R_{0}}) -2)}  > 0
\eeaa

In conclusion choose $r_{0}$ such that $$ r_{0} <   - \frac{2r_{0}}{(3\mu_{0} -2)}  < 0$$

and choose $R_{0}$ as the infimum of all $\hat{R_{0}}$ such that
\bea
\hat{R_{0}}^{*}   \geq  - \frac{2\hat{R_{0}}}{(3\mu(\hat{R_{0}}) -2)}  > 0 \label{definitionofRoastheinfimum}
\eea

Then in the region $ r \leq r_{0}$ or $r \geq R_{0}$ we know that the integrand in $J_{F}^{(K)}( t_{i} \leq t \leq t_{i+1}) $ is negative.\\

Thus,

\beaa
&& J_{\Psi}^{(K)}( t_{i} \leq t \leq t_{i+1}) \\
&\leq&  \int_{t=t_{i}}^{t_{i+1}} \int_{r^{*}= r^{*}_{0} }^{r^{*}= R_{0}^{*} } \int_{\S^{2}} 4t  [ 2 + \frac{ ( 3\mu - 2 )r^{*}}{r} ] .[  \frac{1 }{(1-\mu)^{2}}  | \Psi_{vw}|^{2}  +  \frac{1}{4r^{4}\sin^{2}\th } | \Psi_{\phi\th}|^{2} ]  r^{2}(1-\mu)d\sigma^{2}dr^{*}dt \\
&\leq&  t_{i+1} \int_{t=t_{i}}^{t_{i+1}} \int_{r^{*}= r^{*}_{0} }^{ r^{*} = R_{0}^{*} } \int_{\S^{2}} 4  [ 2 + \frac{ ( 3\mu - 2 )r^{*}}{r} ] .[  \frac{1 }{(1-\mu)^{2}}  | \Psi_{vw}|^{2}  +  \frac{1}{4r^{4}\sin^{2}\th } | \Psi_{\phi\th}|^{2} ]  r^{2}(1-\mu)d\sigma^{2}dr^{*}dt \\
&\lesssim& t_{i+1}  \int_{t=t_{i}}^{t_{i+1}} \int_{r^{*}= r^{*}_{0} }^{ r^{*} = R_{0}^{*} } \int_{\S^{2}} [  \frac{1 }{(1-\mu)^{2}}  | \Psi_{vw}|^{2}  +  \frac{1}{4r^{4}\sin^{2}\th } | \Psi_{\phi\th}|^{2} ]  d\sigma^{2}dr^{*}dt \\
&\les&  t_{i+1}  \int_{t = t_{i}}^{ t= t_{i+1}}  \int_{r^{*}= r^{*}_{0} }^{ r^{*} = R_{0}^{*} } \int_{\S^{2}} (\int_{r^{*}}^{\infty}  \chi_{[r_{0}^{*},  (R_{0}+1)^{*}]} d\overline{r^{*}} )  [    | \Psi_{\hat{v}\hat{w}}|^{2}  +  \frac{1}{4 } | \Psi_{\hat{\phi}\hat{\th}}|^{2} ]  .   6 \mu r (1-\mu)^{2} dr^{*} d\sigma^{2} dt \\
&\les&  t_{i+1} J_{\Psi}^{(G)}( t_{i} \leq t \leq t_{i+1}) (  r^{*}_{0} \le r^{*} \le R_{0}^{*} ) \\
\eeaa

\end{proof}

\subsection{Estimate for $ J_{\Psi}^{(K)} $ in terms of $E_{\Psi}^{(K)}$}\

For,
$$ t_{i+1} \leq t_{i} + (0.1) t_{i} $$ and, $$| r^{*}(r_{0}) | +| r^{*}(R) | \leq 0.4 t_{i}$$

We have:

\bea
\notag
 J_{\Psi}^{(K)} ( t_i \leq t \leq t_{i+1} ) &\les& t_{i+1} J_{\Psi}^{(G)} ( t_i \leq t \leq t_{i+1} ) (r_{0} < r < R_{0} ) \\
&\lesssim& t_{i+1} [  \frac{1}{t_{i}^{2}} E_{\Psi}^{(K)} (t=t_{i}) + \frac{1}{t_{i}^{2}} \sum_{j=1}^{3} E_{\Lie_{\Omega_{j}} \Psi}^{(K)} (t=t_{i}) \label{EKovertsquare}  \\
\notag
&& + \frac{1}{t_{i+1}^{2}} E_{\Psi}^{(K)} (t=t_{i+1}) + \frac{1}{t_{i+1}^{2}} \sum_{j=1}^{3} E_{\Lie_{\Omega_{j}} \Psi}^{(K)} (t=t_{i+1})    ]     
\eea

\begin{proof}\

We have, from assumption \eqref{Assumption1},

\beaa 
&& J_{\Psi}^{(G)} ( t_i \leq t \leq t_{i+1} )  \les  | \hat{E}_{\Psi}^{(\frac{\pa}{\pa t} ) } (t_{i+1})| +  |\hat{E}_{\Psi}^{(\frac{\pa}{\pa t} ) } (t_{i})| + | \hat{E}_{\Lie_{\Omega_{j}} \Psi }^{(\frac{\pa}{\pa t} ) }  (t_{i+1})|  + | \hat{E}_{\Lie_{\Omega_{j}} \Psi }^{(\frac{\pa}{\pa t} ) }  (t_{i})|  
\eeaa

Let $\hat{\Psi}$ be an anti-symmetric 2-tensor, defined by the following,
\bea
\hat{\Psi}_{\hat{t}\hat{\th} } (t= t_{i}, r^{*}, \th, \phi) = \hat{\chi} (\frac{2 r^{*}}{t_{i}})  \Psi_{\hat{t}\hat{\th} } (t=t_{i}, r^{*}, \th, \phi)
\eea
\bea
\hat{\Psi}_{\hat{t}\hat{\phi} } (t= t_{i}, r^{*}, \th, \phi) = \hat{\chi} (\frac{2 r^{*}}{t_{i}})  \Psi_{\hat{t}\hat{\phi} } (t=t_{i}, r^{*}, \th, \phi)
\eea
\bea
\hat{\Psi}_{\hat{r^{*}}\hat{\th} } (t= t_{i}, r^{*}, \th, \phi) = \hat{\chi} (\frac{2 r^{*}}{t_{i}})  \Psi_{\hat{r^{*}}\hat{\th} } (t=t_{i}, r^{*}, \th, \phi)
\eea
\bea
\hat{\Psi}_{\hat{r^{*}}\hat{\phi} } (t= t_{i}, r^{*}, \th, \phi) = \hat{\chi} (\frac{2 r^{*}}{t_{i}})  \Psi_{\hat{r^{*}}\hat{\phi} } (t=t_{i}, r^{*}, \th, \phi)
\eea
\bea
\der_{r^{*} } \hat{\Psi}_{\hat{t}\hat{r^{*} } } (t= t_{i}, r^{*}, \th, \phi) =  \hat{\chi} (\frac{2 r^{*}}{t_{i}})  \der_{r^{*}} \Psi_{\hat{t}\hat{r^{*}} } (t=t_{i}, r^{*}, \th, \phi)
\eea   
\bea
\der_{r^{*} } \hat{\Psi}_{\hat{\th}\hat{\phi} } (t= t_{i}, r^{*}, \th, \phi) = \hat{\chi} (\frac{2 r^{*}}{t_{i}})  \der_{r^{*} } \Psi_{\hat{\th}\hat{\phi} } (t=t_{i}, r^{*}, \th, \phi)
\eea
where $\hat{\chi}$ is a smooth cut-off function equal to one on $[-1,1]$ and zero outside $[-\frac{3}{2}, \frac{3}{2}]$. 

And 
for,
\bea
-\frac{t_{i}}{2} \leq &r^{*}& \leq \frac{t_{i}}{2} \\
\hat{\Psi}_{\hat{r^{*}}\hat{t} } (t= t_{i}, r^{*}, \th, \phi) &=&  \Psi_{\hat{r^{*}}\hat{t} } (t=t_{i}, r^{*}, \th, \phi) \\
\hat{\Psi}_{\hat{\th}\hat{\phi} } (t= t_{i}, r^{*}, \th, \phi) &=&  \Psi_{\hat{\th}\hat{\phi} } (t=t_{i}, r^{*}, \th, \phi) 
\eea

And for,
\bea
t_{i} \leq t \leq t_{i+1} \\
{\der}^{\mu} \hat{\Psi}_{\mu\nu} = 0 \\
\der_{\a} \hat{\Psi}_{\mu\nu} + \der_{\mu} \hat{\Psi}_{\nu\a} + \der_{\nu} \hat{\Psi}_{\a\mu} = 0
\eea

Then, we have that for,

 $$-\frac{t_{i}}{2} \leq r^{*} \leq \frac{t_{i}}{2}$$
$$\hat{\Psi}_{\mu\nu} (t= t_{i}, r^{*}, \th, \phi) =  \Psi_{\mu\nu} (t=t_{i}, r^{*}, \th, \phi)$$ 

And for

$$r^{*} \leq -\frac{3t_{i}}{4} \mbox{        and ,        } r^{*} \geq \frac{3t_{i}}{4}$$

for, $$  (k, l) \in \{  (r^{*}, \th), (r^{*}, \phi), (t, \th), (t, \phi) \} $$

we have,

$$\hat{\Psi}_{kl} (t= t_{i}, r^{*}, \th, \phi) = 0  $$ 

Now, considering the region $$t_{i} \leq t \leq t_{i+1}  \mbox{     ,   and     } r_{0} \leq r \leq R_{0}$$  where $$ t_{i+1} \leq t_{i} + (0.1) t_{i}  \mbox{, and          } | r^{*}(r_{0}) | +| r^{*}(R_{0}) | \leq 0.4 t_{i}$$

clearly, on $t = t_{i}$, we have $\hat{\Psi}_{\mu\nu}(t=t_i) = \Psi_{\mu\nu}(t=t_i) $, in the specified region. However, since the information from the initial data propagates no faster than the speed of light, i.e. along the null cones $t= r^{*}$, and $t= - r^{*}$, then in the region $t_{i} \leq t \leq t_{i+1}$ we have $\hat{\Psi}_{\mu\nu} = \Psi_{\mu\nu} $ if $r^{*}_{0} \geq - \frac{t_{i}}{2} + (0.1) t_{i} = -(0.4)t_{i}$ and $ R_{0}^{*} \leq  \frac{t_{i}}{2} - (0.1) t_{i}= (0.4) t_{i} $ which is satisfied in the specified region because of the condition that  $| r^{*}(r_{0}) | +| r^{*}(R_{0}) | \leq 0.4 t_{i}$. Thus, $$\Psi_{\mu\nu} = \hat{\Psi}_{\mu\nu}$$ in the specified region. Therefore, we have,
\bea
  J_{\Psi}^{(G)} ( t_i \leq t \leq t_{i+1} ) (r_{0} \leq r \leq R_{0}) &\lesssim& J_{\hat{\Psi}}^{(G)} ( t_i \leq t \leq t_{i+1} ) 
\eea

And for the same reason, on $t = t_{i+1}$, we have $\hat{\Psi}_{\mu\nu} (t = t_{i+1})= 0 $ if $r^{*} \geq - \frac{3t_{i}}{4} - (0.1) t_{i} = -(0.85)t_{i}$ and $ r^{*} \leq  \frac{3t_{i}}{4} + (0.1) t_{i}= (0.85) t_{i} $.

\beaa
&& \hat{E}_{ \hat{\Psi}}^{ ( \frac{\pa}{\pa t} ) } (t=t_{i+1}) \\
&=&   \int_{r^{*}= -\infty}^{r^{*} = \infty} \int_{\S^{2}}  (  |  \hat{\Psi}_{\hat{t} \hat{\th}}|^{2} +  |   \hat{\Psi}_{\hat{t} \hat{\phi} }|^{2}  +  |  \hat{\Psi}_{\hat{r^{*}} \hat{\th} }|^{2} +   |  \hat{\Psi}_{ \hat{r^{*}} \hat{\phi} }|^{2}   ) .r^{2} (1-\mu) d\sigma^{2} dr^{*} \\
&\lesssim&   \int_{r^{*}= - (0.85) t_{i}  }^{r^{*} = (0.85) t_{i}  }  \int_{\S^{2}}  (  |  \Psi_{\hat{t} \hat{\th}}|^{2} +  |   \Psi_{\hat{t} \hat{\phi} }|^{2}  +  |  \Psi_{\hat{r^{*}} \hat{\th} }|^{2} +   |  \Psi_{ \hat{r^{*}} \hat{\phi} }|^{2}   ) .r^{2} (1-\mu) d\sigma^{2} dr^{*} \\
&& \text{(where we used the boundedness of $\chi$)}\\
&\lesssim&   \int_{r^{*}= - (0.85) t_{i}  }^{r^{*} = (0.85) t_{i}  }  \int_{\S^{2}}  (  |  \Psi_{\hat{v} \hat{\th}}|^{2} +  |   \Psi_{\hat{v} \hat{\phi} }|^{2}  +  |  \Psi_{\hat{w} \hat{\th} }|^{2} +   |  \Psi_{ \hat{w} \hat{\phi} }|^{2}   ) .r^{2} (1-\mu) d\sigma^{2} dr^{*} 
\eeaa

We also have,

\beaa
&& \hat{E}_{ \hat{\Psi}}^{ ( \frac{\pa}{\pa t} ) } (t=t_{i+1}) \\
&=&   \int_{r^{*}= -\infty}^{r^{*} = \infty} \int_{\S^{2}}  (  |  \hat{\Psi}_{\hat{t} \hat{\th}}|^{2} +  |   \hat{\Psi}_{\hat{t} \hat{\phi} }|^{2}  +  |  \hat{\Psi}_{\hat{r^{*}} \hat{\th} }|^{2} +   |  \hat{\Psi}_{ \hat{r^{*}} \hat{\phi} }|^{2}   ) .r^{2} (1-\mu) d\sigma^{2} dr^{*} \\
&\lesssim&   \int_{r^{*}= - \frac{3t_{i}}{4} }^{r^{*} = \frac{3t_{i}}{4}  }   \int_{\S^{2}}  (  |  \Psi_{\hat{t} \hat{\th}}|^{2} +  |   \Psi_{\hat{t} \hat{\phi} }|^{2}  +  |  \Psi_{\hat{r^{*}} \hat{\th} }|^{2} +   |  \Psi_{ \hat{r^{*}} \hat{\phi} }|^{2}   ) .r^{2} (1-\mu) d\sigma^{2} dr^{*} \\
&& \text{(using the boundedness of $\chi$)}\\
&\lesssim&   \int_{r^{*}= - (0.85) t_{i}  }^{r^{*} = (0.85) t_{i}  }   \int_{\S^{2}}  (  |  \Psi_{\hat{v} \hat{\th}}|^{2} +  |   \Psi_{\hat{v} \hat{\phi} }|^{2}  +  |  \Psi_{\hat{w} \hat{\th} }|^{2} +   |  \Psi_{ \hat{w} \hat{\phi} }|^{2}   ) .r^{2} (1-\mu) d\sigma^{2} dr^{*} \\
\eeaa

\begin{lemma}
\bea
\notag
&&\int_{r^{*}= r_{1}^{*}   }^{r^{*} = r_{2}^{*}   } \int_{\S^{2}} ( \frac{1}{r^{2}(1-\mu)} | \Psi_{w\th} |^{2} + \frac{1}{r^{2}\sin^{2}\th(1-\mu)} | \Psi_{w\phi} |^{2}   +   \frac{1}{r^{2}(1-\mu)} | \Psi_{v\th} |^{2} \\
\notag
&&+ \frac{1}{r^{2}\sin^{2}\th(1-\mu)} | \Psi_{v\phi} |^{2}  +  \frac{1}{(1-\mu)^{2}} |\Psi_{vw}|^{2}   +  \frac{1}{4r^{4}\sin^{2}\th} | \Psi_{\phi\th}|^{2} ). (1-\mu) r^{2}   d\sigma^{2} dr^{*} (t) \\
& \lesssim & \frac{E_{\Psi}^{(K)}(t)}{\min_{w \in \{t\}\cap \{  r_{1}^{*} \leq r^{*} \leq r_{2}^{*} \}  } w^{2}}  + \frac{E_{\Psi}^{(K)}(t)}{\min_{v \in \{t\}\cap  \{  r_{1}^{*} \leq r^{*} \leq r_{2}^{*} \}  } v^{2}} \label{EKoverminvsquaredandwsquared} 
\eea
\end{lemma}

\begin{proof}\

We proved that,
\beaa
&& E_{\Psi}^{(K)}(t_{i})\\
&=& \int_{r^{*} = - \infty}^{r^{*}= \infty} \int_{\S^{2}}   (  w^{2} [\frac{1}{r^{2}} | \Psi_{w\th} |^{2} + \frac{1}{r^{2}\sin^{2}\th} | \Psi_{w\phi} |^{2} ]  + v^{2} [ \frac{1}{r^{2}} | \Psi_{v\th} |^{2} + \frac{1}{r^{2}\sin^{2}\th} | \Psi_{v\phi} |^{2} ]  \\
&&+ (1-\mu) (\om^{2} + v^{2} ) [ \frac{1}{(1-\mu)^{2}} |\Psi_{vw}|^{2}   +  \frac{1}{4r^{4}\sin^{2}\th} | \Psi_{\phi\th}|^{2}]  )  r^{2}  d\sigma^{2} dr^{*} 
\eeaa

Because of this we have,
\beaa
&&\int_{r^{*}= - \infty }^{r^{*} = \infty  } \int_{\S^{2}} ( (1- \mu) w^{2} [\frac{1}{r^{2}(1- \mu)} | \Psi_{w\th} |^{2} + \frac{1}{r^{2} \sin^{2}\th (1- \mu)} | \Psi_{w\phi} |^{2} ] \\
&& + (1- \mu) w^{2} [ \frac{1}{(1-\mu)^{2}} |\Psi_{vw}|^{2}   +  \frac{1}{4r^{4}\sin^{2}\th} | \Psi_{\phi\th}|^{2}]  ). r^{2}   d\sigma^{2} dr^{*} (t_{i}) \\
&\lesssim& E_{\Psi}^{(K)}(t_{i}) 
\eeaa
and,
\beaa
&&\int_{r^{*}= - \infty }^{r^{*} = \infty  } \int_{\S^{2}} ( (1- \mu) v^{2} [ \frac{1}{r^{2}(1- \mu) } | \Psi_{v\th} |^{2} + \frac{1}{r^{2}\sin^{2}\th (1- \mu) } | \Psi_{v\phi} |^{2} ]  \\
&& + (1- \mu)  v^{2} [ \frac{1}{(1-\mu)^{2}} |\Psi_{vw}|^{2}   +  \frac{1}{4r^{4}\sin^{2}\th} | \Psi_{\phi\th}|^{2}] ). r^{2}   d\sigma^{2} dr^{*} (t_{i}) \\
&\lesssim& E_{\Psi}^{(K)}(t_{i}) 
\eeaa
and thus,
\beaa
&&\int_{r^{*}= r_{1}^{*}   }^{r^{*} = r_{2}^{*}   } \int_{\S^{2}} (  (1- \mu)  [\frac{1}{r^{2}(1- \mu)} | \Psi_{w\th} |^{2} + \frac{1}{r^{2}\sin^{2}\th (1- \mu) } | \Psi_{w\phi} |^{2} ] \\
&& + (1- \mu)  [ \frac{1}{(1-\mu)^{2}} |\Psi_{vw}|^{2}   +  \frac{1}{4r^{4}\sin^{2}\th} | \Psi_{\phi\th}|^{2}] ). r^{2}   d\sigma^{2} dr^{*} (t) \\
& \lesssim & \frac{E_{\Psi}^{(K)}(t)}{\min_{w \in \{t\}\cap \{  r_{1}^{*} \leq r^{*} \leq r_{2}^{*} \}  } w^{2}}
\eeaa
and,
\beaa
&&\int_{r^{*}= r_{1}^{*}   }^{r^{*} = r_{2}^{*}   } \int_{\S^{2}} ( (1- \mu)   [ \frac{1}{r^{2}(1- \mu) } | \Psi_{v\th} |^{2} + \frac{1}{r^{2}\sin^{2}\th (1- \mu) } | \Psi_{v\phi} |^{2} ]  \\
&& + (1- \mu)  [ \frac{1}{(1-\mu)^{2}} |\Psi_{vw}|^{2}   +  \frac{1}{4r^{4}\sin^{2}\th} | \Psi_{\phi\th}|^{2}] ). r^{2}   d\sigma^{2} dr^{*} (t) \\
&\lesssim& \frac{E_{\Psi}^{(K)}(t)}{\min_{v \in \{t\}\cap \{  r_{1}^{*} \leq r^{*} \leq r_{2}^{*} \}  } v^{2}}
\eeaa
Summing, we obtain,
\beaa
\notag
&&\int_{r^{*}= r_{1}^{*}   }^{r^{*} = r_{2}^{*}   } \int_{\S^{2}} ( [\frac{1}{r^{2}(1-\mu)} | \Psi_{w\th} |^{2} + \frac{1}{r^{2}\sin^{2}\th(1-\mu)} | \Psi_{w\phi} |^{2} ]  +  [ \frac{1}{r^{2}(1-\mu)} | \Psi_{v\th} |^{2}  \\
\notag
&&+ \frac{1}{r^{2}\sin^{2}\th(1-\mu)} | \Psi_{v\phi} |^{2} ]  +  [ \frac{1}{(1-\mu)^{2}} |\Psi_{vw}|^{2}   +  \frac{1}{4r^{4}\sin^{2}\th} | \Psi_{\phi\th}|^{2}]  ). (1-\mu) r^{2}   d\sigma^{2} dr^{*} (t) \\
& \lesssim & \frac{E_{\Psi}^{(K)}(t)}{\min_{w \in \{t\}\cap \{  r_{1}^{*} \leq r^{*} \leq r_{2}^{*} \}  } w^{2}}  + \frac{E_{\Psi}^{(K)}(t)}{\min_{v \in \{t\}\cap  \{  r_{1}^{*} \leq r^{*} \leq r_{2}^{*} \}  } v^{2}} 
\eeaa

\end{proof}

Inequality \eqref{EKoverminvsquaredandwsquared} gives,

\bea
\notag
&&  \int_{r^{*}= - (0.85) t_{i}  }^{r^{*} = (0.85) t_{i}  }  \int_{\S^{2}} ( \frac{1}{r^{2}(1-\mu)} | \Psi_{w\th} |^{2} + \frac{1}{r^{2}\sin^{2}\th(1-\mu)} | \Psi_{w\phi} |^{2}   +   \frac{1}{r^{2}(1-\mu)} | \Psi_{v\th} |^{2} \\
\notag
&& + \frac{1}{r^{2}\sin^{2}\th(1-\mu)} | \Psi_{v\phi} |^{2}   +   \frac{1}{(1-\mu)^{2}} |\Psi_{vw}|^{2}   +  \frac{1}{4r^{4}\sin^{2}\th} | \Psi_{\phi\th}|^{2}  ). (1-\mu) r^{2}   d\sigma^{2} dr^{*} (t) \\
\notag
& \lesssim & \frac{E_{\Psi}^{(K)}(t)}{\min_{w \in \{t\}\cap\{ - (0.85)t_{i} \leq r^{*} \leq (0.85)t_{i} \} } w^{2}}  + \frac{E_{\Psi}^{(K)}(t)}{\min_{v \in \{t\}\cap\{ - (0.85)t_{i} \leq r^{*} \leq (0.85)t_{i} \} } v^{2}} \\
&\lesssim & \frac{E_{\Psi}^{(K)}(t)}{t_{i}^{2} }
\eea

Examining now the term,

\beaa
| \hat{E}^{(\frac{\pa}{\pa t})}_{ \Lie_{ \Om_{i} }  \hat{\Psi} } (t=t_{i}) |  &\les&  \int_{r^{*}= -\infty}^{r^{*} = \infty} \int_{\S^{2}} (  | \Lie_{ \Om_{i} } \hat{\Psi}_{\hat{w}\hat{\th}} |^{2} +  | \Lie_{ \Om_{i} } \hat{\Psi}_{\hat{w}\hat{\phi}} |^{2}   +   | \Lie_{ \Om_{i} } \hat{\Psi}_{\hat{v}\hat{\th}} |^{2} \\
&&+  | \Lie_{ \Om_{i} } \hat{\Psi}_{\hat{v}\hat{\phi}} |^{2}    ). (1-\mu) r^{2}   d\sigma^{2} dr^{*} (t= t_{i} ) \\
\eeaa

For, $$  (k, l) \in \{  (r^{*}, \th), (r^{*}, \phi), (t, \th), (t, \phi) \} $$

\beaa
\Lie_{ \Om_{i} }  \hat{\Psi}_{kl} (t=t_{i}, r^{*}, \th, \phi)  &=& \Lie_{ \Om_{i} } \chi(\frac{2 r^{*}}{t_{i}})  \Psi_{kl} (t=t_{i}, r^{*}, \th, \phi)  \\
&=& \chi(\frac{2 r^{*}}{t_{i}}) \Lie_{ \Om_{i} }  \Psi_{kl} (t=t_{i}, r^{*}, \th, \phi)
\eeaa

Thus,
\beaa
| \hat{E}^{(\frac{\pa}{\pa t})}_{\Lie_{ \Om_{i} } \hat{\Psi} } (t=t_{i}) | \les 
E^{ (\frac{\pa}{\pa t}) }_{\Lie_{ \Om_{i} } \Psi } ( - \frac{3 t_{i}}{4} \leq r^{*} \leq \frac{3 t_{i}}{4}) (t=t_{i})
\eeaa

\end{proof}

\subsection{Estimate for $E_{\Psi}^{(K)}$}\

\bea
\notag
E_{\Psi}^{(K)} (t) &\les& \sum_{k=1}^{3} \sum_{l=1}^{3} \sum_{j=1}^{3}   E_{\Psi, \Lie_{\Om_{k}} \Psi, \Lie_{\Om_{k}} \Lie_{\Om_{l}} \Lie_{\Om_{j}} \Psi,  \Lie_{\Om_{k}} \Lie_{\Om_{l}} \Lie_{\Om_{j}} \Psi }^{(\frac{\pa}{\pa t})}   (t=t_{0})  +  \sum_{l=1}^{3} \sum_{j=1}^{3}   E_{\Psi, \Lie_{\Om_{l}} \Psi, \Lie_{\Om_{l}} \Lie_{\Om_{j}} \Psi  }^{(K)} (t=t_{0})  \\
&\les& E^{M}_{F} \label{boundingEK}
\eea

where,

\bea
\notag
E^{M}_{\Psi} &=& \sum_{k=1}^{3} \sum_{l=1}^{3} \sum_{j=1}^{3}   E_{\Psi, \Lie_{\Om_{k}} \Psi, \Lie_{\Om_{k}} \Lie_{\Om_{l}} \Lie_{\Om_{j}} \Psi,  \Lie_{\Om_{k}} \Lie_{\Om_{l}} \Lie_{\Om_{j}} \Psi }^{(\frac{\pa}{\pa t})}   (t=t_{0})  +  \sum_{l=1}^{3} \sum_{j=1}^{3}   E_{\Psi, \Lie_{\Om_{l}} \Psi, \Lie_{\Om_{l}} \Lie_{\Om_{j}} \Psi  }^{(K)} (t=t_{0}) \\
&=& \sum_{i=0}^{3} E_{ r^{j} (\rLie)^{j} \Psi  }^{(\frac{\pa}{\pa t})} (t=t_{0})    + \sum_{i=0}^{2}  E_{ r^{j} (\rLie)^{j} \Psi  }^{(K)} (t=t_{0}) 
\eea

\begin{proof}\

Let,

\bea
t_{i+1} = t_{i} + (0.1) t_{i} = (1.1) t_{i} \label{defti}
\eea

For $t_{0}$ big enough, we will have 
$$| r^{*}(r_{0}) | +| r^{*}(R_{0}) | \leq 0.4 t_{0}$$

and therefore will be able to apply \eqref{EKovertsquare}.\\

In view of \eqref{conservationlawdivergncetheorem} and \eqref{theenrgymomuntumtensorofpsiisdivergencefree} applied to the vector field $K$ in the region $t \in [t_{i}, t_{i+1}]$, we get

\beaa
 E_{\Psi}^{(K)} (t=t_{i+1})  &\leq& J_{\Psi}^{(K)}( t_{0} \leq t \leq t_{i+1})   + E_{\Psi}^{(K)} (t=t_{0}) \\
&\leq&  t_{i+1}    J_{\Psi}^{(G)}( t_0 \leq t \leq t_{i+1} ) (r_{0} \leq r \leq R_{0})     + E_{\Psi}^{(K)} (t=t_{0}) \\
\eeaa
(from \eqref{KcontolledbytG})\\

and,
\beaa
&& E_{\Lie_{\Omega_{j}} \Psi }^{(K)} (t=t_{i+1}) \\
&\leq&  t_{i+1}  J_{\Lie_{\Om_{j}}\Psi}^{(G)} ( t_0 \leq t \leq t_{i+1} ) (r_{0} \leq r \leq R_{0})  + E_{\Lie_{\Omega_{j}} \Psi}^{(K)} (t=t_{0}) 
\eeaa
(from \eqref{KcontolledbytG})\\

Since $\hat{\Psi}$ verifies the Maxwell equations, we have 

\bea
&&  \der^{\a} T_{\a\b} (\hat{\Psi}) = 0
\eea

and thus, we can estimate,

\beaa
&& J_{\Psi}^{(G)} ( t_{i} \leq t \leq t_{i+1} ) (r_{0} \leq r \leq R_{0})  \\
\les &&   \frac{1}{t_{i}^{2}} E_{\Psi}^{(K)} (t=t_{i}) + \frac{1}{t_{i}^{2}} \sum_{j=1}^{3} E_{\Lie_{\Omega_{j}} \Psi}^{(K)} (t=t_{i})  + \frac{1}{t_{i+1}^{2}} E_{\Psi}^{(K)} (t=t_{i+1}) + \frac{1}{t_{i+1}^{2}} \sum_{j=1}^{3} E_{\Lie_{\Omega_{j}} \Psi}^{(K)} (t=t_{i+1})         \\
\eeaa

(from \eqref{EKovertsquare} )\\

\beaa
&& J_{\Psi}^{(G)} ( t_{i} \leq t \leq t_{i+1} ) (r_{0} \leq r \leq R_{0}) \\
& \lesssim & \frac{t_{i}}{t_{i}^{2}}     J_{\Psi}^{(G)} ( t_0 \leq t \leq t_{i} ) (r_{0} \leq r \leq R_{0})    +  \frac{1}{t_{i}^{2}} E_{\Psi}^{(K)} (t=t_{0}) \\
&& + \frac{t_{i}}{t_{i}^{2}}  \sum_{j=1}^{3}   J_{ \Lie_{\Om_{j}}  \Psi}^{(G)} (X)( t_0 \leq t \leq t_{i} ) (r_{0} \leq r \leq R_{0})   +  \frac{1}{t_{i}^{2}}  \sum_{j=1}^{3} E_{\Lie_{\Om_{j}} \Psi}^{(K)} (t=t_{0})   \\
&& + \frac{t_{i+1}}{t_{i+1}^{2}}     J_{\Psi}^{(G)} ( t_0 \leq t \leq t_{i+1} ) (r_{0} \leq r \leq R_{0})    +  \frac{1}{t_{i+1}^{2}} E_{\Psi}^{(K)} (t=t_{0}) \\
&& + \frac{t_{i+1}}{t_{i+1}^{2}}   \sum_{j=1}^{3} J_{ \Lie_{\Om_{j}}  \Psi }^{(G)} (X)( t_0 \leq t \leq t_{i+1} ) (r_{0} \leq r \leq R_{0})    +  \frac{1}{t_{i+1}^{2}} \sum_{j=1}^{3} E_{\Lie_{\Om_{j}} \Psi}^{(K)} (t=t_{0})  \\
\eeaa

We will use the notation $J_{\Psi, \Lie_{\Om_{j}} \Psi}^{(G)} = J_{\Psi}^{(G)} + J_{\Lie_{\Om_{j}} \Psi}^{(G)}$, for all letters such as $J$, and for different summations, so as to lighten the notation and be able to write:

\beaa
&& J_{\Psi}^{(G)} ( t_{i} \leq t \leq t_{i+1} ) (r_{0} \leq r \leq R_{0})  \\
 \lesssim && \frac{1}{t_{i}}   \sum_{j=1}^{3}   J_{\Psi,  \Lie_{\Om_{j}} \Psi }^{(G)} ( t_0 \leq t \leq t_{i} ) (r_{0} \leq r \leq R_{0})    +  \frac{1}{t_{i}^{2}}  \sum_{j=1}^{3} E_{\Psi,  \Lie_{\Om_{j}}  \Psi}^{(K)} (t=t_{0}) \\
&& + \frac{1}{t_{i+1}}    \sum_{j=1}^{3}  J_{\Psi,  \Lie_{\Om_{j}} \Psi }^{(G)} ( t_0 \leq t \leq t_{i+1} ) (r_{0} \leq r \leq R_{0})    +  \frac{1}{t_{i+1}^{2}}  \sum_{j=1}^{3} E_{\Psi,  \Lie_{\Om_{j}}  \Psi}^{(K)} (t=t_{0}) 
\eeaa

From this we can deduce the following,

\beaa
&& J_{\Psi}^{(K)}( t_{i} \leq t \leq t_{i+1})  \\
& \les&    t_{i+1}   J_{\Psi}^{(G)} ( t_i \leq t \leq t_{i+1} ) (r_{0} \leq r \leq R_{0})     \\
&& \text{(from \eqref{KcontolledbytG})}\\
& \les &    \sum_{j=1}^{3}    J_{\Psi, \Lie_{\Om_{j}} \Psi}^{(G)} ( t_0 \leq t \leq t_{i+1} ) (r_{0} \leq r \leq R_{0})    + \frac{1}{t_{i+1}}  \sum_{j=1}^{3} E_{\Psi, \Lie_{\Om_{j}} \Psi}^{(K)} (t=t_{0}) 
\eeaa

(from the estimate above, and using the positivity of $J_{\Psi,  \Lie_{\Om_{j}} \Psi}^{(G)} $ ) \\

Since,
$$ t_{i+1}  = (1.1) t_{i} $$

we have,
$$ t_{i+1}  = (1.1)^{i+1} t_{0} $$

and thus,
\beaa
\sum_{i} \frac{1}{t_{i+1}} = \sum_{i} \frac{1}{(1.1)^{i+1} t_{0}} \les 1
\eeaa

Therefore, 

\beaa
&& J_{\Psi}^{(K)}( t_{0} \leq t \leq t_{i+1})\\
 &=& \sum_{i=0}^{i+1} J_{\Psi}^{(K)}( t_{i} \leq t \leq t_{i+1})\\
& \lesssim &  (i+1)   \sum_{j=1}^{3} J_{\Psi,  \Lie_{\Om_{j}} \Psi }^{(G)} ( t_0 \leq t \leq t_{i+1} ) (r_{0} \leq r \leq R_{0})     +   \sum_{j=1}^{3}     E_{\Psi, \Lie_{\Om_{j}} \Psi }^{(K)} (t=t_{0}) \\
&& \text{(from the above)}
\eeaa

and thus,

\beaa
&& E_{\Psi}^{(K)} (t=t_{i+1}) \\
& \lesssim& J_{\Psi}^{(K)}( t_{0} \leq t \leq t_{i+1})   +    E_{\Psi}^{(K)} (t=t_{0}) \\
 & \lesssim &  (i+1)  \sum_{j=1}^{3}   J_{\Psi,  \Lie_{\Om_{j}} \Psi }^{(G)} ( t_0 \leq t \leq t_{i+1} ) (r_{0} \leq r \leq R_{0})     +        \sum_{j=1}^{3}  E_{\Psi, \Lie_{\Om_{j}} \Psi }^{(K)} (t=t_{0}) 
\eeaa

In the same manner, this leads to,

\beaa
&& E_{\Lie_{\Om_{l}} \Psi}^{(K)} (t=t_{i+1})\\
& \lesssim& J_{\Lie_{\Om_{l}}\Psi}^{(K)}( t_{0} \leq t \leq t_{i+1})   +    E_{\Lie_{\Om_{l}} \Psi}^{(K)} (t=t_{0}) \\
 & \lesssim &  (i+1)     \sum_{j=1}^{3}   J_{\Lie_{\Om_{l}}\Psi, \Lie_{\Om_{l}} \Lie_{\Om_{j}} \Psi }^{(G)} ( t_0 \leq t \leq t_{i+1} ) (r_{0} \leq r \leq R_{0})     +  \sum_{j=1}^{3}       E_{\Lie_{\Om_{l}} \Psi, \Lie_{\Om_{l}} \Lie_{\Om_{j}} \Psi }^{(K)} (t=t_{0}) 
\eeaa

(since $\Om_{l}$, $l \in \{ 1, 2, 3 \}$ are Killing, and therefore $\Lie_{\Om_{l}} \Psi$ verifies the Maxwell equations).\\

Repeating the procedure again, we get,

\beaa
&&  J_{\Psi}^{(G)} ( t_i \leq t \leq t_{i+1} ) (r_{0} \leq r \leq R_{0})     \\
& \les &   \frac{1}{t_{i}^{2}} E_{\Psi}^{(K)} (t=t_{i}) + \frac{1}{t_{i}^{2}}  \sum_{l=1}^{3}   E_{\Lie_{\Omega_{l}} \Psi}^{(K)} (t=t_{i})  + \frac{1}{t_{i+1}^{2}} E_{\Psi}^{(K)} (t=t_{i+1}) + \frac{1}{t_{i+1}^{2}} \sum_{l=1}^{3}   E_{\Lie_{\Omega_{l}} \Psi}^{(K)} (t=t_{i+1})         \\
&& \text{(from \eqref{EKovertsquare})} \\
& \lesssim &  \frac{i}{{t_{i}}^{2}}  \sum_{l=1}^{3} \sum_{j=1}^{3}   J_{\Psi, \Lie_{\Om_{l}} \Psi, \Lie_{\Om_{l}} \Lie_{\Om_{j}} \Psi }^{(G)} ( t_0 \leq t \leq t_{i} ) (r_{0} \leq r \leq R_{0})      +  \frac{1}{{t_{i}}^{2}}  \sum_{l=1}^{3} \sum_{j=1}^{3}  E_{\Psi, \Lie_{\Om_{l}} \Psi, \Lie_{\Om_{l}} \Lie_{\Om_{j}} \Psi  }^{(K)} (t=t_{0}) \\
&& +  \frac{(i+1)}{{t_{i+1}}^{2}}  \sum_{l=1}^{3} \sum_{j=1}^{3}  J_{\Psi, \Lie_{\Om_{l}} \Psi, \Lie_{\Om_{l}} \Lie_{\Om_{j}} \Psi }^{(G)} ( t_0 \leq t \leq t_{i+1} ) (r_{0} \leq r \leq R_{0})    \\
&&  +  \frac{1}{{t_{i+1}}^{2}} \sum_{l=1}^{3} \sum_{j=1}^{3}   E_{\Psi, \Lie_{\Om_{l}} \Psi, \Lie_{\Om_{l}} \Lie_{\Om_{j}} \Psi  }^{(K)} (t=t_{0}) \\
&\les&   \frac{(i+1)}{{t_{i+1}}^{2}}  \sum_{l=1}^{3} \sum_{j=1}^{3}  J_{\Psi, \Lie_{\Om_{l}} \Psi, \Lie_{\Om_{l}} \Lie_{\Om_{j}} \Psi }^{(G)} ( t_0 \leq t \leq t_{i+1} ) (r_{0} \leq r \leq R_{0})    \\
&&  +  \frac{1}{{t_{i+1}}^{2}} \sum_{l=1}^{3} \sum_{j=1}^{3}   E_{\Psi, \Lie_{\Om_{l}} \Psi, \Lie_{\Om_{l}} \Lie_{\Om_{j}} \Psi  }^{(K)} (t=t_{0}) 
\eeaa
Thus,
\beaa 
&& J_{\Psi}^{(K)} ( t_{i} \leq t \leq t_{i+1}) ( -\infty \leq r^{*} \leq \infty)    \\
&\les& t_{i+1}  J_{\Psi}^{(G)} ( t_i \leq t \leq t_{i+1} ) (r_{0} \leq r \leq R_{0})     \\
&& \text{(from \eqref{KcontolledbytG})} \\
&\les&   \frac{(i+1)}{t_{i+1}}  \sum_{l=1}^{3} \sum_{j=1}^{3}    J_{\Psi, \Lie_{\Om_{l}} \Psi, \Lie_{\Om_{l}} \Lie_{\Om_{j}}  \Psi }^{(G)} ( t_0 \leq t \leq t_{i+1} ) (r_{0} \leq r \leq R_{0})     \\
&& +  \frac{1}{t_{i+1}}   \sum_{l=1}^{3} \sum_{j=1}^{3}  E_{\Psi, \Lie_{\Om_{j}}\Psi, \Lie_{\Om_{l}} \Lie_{\Om_{j}} \Psi  }^{(K)} (t=t_{0}) 
\eeaa
(using the above).\\

We have,

\beaa
\frac{(i+1)}{t_{i+1}} &=& \frac{(i+1)}{({t_{i+1})}^{\frac{1}{2}}} . \frac{1}{({t_{i+1})}^{\frac{1}{2}}} \le C . \frac{1}{({t_{i+1})}^{\frac{1}{2}}} \\
&& \text{(where we used the fact that $\frac{(i+1)}{({t_{i+1})}^{\frac{1}{2}}} \le C$ )}\\
&\les&  \frac{1}{{(1.1)}^{\frac{i+1}{2}}} \\
&\les&  (\sqrt{ \frac{1}{{1.1}} } )^{i} \\
\eeaa

$ \sum_{i} (\sqrt{ \frac{1}{{1.1}} } )^{i}$ is a geometric series with $ (\sqrt{ \frac{1}{{1.1}} } ) < 1 $, and therefore,
$$  \sum_{i=0}^{\infty}  (\sqrt{ \frac{1}{{1.1}} } )^{i} \les 1 $$

Finally, we have,

\beaa 
&& J_{\Psi}^{(K)} ( t_{0} \leq t \leq t_{i+1}) ( -\infty \leq r^{*} \leq \infty)    \\
&=& \sum_{i=0}^{i+1} J_{\Psi}^{(K)} ( t_{i} \leq t \leq t_{i+1}) ( -\infty \leq r^{*} \leq \infty)    \\
&\les& \sum_{i=0}^{i+1}  (\sqrt{ \frac{1}{{1.1}} } )^{i}  \sum_{l=1}^{3} \sum_{j=1}^{3}  J_{\Psi, \Lie_{\Om_{l}} \Psi, \Lie_{\Om_{l}} \Lie_{\Om_{j}}  \Psi }^{(G)} ( t_0 \leq t \leq t_{i+1} ) (r_{0} \leq r \leq R_{0})    \\
&& + \sum_{i=0}^{i+1} \frac{1}{t_{i+1}}  \sum_{l=1}^{3} \sum_{j=1}^{3}   E_{\Psi, \Lie_{\Om_{l}} \Psi, \Lie_{\Om_{l}} \Lie_{\Om_{j}}  \Psi }^{(K)} (t=t_{0}) \\
&\les&  \sum_{l=1}^{3} \sum_{j=1}^{3}   [ [ J_{\Psi, \Lie_{\Om_{l}} \Psi, \Lie_{\Om_{l}} \Lie_{\Om_{j}}  \Psi }^{(G)} ( t_0 \leq t \leq t_{i+1} ) (r_{0} \leq r \leq R_{0})      +  E_{\Psi, \Lie_{\Om_{l}} \Psi, \Lie_{\Om_{l}} \Lie_{\Om_{j}}  \Psi }^{(K)} (t=t_{0}) ] 
\eeaa
(from the above)\\

which gives,

\beaa
&&  E_{\Psi}^{(K)} (t=t_{i+1}) \\ 
&\les&  J_{\Psi}^{(K)} ( t_{0} \leq t \leq t_{i+1}) ( -\infty \leq r^{*} \leq \infty)  +   E_{ \Psi }^{(K)} (t=t_{0}) \\
&\les&  \sum_{l=1}^{3} \sum_{j=1}^{3}   [  J_{\Psi, \Lie_{\Om_{l}} \Psi, \Lie_{\Om_{l}} \Lie_{\Om_{j}}  \Psi }^{(G)} ( t_0 \leq t \leq t_{i+1} ) (r_{0} \leq r \leq R_{0})      +  E_{\Psi, \Lie_{\Om_{l}} \Psi, \Lie_{\Om_{l}} \Lie_{\Om_{j}}  \Psi }^{(K)} (t=t_{0}) ]  \\
\eeaa

From assumption \eqref{Assumption1},

\beaa
 J_{\Psi }^{(G)} ( t_0 \leq t \leq t_{i+1} ) (r_{0} \leq r \leq R)      &\les&   \sum_{j=1}^{3}   E_{\Psi, \Lie_{\Om_{j}} \Psi   }^{(\frac{\pa}{\pa t})}   \\
\eeaa

Thus,

\beaa
&& E_{\Psi}^{(K)} (t=t_{i+1}) \\
&\les&   \sum_{k=1}^{3} \sum_{l=1}^{3} \sum_{j=1}^{3}   E_{\Psi, \Lie_{\Om_{k}} \Psi, \Lie_{\Om_{k}} \Lie_{\Om_{l}} \Lie_{\Om_{j}} \Psi,  \Lie_{\Om_{k}} \Lie_{\Om_{l}} \Lie_{\Om_{j}} \Psi }^{(\frac{\pa}{\pa t})}   (t=t_{0})  +  \sum_{l=1}^{3} \sum_{j=1}^{3}   E_{\Psi, \Lie_{\Om_{l}} \Psi, \Lie_{\Om_{l}} \Lie_{\Om_{j}} \Psi  }^{(K)} (t=t_{0})   \\
&\les& E^{M}_{\Psi}
\eeaa

because,

\bea
\sum_{j = 1 }^{3} | \Lie_{\Om_{j}  } \Psi |^{2} = r^{2} |\rLie \Psi |^{2} =  | r \rLie \Psi |^{2} \\ \notag
\eea

We have $t_{i+1} = (1.1)^{i+1} t_{0}$, however, our proofs work with any $a$ such that $1 <a <2$, and $t_{i+1} = a^{i} t_{0}$. Since for all $t> t_{0}$ there exist $i$ and $a$ such that $t = a^{i} t_{0}$, $t_{0} > 1$, we get that,

\beaa
E_{\Psi}^{(K)} (t) &\les& E^{M}_{\Psi} \\
\eeaa

for all $t > t_{0}$, and similarly for all $t < - t_{0}$. And since the region $-t_{0} \leq t \leq t_{0}$ is a bounded region, the above inequality holds in this region. Thus, for all $t$, we have,

\beaa
E_{\Psi}^{(K)} (t) &\les& E^{M}_{\Psi}\\
\eeaa

\end{proof}

\section{The Proof of Decay Away from the Horizon}\

We will prove that for all $\hat{\mu}, \hat{\nu} \in \{\hat{\frac{\pa}{\pa w}}, \hat{\frac{\pa}{\pa v}}, \hat{\frac{\pa}{\pa \th}}, \hat{\frac{\pa}{\pa \phi}} \}$, all the components of the Maxwell fields satisfy
\beaa
 |F_{\hat{\mu}\hat{\nu}}|(w, v, \om) &\lesssim& \frac{  E_{F} }{ (1 + |v|) } 
\eeaa
and,
\beaa
 |F_{\hat{\mu}\hat{\nu}}|(w, v, \om)  &\lesssim& \frac{   E_{F}  }{ (1 + |w|) } 
\eeaa

 in $ r \geq R > 2m$, for an arbitrarily fixed $R$, and where $E_{F}$ is defined by

\beaa
E_{F} &= & [ \sum_{i=0}^{1}  \sum_{j=0}^{5} E_{ r^{j} (\rLie)^{j} (\Lie_{t})^{i}  F }^{(\frac{\pa}{\pa t})} (t=t_{0})    + \sum_{i=0}^{1}  \sum_{j=0}^{4} E_{ r^{j} (\rLie)^{j} (\Lie_{t})^{i}  F  }^{(K)} (t=t_{0}) \\
&& +  E_{ r^{6} (\rLie)^{6}  F }^{(\frac{\pa}{\pa t})} (t=t_{0})    +  E_{ r^{5} (\rLie)^{5}  F  }^{(K)} (t=t_{0}) ]^{\frac{1}{2}} \\
\eeaa

\begin{proof}\

We will prove that,
\beaa
 |F_{\hat{\mu}\hat{\nu}}|(w, v, \om) &\lesssim& \frac{ [  E_{  F, \Lie_{t}  F, r \rLie F, \Lie_{t} r \rLie F, r^{2} (\rLie)^{2} F, \Lie_{t} r^{2} (\rLie)^{2} F, r^{3} (\rLie)^{3} F}^{(K)} (t) ]^{\frac{1}{2}} }{ (1 + |v|) } \\
\eeaa
and,
\beaa
 |F_{\hat{\mu}\hat{\nu}}|(w, v, \om) &\lesssim& \frac{ [  E_{  F, \Lie_{t}  F, r \rLie F, \Lie_{t} r \rLie F, r^{2} (\rLie)^{2} F, \Lie_{t} r^{2} (\rLie)^{2} F, r^{3} (\rLie)^{3} F}^{(K)} ]^{\frac{1}{2}}  }{ (1 + |w|) } \\
\eeaa
Recall that,
\beaa
E_{\Psi}^{(K)}  &\les& \sum_{j=0}^{3} E_{ r^{j} (\rLie)^{j} \Psi  }^{(\frac{\pa}{\pa t})} (t=t_{0})    + \sum_{j=0}^{2}  E_{ r^{j} (\rLie)^{j} \Psi  }^{(K)} (t=t_{0}) \\
\eeaa
Thus,
\beaa
&& E_{  F, \Lie_{t}  F, r \rLie F, \Lie_{t} r \rLie F, r^{2} (\rLie)^{2} F, \Lie_{t} r^{2} (\rLie)^{2} F, r^{3} (\rLie)^{3} F}^{(K)}\\
&\les& \sum_{i=0}^{1}   \sum_{j=0}^{5} E_{ r^{j} (\rLie)^{j} (\Lie_{t})^{i}  F }^{(\frac{\pa}{\pa t})} (t=t_{0})    + \sum_{i=0}^{1}   \sum_{j=0}^{4} E_{ r^{j} (\rLie)^{j} (\Lie_{t})^{i}  F  }^{(K)} (t=t_{0}) \\
&& + E_{ r^{6} (\rLie)^{6}  F }^{(\frac{\pa}{\pa t})} (t=t_{0})    +   E_{ r^{5} (\rLie)^{5}  F  }^{(K)} (t=t_{0}) \\
&\les& E_{F}^{2}\\
\eeaa

\begin{definition}
We define positive definite Riemannian metric in the following manner:
\bea
h(e_{\a}, e_{\b}) = \g(e_{\a}, e_{\b}) + 2 \g(e_{\a},  \hat{ \frac{\pr}{\pr t}} ) . \g(e_{\b}, \hat{ \frac{\pr}{\pr t}} ) \label{positiveriemannianmetrich}
\eea

where
\bea
\hat{ \frac{\pr}{\pr t}} = (- \g( \frac{\pr}{\pr t} , \frac{\pr}{\pr t}  ) )^{-\frac{1}{2}} \frac{\pr}{\pr t} \\ \notag
\eea
\end{definition}

\begin{definition}
For any ${\cal G}$-valued 2-tensor $K$, we let
\bea
|K|^{2} =  h_{\a\mu} h_{\b\nu} |K^{\mu\nu}|. |K^{\a\b}| \\ \notag
\eea
\end{definition}

\begin{lemma}
We have,
\bea
\der_{\si} h(e_{\a}, e_{\b}) &=& 2     \g(e_{\a}, \der_{\si} \hat{ \frac{\pr}{\pr t}} ) . \g(e_{\b},\hat{ \frac{\pr}{\pr t}} )  + 2  \g( e_{\a},  \hat{ \frac{\pr}{\pr t}} ) . \g(e_{\b}, \der_{\si} \hat{ \frac{\pr}{\pr t}} )   \label{derivativeofthemetrich}
\eea
\end{lemma}

\begin{proof}

\beaa
\der_{\si} h(e_{\a}, e_{\b}) &=& \pa_{\si} h(e_{\a}, e_{\b})  -   h(\der_{\si}e_{\a}, e_{\b}) -  h(e_{\a}, \der_{\si}e_{\b})  \\
&=& \der_{\si} \g(e_{\a}, e_{\b}) + 2 \der_{\si} [\g(e_{\a}, \hat{ \frac{\pr}{\pr t}} ) . \g(e_{\b},\hat{ \frac{\pr}{\pr t}} ) ] \\
&=& 2 \pa_{\si} [\g(e_{\a}, \hat{ \frac{\pr}{\pr t}} ) . \g(e_{\b},\hat{ \frac{\pr}{\pr t}} ) ] -  2  \g(\der_{\si}e_{\a},  \hat{ \frac{\pr}{\pr t}} ) . \g(e_{\b}, \hat{ \frac{\pr}{\pr t}} )  \\
&& - 2 \g( e_{\a},  \hat{ \frac{\pr}{\pr t}} ) . \g( \der_{\si} e_{\b}, \hat{ \frac{\pr}{\pr t}} ) 
\eeaa
(since the metric $\g$ is Killing)
\beaa
&=& 2  \pa_{\si} \g(e_{\a}, \hat{ \frac{\pr}{\pr t}} ) . \g(e_{\b}, \hat{ \frac{\pr}{\pr t}} )  +  2 \g(e_{\a},  \hat{ \frac{\pr}{\pr t}} ) . \pa_{\si}  \g(e_{\b}, \frac{\pr}{\pr \hat{t}} )  \\
&& -  2  \g(\der_{\si}e_{\a},  \hat{ \frac{\pr}{\pr t}} ) . \g(e_{\b}, \hat{ \frac{\pr}{\pr t}} )   - 2  \g(e_{\a},  \hat{ \frac{\pr}{\pr t}} ) . \g( \der_{\si}  e_{\b}, \hat{ \frac{\pr}{\pr t}} )  \\
&=& 2  [ \pa_{\si} \g(e_{\a},  \hat{ \frac{\pr}{\pr t}} )   -    \g(\der_{\si}e_{\a}, \hat{ \frac{\pr}{\pr t}} )] . \g(e_{\b},\hat{ \frac{\pr}{\pr t}} ) \\
&& + 2 [  \pa_{\si} \g(e_{\b},  \hat{ \frac{\pr}{\pr t}} )    -  \g(\der_{\si} e_{\b}, \hat{ \frac{\pr}{\pr t}} ) ]. \g(e_{\a},\hat{ \frac{\pr}{\pr t}} )  
\eeaa

Using the fact that $\der g = 0$, we get,

\beaa
\der_{\si} h(e_{\a}, e_{\b}) &=& 2     \g(e_{\a}, \der_{\si} \hat{ \frac{\pr}{\pr t}} ) . \g(e_{\b}, \hat{ \frac{\pr}{\pr t}} )  + 2  \g( e_{\a},  \hat{ \frac{\pr}{\pr t}} ) . \g(e_{\b}, \der_{\si} \hat{ \frac{\pr}{\pr t}} )  
\eeaa

\end{proof}

Let,
\beaa
\hat{t}_{\a} = ( \hat{ \frac{\pr}{\pr t}} )_{\a} = \g_{\mu\a}  ( \hat{ \frac{\pr}{\pr t}} )^{\mu}
\eeaa
Hence, we can write \eqref{positiveriemannianmetrich} as,
\bea
h_{\a\b} = \g_{\a\b} + 2 (\hat{ \frac{\pr}{\pr t}})_{\a} (\hat{ \frac{\pr}{\pr t}} )_{\b}
\eea
and \eqref{derivativeofthemetrich} as,
\bea
\der_{\si} h_{\a\b} &=& 2  [ \der_{\si} \hat{t}_{\a} . \hat{t}_{\b} +    \hat{t}_{\a}  .  \der_{\si} \hat{t}_{\b} ] \\ \notag
\eea

\subsection{The region $\om \geq 1$, $r \geq R$}\

We consider the region $w \geq 1$, $r \geq R$, where $R$ is fixed.\\

Let $\Psi$ be a tensor, and $|\Psi_{\mu\nu}| = <\Psi_{\mu\nu}, \Psi_{\mu\nu}> ^{\frac{1}{2}} $. We can compute

\beaa
| \der |\Psi_{\mu\nu} | | &=& | \frac{2 <\Lie \Psi_{\mu\nu}, \Psi_{\mu\nu}>  }{ 2<\Psi_{\mu\nu}, \Psi_{\mu\nu}> ^{\frac{1}{2}}} | \leq \frac{|\Lie \Psi_{\mu\nu}| |\Psi_{\mu\nu}|}{| \Psi_{\mu\nu}|} \\
&\leq& |\Lie \Psi_{\mu\nu} | 
\eeaa

We have the Sobolev inequality,
\beaa
r^{2} |F|^{2}  \lesssim   \int_{\S^{2}} r^{2} |F|^{2} d\sigma^{2} +  \int_{\S^{2}} r^{2}  | \rder |F| |^{2} d\sigma^{2}  +  \int_{\S^{2}} r^{2} | \rder\rder |F||^{2} d\sigma^{2} 
\eeaa
where $\rder$ is the covariant derivative restricted on the 2-spheres. We have

\beaa
| \rLie h_{\a\b}|^{2} &=& \frac{1}{r^{2}} \sum_{j=1}^{3} | \Om_{j} h_{\a\b}|^{2} = \frac{1}{r^{2}} \sum_{i=1}^{3}  | \Om_{j} \g(e_{\a}, e_{\b}) + 2 \Om_{j} \g(e_{\a},  \hat{ \frac{\pr}{\pr t}} ) . \g(e_{\b}, \hat{ \frac{\pr}{\pr t}} ) | \\
&=& 0
\eeaa

 Since $\rLie h_{\mu\nu} = 0$, we have
\beaa
r^{2} |F|^{2}  \lesssim \int_{\S^{2}} r^{2} |F|^{2} d\sigma^{2} +  \int_{\S^{2}}  r^{2} |  \rLie F |^{2} d\sigma^{2}  +  \int_{\S^{2}} r^{2} | \rLie \rLie F|^{2} d\sigma^{2} 
\eeaa

Let, $r_{F}$ be a value of $r$ such that $ R \leq r_{F} \leq R+1 $, and to be determined later. we have, 
\beaa
\int_{\S^{2}} r^{2} |F|^{2}(t, r,  \om)  d\sigma^{2} &\lesssim&  \int_{\S^{2}} r^{2} |F|^{2}(t, r_{F}, \om)  d\sigma^{2}   +   \int_{\S^{2}} \int_{\overline{r}^{*} = r^{*}_{F}}^{\overline{r}^{*} = r^{*} } \der_{r^{*}} [r^{2} |F|^{2}] (t, r, \om)  d\overline{r}^{*} d\sigma^{2} \\
&\lesssim&  \int_{\S^{2}} r_{F}^{2} |F|^{2}(t, r_{F}, \om)  d\sigma^{2}   +   \int_{\S^{2}} \int_{\overline{r}^{*} = r^{*}_{F}}^{\overline{r}^{*} = r^{*} } 2r |F|^{2} (t, r, \om) (1-\mu)  d\overline{r}^{*} d\sigma^{2}  \\
&&+  2   \int_{\S^{2}} \int_{\overline{r}^{*} = r^{*}_{F}}^{\overline{r}^{*} = r^{*} } r^{2}  \der_{r^{*}} |F|^{2} (t, r, \om)  d\overline{r}^{*} d\sigma^{2}  \\
\eeaa
By the same,
\beaa
&&\int_{\S^{2}} r^{2} |\rLie F |^{2}(t, r, \om)  d\sigma^{2} \\
&& \lesssim  \int_{\S^{2}} r_{\rLie F}^{2} | \rLie F |^{2}(t, r_{\rLie F}, \om)  d\sigma^{2}   +   \int_{\S^{2}} \int_{\overline{r}^{*} = r^{*}_{\rLie F}}^{\overline{r}^{*} = r^{*} } 2r |\rLie F |^{2} (t, r, \om) (1-\mu)  d\overline{r}^{*} d\sigma^{2}  \\
&&+  2   \int_{\S^{2}} \int_{\overline{r}^{*} = r^{*}_{\rLie F}}^{\overline{r}^{*} = r^{*} } r^{2}  \der_{r^{*}} | \rLie F |^{2} (t, r, \om)  d\overline{r}^{*} d\sigma^{2}  \\
\eeaa
and,
\beaa
 &&\int_{\S^{2}} r^{2} | \rLie \rLie F |^{2}(t, r, \om)  d\sigma^{2}\\
&& \lesssim  \int_{\S^{2}} r_{\rLie \rLie F}^{2} | \rLie \rLie F |^{2}(t, r_{\rLie \rLie F}, \om)  d\sigma^{2}   +   \int_{\S^{2}} \int_{\overline{r}^{*} = r^{*}_{\rLie \rLie F}}^{\overline{r}^{*} = r^{*} } 2r | \rLie \rLie F |^{2} (t, r, \om) (1-\mu)  d\overline{r}^{*} d\sigma^{2}  \\
&&+  2   \int_{\S^{2}} \int_{\overline{r}^{*} = r^{*}_{\rLie \rLie F}}^{\overline{r}^{*} = r^{*} } r^{2}  \der_{r^{*}} | \rLie  \rLie F |^{2} (t, r, \om)  d\overline{r}^{*} d\sigma^{2}  \\
\eeaa

We showed the following estimate,

\beaa
&&\int_{r^{*}= r_{1}^{*}   }^{r^{*} = r_{2}^{*}   } \int_{\S^{2}} (  | \Psi_{\hat{w}\hat{\th}} |^{2} +  | \Psi_{\hat{w}\hat{\phi}} |^{2}   +   | \Psi_{\hat{v}\hat{\th}} |^{2} +  | \Psi_{\hat{v}\hat{\phi}} |^{2}  +  |\Psi_{\hat{v}\hat{w}}|^{2}   +   | \Psi_{\hat{\phi}\hat{\th}}|^{2} ). (1-\mu) r^{2}   d\sigma^{2} dr^{*} (t) \\
& \lesssim & \frac{E_{\Psi}^{(K)}(t)}{\min_{w \in \{t\}\cap \{  r_{1}^{*} \leq r^{*} \leq r_{2}^{*} \}  } w^{2}}  + \frac{E_{\Psi}^{(K)}(t)}{\min_{v \in \{t\}\cap  \{  r_{1}^{*} \leq r^{*} \leq r_{2}^{*} \}  } v^{2}}
\eeaa

Thus,

$$ \int_{r^{*}= r_{1}^{*}   }^{r^{*} = r_{2}^{*}   }    \int_{\S^{2}}  |\Psi |^{2} (t, \overline{r}, \om) (1-\mu) r^{2} d\sigma^{2}  d\overline{r}^{*}  \lesssim \frac{E_{\Psi}^{(K)}(t)}{\min_{w \in \{t\}\cap \{  r_{1}^{*} \leq r^{*} \leq r_{2}^{*} \}  } w^{2}}  + \frac{E_{\Psi}^{(K)}(t)}{\min_{v \in \{t\}\cap  \{  r_{1}^{*} \leq r^{*} \leq r_{2}^{*} \}  } v^{2}} $$

Therefore,

$$ \int_{\overline{r}^{*} = R^{*} }^{\overline{r}^{*} = (R+1)^{*} }    \int_{\S^{2}}  |\Psi |^{2} (t, \overline{r}, \om) (1-\mu) r^{2} d\sigma^{2}  d\overline{r}^{*}  \lesssim \frac{  E_{\Psi}^{(K)} (t)}{t^{2}} $$
or,
$$ \int_{\overline{r} = R }^{\overline{r} = R+1 }    \int_{\S^{2}}  |\Psi |^{2} (t, \overline{r}, \om) r^{2} d\sigma^{2}  d\overline{r}  \lesssim \frac{  E_{\Psi}^{(K)} (t)}{t^{2}} $$

There exists $r_{\Psi}$, such that $R \leq r_{\Psi} \leq R+1$ and,

\beaa
 \int_{\S^{2}}  r_{F}^{2} |F |^{2} (t, r_{F}, \om) d\sigma^{2}   &\lesssim& \frac{  E_{F}^{(K)} (t)}{t^{2}(R+1-R)} 
\eeaa
which gives,
\beaa
\int_{\S^{2}}  r_{F}^{2} |F |^{2} (t, r_{F}, \om) d\sigma^{2} &\lesssim& \frac{ E_{F}^{(K)} (t)}{t^{2}} 
\eeaa

By same,

\beaa
\int_{\S^{2}}  r_{\rLie  F}^{2} | \rLie  F |^{2} (t, r_{\rLie  F}, \om) d\sigma^{2} &=&  \sum_{j=1}^{3} \frac{1}{r_{\rLie  F}^{2} } \int_{\S^{2}}  r_{\rLie  F}^{2}  | \Lie_{\Om_{j} } F |^{2} (t, r_{\rLie  F}, \om) d\sigma^{2} \\
&\les& \sum_{j=1}^{3} \frac{1}{R^{2} } \int_{\S^{2}}  r_{\rLie  F}^{2}  | \Lie_{\Om_{j} } F |^{2} (t, r_{\rLie  F}, \om) d\sigma^{2} \\
&\lesssim& \sum_{j=1}^{3} \frac{ E_{\Lie_{\Om_{j} } F}^{(K)} (t)}{t^{2}} 
\eeaa
and,
\beaa
\int_{\S^{2}}  r_{\rLie \rLie F}^{2} |\rLie \rLie F |^{2} (t, r_{0}, \om) d\sigma^{2} &\lesssim& \sum_{i=1}^{3} \sum_{j=1}^{3} \frac{ E_{\Lie_{\Om_{i} } \Lie_{\Om_{j} } F}^{(K)} (t)}{t^{2}} 
\eeaa

On the other hand,

\beaa
 \int_{\overline{r}^{*} = r_{F}^{*} }^{\overline{r}^{*} = r^{*} }  \int_{\S^{2}}  \overline{r}  |F |^{2} (1-\mu) (t, \overline{r}, \om) d\sigma^{2}  d\overline{r}^{*}  &\lesssim&
\int_{\overline{r}^{*} = r_{F}^{*} }^{\overline{r}^{*} = r^{*} }  \frac{\overline{r}}{R}   \int_{\S^{2}}  \overline{r} |F |^{2} (1-\mu) (t, \overline{r}, \om) d\sigma^{2}  d\overline{r}^{*} \\
& \lesssim& \frac{  E_{F}^{(K)} (t)}{t^{2}}  + \frac{  E_{F}^{(K)} (t)}{w^{2}} 
\eeaa

Thus,
\beaa
\int_{\overline{r}^{*} = r_{F}^{*} }^{\overline{r}^{*} = r^{*} }  \int_{\S^{2}}  \overline{r} |F |^{2} (1-\mu) (t, \overline{r}, \om) d\sigma^{2}  d\overline{r}^{*}  &\lesssim& \frac{  E_{F}^{(K)} (t)}{t^{2}}  + \frac{  E_{F}^{(K)} (t)}{w^{2}} 
\eeaa

and,

\beaa
&& \int_{\overline{r}^{*} = r_{\rLie F}^{*} }^{\overline{r}^{*} = r^{*} }  \int_{\S^{2}}  \overline{r} |\rLie F |^{2} (1-\mu) (t, \overline{r}, \om) d\sigma^{2}  d\overline{r}^{*}  \\
&=& \sum_{j=1}^{3} \int_{\overline{r}^{*} = r_{\rLie F}^{*} }^{\overline{r}^{*} = r^{*} }  \int_{\S^{2}}  \overline{r} \frac{1}{ \overline{r}^{2} } | \Lie_{\Om_{j}} F |^{2} (1-\mu) (t, \overline{r}, \om) d\sigma^{2}  d\overline{r}^{*} \\
&\leq& \sum_{j=1}^{3}  \frac{1}{ {R }^{2} } \int_{\overline{r}^{*} = r_{\rLie F}^{*} }^{\overline{r}^{*} = r^{*} }  \int_{\S^{2}}  \overline{r}  | \Lie_{\Om_{j}} F |^{2} (1-\mu) (t, \overline{r}, \om) d\sigma^{2}  d\overline{r}^{*} \\
&\lesssim& \sum_{j=1}^{3} \frac{ E_{\Lie_{\Om_{j}} F}^{(K)} (t)}{t^{2}} + \frac{ E_{\Lie_{\Om_{j}}  F}^{(K)} (t)}{w^{2}} 
\eeaa

By same,

\beaa
\int_{\overline{r}^{*} = r_{0}^{*} }^{\overline{r}^{*} = r^{*} }  \int_{\S^{2}}  \overline{r} |\rLie \rLie F |^{2} (1-\mu) (t, \overline{r}, \om) d\sigma^{2}  d\overline{r}^{*}  &\lesssim& \sum_{i=1}^{3} \sum_{j=1}^{3} [ \frac{  E_{ \Lie_{\Om_{i}} \Lie_{\Om_{j}} F}^{(K)} (t)}{t^{2}} +  \frac{  E_{\Lie_{\Om_{i}} \Lie_{\Om_{j}}  F}^{(K)} (t)}{w^{2}}  ]
\eeaa

Now, we want to estimate the term,

\beaa
&& \int_{\S^{2}} \int_{\overline{r}^{*} = r^{*}_{F}}^{\overline{r}^{*} = r^{*} } r^{2}  \der_{r^{*}} |F|^{2} (t, r, \om)  d\overline{r}^{*} d\sigma^{2} \\
&=& \int_{\S^{2}} \int_{\overline{r}^{*} = r^{*}_{F}}^{\overline{r}^{*} = r^{*} } r^{2}  \der_{\hat{r^{*}}} |F|^{2} (t, r, \om)  \sqrt{(1-\mu)} d\overline{r}^{*} d\sigma^{2} \\
&=& \int_{\S^{2}} \int_{\overline{r}^{*} = r^{*}_{F}}^{\overline{r}^{*} = r^{*} } r^{2} [ \der_{\hat{r^{*}}} ( h^{\mu\a} h^{\nu\b} )  |F_{\mu\nu}| |F_{\a\b}| + h^{\mu\a} h^{\nu\b} \der_{\hat{r^{*}}} ( |F_{\mu\nu}| |F_{\a\b}| ) ] (t, r, \om)  \sqrt{(1-\mu)} d\overline{r}^{*} d\sigma^{2} 
\eeaa

We have,
\beaa
\der_{\hat{r^{*}}} h(e_{\a}, e_{\b}) &=& 2     \g(e_{\a}, \der_{\si} \frac{\pr}{\pr \hat{t}} ) . \g(e_{\b}, \frac{\pr}{\pr \hat{t}} )  + 2  \g( e_{\a},  \frac{\pr}{\pr \hat{t}} ) . \g(e_{\b}, \der_{\si} \frac{\pr}{\pr \hat{t}} )  
\eeaa
and
\beaa
\notag
\der_{\hat{r^{*}}} \hat{\frac{\pa  }{\pa t}}  &=&   0 
\eeaa

Hence,
\beaa
\der_{\hat{r^{*}}} h^{\a\b}  &=& \der_{\hat{r^{*}}} h(e^{\a}, e^{\b}) = 2     \g(e^{\a}, \der_{\hat{r^{*}}} \frac{\pr}{\pr \hat{t}} ) . \g(e^{\b}, \frac{\pr}{\pr \hat{t}} )  + 2  \g( e^{\a},  \frac{\pr}{\pr \hat{t}} ) . \g(e^{\b}, \der_{\hat{r^{*}}} \frac{\pr}{\pr \hat{t}} )  \\
&=& 0
\eeaa

Therefore,
\beaa
&& \int_{\S^{2}} \int_{\overline{r}^{*} = r^{*}_{0}}^{\overline{r}^{*} = r^{*} } r^{2}  \der_{r^{*}} |F|^{2} (t, r, \om)  d\overline{r}^{*} d\sigma^{2} \\
&=& \int_{\S^{2}} \int_{\overline{r}^{*} = r^{*}_{0}}^{\overline{r}^{*} = r^{*} } r^{2}  h^{\mu\a} h^{\nu\b} \der_{\hat{r^{*}}} ( |F_{\mu\nu}| |F_{\a\b}| )  (t, r, \om)  \sqrt{(1-\mu)} d\overline{r}^{*} d\sigma^{2} \\
&\les& \int_{\S^{2}} \int_{\overline{r}^{*} = r^{*}_{0}}^{\overline{r}^{*} = r^{*} } r^{2} [ | h^{\mu\a} h^{\nu\b}|  | \der_{\hat{r^{*}}} |F_{\mu\nu}| |.  |F_{\a\b}| +  | h^{\mu\a} h^{\nu\b}|   |F_{\mu\nu}| . | \der_{\hat{r^{*}}}  |F_{\a\b}| |  ]  (t, r, \om)  \sqrt{(1-\mu)} d\overline{r}^{*} d\sigma^{2} \\
&\les& \int_{\S^{2}} \int_{\overline{r}^{*} = r^{*}_{0}}^{\overline{r}^{*} = r^{*} } r^{2} [ | h^{\mu\a} h^{\nu\b}|  | \Lie_{\hat{r^{*}}}  F_{\mu\nu}| |.  |F_{\a\b}| +  | h^{\mu\a} h^{\nu\b}|   |F_{\mu\nu}| . | \Lie_{\hat{r^{*}}}  F_{\a\b}  |  ]  (t, r, \om)  \sqrt{(1-\mu)} d\overline{r}^{*} d\sigma^{2} \\
&\les& \int_{\S^{2}} \int_{\overline{r}^{*} = r^{*}_{0}}^{\overline{r}^{*} = r^{*} } r^{2} [  h^{\hat{\mu}\hat{\a}} h^{\hat{\nu}\hat{\b}}  | \Lie_{\hat{r^{*}}}  F_{\hat{\mu}\hat{\nu}}| |.  |F_{\hat{\a}\hat{\b}}| +   h^{\hat{\mu}\hat{\a}} h^{\hat{\nu}\hat{\b}}  |F_{\hat{\mu}\hat{\nu}}| . | \Lie_{\hat{r^{*}}}  F_{\hat{\a}\hat{\b}}  |  ]  (t, r, \om)  \sqrt{(1-\mu)} d\overline{r}^{*} d\sigma^{2} 
\eeaa

where $\hat{\mu}, \hat{\nu}, \hat{\a}, \hat{\b} \in \{ \hat{\frac{\pa}{\pa t }}, \hat{\frac{\pa}{\pa r^{*}}}, \hat{\frac{\pa}{\pa \th}},\hat{\frac{\pa}{\pa \phi }} \}$.\\

As we have (see \eqref{The compatible symmetric connection} in the Appendix),

\beaa
\notag
\der_{\hat{r^{*}}} \hat{\frac{\pa  }{\pa t}}  &=&   0 \\
\notag
\der_{\hat{r^{*}}} \hat{\frac{\pa}{\pa r^{*}}}  &=&   0 \\
\notag
\der_{\hat{r^{*}}} \frac{\pa  }{\pa \hat{\th} }  &=&   0 \\
\notag
\der_{\hat{r^{*}}} \frac{\pa  }{\pa \hat{\phi} }  &=&   0 
\eeaa

we get,
\beaa
\Lie_{\hat{r^{*}}}  F_{\hat{\mu}\hat{\nu}}  &=& \der_{\hat{r^{*}}}  F_{\hat{\mu}\hat{\nu}}
\eeaa

Consequently, using Cauchy-Schwarz, we obtain

\beaa
&& \int_{\S^{2}} \int_{\overline{r}^{*} = r^{*}_{0}}^{\overline{r}^{*} = r^{*} } r^{2}  \der_{r^{*}} |F|^{2} (t, r, \om)  d\overline{r}^{*} d\sigma^{2} \\
&\lesssim&  [\int_{\overline{r}^{*} = r^{*}_{0}}^{\overline{r}^{*} = r^{*} }  \int_{\S^{2}}  \overline{r}^{2}  | \der_{\hat{r^{*}}}F |^{2} (1-\mu)  d\sigma^{2}  d\overline{r}^{*} ]^{\frac{1}{2}}  .  [\int_{\overline{r}^{*} = r^{*}_{0}}^{\overline{r}^{*} = r^{*} }   \int_{\S^{2}}  \overline{r}^{2}  | F |^{2} d\sigma^{2}  d\overline{r}^{*} ]^{\frac{1}{2}}
\eeaa

We have,
\beaa
 \int_{\overline{r}^{*} = r_{0}^{*} }^{\overline{r}^{*} = r^{*} }  \int_{\S^{2}}  \overline{r}^{2} |F|^{2} (1-\mu) (t, \overline{r}, \om) d\sigma^{2}  d\overline{r}^{*}  & \lesssim& \frac{  E_{F}^{(K)} (t)}{t^{2}}  +  \frac{  E_{F}^{(K)} (t)}{w^{2}} 
\eeaa

Thus,
\beaa
[\int_{\overline{r}^{*} = r^{*}_{0}}^{\overline{r}^{*} = r^{*} }   \int_{\S^{2}}  \overline{r}^{2}  | F |^{2} d\sigma^{2}  d\overline{r}^{*} ]^{\frac{1}{2}} \lesssim \frac{  \sqrt{E_{F}^{(K)} (t)}}{t}  + \frac{  \sqrt{E_{F}^{(K)} (t)}}{w} 
\eeaa

Now, considering the case where  $\hat{\mu} = \hat{r^{*}}$, (or similarly if $\hat{\nu} = \hat{r^{*}}$), we can compute,
\beaa
{\der}^{\hat{r^{*}}} F_{\hat{r^{*}}\hat{\nu}} = - {\der}^{\hat{t}} F_{\hat{t}\hat{\nu}} - {\der}^{\hat{\th}} F_{\hat{\th}\hat{\nu}}  - {\der}^{\hat{\phi}} F_{\hat{\phi}\hat{\nu}}  
\eeaa
(since the Maxwell fields are divergence free)\\

Therefore,
\beaa
{\der}_{\hat{r^{*}}} F_{\hat{r^{*}}\hat{\nu}} =  {\der}_{\hat{t}} F_{\hat{t}\hat{\nu}} - {\der}_{\hat{\th}} F_{\hat{\th}\hat{\nu}}  - {\der}_{\hat{\phi}} F_{\hat{\phi}\hat{\nu}}  
\eeaa

And if both $\mu \ne \hat{r}$, and $\nu \ne \hat{r}$, we can compute,
\beaa
{\der}^{\hat{r^{*}}} F_{\hat{\mu}\hat{\nu}} = - {\der}_{\hat{\mu}} F_{\hat{\nu} \hat{r^{*}}} - {\der}_{\hat{\nu}} F_{\hat{r^{*}}\hat{\mu}}  
\eeaa
(by using the Bianchi identities)

and therefore,
\beaa
{\der}_{\hat{r^{*}}} F_{\hat{\mu}\hat{\nu}} = - {\der}_{\hat{\mu}} F_{\hat{\nu} \hat{r^{*}}} - {\der}_{\hat{\nu}} F_{\hat{r^{*}}\hat{\mu}}  
\eeaa

where $\hat{\mu}, \hat{\nu} \in \{ \hat{\frac{\pa}{\pa t }}, \hat{\frac{\pa}{\pa \th}},\hat{\frac{\pa}{\pa \phi }} \}$.\\

In all cases, we get,
\beaa
 |{\der}_{\hat{r^{*}}} F_{\hat{\mu}\hat{\nu}}|^{2} (1-\mu) &\les& [ |{\der}_{\hat{t}} F |^{2} + |\rder F|^{2}  ](1-\mu) \\
\eeaa
(using the triangular inequality and $a.b \les a^{2} + b^{2}$ )\\

Hence,
\beaa
 |{\der}_{\hat{r^{*}}} F |^{2} (1-\mu) &\les& [ |{\der}_{\hat{t}} F |^{2} + \frac{1}{r^{2}} \sum_{j=1}^{3} |\der_{\Om_{j}} F|^{2}  ](1-\mu)  \\
&\les&  |{\der}_{t} F |^{2} +  \sum_{j=1}^{3} |\der_{\Om_{j}}  F|^{2}  \\
&\les&  h^{\mu\a} h^{\nu\b} [ |{\der}_{t} F_{\mu\nu} |.|{\der}_{t} F_{\a\b} | +  \sum_{j=1}^{3} h^{\mu\a} h^{\nu\b} [ |{\der}_{\Om_{j}} F_{\mu\nu} |.|{\der}_{\Om_{j}} F_{\a\b} |
\eeaa
We have,
\beaa
&& |{\der}_{t} F |^{2} \\
&=& h^{\mu\a} h^{\nu\b} [ |{\der}_{t} F_{\mu\nu} |.|{\der}_{t} F_{\a\b} | \\
&=&  h^{\mu\a} h^{\nu\b} [ |{\Lie}_{t} F_{\mu\nu} - F(\der_{t} \frac{\pa}{\pa \mu}, \frac{\pa}{\pa \nu}) - F(\frac{\pa}{\pa \mu}, \der_{t} \frac{\pa}{\pa \nu})  |.|{\Lie}_{t} F_{\a\b} - F(\der_{t} \frac{\pa}{\pa \a}, \frac{\pa}{\pa \b}) - F(\frac{\pa}{\pa \a}, \der_{t} \frac{\pa}{\pa \b})  |  \\
&\les& |\Lie_{t} F |^{2} + | h^{\mu\a} h^{\nu\b}| | F(\der_{t} \frac{\pa}{\pa \mu}, \frac{\pa}{\pa \nu}) |  .| F(\der_{t} \frac{\pa}{\pa \a}, \frac{\pa}{\pa \b})   |  + | h^{\mu\a} h^{\nu\b}  | | F(\frac{\pa}{\pa \mu}, \der_{t} \frac{\pa}{\pa \nu})  | . | F(\frac{\pa}{\pa \a}, \der_{t} \frac{\pa}{\pa \b})  |  \\
\eeaa
(using Cauchy-Schwarz inequality and $a.b \les a^{2} + b^{2}$ )\\

Since, $ \frac{\pa}{\pa t} $ is a smooth vector field away from the horizon, choosing a system of coordinates to compute the contractions above, we get
\beaa
 |{\der}_{t} F |^{2} (1-\mu) &\les&  |\Lie_{t} F |^{2} +  | F|^{2}  
\eeaa

Similarly, we have,
\beaa
&& |{\der}_{\Om_{j}} F |^{2} \\
&\les& |\Lie_{\Om_{j}} F |^{2} + | h^{\mu\a} h^{\nu\b}| | F(\der_{\Om_{j}} \frac{\pa}{\pa \mu}, \frac{\pa}{\pa \nu}) |  .| F(\der_{\Om_{j}} \frac{\pa}{\pa \a}, \frac{\pa}{\pa \b})   |  \\
&& + | h^{\mu\a} h^{\nu\b}  | | F(\frac{\pa}{\pa \mu}, \der_{\Om_{j}} \frac{\pa}{\pa \nu})  | . | F(\frac{\pa}{\pa \a}, \der_{\Om_{j}} \frac{\pa}{\pa \b})  |  \\
\eeaa

Since $\Om_{j}$, $j \in \{1, 2, 3 \}$ are smooth vector fields, computing the contractions above in a system of coordinates, we obtain
\beaa
\sum_{j=1}^{3} |{\der}_{\Om_{j}} F_{\hat{\mu}\hat{\nu}}|^{2}  &\les&   \sum_{j=1}^{3} |\Lie_{\Om_{j}}  F|^{2}  +  | F|^{2}
\eeaa

Finally, we have
\bea
 |{\der}_{\hat{r^{*}}} F |^{2} (1-\mu) &\les&  | F|^{2}  + |\Lie_{t} F |^{2} +   \sum_{j=1}^{3} |\Lie_{\Om_{j}}  F|^{2} 
\eea
Thus,
\beaa
&& \int_{\overline{r}^{*} = r^{*}_{0}}^{\overline{r}^{*} = r^{*} }   \int_{\S^{2}} \overline{r}^{2} | \Lie_{r^{*}}F |^{2} d\sigma^{2}  d\overline{r}^{*} \\
&\les& \int_{\overline{r}^{*} = r^{*}_{0}}^{\overline{r}^{*} = r^{*} }   \int_{\S^{2}} \overline{r}^{2} [  |{\Lie}_{t} F |^{2} + |\rLie F|^{2} + |F|^{2} ]  d\sigma^{2}  d\overline{r}^{*} \\
 &\les& \int_{\overline{r}^{*} = r^{*}_{0}}^{\overline{r}^{*} = r^{*} }   \int_{\S^{2}} \overline{r}^{2}  [ |F|^{2} + |{\Lie}_{t} F |^{2}  ] d\sigma^{2}  d\overline{r}^{*} +  \sum_{j=1}^{3} \int_{\overline{r}^{*} = r^{*}_{0}}^{\overline{r}^{*} = r^{*} }   \int_{\S^{2}}   |\Lie_{\Om_{j}} F|^{2}    d\sigma^{2}  d\overline{r}^{*} \\
&\les& \int_{\overline{r}^{*} = r^{*}_{0}}^{\overline{r}^{*} = r^{*} }   \int_{\S^{2}} \overline{r}^{2}  [ |F|^{2} + |{\Lie}_{t} F |^{2}   ] d\sigma^{2}  d\overline{r}^{*} +  \sum_{j=1}^{3} \int_{\overline{r}^{*} = r^{*}_{0}}^{\overline{r}^{*} = r^{*} }   \int_{\S^{2}}  \frac{\overline{r}^{2}}{R^{2}} |\Lie_{\Om_{j}} F|^{2}    d\sigma^{2}  d\overline{r}^{*} \\
&\lesssim& \frac{E_{ F}^{(K)} (t) + E_{ \Lie_{t} F}^{(K)} (t)  + \sum_{j=1}^{3} E_{ \Lie_{\Om_{j}} F}^{(K)} (t) }{t^{2}} + \frac{E_{ F}^{(K)} (t) + E_{ \Lie_{t} F}^{(K)} (t)  + \sum_{j=1}^{3} E_{ \Lie_{\Om_{j}}  F}^{(K)} (t) }{w^{2}} 
\eeaa

and therefore,
\beaa
&& [\int_{\overline{r}^{*} = r^{*}_{0}}^{\overline{r}^{*} = r^{*} }  \int_{\S^{2}}  \overline{r}^{2}  | \Lie_{r^{*}}F |^{2} d\sigma^{2}  d\overline{r}^{*} ]^{\frac{1}{2}}  \\
&\lesssim& \frac{\sqrt{E_{ F}^{(K)} (t) + E_{ \Lie_{t} F}^{(K)} (t) + \sum_{j=1}^{3} E_{ \Lie_{\Om_{j}} F}^{(K)} (t) }}{|t|} + \frac{\sqrt{E_{ F}^{(K)} (t) + E_{ \Lie_{t} F}^{(K)} (t) + \sum_{j=1}^{3} E_{ \Lie_{\Om_{j}} F}^{(K)} (t) }}{|w|}
\eeaa

Thus,
\beaa
&& \int_{\overline{r}^{*} = r^{*}_{0}}^{\overline{r}^{*} = r^{*} } \int_{\S^{2}}  \overline{r}^{2}  <\Lie_{r^{*}}F, F>_{h} (t, r, \om) d\sigma^{2}  d\overline{r}^{*}  \\
&\lesssim& ( \frac{\sqrt{E_{  F}^{(K)} (t) + E_{ \Lie_{t} F}^{(K)} (t) + \sum_{j=1}^{3} E_{ \Lie_{\Om_{j}} F}^{(K)} (t) }}{|t|} + \frac{\sqrt{E_{ F}^{(K)} (t) + E_{ \Lie_{t} F}^{(K)} (t) + \sum_{j=1}^{3} E_{ \Lie_{\Om_{j}} F}^{(K)} (t) }}{|w|} ) \\
&& . ( \frac{  \sqrt{E_{F}^{(K)} (t)}}{t}  + \frac{  \sqrt{E_{F}^{(K)} (t)}}{w} ) \\
&\lesssim& \frac{E_{  F}^{(K)} (t) + E_{ \Lie_{t} F}^{(K)} (t) + \sum_{j=1}^{3} E_{ \Lie_{\Om_{j}} F}^{(K)} (t) }{t^{2}}  + \frac{E_{  F}^{(K)} (t)  + E_{ \Lie_{t} F}^{(K)} (t) +  \sum_{j=1}^{3} E_{ \Lie_{\Om_{j}} F}^{(K)} (t) }{w^{2}}
\eeaa

By same,

\beaa
&& \int_{\overline{r}^{*} = r^{*}_{\rLie F}}^{\overline{r}^{*} = r^{*} } \int_{\S^{2}}  \overline{r}^{2}  <\Lie_{r^{*}}\rLie F, \rLie F>_{h}  (t, r, \om) d\sigma^{2}  d\overline{r}^{*}  \\ &\lesssim&  \sum_{i=1}^{3} \sum_{j=1}^{3} [\frac{  E_{ \Lie_{\Om_{j}} F, \Lie_{t}  \Lie_{\Om_{j}} F,  \Lie_{\Om_{i}} \Lie_{\Om_{j}} F }^{(K)} (t)}{ t^{2} }  + \frac{  E_{ \Lie_{\Om_{j}} F, \Lie_{t}  \Lie_{\Om_{j}} F,  \Lie_{\Om_{i}} \Lie_{\Om_{j}} F }^{(K)} (t)}{ w^{2} } ]
\eeaa 

and,
\beaa
&& \int_{\overline{r}^{*} = r^{*}_{\rLie \rLie F}}^{\overline{r}^{*} = r^{*} } \int_{\S^{2}}  \overline{r}^{2}  <\Lie_{r^{*}}\rLie \rLie F, \rLie \rLie F >_{h}  (t, r, \om) d\sigma^{2}  d\overline{r}^{*} \\
 &\lesssim&  \sum_{l=1}^{3} \sum_{i=1}^{3} \sum_{j=1}^{3} [ \frac{  E_{ \Lie_{\Om_{i}} \Lie_{\Om_{j}}  F, \Lie_{t} \Lie_{\Om_{i}} \Lie_{\Om_{j}}  F, \Lie_{\Om_{l}} \Lie_{\Om_{i}} \Lie_{\Om_{j}}  F }^{(K)} (t)}{t^{2} }  +  \frac{  E_{ \Lie_{\Om_{i}} \Lie_{\Om_{j}}  F, \Lie_{t} \Lie_{\Om_{i}} \Lie_{\Om_{j}}  F, \Lie_{\Om_{l}} \Lie_{\Om_{i}} \Lie_{\Om_{j}}  F }^{(K)} (t)}{w^{2} } ]
\eeaa

Using the fact that,

\beaa
\sum_{j = 1 }^{3} | \Lie_{\Om_{j}  } \Psi |^{2} = r^{2} |\rLie \Psi |^{2} =  | r \rLie \Psi |^{2}
\eeaa

We have,

\beaa
| r \rLie r \rLie \Psi |^{2} &=& r^{2} [ r^{2}| \rLie  \rLie \Psi  |^{2} ] \\
&=& r^{2} [ \sum_{j=1}^{3} | \Lie_{\Om_{j}}  \rLie \Psi  |^{2} ] \\
&=&   \sum_{j=1}^{3} r^{2}  | \rLie   \Lie_{\Om_{j}} \Psi  |^{2}  \\
&=&    \sum_{i=1}^{3} \sum_{j=1}^{3}   | \Lie_{\Om_{i}}   \Lie_{\Om_{j}} \Psi  |^{2}  \\
\eeaa

Finally, we obtain,

\beaa
&& r^{2} |F |^{2} (t, r, \om)  \\
&\lesssim& \frac{  E_{ F, \Lie_{t}  F, r \rLie F }^{(K)} (t)}{ t^{2} }  + \frac{  E_{ F, \Lie_{t}  F, r \rLie F }^{(K)} (t)}{ w^{2} } \\
&&+ \frac{  E_{r \rLie F, \Lie_{t} r \rLie F, r^{2} (\rLie)^{2} F }^{(K)} (t)}{ t^{2} }  + \frac{  E_{r{\rLie} F, \Lie_{t} r{\rLie} F, r^{2} (\rLie)^{2} F }^{(K)} (t)}{ w^{2} }  \\
&&+ \frac{  E_{ r^{2} (\rLie)^{2} F, \Lie_{t} r^{2} (\rLie)^{2} F, r^{3} (\rLie)^{3} F }^{(K)} (t)}{t^{2} }  +  \frac{  E_{ r^{2} (\rLie)^{2} F, \Lie_{t} r^{2}(\rLie)^{2} F, r^{3} (\rLie)^{3} F }^{(K)} (t)}{w^{2} }
\eeaa

Since what is on the left hand side of the previous inequality is a contraction, it can be computed in the basis  $\{\hat{\frac{\pa}{\pa w}}, \hat{\frac{\pa}{\pa v}}, \hat{\frac{\pa}{\pa \th}}, \hat{\frac{\pa}{\pa \phi}} \}$. Thus,

\beaa
&& \sum_{ \hat{\mu}, \hat{\nu} \in \{\hat{\frac{\pa}{\pa w}}, \hat{\frac{\pa}{\pa v}}, \hat{\frac{\pa}{\pa \th}}, \hat{\frac{\pa}{\pa \phi}} \} } |F_{\hat{\mu}\hat{\nu}}|(w, v, \om)  \\
&\lesssim& \frac{ [ E_{  F, \Lie_{t}  F, r{\rLie} F, \Lie_{t} r{\rLie} F, r^{2} (\rLie)^{2} F, \Lie_{t} r^{2}(\rLie)^{2} F, r^{3}(\rLie)^{3} F}^{(K)} (t) ]^{\frac{1}{2}} }{ r t } \\
&& + \frac{  [ E_{  F, \Lie_{t}  F, r{\rLie} F, \Lie_{t} r{\rLie} F, r^{2}(\rLie)^{2} F, \Lie_{t} r^{2}(\rLie)^{2} F, r^{3}(\rLie)^{3} F}^{(K)} (t) ]^{\frac{1}{2}}  }{ r w } \\
\eeaa

For $R$ fixed, consider first the region where $t \geq 1 $, and thus we have $r + t \lesssim rt $. Consequently, $v = r^{*} + t \lesssim 2r +t \lesssim 2r + 2t \lesssim rt $, and $v \lesssim t + r \lesssim r + t - r^{*} \lesssim r + w \lesssim rw$ (since $w \geq 1$).\\

For $t \leq 1$, the region,  $\om \geq 1$, $r \geq R$, $t \leq 1 $ is a bounded compact region, and therefore, in this region
\beaa
| F_{\hat{\mu}\hat{\nu}}|(w, v, \om) &\les& E_{ F }^{(\frac{\pa}{\pa t})} (t=t_{0}) + E_{ \der F }^{(\frac{\pa}{\pa t})} (t=t_{0}) \\
&& \text{(see [G])} \\
&\les&  E_{ F, \Lie_{t} F, r \rLie F }^{(\frac{\pa}{\pa t})} (t=t_{0})
\eeaa

Thus, we have,

\beaa
 |F_{\hat{\mu}\hat{\nu}}|(w, v, \om) &\lesssim& \frac{  [ E_{  F, \Lie_{t}  F, r{\rLie} F, \Lie_{t} r{\rLie} F, r^{2}(\rLie)^{2} F, \Lie_{t} r^{2}(\rLie)^{2} F, r^{3}(\rLie)^{3} F}^{(K)} (t) ]^{\frac{1}{2}}  }{ (1 + |v|) } \\
&\lesssim& \frac{  E_{F} }{ (1 + |v|) } \\
\eeaa

\subsection{The region $w \leq -1$, $r \geq R$, $|t| \geq 1$ }\

Again, we have the Sobolev inequality

\beaa
r^{2} |F_{\hat{\mu}\hat{\nu}}|^{2}  \lesssim \int_{\S^{2}} r^{2} |F|^{2} d\sigma^{2} +  \int_{\S^{2}} r^{2} | \rLie F |^{2} d\sigma^{2}  +  \int_{\S^{2}} r^{2} |\rLie \rLie F|^{2} d\sigma^{2} 
\eeaa

We have, 
\beaa
\int_{\S^{2}} r^{2} |F|^{2}(t, r,  \om)  d\sigma^{2} &\lesssim&  \int_{\S^{2}} r^{2} |F|^{2}(t, \infty, \om)  d\sigma^{2}   +   \int_{\S^{2}} \int_{\overline{r}^{*} = r^{*}}^{\overline{r}^{*} = \infty } \der_{r^{*}} [r^{2} |F |^{2}] (t, r, \om)  d\overline{r}^{*} d\sigma^{2} \\
&\lesssim&    \int_{\S^{2}} \int_{\overline{r}^{*} }^{\overline{r}^{*} = \infty } 2r |F |^{2} (t, r, \om) (1-\mu)  d\overline{r}^{*} d\sigma^{2}  \\
&&+  2   \int_{\S^{2}} \int_{\overline{r}^{*} = r^{*} }^{\overline{r}^{*} = \infty } r^{2}  <\der_{r^{*}}F, F>_{h} (t, r, \om)  d\overline{r}^{*} d\sigma^{2}  \\
\eeaa
By the same,
\beaa
&&\int_{\S^{2}} r^{2} |\rLie F |^{2}(t, r, \om)  d\sigma^{2} \\
& \lesssim&     \int_{\S^{2}} \int_{\overline{r}^{*} = r^{*} }^{\overline{r}^{*} = \infty } 2r |\rLie F|^{2} (t, r, \om) (1-\mu)  d\overline{r}^{*} d\sigma^{2}  \\
&& +  2   \int_{\S^{2}} \int_{\overline{r}^{*} = r^{*} }^{\overline{r}^{*} = \infty } r^{2}  <\der_{r^{*}} \rLie F, \rLie F >_{h} (t, r, \om)  d\overline{r}^{*} d\sigma^{2}  \\
\eeaa
and,
\beaa
 &&\int_{\S^{2}} r^{2} | \rLie \rLie F |^{2}(t, r, \om)  d\sigma^{2}\\
& \lesssim &  \int_{\S^{2}} \int_{\overline{r}^{*} = r^{*}  }^{\overline{r}^{*} = \infty } 2r | \rLie \rLie F|^{2} (t, r, \om) (1-\mu)  d\overline{r}^{*} d\sigma^{2}  \\
&&+  2   \int_{\S^{2}} \int_{\overline{r}^{*} = r^{*} }^{\overline{r}^{*} = \infty } r^{2}  <\der_{r^{*}} \rLie  \rLie F,  \rLie \rLie F >_{h} (t, r, \om)  d\overline{r}^{*} d\sigma^{2}  \\
\eeaa

We have shown,

\beaa
&&\int_{r^{*}= r_{1}^{*}   }^{r^{*} = r_{2}^{*}   } \int_{\S^{2}} (  | \Psi_{\hat{w}\hat{\th}} |^{2} +  | \Psi_{\hat{w}\hat{\phi}} |^{2}   +   | \Psi_{\hat{v}\hat{\th}} |^{2} +  | \Psi_{\hat{v}\hat{\phi}} |^{2}  +  |\Psi_{\hat{v}\hat{w}}|^{2}   +   | \Psi_{\hat{\phi}\hat{\th}}|^{2} ). (1-\mu) r^{2}   d\sigma^{2} dr^{*} (t) \\
& \lesssim & \frac{E_{\Psi}^{(K)}(t)}{\min_{w \in \{t\}\cap \{  r_{1}^{*} \leq r^{*} \leq r_{2}^{*} \}  } w^{2}}  + \frac{E_{\Psi}^{(K)}(t)}{\min_{v \in \{t\}\cap  \{  r_{1}^{*} \leq r^{*} \leq r_{2}^{*} \}  } v^{2}}
\eeaa

Thus,

\beaa
 \int_{r^{*}= r_{1}^{*}   }^{r^{*} = r_{2}^{*}   }    \int_{\S^{2}}  |\Psi |^{2} (t, \overline{r}, \om) (1-\mu) r^{2} d\sigma^{2}  d\overline{r}^{*}  &\lesssim& \frac{E_{\Psi}^{(K)}(t)}{\min_{w \in \{t\}\cap \{  r_{1}^{*} \leq r^{*} \leq r_{2}^{*} \}  } w^{2}}  + \frac{E_{\Psi}^{(K)}(t)}{\min_{v \in \{t\}\cap  \{  r_{1}^{*} \leq r^{*} \leq r_{2}^{*} \}  } v^{2}} 
\eeaa

We have,

\beaa
 \int_{\overline{r}^{*} = r^{*} }^{\overline{r}^{*} = \infty }  \int_{\S^{2}}  \overline{r}  |F |^{2} (1-\mu) (t, \overline{r}, \om) d\sigma^{2}  d\overline{r}^{*}  &\lesssim&
\int_{\overline{r}^{*} = r^{*} }^{\overline{r}^{*} = \infty }  \frac{\overline{r}}{R}   \int_{\S^{2}}  \overline{r} |F |^{2} (1-\mu) (t, \overline{r}, \om) d\sigma^{2}  d\overline{r}^{*} \\
& \lesssim& \frac{  E_{F}^{(K)} (t)}{t^{2}}  + \frac{  E_{F}^{(K)} (t)}{w^{2}} 
\eeaa

Thus,
\beaa
\int_{\overline{r}^{*} = r^{*} }^{\overline{r}^{*} = \infty }  \int_{\S^{2}}  \overline{r} |F |^{2} (1-\mu) (t, \overline{r}, \om) d\sigma^{2}  d\overline{r}^{*}  &\lesssim& \frac{  E_{F}^{(K)} (t)}{t^{2}}  + \frac{  E_{F}^{(K)} (t)}{w^{2}} 
\eeaa

and,

\beaa
 \int_{\overline{r}^{*} = r^{*} }^{\overline{r}^{*} = \infty }  \int_{\S^{2}}  \overline{r} |\rLie F |^{2} (1-\mu) (t, \overline{r}, \om) d\sigma^{2}  d\overline{r}^{*} &=&  \sum_{j=1}^{3} \int_{\overline{r}^{*} = r^{*} }^{\overline{r}^{*} = \infty }  \int_{\S^{2}}  \overline{r} \frac{1}{\overline{r}^{2}} | \Lie_{\Om_{j}} F |^{2} (1-\mu) (t, \overline{r}, \om) d\sigma^{2}  d\overline{r}^{*} \\
&\les&  \sum_{j=1}^{3} \frac{1}{R^{2} } \int_{\overline{r}^{*} = r^{*} }^{\overline{r}^{*} = \infty }  \int_{\S^{2}}  \overline{r}  | \Lie_{\Om_{j}} F |^{2} (1-\mu) (t, \overline{r}, \om) d\sigma^{2}  d\overline{r}^{*} \\
&\lesssim& \sum_{j=1}^{3} [ \frac{ E_{\Lie_{\Om_{j}} F}^{(K)} (t)}{t^{2}} + \frac{ E_{\Lie_{\Om_{j}} F}^{(K)} (t)}{w^{2}} ]
\eeaa

By same,

\beaa
\int_{\overline{r}^{*} = r^{*} }^{\overline{r}^{*} = \infty }  \int_{\S^{2}}  \overline{r} |\rLie \rLie F |^{2} (1-\mu) (t, \overline{r}, \om) d\sigma^{2}  d\overline{r}^{*}  &\lesssim& \sum_{i=1}^{3} \sum_{j=1}^{3} [ \frac{  E_{\Lie_{\Om_{i}} \Lie_{\Om_{j}} F}^{(K)} (t)}{t^{2}} + \frac{  E_{\Lie_{\Om_{i}} \Lie_{\Om_{j}} F}^{(K)} (t)}{w^{2}} ]
\eeaa

Now, we can estimate the term,

\beaa
&&\int_{\overline{r}^{*} = r^{*} }^{\overline{r}^{*} = \infty } \int_{\S^{2}}  \overline{r}^{2}  <\der_{r^{*}}F, F >_{h} (t, r, \om) d\sigma^{2}  d\overline{r}^{*}  \\
&&\lesssim  [\int_{\overline{r}^{*} = r^{*} }^{\overline{r}^{*} = \infty }  \int_{\S^{2}}  \overline{r}^{2}  | \der_{r^{*}}F  |^{2} d\sigma^{2}  d\overline{r}^{*} ]^{\frac{1}{2}}  .  [\int_{\overline{r}^{*} = r^{*} }^{\overline{r}^{*} = \infty }   \int_{\S^{2}}  \overline{r}^{2}  | F  |^{2} d\sigma^{2}  d\overline{r}^{*} ]^{\frac{1}{2}}
\eeaa

We have,
\beaa
 \int_{\overline{r}^{*} = r^{*} }^{\overline{r}^{*} = \infty }  \int_{\S^{2}}  \overline{r}^{2} |F |^{2} (1-\mu) (t, \overline{r}, \om) d\sigma^{2}  d\overline{r}^{*}  & \lesssim& \frac{  E_{F}^{(K)} (t)}{t^{2}}  +  \frac{  E_{F}^{(K)} (t)}{w^{2}} 
\eeaa

Thus,
\beaa
[\int_{\overline{r}^{*} = r^{*} }^{\overline{r}^{*} = \infty }   \int_{\S^{2}}  \overline{r}^{2}  | F |^{2} d\sigma^{2}  d\overline{r}^{*} ]^{\frac{1}{2}} \lesssim \frac{  \sqrt{E_{F}^{(K)} (t)}}{t}  + \frac{  \sqrt{E_{F}^{(K)} (t)}}{w} 
\eeaa

We also have the estimate:

\beaa
&& \int_{\overline{r}^{*} = r^{*} }^{\overline{r}^{*} = \infty }   \int_{\S^{2}} \overline{r}^{2} | \der_{r^{*}}F |^{2} d\sigma^{2}  d\overline{r}^{*} \\
&\lesssim& \int_{\overline{r}^{*} = r^{*} }^{\overline{r}^{*} = \infty }   \int_{\S^{2}} \overline{r}^{2} ( |F |^{2} + |\Lie_{t} F |^{2} + |\rLie F|^{2} ) d\sigma^{2}  d\overline{r}^{*} \\
&\lesssim& \int_{\overline{r}^{*} = r^{*} }^{\overline{r}^{*} = \infty }   \int_{\S^{2}} \overline{r}^{2}  [|F |^{2} + |{\Lie}_{t} F |^{2} ] d\sigma^{2}  d\overline{r}^{*} +  \sum_{j=1}^{3} \int_{\overline{r}^{*} = r^{*} }^{\overline{r}^{*} = \infty }   \int_{\S^{2}} \overline{r}^{2} \frac{1}{\overline{r}^{2}  } |\Lie_{\Om_{j}} F|^{2}  d\sigma^{2}  d\overline{r}^{*} \\
&\lesssim& \int_{\overline{r}^{*} = r^{*} }^{\overline{r}^{*} = \infty }   \int_{\S^{2}} \overline{r}^{2} [|F |^{2} +  |{\Lie}_{t} F |^{2}  ]  d\sigma^{2}  d\overline{r}^{*} +   \sum_{j=1}^{3} \frac{1}{R^{2} } \int_{\overline{r}^{*} = r^{*} }^{\overline{r}^{*} = \infty }   \int_{\S^{2}} \overline{r}^{2} |\Lie_{\Om_{j}} F|^{2}  d\sigma^{2}  d\overline{r}^{*} \\
&\lesssim&   \frac{E_{ F}^{(K)} (t) + E_{ \Lie_{t} F}^{(K)} (t) +  \sum_{j=1}^{3} E_{ \Lie_{\Om_{j}} F}^{(K)} (t) }{t^{2}} + \frac{E_{F}^{(K)} (t) + E_{ \Lie_{t} F}^{(K)} (t) +  \sum_{j=1}^{3} E_{ \Lie_{\Om_{j}} F}^{(K)} (t) }{w^{2}} 
\eeaa

and therefore,
\beaa
&& [\int_{\overline{r}^{*} = r^{*} }^{\overline{r}^{*} = \infty }  \int_{\S^{2}}  \overline{r}^{2}  | \der_{r^{*}}F  |^{2} d\sigma^{2}  d\overline{r}^{*} ]^{\frac{1}{2}}  \\
&\lesssim& \frac{\sqrt{E_{ F}^{(K)} (t)  + E_{ \Lie_{t} F}^{(K)} (t) + \sum_{j=1}^{3}  E_{ \Lie_{\Om_{j}} F}^{(K)} (t) }}{|t|} + \frac{\sqrt{E_{ F}^{(K)} (t)  + E_{ \Lie_{t} F}^{(K)} (t) + \sum_{j=1}^{3}  E_{ \Lie_{\Om_{j}} F}^{(K)} (t) }}{|w|} 
\eeaa

Thus,
\beaa
&& \int_{\overline{r}^{*} = r^{*}  }^{\overline{r}^{*} = \infty } \int_{\S^{2}}  \overline{r}^{2}  <\der_{r^{*}}F, F >_{h} (t, r, \om) d\sigma^{2}  d\overline{r}^{*} \\  
&\lesssim&   ( \frac{\sqrt{E_{ F}^{(K)} (t)  + E_{ \Lie_{t} F}^{(K)} (t) + \sum_{j=1}^{3}   E_{ \Lie_{\Om_{j}} F}^{(K)} (t) }}{|t|} + \frac{\sqrt{E_{ F}^{(K)} (t)  + E_{ \Lie_{t} F}^{(K)} (t) + \sum_{j=1}^{3}  E_{ \Lie_{\Om_{j}} F}^{(K)} (t) }}{|w|} )  \\
&& . ( \frac{  \sqrt{E_{F}^{(K)} (t)}}{t}  + \frac{  \sqrt{E_{F}^{(K)} (t)}}{w} ) \\
&\lesssim&   \frac{E_{  F}^{(K)} (t) + E_{ \Lie_{t} F}^{(K)} (t) + \sum_{j=1}^{3}  E_{ \Lie_{\Om_{j}} F}^{(K)} (t) }{t^{2}} + \frac{E_{  F}^{(K)} (t)  + E_{ \Lie_{t} F}^{(K)} (t) + \sum_{j=1}^{3}  E_{ \Lie_{\Om_{j}} F}^{(K)} (t) }{w^{2}}   \\
&\lesssim& \frac{E_{  F}^{(K)} (t) + E_{ \Lie_{t} F}^{(K)} (t) + E_{ r{\rLie} F}^{(K)} (t) }{t^{2}} + \frac{E_{  F}^{(K)} (t)  + E_{ \Lie_{t} F}^{(K)} (t) + E_{ r{\rLie} F}^{(K)} (t) }{w^{2}}
\eeaa

By same,

\beaa
&& \int_{\overline{r}^{*} = r^{*} }^{\overline{r}^{*} = \infty } \int_{\S^{2}}  \overline{r}^{2}  <\der_{r^{*}}\rLie F, \rLie F >_{h} (t, r, \om) d\sigma^{2}  d\overline{r}^{*}  \\
&\lesssim&  \frac{  E_{r{\rLie} F, \Lie_{t} r{\rLie} F, r^{2}(\rLie)^{2} F }^{(K)} (t)}{ t^{2} }  + \frac{  E_{r{\rLie} F, \Lie_{t} r{\rLie} F, r^{2}(\rLie)^{2} F }^{(K)} (t)}{ w^{2} }
\eeaa 

and,
\beaa
&& \int_{\overline{r}^{*} = r^{*} }^{\overline{r}^{*} = \infty } \int_{\S^{2}}  \overline{r}^{2}  <\der_{r^{*}}\rLie\rLie F, \rLie \rLie F >_{h} (t, r, \om) d\sigma^{2}  d\overline{r}^{*} \\
 &\lesssim&  \frac{  E_{ r^{2}(\rLie)^{2} F, \Lie_{t} r^{2}(\rLie)^{2} F, r^{3}(\rLie)^{3} F }^{(K)} (t)}{t^{2} }  +  \frac{  E_{ r^{2}(\rLie)^{2} F, \Lie_{t} r^{2}(\rLie)^{2} F, r^{3}(\rLie)^{3} F }^{(K)} (t)}{w^{2} }
\eeaa

Finally, we obtain,

\beaa
&& r^{2} |F |^{2} (t, r, \om)  \\
&\lesssim& \frac{  E_{ F, \Lie_{t}  F, r{\rLie} F }^{(K)} (t)}{ t^{2} }  + \frac{  E_{ F, \Lie_{t}  F, r{\rLie} F }^{(K)} (t)}{ w^{2} } \\
&&+ \frac{  E_{r{\rLie} F, \Lie_{t} r{\rLie} F, r^{2}(\rLie)^{2} F }^{(K)} (t)}{ t^{2} }  + \frac{  E_{r{\rLie} F, \Lie_{t} r{\rLie} F, r^{2}(\rLie)^{2} F }^{(K)} (t)}{ w^{2} }  \\
&&+ \frac{  E_{ r^{2}(\rLie)^{2} F, \Lie_{t} r^{2}(\rLie)^{2} F, r^{3}(\rLie)^{3} F }^{(K)} (t)}{t^{2} }  +  \frac{  E_{ r^{2}(\rLie)^{2} F, \Lie_{t} r^{2}(\rLie)^{2} F, r^{3} (\rLie)^{3} F }^{(K)} (t)}{w^{2} }
\eeaa

Thus,

\beaa
 |F|(w, v, \om)  &\lesssim& \frac{  [ E_{  F, \Lie_{t}  F, r{\rLie} F, \Lie_{t} r{\rLie} F, r^{2}(\rLie)^{2} F, \Lie_{t} r^{2}(\rLie)^{2} F, r^{3}(\rLie)^{3} F}^{(K)} (t) ]^{\frac{1}{2}}  }{ r |t| } \\
&& + \frac{  [ E_{  F, \Lie_{t}  F, r{\rLie} F, \Lie_{t} r{\rLie} F, r^{2}(\rLie)^{2} F, \Lie_{t} r^{2}(\rLie)^{2} F, r^{3}(\rLie)^{3} F}^{(K)} (t)  ]^{\frac{1}{2}}  }{ r |w| } \\
\eeaa

We have, $$r^{*} = \frac{v - w}{2}$$ Thus, for $w \leq -1$, we have, $r^{*} \geq v$, and hence $r \geq v$. Therefore, $$\frac{1}{r} \lesssim \frac{1}{|v|}$$

Since in this region we have $|w| \geq 1$, and $|t| \geq 1$, we get, 

\beaa
 |F_{\hat{\mu}\hat{\nu}}|(w, v, \om) &\lesssim& \frac{  [ E_{  F, \Lie_{t}  F, r{\rLie} F, \Lie_{t} r{\rLie} F, r^{2}(\rLie)^{2} F, \Lie_{t} r^{2}(\rLie)^{2} F, r^{3}(\rLie)^{3} F}^{(K)} (t)  ]^{\frac{1}{2}}  }{ (1 + |v|) } \\
&\lesssim& \frac{  E_{F}  }{ (1 + |v|) } \\
\eeaa

We also have for fixed $R$, and $|t| \geq 1 $, $|w| = |t - r^{*} | \lesssim |r^{*} | + |t| \lesssim 2r +t \lesssim 2r + 2|t| \lesssim r |t| $. Thus, in this region, we also have,

\beaa
 |F_{\hat{\mu}\hat{\nu}}|(w, v, \om) &\lesssim& \frac{ [  E_{  F, \Lie_{t}  F, r{\rLie} F, \Lie_{t} r{\rLie} F, r^{2}(\rLie)^{2} F, \Lie_{t} r^{2}(\rLie)^{2} F, r^{3}(\rLie)^{3} F}^{(K)} (t) ]^{\frac{1}{2}}  }{ (1 + |w|) } \\
&\lesssim& \frac{   E_{F}  }{ (1 + |w|) } \\
\eeaa

\subsection{The region $w \leq -1$, $r \geq R$, $-1 \leq t \leq 1$ }\

Let, $$t^{\#} = t - 2$$

When $$-1 \leq t \leq 1$$ we have $$-3 \leq t^{\#} \leq -1$$

Let, $$w^{\#} = t^{\#} - r^{*} = t - r^{*} - 2$$

When $$w \leq -1 $$

we have $$w^{\#} \leq -3 $$

Thus, the region $w \leq -1$, $r \geq R$, $-1 \leq t \leq 1$, is in the new system of coordinates included in the region $w^{\#} \leq -1$, $r \geq R$, $t^{\#} \leq -1$.\\

$\frac{\pa}{\pa t}$ is a Killing vector field, therefore, the time translation will keep the metric invariant, i.e. in the new system of coordinates $ \{t^{\#}, r, \th, \phi \}$ the metric is written exactly as in the former system of coordinates $ \{t, r, \th, \phi \}$. Consequently, we will have the same results proven previously, i.e., in the region $w^{\#} \leq -1$, $r \geq R$, $t^{\#} \leq -1$, we have:
\beaa
 |F_{\hat{\mu}\hat{\nu}}|(w^{\#}, v^{\#}, \om) &\lesssim& \frac{ [  E_{ F, \Lie_{t}  F, r{\rLie} F, \Lie_{t} r{\rLie} F, r^{2}(\rLie)^{2} F, \Lie_{t} r^{2}(\rLie)^{2} F, r^{3}(\rLie)^{3} F}^{M} ]^{\frac{1}{2}}  }{ (1 + |v^{\#}|) } \\
&\lesssim& \frac{ [  E_{  F, \Lie_{t}  F, r{\rLie} F, \Lie_{t} r{\rLie} F, r^{2}(\rLie)^{2} F, \Lie_{t} r^{2}(\rLie)^{2} F, r^{3}(\rLie)^{3} F}^{M} ]^{\frac{1}{2}}  }{ (1 + |v|) } 
\eeaa
and,
\beaa
 |F_{\hat{\mu}\hat{\nu}}|(w^{\#}, v^{\#}, \om) &\lesssim& \frac{ [ E_{  F, \Lie_{t}  F, r{\rLie} F, \Lie_{t} r{\rLie} F, r^{2}(\rLie)^{2} F, \Lie_{t} r^{2}(\rLie)^{2} F, r^{3}(\rLie)^{3} F}^{M} ]^{\frac{1}{2}}  }{ (1 + |w^{\#}|) } \\
&\lesssim& \frac{ [   E_{  F, \Lie_{t}  F, r{\rLie} F, \Lie_{t} r{\rLie} F, r^{2}(\rLie)^{2} F, \Lie_{t} r^{2}(\rLie)^{2} F, r^{3}(\rLie)^{3} F}^{M} ]^{\frac{1}{2}} }{ (1 + |w|) } \\
\eeaa
which gives,
\beaa
 |F_{\hat{\mu}\hat{\nu}}|(w, v, \om) &\lesssim& \frac{ [  E_{  F, \Lie_{t}  F, r{\rLie} F, \Lie_{t} r{\rLie} F, r^{2}(\rLie)^{2} F, \Lie_{t} r^{2}(\rLie)^{2} F, r^{3}(\rLie)^{3} F}^{M} ]^{\frac{1}{2}}  }{ (1 + |v|) } 
\eeaa
and, 
\beaa
 |F_{\hat{\mu}\hat{\nu}}|(w, v, \om) &\lesssim& \frac{  [  E_{  F, \Lie_{t}  F, r{\rLie} F, \Lie_{t} r{\rLie} F, r^{2}(\rLie)^{2} F, \Lie_{t} r^{2}(\rLie)^{2} F, r^{3}(\rLie)^{3} F}^{M} ]^{\frac{1}{2}}  }{ (1 + |w|) } 
\eeaa
in the region $w \leq -1$, $r \geq R$, $-1 \leq t \leq 1$.

\subsection{The region $-1 \leq w \leq 1$, $r \geq R$}\

Let, $$r^{*\vartriangleright} = r^{*} + 2$$

Then, when $$-1 \leq  w \leq 1 $$

we have, 
$$-3 \leq  w^{\#} \leq -1 $$

and, when $$r^{*} \geq R^{*}$$

then, $$r^{*\vartriangleright} \geq R^{*} + 2 \geq R^{*} $$

Thus, the region $-1 \leq w \leq 1$, $r \geq R$ is included, in the new system of coordinates, in the region $ w^{\vartriangleright} \leq -1$, $r^{*\vartriangleright} \geq R^{*}$.

Notice that $ r^{*} $ is defined up to a constant. With the new definition of $r^{*}$ everything we have proven with $r^{*}$ works with $r^{*\vartriangleright}$ by replacing in \eqref{Assumption1}, $J_{ F }^{(G)}  (  t_{i} \leq t \leq t_{i+1} ) (r_{0} < r < R_{0} )$ by $J_{ F }^{(G) \vartriangleright}  (  t_{i} \leq t \leq t_{i+1} ) (r_{0}^{\vartriangleright} < r < R_{0}^{\vartriangleright} )$ defined by
\beaa
\notag
&& J_{\Psi}^{(G) \vartriangleright}  (r_{0}^{\vartriangleright} < r < R_{0}^{\vartriangleright} ) (  t_{i} \leq t \leq t_{i+1} )   \\
\notag
&=&   \int_{t = t_{i}}^{ t= t_{i+1}}  \int_{r^{*} =  r_{0}^{*\vartriangleright} }^{r^{*}= R_{0}^{*\vartriangleright} } \int_{\S^{2}}  [    | \Psi_{\hat{v}\hat{w}}|^{2}  +  \frac{1}{4 } | \Psi_{\hat{\phi}\hat{\th}}|^{2} ] .  dr^{*\vartriangleright} d\sigma^{2} dt \\
&=&  \int_{t = t_{i}}^{ t= t_{i+1}}  \int_{r^{*} =  r_{0}^{*}+2 }^{r^{*}= R_{0}^{*}+2 } \int_{\S^{2}}  [    | \Psi_{\hat{v}\hat{w}}|^{2}  +  \frac{1}{4 } | \Psi_{\hat{\phi}\hat{\th}}|^{2} ] .  dr^{*} d\sigma^{2} dt  
\eeaa
However, since the length of the interval $w \in [-1, 1]$ was arbitrary; we only wanted in the previous subsections to avoid $w = 0$, what is actually only needed is
\bea
\notag
&& J_{\Psi}^{(G)}  (r_{0} < r < R_{0} + \eps ) (  t_{i} \leq t \leq t_{i+1} ) \\
&=&  \int_{t = t_{i}}^{ t= t_{i+1}}  \int_{r^{*} =  r_{0}^{*} }^{r^{*}= R_{0}^{*}+ \eps } \int_{\S^{2}}  [    | \Psi_{\hat{v}\hat{w}}|^{2}  +  \frac{1}{4 } | \Psi_{\hat{\phi}\hat{\th}}|^{2} ] .  dr^{*} d\sigma^{2} dt  
\eea
in assumption \eqref{Assumption1} with $\eps$ arbitrary small. Since we assume \eqref{Assumption1}, we would obtain the above, or take $R_{0}$ as being the infimum of all $\hat{R_{0}}$ in \eqref{definitionofRoastheinfimum} plus $\eps$ fixed. Therefore, in the region $ w^{\vartriangleright} \leq -1$, $r^{*\vartriangleright} \geq R^{*}$. we have,
\beaa
 |F_{\hat{\mu}\hat{\nu}}|(w^{\vartriangleright}, v^{\#}, \om) &\lesssim& \frac{ [   E_{  F, \Lie_{t}  F, r{\rLie} F, \Lie_{t} r{\rLie} F, r^{2}(\rLie)^{2} F, \Lie_{t} r^{2}(\rLie)^{2} F, r^{3}(\rLie)^{3} F}^{M} ] }{ (1 + |v^{\vartriangleright}|) } \\
&\lesssim& \frac{ [  E_{  F, \Lie_{t}  F, r{\rLie} F, \Lie_{t} r{\rLie} F, r^{2}(\rLie)^{2} F, \Lie_{t} r^{2}(\rLie)^{2} F, r^{3}(\rLie)^{3} F}^{M} ] }{ (1 + |v|) } \\
\eeaa

which gives,

\beaa
 |F |(w, v, \om) &\lesssim& \frac{ E_{F}  }{ (1 + |v|) } \\
\eeaa

in the region $-1 \leq w \leq 1$, $r \geq R$.

\end{proof}

\section{Decay of the Energy to Observers Traveling to the Black Hole on $v = constant$ Hypersurfaces Near the Horizon}\

\subsection{The vector field $H$}\

Let
\bea
H &=& - \frac{ h(r^{*})}{(1-\mu)} \frac{\pa}{\pa w}  - h(r^{*}) \frac{\pa}{\pa v}  \\
\notag
&=&  H^{w} \frac{\pa}{\pa w}  + H^{v} \frac{\pa}{\pa v}  
\eea

Computing,
\beaa
\frac{\pa}{\pa w} h^{w} &=& \frac{\pa r^{*}}{\pa w} \frac{\pa}{\pa r^{*}} H^{w}   +    \frac{\pa t}{\pa w} \frac{\pa}{\pa t} H^{w} \\
&=& -\frac{1}{2} \frac{\pa}{\pa r^{*}} H^{w}   + 0 \\
&=&  \frac{1}{2(1-\mu)} [ h^{'} - \frac{\mu}{r} h ]
\eeaa
where $h^{'} = \frac{\pa}{\pa r^{*}} h$\\
And,
\beaa
\frac{\pa}{\pa v} H^{w} &=& \frac{\pa r^{*}}{\pa v} \frac{\pa}{\pa r^{*}} H^{w}   +    \frac{\pa t}{\pa v} \frac{\pa}{\pa t} H^{w} \\
&=& \frac{1}{2} \frac{\pa}{\pa r^{*}} H^{w}   + 0 \\
&=&  \frac{-1}{2(1-\mu)} [ h^{'} - \frac{\mu}{r} h ] 
\eeaa

Similarly,
\beaa
\frac{\pa}{\pa v} H^{v} &=& \frac{\pa r^{*}}{\pa v} \frac{\pa}{\pa r^{*}} H^{v}   +    \frac{\pa t}{\pa v} \frac{\pa}{\pa t} H^{v} \\
&=& \frac{1}{2} \frac{\pa}{\pa r^{*}} H^{v}   + 0 \\
&=& -\frac{1}{2}  h^{'} \\
\frac{\pa}{\pa w} H^{v} &=& \frac{\pa r^{*}}{\pa w} \frac{\pa}{\pa r^{*}} H^{v}   +    \frac{\pa t}{\pa w} \frac{\pa}{\pa t} H^{v} \\
&=& - \frac{1}{2} \frac{\pa}{\pa r^{*}} H^{v}   + 0 \\
&=&  \frac{1}{2} h^{'} 
\eeaa

\beaa
&& \pi^{\a\b}(H)T_{\a\b}(\Psi) \\
&=& [\frac{1}{r^{2}} | \Psi_{w\th} |^{2} + \frac{1}{r^{2}\sin^{2}\th} | \Psi_{w\phi} |^{2}]( \frac{-2}{(1 - \mu)} \pa_{v} H^{w}) \\
&& + [ \frac{1}{r^{2}} | \Psi_{v\th} |^{2} + \frac{1}{r^{2}\sin^{2}\th} | \Psi_{v\phi} |^{2} ] (\frac{-2}{(1 - \mu)} \pa_{w} H^{v})\\
&& + [ \frac{1}{(1-\mu)^{2}} |\Psi_{vw}|^{2}   +  \frac{1}{4r^{4}\sin^{2}\th} | \Psi_{\phi\th}|^{2}  ] ( -2 [    \pa_{v} H^{v}  + \pa_{w} H^{w}   +      \frac{(3\mu-2)}{2r} ( H^{v}  -   H^{w} ) ]  ) \\
\eeaa

Thus,

\bea
\notag
&& \pi^{\a\b}(H)T_{\a\b}(\Psi) \\
\notag
&=& [\frac{1}{r^{2}(1-\mu)} | \Psi_{w\th} |^{2} + \frac{1}{r^{2}\sin^{2}\th (1-\mu)} | \Psi_{w\phi} |^{2}]( \frac{1}{(1-\mu)} [ h^{'} - \frac{\mu}{r} h ] ) \\
\notag
&& + [ \frac{1}{r^{2}} | \Psi_{v\th} |^{2} + \frac{1}{r^{2}\sin^{2}\th} | \Psi_{v\phi} |^{2} ] ( \frac{-1}{(1-\mu)}  h^{'}  )\\
\notag
&& + [ \frac{1}{(1-\mu)^{2}} |\Psi_{vw}|^{2}   +  \frac{1}{4r^{4}\sin^{2}\th} | \Psi_{\phi\th}|^{2}  ]  \\
\notag
&& .  [ -2 (      \frac{1}{2(1-\mu)} [ (h^{'} - (1-\mu) h^{'})  - \frac{\mu}{r} h  ])    +      \frac{(2-3\mu)}{(1-\mu)r}   (h - (1-\mu) h)  ] \\
\eea

We have,
\bea
F_{\Psi}^{(H)}  ( w = w_{i} ) ( v_{i} \leq v \leq v_{i+1} ) = \int_{v = v_{i}}^{v = v_{i+1}} \int_{\S^{2}} J_{\a}(H) n^{\a} dVol_{ w = w_{i}} ( w = w_{i} )
\eea

where 
\beaa
n^{\a} &=& \g(\frac{\pa}{\pa v}, \frac{\pa}{\pa t} )^{-1} (\frac{\pa}{\pa v})^{\a} \\
&=& \frac{-2}{(1-\mu)}  (\frac{\pa}{\pa v})^{\a}
\eeaa
and
\beaa
dVol_{ w = w_{i}} &=& \g(\frac{\pa}{\pa t} , \frac{\pa}{\pa t} ) r^{2} d\sigma^{2} dv \\
&=&  - (1-\mu)  r^{2} d\sigma^{2} dv
\eeaa

We get
\bea
\notag
&& F_{\Psi}^{(H)}  ( w = w_{i} ) ( v_{i} \leq v \leq v_{i+1} ) \\
\notag
&=& \int_{v = v_{i}}^{v = v_{i+1}} \int_{\S^{2}}  - 2 [ \frac{ h(r^{*})}{(1-\mu)} T_{w v}  +   h(r^{*}) T_{v v}  ] r^{2} d\sigma^{2} d v \\
\notag
&=& \int_{v = v_{i}}^{v = v_{i+1}} \int_{\S^{2}}  - 2 [ \frac{ h(r^{*})}{(1-\mu)} (  \frac{1}{(1-\mu)} |\Psi_{vw}|^{2}   +  \frac{(1-\mu)}{4r^{4}\sin^{2}\th} | \Psi_{\phi\th}|^{2} ) \\
\notag
&& +  h(r^{*})   (  \frac{1}{r^{2}} | \Psi_{v\th} |^{2} + \frac{1}{r^{2}\sin^{2}\th} | \Psi_{v\phi} |^{2}  )    ] r^{2} d\sigma^{2} d v \\
\notag
&=& \int_{v = v_{i}}^{v = v_{i+1}} \int_{\S^{2}}  - 2 [  h(r^{*}) (  \frac{1}{(1-\mu)^{2}} |\Psi_{vw}|^{2}   +  \frac{1}{4r^{4}\sin^{2}\th} | \Psi_{\phi\th}|^{2} )   \\
&& +  h(r^{*})   (  \frac{1}{ r^{2}} | \Psi_{v\th} |^{2} + \frac{1}{ r^{2}\sin^{2}\th} | \Psi_{v\phi} |^{2}  )    ] r^{2} d\sigma^{2} d v  \label{firstexpressionforthefluxofHlonwequalconstant}
\eea

and,
\bea
F_{\Psi}^{(H)} ( v = v_{i} ) ( w_{i} \leq w \leq w_{i+1} ) = \int_{w = w_{i}}^{w = w_{i+1}} \int_{\S^{2}} J_{\a}(H) n^{\a} dVol_{ v = v_{i}} ( v = v_{i} )
\eea

where 
\beaa
n^{\a} &=& \g(\frac{\pa}{\pa w}, \frac{\pa}{\pa t} )^{-1} (\frac{\pa}{\pa w})^{\a} \\
&=& \frac{-2}{(1-\mu)}  (\frac{\pa}{\pa w})^{\a}
\eeaa
and
\beaa
dVol_{ v = v_{i}} &=& \g(\frac{\pa}{\pa t} , \frac{\pa}{\pa t} ) r^{2} d\sigma^{2} dw \\
&=&  - (1-\mu)  r^{2} d\sigma^{2} d w
\eeaa

Thus,

\beaa
\notag
&& F_{\Psi}^{(H)} ( v = v_{i} ) ( w_{i} \leq w \leq w_{i+1} ) \\
&=& \int_{w = w_{i}}^{w = w_{i+1}} \int_{\S^{2}}  - 2 [ \frac{ h(r^{*})}{(1-\mu)} T_{w w}  + h(r^{*})  T_{v w}  ] r^{2} d\sigma^{2} d w \\
\notag
&=& \int_{w = w_{i}}^{w = w_{i+1}} \int_{\S^{2}}  - 2 [ \frac{ h(r^{*})}{(1-\mu)} (  \frac{1}{r^{2}} | \Psi_{w\th} |^{2} + \frac{1}{r^{2}\sin^{2}\th} | \Psi_{w\phi} |^{2} )  \\ 
\notag
&& +  h (r^{*})  (  \frac{1}{(1-\mu)} |\Psi_{vw}|^{2}   +  \frac{(1-\mu)}{4r^{4}\sin^{2}\th} | \Psi_{\phi\th}|^{2} ) 
] r^{2} d\sigma^{2} d w 
\eeaa
We get

\bea
\notag
&& F_{\Psi}^{(H)} ( v = v_{i} ) ( w_{i} \leq w \leq w_{i+1} ) \\ 
\notag
&=& \int_{w = w_{i}}^{w = w_{i+1}} \int_{\S^{2}}  - 2 [ (1-\mu) h (r^{*})  (  \frac{1}{r^{2} (1-\mu)^{2} } | \Psi_{w\th} |^{2} + \frac{1}{r^{2}\sin^{2}\th (1-\mu)^{2} } | \Psi_{w\phi} |^{2} )  \\
&& + (1-\mu) h (r^{*})  (  \frac{1}{(1-\mu)^{2}} |\Psi_{vw}|^{2}   +  \frac{1}{4r^{4}\sin^{2}\th} | \Psi_{\phi\th}|^{2} )   ] r^{2} d\sigma^{2} d w 
\eea

Applying the divergence theorem for $\Psi_{\mu\nu} $ in a rectangle in the Penrose diagram representing the exterior of the Schwarzschild space-time of which one side contains the horizon, say in the region $ [w_{i}, \infty ] . [ v_{i}, v_{i+1} ] $ :

\beaa
&& \int_{v = v_{i}}^{v = v_{i+1}} \int_{w = w_{i}}^{w = \infty } \int_{\S^{2}}  ( [\frac{1}{r^{2} (1-\mu) } | \Psi_{w\th} |^{2} + \frac{1}{r^{2}\sin^{2}\th (1-\mu) } | \Psi_{w\phi} |^{2}]( \frac{1}{(1-\mu)} [ h^{'} - \frac{\mu}{r} h ] )\\
&& + ( [\frac{1}{r^{2} (1-\mu) } | \Psi_{v\th} |^{2} + \frac{1}{r^{2}\sin^{2}\th (1-\mu) } | \Psi_{v\phi} |^{2}]( - h^{'}  )\\
&& + [ \frac{1}{(1-\mu)^{2}} |\Psi_{vw}|^{2}   +  \frac{1}{4r^{4}\sin^{2}\th} | \Psi_{\phi\th}|^{2}  ]\\
&& \quad . [    \frac{-1}{(1-\mu)} [ ( h^{'} - (1-\mu) h^{'} ) - \frac{\mu}{r}  h  ]      +      \frac{(2-3\mu)}{(1-\mu)r}  ( h - (1-\mu) h )  ] )   . r^{2} d\sigma^{2} (1-\mu) dw dv \\
&=& - F_{\Psi}^{(H)} ( w = \infty ) ( v_{i} \leq v \leq v_{i+1} )  +  F_{\Psi}^{(H)} ( w = w_{i} ) ( v_{i} \leq v \leq v_{i+1} ) \\
&& - F_{\Psi}^{(H)} ( v = v_{i+1} ) ( w_{i} \leq w \leq \infty )  +  F_{\Psi}^{(H)} ( v = v_{i} ) ( w_{i} \leq w \leq \infty ) \\
\eeaa

We are going to choose $h$ such that

$$ h(r^{*} = - \infty ) =  1 $$

and for all $r > 2m$ : 
$$ h \geq 0$$

Furthermore, we let $h$ be supported in the region $ 2m \leq r \leq (1.2) r_{1} $ for $r_{1}$ chosen such that, $ 2m < r_{0} \leq r_{1} < (1.2) r_{1} < 3m $. We choose $h$ such that, for all $  r \leq r_{1}$, we have

\bea
\frac{\mu}{r} h - h^{'} &\geq& 0 \\
h  &>& 0  \\
h'  &\geq& 0  \\
\frac{-1}{(1-\mu)}  h^{'}  + \frac{3}{r}   h   &\leq&  0 \\
\mu [\frac{-1}{(1-\mu)}  h^{'}  + \frac{3}{r}   h]   &\leq&  - h
\eea

Computing
\bea
\notag
&& \int_{v = v_{i}}^{v = v_{i+1}} \int_{w = w_{i}}^{w = \infty } \int_{\S^{2}} [ \frac{1}{(1-\mu)^{2}} |\Psi_{vw}|^{2}   +  \frac{1}{4r^{4}\sin^{2}\th} | \Psi_{\phi\th}|^{2}  ]\\
\notag
&& . [ \frac{-1}{(1-\mu)} [ (h^{'} - (1-\mu) h^{'}) - \frac{\mu}{r} h  ] + \frac{(2-3\mu)}{(1-\mu)r} ( h - (1-\mu) h )  ]   . r^{2} d\sigma^{2} (1-\mu) dw dv \\
\notag
&=& \int_{v = v_{i}}^{v = v_{i+1}} \int_{w = w_{i}}^{w = \infty } \int_{\S^{2}} [ \frac{1}{(1-\mu)^{2}} |\Psi_{vw}|^{2}   +  \frac{1}{4r^{4}\sin^{2}\th} | \Psi_{\phi\th}|^{2}  ]\\
\notag
&& . [ \frac{-1}{(1-\mu)} [ \mu h^{'} - \frac{\mu}{r} h  ] + \frac{(2-3\mu)}{(1-\mu)r} ( \mu h )  ]   . r^{2} d\sigma^{2} (1-\mu) dw dv \\
\notag
&=& \int_{v = v_{i}}^{v = v_{i+1}} \int_{w = w_{i}}^{w = \infty } \int_{\S^{2}} [ \frac{1}{(1-\mu)^{2}} |\Psi_{vw}|^{2}   +  \frac{1}{4r^{4}\sin^{2}\th} | \Psi_{\phi\th}|^{2}  ]\\
\notag
&& . \frac{\mu}{(1-\mu)} [  ( \frac{h}{r}  -  h^{'} ) + \frac{(2-3\mu)}{r}   h   ]   . r^{2} d\sigma^{2} (1-\mu) dw dv \\
\notag
&=& \int_{v = v_{i}}^{v = v_{i+1}} \int_{w = w_{i}}^{w = \infty } \int_{\S^{2}} [ \frac{1}{(1-\mu)^{2}} |\Psi_{vw}|^{2}   +  \frac{1}{4r^{4}\sin^{2}\th} | \Psi_{\phi\th}|^{2}  ]\\
&& . \mu [   \frac{-1}{(1-\mu)}  h^{'}  + \frac{3}{r}   h   ]   . r^{2} d\sigma^{2} (1-\mu) dw dv 
\eea

Let,

\bea
\notag
&& I^{(H)}_{ \Psi} (  v_{i} \leq v \leq v_{i+1} ) ( w_{i} \leq w \leq \infty) \\
\notag
&=& \int_{v = v_{i}}^{v = v_{i+1}} \int_{w = w_{i}}^{w = \infty } \int_{\S^{2}} ( [\frac{1}{r^{2} (1-\mu)^{2} } | \Psi_{w\th} |^{2} + \frac{1}{r^{2}\sin^{2}\th (1-\mu)^{2} } | \Psi_{w\phi} |^{2}](    h^{'} - \frac{\mu}{r} h  )  \\
\notag
&& + ( [\frac{1}{r^{2} } | \Psi_{v\th} |^{2} + \frac{1}{r^{2}\sin^{2}\th } | \Psi_{v\phi} |^{2}](  \frac{- h^{'}}{(1-\mu)}  ) \\
\notag
&& + [ \frac{1}{(1-\mu)^{2}} |\Psi_{vw}|^{2}   +  \frac{1}{4r^{4}\sin^{2}\th} | \Psi_{\phi\th}|^{2}  ]  . \mu [   \frac{-1}{(1-\mu)}  h^{'}  + \frac{3}{r}   h   ]     ). r^{2} d\sigma^{2} (1-\mu) dw dv \\
\eea

Then, we have,

\bea
\notag
&& - F_{\Psi}^{(H)} ( w = w_{i} ) ( v_{i} \leq v \leq v_{i+1} ) -   F_{\Psi}^{(H)} ( v = v_{i} ) ( w_{i} \leq w \leq \infty ) \\
\notag
&=&- I_{\Psi}^{(H)} (  v_{i} \leq v \leq v_{i+1} ) ( w_{i} \leq w \leq \infty)   \\
\notag
&& - F_{\Psi}^{(H)} ( w = \infty ) ( v_{i} \leq v \leq v_{i+1} )   - F_{\Psi}^{(H)} ( v = v_{i+1} ) ( w_{i} \leq w \leq \infty )  \\
\eea

We choose $r_{1}$ small enough such that $(1.2)r_{1}  < 3m $ .

\subsection{Estimate 1}\

For $(w_{i}, v_{i})$ such that $r(w_{i}, v_{i}) = r_{1}$, where $r_{1}$ is as determined in the construction of the vector field $H$, and for $v_{i+1} \geq v_{i} $, we have

\bea
- F_{\Psi}^{(H)} ( w = w_{i} ) ( v_{i} \leq v \leq v_{i+1} )  \lesssim  F_{\Psi}^{(\frac{\pa}{\pa t})}  ( w = w_{i} ) ( v_{i} \leq v \leq v_{i+1} )  \label{estimate1H}
\eea

\begin{proof}\

We have,
\beaa
F_{\Psi}^{(\frac{\pa}{\pa t})}  ( w = w_{i} ) ( v_{i} \leq v \leq v_{i+1} ) =  \int_{v = v_{i}}^{v = v_{i+1}} \int_{\S^{2}} J_{\a}(  \frac{\pa }{\pa t} ) n^{\a} dVol_{ w = w_{i}} ( w = w_{i} )
\eeaa

where 
\beaa
n^{\a} &=& \g(\frac{\pa}{\pa v}, \frac{\pa}{\pa t} )^{-1} (\frac{\pa}{\pa v})^{\a} \\
&=& \frac{-2}{(1-\mu)}  (\frac{\pa}{\pa v})^{\a}
\eeaa
and
\beaa
dVol_{ w = w_{i}} &=& \g(\frac{\pa}{\pa t} , \frac{\pa}{\pa t} ) r^{2} d\sigma^{2} dv \\
&=&  - (1-\mu)  r^{2} d\sigma^{2} dv
\eeaa

We get
\beaa
 F_{\Psi}^{(\frac{\pa}{\pa t})}  ( w = w_{i} ) ( v_{i} \leq v \leq v_{i+1} ) &=& \int_{v = v_{i}}^{v = v_{i+1}} \int_{\S^{2}}   2 [  T_{v v}  +  T_{w v} ] r^{2} d\sigma^{2} d v \\
&=& \int_{v = v_{i}}^{v = v_{i+1}} \int_{\S^{2}}   2 [    \frac{1}{(1-\mu)} |\Psi_{vw}|^{2}   +  \frac{(1-\mu)}{4r^{4}\sin^{2}\th} | \Psi_{\phi\th}|^{2}   \\
&& +     \frac{1}{r^{2}} | \Psi_{v\th} |^{2} + \frac{1}{r^{2}\sin^{2}\th} | \Psi_{v\phi} |^{2}    ] r^{2} d\sigma^{2} d v \\
&=& \int_{v = v_{i}}^{v = v_{i+1}} \int_{\S^{2}}   2 [    \frac{1}{(1-\mu)^{2}} |\Psi_{vw}|^{2}   +  \frac{1}{4r^{4}\sin^{2}\th} | \Psi_{\phi\th}|^{2}   \\
&& +     \frac{1}{r^{2} (1-\mu)} | \Psi_{v\th} |^{2} + \frac{1}{r^{2}\sin^{2}\th (1-\mu) } | \Psi_{v\phi} |^{2}    ] r^{2} (1-\mu) d\sigma^{2} d v \\
\eeaa

We showed that,
\beaa
&& F_{\Psi}^{(H)}  ( w = w_{i} ) ( v_{i} \leq v \leq v_{i+1} ) \\
&=& \int_{v = v_{i}}^{v = v_{i+1}} \int_{\S^{2}}  - 2 [  h(r^{*}) (  \frac{1}{(1-\mu)^{2}} |\Psi_{vw}|^{2}   +  \frac{1}{4r^{4}\sin^{2}\th} | \Psi_{\phi\th}|^{2} ) \\
&& + h(r^{*}) (  \frac{1}{ r^{2}} |\Psi_{v\th}|^{2}   +  \frac{1}{r^{2}\sin^{2}\th} | \Psi_{v\phi}|^{2} )  ] r^{2} d\sigma^{2} d v
\eeaa

The region $w = w_{i}$ and $v_{i} \leq v \leq v_{i+1} $ is in the region $r \geq r_{1}$ as $ r(w_{i}, v_{i}) = r_{1}$, and $v_{i+1} \geq v_{i} $. Thus, in this region 
$$\frac{h(r^{*})}{(1-\mu)}  \lesssim 1 $$

which gives immediately \eqref{estimate1H}.

\end{proof}

\subsection{Estimate 2}\

Let $$t_{i+1} = (1.1) t_{i}$$

Define,
\bea
\notag
J_{ \Psi }^{(C)}  (  t_{i} \leq t \leq t_{i+1} ) ( r_{0} \leq r \leq R_{0}  )  &=& \int_{t = t_{i}}^{ t= t_{i+1}}  \int_{r^{*} =  r_{0}^{*} }^{r^{*}= R_{0}^{*} } \int_{\S^{2}}  [    | \Psi_{\hat{v}\hat{w}}|^{2}  +  \frac{1}{4 } | \Psi_{\hat{\phi}\hat{\th}}|^{2} ] .|r^{*} - (3m)^{*}|.  dr^{*} d\sigma^{2} dt \\
\eea
From \eqref{Assumption2}, we have
\bea
\notag
J_{ \Psi }^{(C)}  (  t_{i} \leq t \leq t_{i+1} ) ( r_{0} \leq r \leq R_{0}  ) &\les&  | \hat{E}^{(\frac{\pa}{\pa t})}_{\Psi} (t_{i}) | + | \hat{E}^{(\frac{\pa}{\pa t})}_{\Psi} (t_{i+1})| \\
\eea

Recall that $$  2m < r_{0}  \leq r_{1}  < (1.2)r_{1}  < 3m $$

We have,

\bea
\notag
&& \int_{t = t_{i}}^{ t= t_{i+1}}  \int_{r^{*} = - \infty}^{r^{*}= \infty} \int_{\S^{2}}    [| \Psi_{\hat{r^{*}} \hat{\th}} |^{2} +  | \Psi_{\hat{r^{*}}\hat{\phi}} |^{2} + | \Psi_{\hat{t}\hat{\th}} |^{2} +  | \Psi_{\hat{t}\hat{\phi}} |^{2} ] .\chi_{[ r^{*}_{1}, (1.2) r^{*}_{1} ]}  r^{2} (1 -\mu)  d\sigma^{2} dr^{*} dt \\
\notag
&\les&   | E^{(\frac{\pa}{\pa t})}_{\Psi}  (  - (0.85) t_{i} \leq r^{*} \leq (0.85) t_{i}  ) (t= t_{i}) |  \label{estimate2H} \\
\eea

\begin{proof}\

Let, 
\bea
f (r^{*}) &=&  \int_{r^{*} = - \infty}^{r^{*}}   \chi_{[ r^{*}_{1}, (1.2) r^{*}_{1}]} (r^{*})   dr^{*}
\eea  

where $\chi$ is the sharp cut-off function, such that,
 
$$ f (r^{*}) =  1, \qquad \mbox{for} \qquad  r_{1}^{*} < r^{*} < (1.2)r_{1}^{*} $$
and
$$ f (r) =   0, \qquad \mbox{for} \qquad  r^{*} \in ]-\infty, r_{1}^{*}] \cup   [(1.2)r_{1}^{*}, \infty [$$

We get from \eqref{contracteddeformationforG} applied to $f$,

\beaa
\notag
&& T^{\a\b}(\Psi_{\mu\nu})\pi_{\a\b}( G ) \\
&=&  [\frac{1}{r^{2}} | \Psi_{w\th} |^{2} + \frac{1}{r^{2}\sin^{2}\th} | \Psi_{w\phi} |^{2} + \frac{1}{r^{2}} | \Psi_{v\th} |^{2} + \frac{1}{r^{2}\sin^{2}\th} | \Psi_{v\phi} |^{2} ] \frac{\chi_{[ r^{*}_{1}, (1.2) r^{*}_{1}]}}{(1 - \mu)}  \\
&&-2 [  \frac{1 }{(1-\mu)^{2}}  | \Psi_{vw}|^{2}  +  \frac{1}{4r^{4}\sin^{2}\th } | \Psi_{\phi\th}|^{2} ] (  \chi_{[ r^{*}_{1}, (1.2) r^{*}_{1}]} +   \frac{(3\mu -2)  }{r}  \int_{r^{*} = - \infty}^{r^{*}}   \chi_{[ r^{*}_{1}, (1.2) r^{*}_{1}]} dr^{*} )
\eeaa

Applying the divergence theorem between the two hypersurfaces  $\{ t=t_{i} \}$ and $\{ t =t_{i+1} \}$, we obtain
\beaa
&& \int_{t = t_{i}}^{ t= t_{i+1}}  \int_{r^{*} = -\infty }^{r^{*}= \infty } \int_{\S^{2}}   [\frac{1}{r^{2}} | \Psi_{w\th} |^{2} + \frac{1}{r^{2}\sin^{2}\th} | \Psi_{w\phi} |^{2} + \frac{1}{r^{2}} | \Psi_{v\th} |^{2} + \frac{1}{r^{2}\sin^{2}\th} | \Psi_{v\phi} |^{2} ] \frac{\chi_{[ r^{*}_{1}, (1.2) r^{*}_{1}]}}{(1 - \mu)} \\
&& . r^{2} (1-\mu) dr^{*} d\sigma^{2} dt  \\
&=&  \int_{t = t_{i}}^{ t= t_{i+1}}  \int_{r^{*} = -\infty }^{r^{*}= \infty } \int_{\S^{2}} 2 [  \frac{1 }{(1-\mu)^{2}}  | \Psi_{vw}|^{2}  +  \frac{1}{4r^{4}\sin^{2}\th } | \Psi_{\phi\th}|^{2} ] (  \chi_{[ r^{*}_{1}, (1.2) r^{*}_{1}]} \\
&&+   \frac{(3\mu -2)  }{r}  \int_{r^{*} = - \infty}^{r^{*}}   \chi_{[ r^{*}_{1}, (1.2) r^{*}_{1}]} dr^{*} )  .  r^{2} (1-\mu) dr^{*} d\sigma^{2} dt \\
&&+ E^{(G)}_{\Psi} (t_{i+1}) - E^{(G)}_{\Psi} (t_{i}) \\
&\leq&  \int_{t = t_{i}}^{ t= t_{i+1}}  \int_{r^{*} = r_{1}^{*} }^{r^{*}= (1.2) r_{1}^{*} } \int_{\S^{2}} 2 [  \frac{1 }{(1-\mu)^{2}}  | \Psi_{vw}|^{2}  +  \frac{1}{4r^{4}\sin^{2}\th } | \Psi_{\phi\th}|^{2} ]    . r^{2} (1-\mu) dr^{*} d\sigma^{2} dt \\
&& + \int_{t = t_{i}}^{ t= t_{i+1}}  \int_{r^{*} = r_{0}^{*} }^{r^{*}= (3m)^{*} } \int_{\S^{2}} 2 [  \frac{1 }{(1-\mu)^{2}}  | \Psi_{vw}|^{2}  +  \frac{1}{4r^{4}\sin^{2}\th } | \Psi_{\phi\th}|^{2} ] (   \frac{1}{r} [ (3m)^{*} - r_{0}^{*}] ) .  r^{2} (1-\mu) dr^{*} d\sigma^{2} dt \\
&&+ E^{(G)}_{\Psi} (t_{i+1}) - E^{(G)}_{\Psi} (t_{i}) \\
&\les& J_{ \Psi }^{(C)}  (  t_{i} \leq t \leq t_{i+1} ) ( r_{0} \leq r \leq R_{0}  )  + E^{(G)}_{\Psi} (t_{i+1}) - E^{(G)}_{\Psi} (t_{i}) 
\eeaa

Instead of $\Psi$, take $\hat{\Psi}$ as in the proof of \eqref{EKovertsquare}, we get,

\beaa
&& \int_{t = t_{i}}^{ t= t_{i+1}}  \int_{r^{*} = -\infty }^{r^{*}= \infty } \int_{\S^{2}}   [\frac{1}{r^{2}} | \Psi_{w\th} |^{2} + \frac{1}{r^{2}\sin^{2}\th} | \Psi_{w\phi} |^{2} + \frac{1}{r^{2}} | \Psi_{v\th} |^{2} + \frac{1}{r^{2}\sin^{2}\th} | \Psi_{v\phi} |^{2} ] \frac{\chi_{[ r^{*}_{1}, (1.2) r^{*}_{1}]}}{(1 - \mu)} \\
&& . r^{2} (1-\mu) dr^{*} d\sigma^{2} dt  \\
&\leq& \int_{t = t_{i}}^{ t= t_{i+1}}  \int_{r^{*} = -\infty }^{r^{*}= \infty } \int_{\S^{2}}   [\frac{1}{r^{2}} | \hat{\Psi}_{w\th} |^{2} + \frac{1}{r^{2}\sin^{2}\th} | \hat{\Psi}_{w\phi} |^{2} + \frac{1}{r^{2}} | \hat{\Psi}_{v\th} |^{2} + \frac{1}{r^{2}\sin^{2}\th} | \hat{\Psi}_{v\phi} |^{2} ] \frac{\chi_{[ r^{*}_{1}, (1.2) r^{*}_{1}]}}{(1 - \mu)} \\
&& . r^{2} (1-\mu) dr^{*} d\sigma^{2} dt  \\
&\les& J_{ \hat{\Psi} }^{(C)}  (  t_{i} \leq t \leq t_{i+1} ) ( r_{0} \leq r \leq R_{0}  )  + E^{(G)}_{\hat{\Psi}} (t_{i+1}) - E^{(G)}_{\hat{\Psi}} (t_{i}) 
\eeaa
Recall that we have,
\beaa
| E^{(G)}_{\hat{\Psi}}  (  -\infty \leq r^{*} \leq \infty  ) (t_{i}) |  \les   | E^{(\frac{\pa}{\pa t})}_{\Psi}  (  -(0.85)t_{i} \leq r^{*} \leq (0.85)t_{i}  ) (t= t_{i}) | 
\eeaa
Thus,
\beaa
&&  \int_{t = t_{i}}^{ t= t_{i+1}}  \int_{r^{*} = - \infty}^{r^{*}= \infty} \int_{\S^{2}}  \frac{1}{2} [  | \hat{\Psi}_{\hat{r^{*}} \hat{t}}|^{2}  +    |\hat{\Psi}_{\hat{\th} \hat{\phi}}|^{2}  ]  \chi_{[ r^{*}_{1}, 1.2 r^{*}_{1} ]}  r^{2}(1-\mu)  dr^{*} d\sigma^{2} dt\\
&\les& J_{ \hat{\Psi} }^{(C)} (  t_{i} \leq t \leq t_{i+1} ) ( - \infty \leq r^{*} \leq \infty  ) \\
&\les&  | E^{(G)}_{\hat{\Psi}}  (  -\infty \leq r^{*} \leq \infty  )  |  \\
&\les&   | E^{(\frac{\pa}{\pa t})}_{\Psi}  (  -(0.85)t_{i} \leq r^{*} \leq (0.85)t_{i}  ) (t= t_{i}) |   
\eeaa

\end{proof}

\subsection{Estimate 3}\

For
\beaa
w_{i} &=& t_{i} - r_{1}^{*} \\
v_{i} &=& t_{i} + r_{1}^{*} 
\eeaa
where $r_{1}$ is as determined in the construction of the vector field $H$, we have

\bea
\notag
&& - I_{\Psi}^{(H)} (  v_{i} \leq v \leq v_{i+1} ) ( w_{i} \leq w \leq \infty) \\
\notag
&& - F_{\Psi}^{(H)} ( w = \infty ) ( v_{i} \leq v \leq v_{i+1} )   - F_{\Psi}^{(H)} ( v = v_{i+1} ) ( w_{i} \leq w \leq \infty )  \\
&\leq& C F_{\Psi}^{(\frac{\pa}{\pa t} )} ( w = w_{i} ) ( v_{i} \leq v \leq v_{i+1} ) -   F_{\Psi}^{(H)} ( v = v_{i} ) ( w_{i} \leq w \leq \infty )   \label{pseudoestimate3H1}
\eea
(where $C$ is a constant)

And,

\bea
\notag
&& - I_{\Psi}^{(H)}  (  v_{i} \leq v \leq v_{i+1} ) ( w_{i} \leq w \leq \infty) ( r \leq r_{1} )   \\
\notag
&& - F_{\Psi}^{(H)} ( v = v_{i+1} ) ( w_{i} \leq w \leq \infty )  - F_{\Psi}^{(H)} ( w = \infty ) ( v_{i} \leq v \leq v_{i+1} )  \\
\notag
&\lesssim& F_{\Psi}^{(\frac{\pa}{\pa t} )} ( w = w_{i} ) ( v_{i} \leq v \leq v_{i+1} )  - F_{\Psi}^{(H)} ( v = v_{i} ) ( w_{i} \leq w \leq \infty ) \\
&&  +  | E^{(\frac{\pa}{\pa t})}_{\Psi}  (  -(0.85)t_{i} \leq r^{*} \leq (0.85)t_{i}  ) (t= t_{i}) | |  \label{estimate3H1}
\eea

\begin{proof}\

We showed that,

\beaa
&& - I_{\Psi}^{(H)} (  v_{i} \leq v \leq v_{i+1} ) ( w_{i} \leq w \leq \infty) \\
&& - F_{\Psi}^{(H)} ( w = \infty ) ( v_{i} \leq v \leq v_{i+1} )   - F_{\Psi}^{(H)} ( v = v_{i+1} ) ( w_{i} \leq w \leq \infty )  \\
&=& - F_{\Psi}^{(H)} ( w = w_{i} ) ( v_{i} \leq v \leq v_{i+1} ) -   F_{\Psi}^{(H)} ( v = v_{i} ) ( w_{i} \leq w \leq \infty ) \\
\eeaa

From \eqref{estimate1H}, we have,
\beaa
- F_{\Psi}^{(H)} ( w = w_{i} ) ( v_{i} \leq v \leq v_{i+1} )  \lesssim  F_{\Psi}^{(\frac{\pa}{\pa t} )} ( w = w_{i} ) ( v_{i} \leq v \leq v_{i+1} )
\eeaa

This proves \eqref{pseudoestimate3H1}. On the other hand, for all $ r \leq r_{1}$, we have,

\bea
\frac{-1}{(1-\mu)}  h^{'}  + \frac{3}{r}   h   &\leq&  0
\eea

Thus,

\bea
\notag
&& - \int_{v = v_{i}, r \leq r_{1} }^{v = v_{i+1}} \int_{w = w_{i} }^{w = \infty } \int_{\S^{2}} [ \frac{1}{(1-\mu)^{2}} |\Psi_{vw}|^{2}   +  \frac{1}{4r^{4}\sin^{2}\th} | \Psi_{\phi\th}|^{2}  ] . \mu [   \frac{-1}{(1-\mu)}  h^{'}  + \frac{3}{r}   h   ]  \\
\notag
&&  . r^{2} d\sigma^{2} (1-\mu) dw dv \\
&\geq & 0
\eea

whereas,

\bea
\notag
&& | \int_{v = v_{i}, r \geq r_{1}}^{v = v_{i+1}} \int_{w = w_{i}}^{w = \infty } \int_{\S^{2}} [ \frac{1}{(1-\mu)^{2}} |\Psi_{vw}|^{2}   +  \frac{1}{4r^{4}\sin^{2}\th} | \Psi_{\phi\th}|^{2}  ] .  \mu [   \frac{-1}{(1-\mu)}  h^{'}  + \frac{3}{r}   h   ]\\
\notag
&& . r^{2} (1-\mu) d\sigma^{2} dw dv | \\
\notag
& \lesssim & J_{\Psi}^{(C)} (  v_{i} \leq v \leq v_{i+1} ) ( w_{i} \leq w \leq \infty) ( r_{1} \leq r \leq (1.2) r_{1}  ) \\
&\les&  | E^{(\frac{\pa}{\pa t})}_{\Psi}  (  -(0.85)t_{i} \leq r^{*} \leq (0.85)t_{i}  )  (t= t_{i}) |  
\eea

\beaa
&& J_{\Psi}^{(C)} (  v_{i} \leq v \leq v_{i+1} ) ( w_{i} \leq w \leq \infty) ( r_{1} \leq r \leq 1.2 r_{1}  ) \\
&\les&  | E^{(\frac{\pa}{\pa t})}_{\Psi}  (  -(0.85)t_{i} \leq r^{*} \leq (0.85)t_{i}  ) (t= t_{i})  | 
\eeaa

and,

\beaa
&& | \int_{v = v_{i}, r \geq r_{1}}^{v = v_{i+1}} \int_{w = w_{i}}^{w = \infty } \int_{\S^{2}}  ( [\frac{1}{r^{2} (1-\mu)} | \Psi_{w\th} |^{2} + \frac{1}{r^{2}\sin^{2}\th (1-\mu)} | \Psi_{w\phi} |^{2}]( \frac{1}{(1-\mu)} | h^{'}  - \frac{\mu}{r} h  | ) \\
&& + ( [\frac{1}{r^{2} } | \Psi_{v\th} |^{2} + \frac{1}{r^{2}\sin^{2}\th )} | \Psi_{v\phi} |^{2}]( \frac{-h^{'}}{(1-\mu)}  )  ). r^{2} d\sigma^{2} (1-\mu) dw dv |\\
& \lesssim &  \int_{v = v_{i}, r \geq r_{1}}^{v = v_{i+1}} \int_{w = w_{i}}^{w = \infty } \int_{\S^{2}}  [\frac{1}{r^{2} (1-\mu) } | \Psi_{w\th} |^{2} + \frac{1}{r^{2}\sin^{2}\th (1-\mu)} | \Psi_{w\phi} |^{2} \\
&& +  \frac{1}{r^{2} } | \Psi_{v\th} |^{2} + \frac{1}{r^{2}\sin^{2}\th} | \Psi_{v\phi} |^{2} ] .  \chi_{[ r^{*}_{1}, 1.2 r^{*}_{1} ]}   . r^{2} d\sigma^{2}  dw dv  \\
\eeaa

where $ \chi_{[ r_{1}, 1.2 r_{1} ]} $ is a smooth positive cut-off function supported on $[ r^{*}_{1}, 1.2 r^{*}_{1} ]$. \\

From \eqref{estimate2H},
\beaa
&& \int_{v = v_{i}, r \geq r_{1}}^{v = v_{i+1}} \int_{w = w_{i}}^{w = \infty } \int_{\S^{2}}  [\frac{1}{r^{2} (1-\mu) } | \Psi_{w\th} |^{2} + \frac{1}{r^{2}\sin^{2}\th (1-\mu) } | \Psi_{w\phi} |^{2} \\
&& +  \frac{1}{r^{2} } | \Psi_{v\th} |^{2}  + \frac{1}{r^{2}\sin^{2}\th } | \Psi_{v\phi} |^{2} ] .  \chi_{[ r^{*}_{1}, 1.2 r^{*}_{1} ]}   . r^{2} d\sigma^{2}  dw dv \\
& \lesssim &  | E^{(\frac{\pa}{\pa t})}_{\Psi}  (  -(0.85)t_{i} \leq r^{*} \leq (0.85)t_{i}  ) (t= t_{i}) |
\eeaa

We get,

\beaa
| I_{\Psi}^{(H)} | (  v_{i} \leq v \leq v_{i+1} ) ( w_{i} \leq w \leq \infty) ( r \geq r_{1} ) &\les&  | E^{(\frac{\pa}{\pa t})}_{\Psi}  (  -(0.85)t_{i} \leq r^{*} \leq (0.85)t_{i}  )(t= t_{i})  | 
\eeaa

This proves \eqref{estimate3H1}. 

\end{proof}

\subsection{Estimate  4 }\

Let $v_{i+1} \geq v_{i}$. We have

\bea
&& \inf_{ v_{i} \leq v \leq v_{i+1} }  - F_{\Psi}^{(H)} ( v  ) ( w_{i} \leq w \leq \infty )   \label{estimate4H}  \\
\notag
&\lesssim& \frac{- I_{\Psi}^{(H)} (  v_{i} \leq v \leq v_{i+1} ) ( w_{i} \leq w \leq \infty) ( r \leq r_{1})  }{  (v_{i+1} - v_{i} ) } + \sup_{ v_{i} \leq v \leq v_{i+1} } F_{\Psi}^{(\frac{\pa}{\pa t} )} ( v  ) ( w_{i} \leq w \leq \infty ) ( r \geq r_{1} )
\eea

\begin{proof}\

We have,\

\beaa
- F_{\Psi}^{(H)} ( v   ) ( w_{i} \leq w \leq \infty )  &=& \int_{w = w_{1}}^{w = \infty } \int_{\S^{2}}   2 [  h(r^{*}) (  \frac{1}{r^{2}(1-\mu)} | \Psi_{w\th} |^{2} + \frac{1}{r^{2}\sin^{2}\th (1-\mu)} | \Psi_{w\phi} |^{2} )  \\
&& + (1-\mu)  h(r^{*}) (  \frac{1}{(1-\mu)^{2}} | \Psi_{vw} |^{2} + \frac{1}{4r^{4}\sin^{2}\th (1-\mu)} | \Psi_{\th\phi} |^{2} ) ] r^{2} d\sigma^{2} d w
\eeaa

and we have

\beaa
&&- I_{\Psi}^{(H)} (  v_{i} \leq v \leq v_{i+1} ) ( w_{i} \leq w \leq \infty) \\
&=& \int_{v = v_{i}}^{v = v_{i+1}} \int_{w = w_{i}}^{w = \infty } \int_{\S^{2}}  ( [\frac{1}{r^{2} (1-\mu) } | \Psi_{w\th} |^{2} + \frac{1}{r^{2}\sin^{2}\th (1-\mu) } | \Psi_{w\phi} |^{2}]( \frac{1}{(1-\mu) } [  \frac{\mu}{r} h - h^{'} ] ) \\
\notag
&& + ( [\frac{1}{r^{2}  } | \Psi_{v\th} |^{2} + \frac{1}{r^{2}\sin^{2}\th } | \Psi_{v\phi} |^{2}]( \frac{ h^{'}}{(1-\mu)}  ) \\
\notag
&& + [ \frac{1}{(1-\mu)^{2}} |\Psi_{vw}|^{2}   +  \frac{1}{4r^{4}\sin^{2}\th} | \Psi_{\phi\th}|^{2}  ]  . \mu [   \frac{1}{(1-\mu)}  h^{'}  - \frac{3}{r}   h   ]     ). r^{2} d\sigma^{2} (1-\mu) dw dv 
\eeaa

Given the expression of $h$, $h'$, in the region $ r \leq r_{1} $, we get for $v \geq v_{i}$,
\beaa
&& - F_{\Psi}^{(H)} ( v   ) ( w_{i} \leq w \leq \infty ) ( r \leq r_{1} )  \\
& \lesssim&   \int_{w = w_{i}, r \leq r_{1} }^{w = \infty } \int_{\S^{2}} ( [\frac{1}{r^{2} (1-\mu) } | \Psi_{w\th} |^{2} + \frac{1}{r^{2}\sin^{2}\th (1-\mu) } | \Psi_{w\phi} |^{2}]( \frac{1}{(1-\mu) } [  \frac{\mu}{r} h - h^{'} ] ) \\
\notag
&& + ( [\frac{1}{r^{2}  } | \Psi_{v\th} |^{2} + \frac{1}{r^{2}\sin^{2}\th } | \Psi_{v\phi} |^{2}]( \frac{ h^{'}}{(1-\mu)}  ) \\
\notag
&& + [ \frac{1}{(1-\mu)^{2}} |\Psi_{vw}|^{2}   +  \frac{1}{4r^{4}\sin^{2}\th} | \Psi_{\phi\th}|^{2}  ]  . \mu [   \frac{1}{(1-\mu)}  h^{'}  - \frac{3}{r}   h   ]     ). r^{2} d\sigma^{2} (1-\mu) dw 
\eeaa

On the other hand, we have,
\beaa
&& F_{\Psi}^{(\frac{\pa}{\pa t})}  ( v = v_{i} ) ( w_{i} \leq w \leq w_{i+1} ) \\
 &=& \int_{w = w_{i}}^{w = w_{i+1}} \int_{\S^{2}}   2 [  T_{w w}  +  T_{w v} ] r^{2} d\sigma^{2} d w \\
&=& \int_{w = w_{i}}^{w = w_{i+1}} \int_{\S^{2}}   2 [    \frac{1}{(1-\mu)} |\Psi_{vw}|^{2}   +  \frac{(1-\mu)}{4r^{4}\sin^{2}\th} | \Psi_{\phi\th}|^{2}   \\
&& +     \frac{1}{r^{2}} | \Psi_{w\th} |^{2} + \frac{1}{r^{2}\sin^{2}\th} | \Psi_{w \phi} |^{2}    ] r^{2} d\sigma^{2} d v \\
&=& \int_{w = w_{i}}^{w = w_{i+1}} \int_{\S^{2}}   2 [    \frac{1}{(1-\mu)^{2} } |\Psi_{vw}|^{2}   +  \frac{1}{4r^{4}\sin^{2}\th} | \Psi_{\phi\th}|^{2}   \\
&& +     \frac{1}{r^{2} (1-\mu) } | \Psi_{w\th} |^{2} + \frac{1}{r^{2}\sin^{2}\th (1-\mu) } | \Psi_{w\phi} |^{2}    ] r^{2} (1-\mu) d\sigma^{2} d w \\
\eeaa

Thus, from the boundedness of $h$, $h'$, we have in $ r \geq r_{1} $,

\beaa
&& - F_{\Psi}^{(H)} ( v   ) ( w_{i} \leq w \leq \infty ) ( r \geq r_{1} ) \\
& \lesssim&   \int_{w = w_{i}, r \geq r_{1} }^{w = \infty } \int_{\S^{2}} (  \frac{1}{r^{2}(1-\mu)} | \Psi_{w\th} |^{2} + \frac{1}{r^{2}\sin^{2}\th(1-\mu)} | \Psi_{w\phi} |^{2}  \\
&& +   \frac{1}{(1-\mu)^{2}} |\Psi_{vw}|^{2}   +  \frac{1}{4r^{4}\sin^{2}\th} | \Psi_{\phi\th}|^{2}  ) . r^{2} d\sigma^{2} (1-\mu) dw \\
& \lesssim & F_{\Psi}^{(\frac{\pa}{\pa t} )} ( v   ) ( w_{i} \leq w \leq \infty ) ( r \geq r_{1} )
\eeaa

Thus,
\beaa
&& - F_{\Psi}^{(H)} ( v   ) ( w_{i} \leq w \leq \infty ) \\
& \lesssim&   \int_{w = w_{i}, r \leq r_{1} }^{w = \infty } \int_{\S^{2}}  ( [\frac{1}{r^{2} (1-\mu) } | \Psi_{w\th} |^{2} + \frac{1}{r^{2}\sin^{2}\th (1-\mu) } | \Psi_{w\phi} |^{2}]( \frac{1}{(1-\mu) } [  \frac{\mu}{r} h - h^{'} ] ) \\
\notag
&& + ( [\frac{1}{r^{2}  } | \Psi_{v\th} |^{2} + \frac{1}{r^{2}\sin^{2}\th } | \Psi_{v\phi} |^{2}]( \frac{ h^{'}}{(1-\mu)}  ) \\
\notag
&& + [ \frac{1}{(1-\mu)^{2}} |\Psi_{vw}|^{2}   +  \frac{1}{4r^{4}\sin^{2}\th} | \Psi_{\phi\th}|^{2}  ]  . \mu [   \frac{1}{(1-\mu)}  h^{'}  - \frac{3}{r}   h   ]     ). r^{2} d\sigma^{2} (1-\mu) dw \\
\notag
&& + F_{\Psi}^{(\frac{\pa}{\pa t} )} ( v   ) ( w_{i} \leq w \leq \infty )( r \geq r_{1} ) 
\eeaa

We have,
\beaa
&& (v_{i+1} - v_{i} ) \inf_{ v_{i} \leq v \leq v_{i+1} } - F_{\Psi}^{(H)} ( v   ) ( w_{i} \leq w \leq \infty ) \\
& \leq&  \int_{v = v_{i}}^{v = v_{i+1}} - F_{\Psi}^{(H)} ( v   ) ( w_{i} \leq w \leq \infty ) dv \\
& \lesssim&   \int_{v = v_{i}, r \leq r_{1} }^{v = v_{i+1}} \int_{w = w_{i} }^{w = \infty } \int_{\S^{2}}  ( [\frac{1}{r^{2} (1-\mu) } | \Psi_{w\th} |^{2} + \frac{1}{r^{2}\sin^{2}\th (1-\mu) } | \Psi_{w\phi} |^{2}]( \frac{1}{(1-\mu) } [  \frac{\mu}{r} h - h^{'} ] ) \\
\notag
&& + ( [\frac{1}{r^{2}  } | \Psi_{v\th} |^{2} + \frac{1}{r^{2}\sin^{2}\th } | \Psi_{v\phi} |^{2}]( \frac{ h^{'}}{(1-\mu)}  ) \\
\notag
&& + [ \frac{1}{(1-\mu)^{2}} |\Psi_{vw}|^{2}   +  \frac{1}{4r^{4}\sin^{2}\th} | \Psi_{\phi\th}|^{2}  ]  . \mu [   \frac{1}{(1-\mu)}  h^{'}  - \frac{3}{r}   h   ]     ). r^{2} d\sigma^{2} (1-\mu) dw dv  \\
&& + \int_{v = v_{i}}^{v = v_{i+1}} F_{\Psi}^{(\frac{\pa}{\pa t} )} ( v   ) ( w_{i} \leq w \leq \infty ) ( r \geq r_{1} )  dv  \\
&\lesssim & - I_{\Psi}^{(H)} (  v_{i} \leq v \leq v_{i+1} ) ( w_{i} \leq w \leq \infty) ( r \leq r_{1} ) \\
&&  + (v_{i+1} - v_{i} ) \sup_{ v_{i} \leq v \leq v_{i+1} } F_{\Psi}^{(\frac{\pa}{\pa t} )} ( v  ) ( w_{i} \leq w \leq \infty ) ( r \geq r_{1} )
\eeaa

\end{proof}

\subsection{ Estimate  5 }\

For
\beaa
w_{i} &=& t_{i} + r_{1} \\
v_{i} &=& t_{i} + r_{1} 
\eeaa
where $r_{1}$ is as determined in the construction of the vector field $H$, we have

\bea
\notag
0 \leq - I_{\Psi}^{(H)}  (  v_{i} \leq v \leq v_{i+1} ) ( w_{i} \leq w \leq \infty)  ( r \leq r_{1} ) &\lesssim&  |E_{\Psi}^{  (\frac{\pa}{\pa t} )} | + E_{\Psi}^{ \# (\frac{\pa}{\pa t} )} ( t= t_{0} )  \label{firstestimateuniformboundonI1I2foreachH1andH2} \\
\eea

where,

\bea
\notag
&& E_{\Psi}^{ \# (\frac{\pa}{\pa t} )} ( t= t_{0} ) \\
\notag
&=& \int_{r^{*} = - \infty }^{r^{*} = \infty } \int_{\S^{2}}    [    \frac{1}{r^{2}(1-\mu)} | \Psi_{w\th} |^{2} + \frac{1}{r^{2}\sin^{2}\th  (1-\mu)} | \Psi_{w\phi} |^{2}  +   \frac{1}{r^{2}} | \Psi_{v\th} |^{2} + \frac{1}{r^{2}\sin^{2}\th } | \Psi_{v\phi} |^{2}  \\
\notag
&& +   \frac{1}{(1-\mu)^{2}} |\Psi_{vw}|^{2}   +  \frac{ 1 }{4r^{4}\sin^{2}\th} | \Psi_{\phi\th}|^{2}   ] . r^{2}   d\sigma^{2} dr^{*} ( t = t_{0} ) \label{defmodifiedenergy} \\
\eea

\begin{proof}\
 
Computing,
\bea
&& E^{(H)}_{\Psi} (t) =   \int_{r^{*}  = - \infty}^{r^{*} = \infty} \int_{\S^{2}} J_{\a} (H) (\Psi_{\mu\nu}) n^{\a} dVol_{t } 
\eea
where $n^{\a} = - \frac { (\frac{\pa}{\pa t})^{\a} }{\sqrt{(1-\mu) } }$ and $dVol_{t=t_{i}} = r^{2} \sqrt{(1-\mu) } d\sigma^{2} dr^{*}$, we get,
\beaa
E^{(H) }_{\Psi}(t ) = \int_{r^{*}  = - \infty }^{r^{*} = \infty} \int_{\S^{2}}   - \frac {1}{\sqrt{(1-\mu) } }  (\frac{\pa}{\pa t})^{\a}     J_{\a}(H) (\Psi_{\mu\nu})  r^{2} \sqrt{(1-\mu) } d\sigma^{2} dr^{*}  
\eeaa

Therefore,

\beaa
E^{(H)}_{\Psi} (t) &=& \int_{r^{*}  = -\infty }^{r^{*} = \infty} \int_{\S^{2}} - (H) ^{\a}   T_{t \a } (\Psi_{\mu\nu})  r^{2}  d\sigma^{2} dr^{*}   \\
&=& \int_{r^{*}  = -\infty}^{r^{*} = \infty} \int_{\S^{2}}   [  \frac{h }{(1-\mu)}  T_{t w } +   h   T_{t v  } ] r^{2}  d\sigma^{2} dr^{*}  \\
&=& \int_{r^{*}  = -\infty}^{r^{*} = \infty} \int_{\S^{2}}   [  \frac{h }{(1-\mu)} T_{w w  } + \frac{h }{(1-\mu)}  T_{v w  } +  h  T_{w v  } + h  T_{v v }  ] r^{2}  d\sigma^{2}  dr^{*}  \\
&=&  \int_{r^{*}  = -\infty }^{r^{*} = \infty} \int_{\S^{2}}   [ \frac{h }{(1-\mu)} (\frac{1}{r^{2}} | \Psi_{w\th} |^{2} + \frac{1}{r^{2}\sin^{2}\th} | \Psi_{w\phi} |^{2} ) \\
&& + \frac{h}{(1-\mu)} (   \frac{1}{(1-\mu)} |\Psi_{vw}|^{2}   +  \frac{(1-\mu)}{4r^{4}\sin^{2}\th} | \Psi_{\phi\th}|^{2} ) \\
&& +  h(  \frac{1}{(1-\mu)} |\Psi_{vw}|^{2}   +  \frac{(1-\mu)}{4r^{4}\sin^{2}\th} | \Psi_{\phi\th}|^{2} ) \\
&&  +  h ( \frac{1}{r^{2}} | \Psi_{v\th} |^{2} + \frac{1}{r^{2}\sin^{2}\th} | \Psi_{v\phi} |^{2})  ]  r^{2}  d\sigma^{2} dr^{*}  
\eeaa

Hence,
\beaa
E^{(H)}_{\Psi} (t)  &=& \int_{r^{*}  = -\infty }^{r^{*} = \infty} \int_{\S^{2}}   [ h (1-\mu) (  | \Psi_{\hat{w}\hat{\th}} |^{2} + \frac{1}{r^{2}\sin^{2}\th} | \Psi_{\hat{w}\hat{\phi}} |^{2} ) \\
&& + h (    |\Psi_{\hat{v}\hat{w}}|^{2}   +  \frac{1}{4 } | \Psi_{\hat{\phi}\hat{\th}}|^{2} ) +  h (1-\mu) (  |\Psi_{\hat{v} \hat{w}}|^{2}   +  \frac{1}{4} | \Psi_{\hat{\phi}\hat{\th}}|^{2} ) \\
&&  +  h (  | \Psi_{\hat{v}\hat{\th}} |^{2} +  | \Psi_{\hat{v}\hat{\phi}} |^{2})  ]   r^{2}  d\sigma^{2}  dr^{*}   
\eeaa

By using the divergence theorem in the region $(   v \leq v_{0} ) ( t_{0} \leq t \leq \infty) ( r \leq r_{1} )$, we

\bea
\notag
&& - I_{\Psi}^{(H)}  (   v \leq v_{0} ) ( t_{0} \leq t \leq \infty) ( r \leq r_{1} )   \\
\notag
&&   - F_{\Psi}^{(H)} ( v = v_{0} ) ( w_{0} \leq w \leq \infty )   - F_{\Psi}^{(H)} ( w = \infty ) ( - \infty \leq v \leq v_{0} )  \\
\notag
&=& E^{(H)}_{\Psi} (t_{0})  \\ 
\eea
Due to the positivity of the terms on the left hand side, we get
\bea
\notag
  - F_{\Psi}^{(H)} ( v = v_{0} ) ( w_{0} \leq w \leq \infty ) &\les& E^{(H)}_{\Psi} (t_{0})  \\
&\les&  E_{\Psi}^{ \# (\frac{\pa}{\pa t} )} ( t= t_{0} )  \label{firstboundenergyfromHforfirstflux}
\eea
where,
\beaa
&& E_{\Psi}^{ \# (\frac{\pa}{\pa t} )} ( t= t_{0} ) \\
&=& \int_{r^{*} = - \infty }^{r^{*} = \infty } \int_{\S^{2}}    [    \frac{1}{r^{2}(1-\mu)} | \Psi_{w\th} |^{2} + \frac{1}{r^{2}\sin^{2}\th (1-\mu) } | \Psi_{w\phi} |^{2}   +  \frac{1}{r^{2}} | \Psi_{v\th} |^{2} + \frac{1}{r^{2}\sin^{2}\th  } | \Psi_{v\phi} |^{2} \\
&& +     \frac{1}{(1-\mu)^{2}} |\Psi_{vw}|^{2}   +  \frac{ 1 }{4r^{4}\sin^{2}\th} | \Psi_{\phi\th}|^{2}   ] . r^{2}  d\sigma^{2} dr^{*} ( t = t_{0} )\\
\eeaa

From the divergence theorem and the fact that $\frac{\pa}{\pa t} $ is Killing, it is easy to see that by integrating in a suitable region and using the positivity of the energy we get,

\bea
\notag
F_{\Psi}^{(\frac{\pa}{\pa t} )} ( w = w_{i} ) ( v_{i} \leq v \leq v_{i+1} ) &=& F_{\Psi}^{(\frac{\pa}{\pa t} )} ( w = w_{i} ) ( v_{i} \leq v \leq v_{i+1} ) ( r \geq r_{1}) \\
\notag
&\lesssim&  | E^{(\frac{\pa}{\pa t})}_{\Psi}  (  -(0.85)t_{i} \leq r^{*} \leq (0.85)t_{i}  ) (t= t_{i}) |  \\ \label{forrecurrence1below}
\eea

From \eqref{estimate3H1}, we get,
\beaa
\notag
&& - I_{\Psi}^{(H)}  (  v_{0} \leq v \leq v_{1} ) ( w_{0} \leq w \leq \infty) ( r \leq r_{1} )  \\
\notag
&&- F_{\Psi}^{(H)} ( v = v_{1} ) ( w_{0} \leq w \leq \infty ) - F_{\Psi}^{(H)}  ( w = \infty ) ( v_{0} \leq v \leq v_{1} ) \\
\notag
&\lesssim& F_{\Psi}^{(\frac{\pa}{\pa t} )} ( w = w_{0} ) ( v_{0} \leq v \leq v_{1} ) -   F_{\Psi}^{(H)} ( v = v_{0} ) ( w_{0} \leq w \leq \infty ) \\
\notag
&&  +  | E^{(\frac{\pa}{\pa t})}_{\Psi}  (  -(0.85)t_{0} \leq r^{*} \leq (0.85)t_{0}  ) (t= t_{0}) | | \\ 
&\lesssim& | E^{(\frac{\pa}{\pa t})}_{\Psi}  (  -(0.85)t_{0} \leq r^{*} \leq (0.85)t_{0}  ) (t= t_{0}) |  + E_{\Psi}^{ \# (\frac{\pa}{\pa t} )} ( t= t_{0} )
\eeaa

(from \eqref{firstboundenergyfromHforfirstflux} and \eqref{forrecurrence1below}).\\

By recurrence from inequality \eqref{estimate3H1}, and using \eqref{forrecurrence1below}, we obtain for all $i$ integer 
\beaa
\notag
&& - I_{\Psi}^{(H)}  (  v_{i} \leq v \leq v_{i+1} ) ( w_{i} \leq w \leq \infty) ( r \leq r_{1} )  \\
\notag
&&- F_{\Psi}^{(H)} ( v = v_{i+1} ) ( w_{i} \leq w \leq \infty ) - F_{\Psi}^{(H)}  ( w = \infty ) ( v_{i} \leq v \leq v_{i+1} ) \\
\notag
&\lesssim& F_{\Psi}^{(\frac{\pa}{\pa t} )} ( w = w_{i} ) ( v_{i} \leq v \leq v_{i+1} ) -   F_{\Psi}^{(H)} ( v = v_{i} ) ( w_{i} \leq w \leq \infty ) \\
\notag
&&  +  | E^{(\frac{\pa}{\pa t})}_{\Psi}  (  -(0.85)t_{i} \leq r^{*} \leq (0.85)t_{i}  ) (t= t_{i}) | | \\ 
&\lesssim& | E^{(\frac{\pa}{\pa t})}_{\Psi}  (  -(0.85)t_{i} \leq r^{*} \leq (0.85)t_{i}  ) (t= t_{i}) |  + E_{\Psi}^{ \# (\frac{\pa}{\pa t} )} ( t= t_{0} )\\ 
\eeaa

Due to sign of $h$, and the definition of $h$, we have that the terms in each of the integrands on the left hand side are positive, hence, we obtain \eqref{firstestimateuniformboundonI1I2foreachH1andH2}.

\end{proof}

\subsection{ Estimate  6 }\

For all $v$, let $$w_{0}(v) = v - 2 r_{1}^{*} $$

Let
\bea
v_{+} = \max\{1, v \} \label{definitionvplus}
\eea

We have,

\bea
 - F_{\Psi}^{(H)} ( v  ) ( w_{0}(v) \leq w \leq \infty )  &\lesssim &  \frac{ [ |E_{\Psi}^{  (\frac{\pa}{\pa t} )} | + E_{\Psi}^{ \# (\frac{\pa}{\pa t} )} ( t= t_{0} ) + E_{\Psi}^{M} ] }{  v_{+}^{2}  }  \label{estimate6Hflux}
\eea
and,
\bea
 - F_{\Psi}^{(H)} ( w ) (  v-1 \leq \overline{v} \leq v )  &\lesssim &  \frac{ [ |E_{\Psi}^{  (\frac{\pa}{\pa t} )} | + E_{\Psi}^{ \# (\frac{\pa}{\pa t} )} ( t= t_{0} ) + E_{\Psi}^{M} ] }{  v_{+}^{2}  }  \label{estimate6Hwequalconstantflux}
\eea

\begin{proof}\

Let,
\bea
v_{i} &=& t_{i} + r_{1}^{*} \\
w_{i} &=& t_{i} - r_{1}^{*} 
\eea
where $t_{i}$ is defined as in \eqref{defti}: $$t_{i} = (1.1)^{i} t_{0}$$

We have shown that,
\beaa
&& \inf_{ v_{i} \leq v \leq v_{i+1} } - F_{\Psi}^{(H)} ( v  ) ( w_{i} \leq w \leq \infty )  \\
&\lesssim& \frac{ - I_{\Psi}^{(H)} (  v_{i} \leq v \leq v_{i+1} ) ( w_{i} \leq w \leq \infty)  ( r \leq r_{1} )  }{  (v_{i+1} - v_{i} ) } + \sup_{ v_{i} \leq v \leq v_{i+1} } F_{\Psi}^{(\frac{\pa}{\pa t} )} ( v  ) ( w_{i} \leq w \leq \infty ) ( r \geq r_{1} ) \\
\eeaa

\begin{lemma}

We have,
\bea
 \sup_{ v_{i} \leq v \leq v_{i+1} } F_{\Psi}^{(\frac{\pa}{\pa t} )} ( v  ) ( w_{i} \leq w \leq \infty ) ( r \geq r_{1} ) & \lesssim & \frac{E_{\Psi}^{(K)}(t_{i})}{t_{i}^{2} }  
\eea

\end{lemma}

\begin{proof}

By integrating in a well chosen region and using the divergence theorem we get that,

\bea
\notag
&&  \sup_{ v_{i} \leq v \leq v_{i+1} } F_{\Psi}^{(\frac{\pa}{\pa t} )} ( v  ) ( w_{i} \leq w \leq \infty ) ( r \geq r_{1} ) \\
\notag
& \lesssim & \int_{r^{*}= c r_{1}^{*}   }^{r^{*} = C r_{1}^{*}   } \int_{\S^{2}} ( [\frac{1}{r^{2}(1-\mu)} | \Psi_{w\th} |^{2} + \frac{1}{r^{2}\sin^{2}\th(1-\mu)} | \Psi_{w\phi} |^{2} ]  +  [ \frac{1}{r^{2}(1-\mu)} | \Psi_{v\th} |^{2} \\
&& + \frac{1}{r^{2}\sin^{2}\th(1-\mu)} | \Psi_{v\phi} |^{2} ]  +  [ \frac{1}{(1-\mu)^{2}} |\Psi_{vw}|^{2}   +  \frac{1}{4r^{4}\sin^{2}\th} | \Psi_{\phi\th}|^{2}]  ). (1-\mu) r^{2}   d\sigma^{2} dr^{*} (t_{i})  \label{boundingthesup}
\eea

We showed that,

\beaa
&&\int_{r^{*}= r_{1}^{*}   }^{r^{*} = r_{2}^{*}   } \int_{\S^{2}} ( [\frac{1}{r^{2}(1-\mu)} | \Psi_{w\th} |^{2} + \frac{1}{r^{2}\sin^{2}\th(1-\mu)} | \Psi_{w\phi} |^{2} ]  +  [ \frac{1}{r^{2}(1-\mu)} | \Psi_{v\th} |^{2} \\
&& + \frac{1}{r^{2}\sin^{2}\th(1-\mu)} | \Psi_{v\phi} |^{2} ]  +  [ \frac{1}{(1-\mu)^{2}} |\Psi_{vw}|^{2}   +  \frac{1}{4r^{4}\sin^{2}\th} | \Psi_{\phi\th}|^{2}]  ). (1-\mu) r^{2}   d\sigma^{2} dr^{*} (t) \\
& \lesssim & \frac{E_{\Psi}^{(K)}(t)}{\min_{w \in \{t\}\cap \{  r_{1}^{*} \leq r^{*} \leq r_{2}^{*} \}  } w^{2}}  + \frac{E_{\Psi}^{(K)}(t)}{\min_{v \in \{t\}\cap  \{  r_{1}^{*} \leq r^{*} \leq r_{2}^{*} \}  } v^{2}}
\eeaa

Thus, \eqref{boundingthesup} gives,
\bea
 \sup_{ v_{i} \leq v \leq v_{i+1} } F_{\Psi}^{(\frac{\pa}{\pa t} )} ( v  ) ( w_{i} \leq w \leq \infty ) ( r \geq r_{1} ) & \lesssim & \frac{E_{\Psi}^{(K)}(t_{i})}{t_{i}^{2} }  
\eea

\end{proof}

And we showed that,

\beaa
 -  I_{\Psi}^{(H)}  (  v_{i} \leq v \leq v_{i+1} ) ( w_{i} \leq w \leq \infty) ( r \leq r_{1} ) &\lesssim&  |E_{\Psi}^{  (\frac{\pa}{\pa t} )} | + E_{\Psi}^{ \# (\frac{\pa}{\pa t} )} ( t= t_{0} )\\ 
\eeaa

Thus, we obtain,

\bea
\notag
 \inf_{ v_{i} \leq v \leq v_{i+1} } F_{\Psi}^{(H)} ( v  ) ( w_{i} \leq w \leq \infty )  &\lesssim & \frac{1}{ (v_{i+1} - v_{i} ) } [ |E_{\Psi}^{  (\frac{\pa}{\pa t} )} | + E_{\Psi}^{ \# (\frac{\pa}{\pa t} )} ( t= t_{0} )  ]  + \frac{E_{\Psi}^{(K)}(t_{i})}{t_{i}^{2} } \\ \label{firstdecay} 
\eea

and thus, there exists a $ v_{i}^{\#} \in [ v_{i}, v_{i+1} ] $ where above inequality holds.\\

We have,

\bea
\notag
v_{i+1} - v_{i} &=& t_{i+1} + r_{1}^{*} -   (t_{i} + r_{1}^{*}) =  t_{i+1} - t_{i} = (1.1)^{i+1}t_{0} - (1.1)^{i}t_{0} = (1.1)^{i}t_{0} ( 1.1 - 1 ) \\
\notag
&=&  (0.1)  (1.1)^{i} t_{0} \\
&=& 0.1 t_{i}
\eea

Let,
\bea
w_{i}^{\#} = v_{i}^{\#}  - 2 r_{1}^{*}
\eea

Therefore,

\beaa
- F_{\Psi}^{(H)} ( v_{i}^{\#}  ) ( w_{i}^{\#} \leq w \leq \infty )  &\lesssim & - F_{\Psi}^{(H)} ( v_{i}^{\#}  ) ( w_{i} \leq w \leq \infty )  \\
&& \text{(due to the positivity of $- F_{\Psi}^{(H)} ( v_{i}^{\#}  ) ( w_{i} \leq w \leq \infty )  $)}\\
&\les& \frac{1}{ t_{i} } [ |E_{\Psi}^{  (\frac{\pa}{\pa t} )} | + E_{\Psi}^{ \# (\frac{\pa}{\pa t} )} ( t= t_{0} )  ]  + \frac{E_{\Psi}^{(K)}(t_{i})}{t_{i}^{2} } \\
&& \text{(by \eqref{firstdecay}).} \\
\eeaa

From \eqref{estimate3H1}, applied in the region $[w_{i}^{\#}, \infty].[ v_{i}^{\#}, v_{i+1}] $, we get due to the positivity of $ -I_{\Psi}^{(H)}  (  v_{i}^{\#} \leq v \leq v_{i+1} ) ( w_{i}^{\#} \leq w \leq \infty) ( r \leq r_{1} ) $,  and $- F_{\Psi}^{(H)} ( w = \infty ) ( v_{i}^{\#} \leq v \leq v_{i+1} )$, that,

\beaa
&&  - F_{\Psi}^{(H)} ( v = v_{i+1} ) ( w_{i}^{\#} \leq w \leq \infty )  \\
&\lesssim& F_{\Psi}^{(\frac{\pa}{\pa t} )} ( w = w_{i}^{\#} ) ( v_{i}^{\#} \leq v \leq v_{i+1} ) -   F_{\Psi}^{(H)} ( v = v_{i}^{\#} ) ( w_{i}^{\#} \leq w \leq \infty ) \\
&& + | E^{(\frac{\pa}{\pa t})}_{\Psi}  (  -(0.85)t_{i} \leq r^{*} \leq (0.85)t_{i}  ) (t= t_{i}) |  \\
\eeaa

\begin{lemma}

\beaa
F_{\Psi}^{(\frac{\pa}{\pa t} )} ( w = w_{i}^{\#} ) ( v_{i}^{\#} \leq v \leq v_{i+1} )  &\les& \frac{E_{\Psi}^{(K)}(t_{i})}{t_{i}^{2} } \\
\eeaa

\end{lemma}

\begin{proof}

By applying the divergence theorem in a well chosen region, we get,

\bea
\notag
F_{\Psi}^{(\frac{\pa}{\pa t} )} ( w = w_{i}^{\#} ) ( v_{i}^{\#} \leq v \leq v_{i+1} )  &\les& | E^{(\frac{\pa}{\pa t})}_{\Psi}  (  r_{1}^{*} \leq r^{*} \leq \frac{v_{i+1} - (2t_{i} - v_{i+1}) }{2}  ) (t= t_{i}) | \\
\notag
&\les& | E^{(\frac{\pa}{\pa t})}_{\Psi}  (  r_{1}^{*} \leq r^{*} \leq v_{i+1} - t_{i}   ) (t= t_{i}) | \\
\notag
&\les& | E^{(\frac{\pa}{\pa t})}_{\Psi}  (  r_{1}^{*} \leq r^{*} \leq t_{i+1} + r_{1}^{*} - t_{i}   ) (t= t_{i}) | \\
\notag
&\les& | E^{(\frac{\pa}{\pa t})}_{\Psi}  (  r_{1}^{*} \leq r^{*} \leq (1.1)t_{i} + r_{1}^{*} - t_{i}   ) (t= t_{i}) | \\
\notag
&\les& | E^{(\frac{\pa}{\pa t})}_{\Psi}  (  r_{1}^{*} \leq r^{*} \leq (0.1)t_{i} + r_{1}^{*}   ) (t= t_{i}) | \\
\eea

We proved that,

\bea
\notag
&& | E^{(\frac{\pa}{\pa t})}_{\Psi}  (  r_{1}^{*} \leq r^{*} \leq (0.1)t_{i} + r_{1}^{*}   ) (t= t_{i}) | \\
\notag
&=&  \int_{r^{*}= r_{1}^{*} }^{r^{*} = (0.1)t_{i} + r_{1}^{*}    }  \int_{\S^{2}} ( [\frac{1}{r^{2}(1-\mu)} | \Psi_{w\th} |^{2} + \frac{1}{r^{2}\sin^{2}\th(1-\mu)} | \Psi_{w\phi} |^{2} ]  +  [ \frac{1}{r^{2}(1-\mu)} | \Psi_{v\th} |^{2} \\
\notag
&&  + \frac{1}{r^{2}\sin^{2}\th(1-\mu)} | \Psi_{v\phi} |^{2} ]  +  [ \frac{1}{(1-\mu)^{2}} |\Psi_{vw}|^{2}   +  \frac{1}{4r^{4}\sin^{2}\th} | \Psi_{\phi\th}|^{2}]  ). (1-\mu) r^{2}   d\sigma^{2} dr^{*} (t) \\
\notag
& \lesssim & \frac{E_{\Psi}^{(K)}(t_{i})}{\min_{w \in \{t_{i}\}\cap\{ -  r_{1}^{*}   \leq r^{*} \leq (0.1)t_{i} + r_{1}^{*}   \} } w^{2}}  + \frac{E_{\Psi}^{(K)}(t_{i})}{\min_{v \in \{t_{i}\}\cap\{  r_{1}^{*}   \leq r^{*} \leq (0.1)t_{i} + r_{1}^{*}   \} } v^{2}} \\
\notag
& \lesssim & \frac{E_{\Psi}^{(K)}(t_{i})}{\min_{r^{*} \in \{   r_{1}^{*}   \leq r^{*} \leq (0.1)t_{i} + r_{1}^{*}   \} } |t_{i} - r^{*}|^{2} }  + \frac{E_{\Psi}^{(K)}(t_{i})}{\min_{r^{*} \in \{  r_{1}^{*}   \leq r^{*} \leq (0.1)t_{i} + r_{1}^{*}   \} } |t_{i} + r^{*}|^{2} } \\
\eea

For $r^{*} \in [   r_{1}^{*}, (0.1)t_{i} + r_{1}^{*}   ]$,

\beaa
t_{i} - r_{1}^{*} \geq t_{i} - r^{*} \geq  t_{i} - [ (0.1)t_{i} + r_{1}^{*} ] = (0.9)t_{i} - r_{1}^{*} 
\eeaa
We have,

\beaa
(0.9)t_{i} - r_{1}^{*}  \geq 0 
\eeaa
(for $t_{i}$ large enough)\\

Thus,

\beaa
|(0.9)t_{i} - r_{1}^{*} |^{2}  \leq  |t_{i} - r^{*}|^{2} \leq |t_{i} - r_{1}^{*}|^{2}
\eeaa

\beaa
\min_{r^{*} \in \{   r_{1}^{*}   \leq r^{*} \leq (0.1)t_{i} + r_{1}^{*}   \} } |t_{i} - r^{*}|^{2} \geq |(0.9)t_{i} - r_{1}^{*} |^{2} 
\eeaa

Therefore,

\beaa
\frac{E_{\Psi}^{(K)}(t_{i})}{\min_{r^{*} \in \{   r_{1}^{*}   \leq r^{*} \leq (0.1)t_{i} + r_{1}^{*}   \} } |t_{i} - r^{*}|^{2} } &\lesssim & \frac{E_{\Psi}^{(K)}(t_i)}{t_{i}^{2} }
\eeaa

and thus,

\beaa
 | E^{(\frac{\pa}{\pa t})}_{\Psi}  (  r_{1}^{*} \leq r^{*} \leq (0.1)t_{i} + r_{1}^{*}   ) (t= t_{i}) | & \lesssim&  \frac{E_{\Psi}^{(K)}(t_{i})}{t_{i}^{2} } \\
\eeaa

Therefore,

\beaa
F_{\Psi}^{(\frac{\pa}{\pa t} )} ( w = w_{i}^{\#} ) ( v_{i}^{\#} \leq v \leq v_{i+1} )  &\les& \frac{E_{\Psi}^{(K)}(t_{i})}{t_{i}^{2} } \\
\eeaa

\end{proof}

We also have,

\beaa
| E^{(\frac{\pa}{\pa t})}_{\Psi}  (  -(0.85)t_{i} \leq r^{*} \leq (0.85)t_{i}  ) (t= t_{i}) |   &\les& \frac{E_{\Psi}^{(K)}(t_{i})}{t_{i}^{2} } \\
\eeaa

Thus,

\beaa
  - F_{\Psi}^{(H)} ( v = v_{i+1} ) ( w_{i}^{\#} \leq w \leq \infty ) &\les& \frac{1}{ t_{i} } [ |E_{\Psi}^{  (\frac{\pa}{\pa t} )} | + E_{\Psi}^{ \# (\frac{\pa}{\pa t} )} ( t= t_{0} )  ]  + \frac{E_{\Psi}^{(K)}(t_{i})}{t_{i}^{2} } \\
&\les& \frac{ [ |E_{\Psi}^{  (\frac{\pa}{\pa t} )} | + E_{\Psi}^{ \# (\frac{\pa}{\pa t} )} ( t= t_{0} ) + E_{\Psi}^{(K)} (t_{i}) ] }{ t_{i} } \\
\eeaa

(from the above)\\

and thus,

\bea
\notag
  - F_{\Psi}^{(H)} ( v = v_{i+1} ) ( w_{i+1} \leq w \leq \infty ) &\les&   - F_{\Psi}^{(H)} ( v = v_{i+1} ) ( w_{i}^{\#} \leq w \leq \infty ) \\
\notag
&\les& \frac{ [ |E_{\Psi}^{  (\frac{\pa}{\pa t} )} | + E_{\Psi}^{ \# (\frac{\pa}{\pa t} )} ( t= t_{0} ) + E_{\Psi}^{(K)} (t_{i})  ] }{ t_{i} } \\
\eea

(due to the positivity of $- F_{\Psi}^{(H)} ( v = v_{i+1} ) ( w_{i}^{\#} \leq w \leq \infty ) $).\\

Repeating the same procedure, 

\begin{lemma}
\beaa
 -  I_{\Psi}^{(H)}  (  v_{i+1} \leq v \leq v_{i+2} ) ( w_{i+1} \leq w \leq \infty) ( r \leq r_{1} ) &\lesssim& \frac{ [ |E_{\Psi}^{  (\frac{\pa}{\pa t} )} | + E_{\Psi}^{ \# (\frac{\pa}{\pa t} )} ( t= t_{0} ) + E_{\Psi}^{M} ] }{ t_{i} } \\
\eeaa

\end{lemma}

\begin{proof}

We get from \eqref{estimate3H1},

\beaa
&& -  I_{\Psi}^{(H)}  (  v_{i+1} \leq v \leq v_{i+2} ) ( w_{i+1} \leq w \leq \infty)  ( r \leq r_{1} ) \\
&\lesssim& F_{\Psi}^{(\frac{\pa}{\pa t} )} ( w = w_{i+1} ) ( v_{i+1} \leq v \leq v_{i+2} ) -   F_{\Psi}^{(H)} ( v = v_{i+1} ) ( w_{i+1} \leq w \leq \infty ) \\
&& +  | E^{(\frac{\pa}{\pa t})}_{\Psi}  (  -(0.85)t_{i+1} \leq r^{*} \leq (0.85)t_{i+1}  ) (t= t_{i+1}) | \\
\eeaa

(due to the positivity of $ - F_{\Psi}^{(H)} ( w = \infty ) ( v_{i+1} \leq v \leq v_{i+2} )  $, $ - F_{\Psi}^{(H)} ( v = v_{i+2} ) ( w_{i+1} \leq w \leq \infty )  $ and $-  I_{\Psi}^{(H)}  (  v_{i+1} \leq v \leq v_{i+2} ) ( w_{i+1} \leq w \leq \infty)  ( r \leq r_{1} )$) \\

We have,

\beaa
| E^{(\frac{\pa}{\pa t})}_{\Psi}  (  -(0.85)t_{i+1} \leq r^{*} \leq (0.85)t_{i+1}  ) (t= t_{i+1}) |   &\les& \frac{E_{\Psi}^{(K)}(t_{i+1})}{t_{i+1}^{2} } \\
\eeaa

and,

\beaa
  - F_{\Psi}^{(H)} ( v = v_{i+1} ) ( w_{i+1} \leq w \leq \infty ) 
&\les& \frac{ [ |E_{\Psi}^{  (\frac{\pa}{\pa t} )} | + E_{\Psi}^{ \# (\frac{\pa}{\pa t} )} ( t= t_{0} ) + E_{\Psi}^{(K)}(t_{i}) ] }{ t_{i} } \\
\eeaa

(from above)\\

And,

\beaa
F_{\Psi}^{(\frac{\pa}{\pa t} )} ( w = w_{i+1} ) ( v_{i+1} \leq v \leq v_{i+2} )  &\les& | E^{(\frac{\pa}{\pa t})}_{\Psi}  (  r_{1}^{*} \leq r^{*} \leq \frac{v_{i+2} - (2t_{i+1} - v_{i+2}) }{2}  ) (t= t_{i+1}) | \\
&\les& | E^{(\frac{\pa}{\pa t})}_{\Psi}  (  r_{1}^{*} \leq r^{*} \leq (0.1)t_{i+1} + r_{1}^{*}   ) (t= t_{i+1}) | \\
\eeaa

We proved,

\beaa
 | E^{(\frac{\pa}{\pa t})}_{\Psi}  (  r_{1}^{*} \leq r^{*} \leq (0.1)t_{i+1} + r_{1}^{*}   ) (t= t_{i+1}) | & \lesssim&  \frac{E_{\Psi}^{(K)}(t_{i+1})}{t_{i+1}^{2} } \\
\eeaa

hence,

\beaa
F_{\Psi}^{(\frac{\pa}{\pa t} )} ( w = w_{i+1} ) ( v_{i+1} \leq v \leq v_{i+2} )  &\les& \frac{E_{\Psi}^{(K)}(t_{i+1})}{t_{i+1}^{2} } \\
\eeaa

Therefore,

\beaa
 -  I_{\Psi}^{(H)}  (  v_{i+1} \leq v \leq v_{i+2} ) ( w_{i+1} \leq w \leq \infty) ( r \leq r_{1} ) &\lesssim& \frac{ [ |E_{\Psi}^{  (\frac{\pa}{\pa t} )} | + E_{\Psi}^{ \# (\frac{\pa}{\pa t} )} ( t= t_{0} ) + E_{\Psi}^{M} ] }{ t_{i} } \\
\eeaa

\end{proof}

Using \eqref{estimate4H}, we get,

\beaa
 \inf_{ v_{i+1} \leq v \leq v_{i+2} } F_{\Psi}^{(H)} ( v  ) ( w_{i+1} \leq w \leq \infty )  &\lesssim & \frac{ [ |E_{\Psi}^{  (\frac{\pa}{\pa t} )} | + E_{\Psi}^{ \# (\frac{\pa}{\pa t} )} ( t= t_{0} ) + E_{\Psi}^{M} ] }{ t_{i+1} (v_{i+2} - v_{i+1} ) }  + \frac{E_{\Psi}^{M}}{t_{i+1}^{2} } \\
 &\lesssim & \frac{ [ |E_{\Psi}^{  (\frac{\pa}{\pa t} )} | + E_{\Psi}^{ \# (\frac{\pa}{\pa t} )} ( t= t_{0} ) + E_{\Psi}^{M} ] }{ t_{i+1}^{2}  }  \\
\eeaa

and thus, there exists a $ v_{i+1}^{\#} \in [ v_{i+1}, v_{i+2} ] $ where above inequality holds. We get,

\beaa
 F_{\Psi}^{(H)} ( v_{i+1}^{\#}  ) ( w_{i+1}^{\#} \leq w \leq \infty )  &\lesssim & F_{\Psi}^{(H)} ( v_{i+1}^{\#}  ) ( w_{i+1} \leq w \leq \infty )  \\
&\lesssim & \frac{ [ |E_{\Psi}^{  (\frac{\pa}{\pa t} )} | + E_{\Psi}^{ \# (\frac{\pa}{\pa t} )} ( t= t_{0} ) + E_{\Psi}^{M} ] }{ t_{i+1}^{2}  }  \\
\eeaa

As before, we let,
$$w_{i+1}^{\#} = v_{i+1}^{\#}  - 2 r_{1}^{*} $$

From \eqref{estimate3H1}, applied in the region $[w_{i+1}^{\#}, \infty].[ v_{i+1}^{\#}, v_{i+2}] $, we get due to the positivity of $-  I_{\Psi}^{(H)}  (  v_{i+1}^{\#} \leq v \leq v_{i+2} ) ( w_{i+1}^{\#} \leq w \leq \infty)  $, and $- F_{\Psi}^{(H)} ( w = \infty ) ( v_{i+1}^{\#} \leq v \leq v_{i+2} )$, that,

\beaa
&&  - F_{\Psi}^{(H)} ( v = v_{i+2} ) ( w_{i+1}^{\#} \leq w \leq \infty )  \\
&\lesssim& F_{\Psi}^{(\frac{\pa}{\pa t} )} ( w = w_{i+1}^{\#} ) ( v_{i+1}^{\#} \leq v \leq v_{i+2} ) -   F_{\Psi}^{(H)} ( v = v_{i+1}^{\#} ) ( w_{i+1}^{\#} \leq w \leq \infty ) \\
&& + | E^{(\frac{\pa}{\pa t})}_{\Psi}  (  -(0.85)t_{i+1} \leq r^{*} \leq (0.85)t_{i+1}  ) (t= t_{i+1}) |  \\
\eeaa

We proved that,

\beaa
F_{\Psi}^{(\frac{\pa}{\pa t} )} ( w = w_{i+1}^{\#} ) ( v_{i+1}^{\#} \leq v \leq v_{i+2} )  &\les& \frac{E_{\Psi}^{(K)}(t_{i+1})}{t_{i+1}^{2} } \\
\eeaa

We also have,

\beaa
| E^{(\frac{\pa}{\pa t})}_{\Psi}  (  -(0.85)t_{i+1} \leq r^{*} \leq (0.85)t_{i+1}  ) (t= t_{i+1}) |   &\les& \frac{E_{\Psi}^{(K)}(t_{i+1})}{t_{i+1}^{2} } \\
\eeaa

Therefore, 

\beaa
  - F_{\Psi}^{(H)} ( v = v_{i+2} ) ( w_{i+1}^{\#} \leq w \leq \infty )  &\lesssim & \frac{ [ |E_{\Psi}^{  (\frac{\pa}{\pa t} )} | + E_{\Psi}^{ \# (\frac{\pa}{\pa t} )} ( t= t_{0} ) + E_{\Psi}^{M} ] }{ t_{i+1}^{2}  }  \\
\eeaa

and finally,

\beaa
 - F_{\Psi}^{(H)} ( v = v_{i+2} ) ( w_{i+2} \leq w \leq \infty )  &\lesssim &  - F_{\Psi}^{(H)} ( v = v_{i+2} ) ( w_{i+1}^{\#} \leq w \leq \infty )  \\
 &\lesssim & \frac{ [ |E_{\Psi}^{  (\frac{\pa}{\pa t} )} | + E_{\Psi}^{ \# (\frac{\pa}{\pa t} )} ( t= t_{0} ) + E_{\Psi}^{M} ] }{ t_{i+1}^{2}  }  \\
 &\lesssim & \frac{ [ |E_{\Psi}^{  (\frac{\pa}{\pa t} )} | + E_{\Psi}^{ \# (\frac{\pa}{\pa t} )} ( t= t_{0} ) + E_{\Psi}^{M} ] }{  v_{+ (i+1)}^{2} }  \\
\eeaa

(due to the positivity of $- F_{\Psi}^{(H)} ( v = v_{i+2} ) ( w_{i+1}^{\#} \leq w \leq \infty ) $).\\

\end{proof}

\section{The Proof of Decay Near the Horizon}\

Let $v_{+}$ be as defined in \eqref{definitionvplus}, we will prove that for all $r \leq r_{1}$,

\beaa
 |F_{\hat{v} \hat{w}} (v, w, \om) | &\lesssim & \frac{ E_{1}  }{ v_{+}  }  \\
 |F_{e_{1} e_{2}} (v, w, \om) | &\lesssim & \frac{ E_{1}  }{ v_{+}  }  
\eeaa

where,

\beaa
E_{1} &= & [ \sum_{j = 0}^{6}    E_{   r^{j}(\rLie)^{j}   F }^{(\frac{\pa}{\pa t}) } (t=t_{0})  + \sum_{j = 0}^{5}    E_{   r^{j}(\rLie)^{j}   F }^{( K ) } (t=t_{0})  +  \sum_{j=1}^{3}   E_{r^{j}(\rLie)^{j} F }^{ \# (\frac{\pa}{\pa t} )} ( t= t_{0} ) ]^{\frac{1}{2}}  
\eeaa

and,

\beaa
 |F_{ \hat{v}e_{1} } (v, w, \om ) | &\les&  \frac{  E_{2} }{  v_{+} } \\
|F_{ \hat{v}e_{2} } (v, w, \om ) | &\les&  \frac{  E_{2} }{  v_{+} } \\
| \sqrt{1-\frac{2m}{r}}  F_{ \hat{w}e_{1} } (v, w, \om ) | &\lesssim &  \frac{ E_{2}  }{ v_{+}  }  \\ 
| \sqrt{1-\frac{2m}{r}} F_{ \hat{w}e_{2} } (v, w, \om ) | &\lesssim &  \frac{ E_{2}  }{ v_{+}  }  
\eeaa

where,

\beaa
E_{2} &=& [ E_{F}^{2} + \sum_{i=0}^{1}  \sum_{j=1}^{2}    E_{ r^{j}(\rLie)^{j} (\Lie_{t})^{i} F }^{ \# (\frac{\pa}{\pa t} )} ( t= t_{0} )   +    E_{r^{3}(\rLie)^{3} F }^{ \# (\frac{\pa}{\pa t} )} ( t= t_{0} )  ]^{\frac{1}{2}} \\
&= & [ \sum_{i=0}^{1} \sum_{j = 0}^{5}  E_{   r^{j}(\rLie)^{j}  (\Lie_{t})^{i}  \Psi }^{(\frac{\pa}{\pa t}) } (t=t_{0})    +  E_{  r^{6}(\rLie)^{6}   \Psi }^{(\frac{\pa}{\pa t}) } (t=t_{0}) \\
&& + \sum_{i=0}^{1}  \sum_{j = 0}^{4}   E_{   r^{j}(\rLie)^{j}  (\Lie_{t})^{i}  \Psi }^{( K ) } (t=t_{0})  +  E_{     r^{5}(\rLie)^{5} }^{( K ) } (t=t_{0}) \\
&& + \sum_{i=0}^{1} \sum_{j=1}^{2}      E_{ r^{j}(\rLie)^{j} (\Lie_{t})^{i} F }^{ \# (\frac{\pa}{\pa t} )} ( t= t_{0} )   +    E_{r^{3}(\rLie)^{3} F }^{ \# (\frac{\pa}{\pa t} )} ( t= t_{0} )  ]^{\frac{1}{2}}
\eeaa

\begin{proof}\

\subsection{Decay for $F_{\hat{v} \hat{w}}$  and $F_{e_{1} e_{2}}$ } \

We have the Sobolev inequality, 

\beaa
|F_{ \hat{v}\hat{w} } (v, w, \om ) |^{2}   
&\lesssim&   \int_{ \overline{v} = v - 1 }^{ \overline{v} = v } \int_{\S^{2}} ( | F_{ \hat{v}\hat{w} } |^{2} + |  \Lie_{v}   F_{ \hat{v}\hat{w} } |^{2}  + | \rLie F_{ \hat{v}\hat{w} } |^{2}  \\
&& +  | \rLie \Lie_{v} F_{ \hat{v}\hat{w} } |^{2}  +  | (\rLie)^{2}   F_{ \hat{v}\hat{w} } |^{2}  + | (\rLie)^{2}  \Lie_{v}  F_{ \hat{v}\hat{w} } |^{2}  ) d\sigma^{2} d\overline{v}
\eeaa

\beaa
\Lie_{\hat{v} } \Psi_{\hat{v} \hat{w}} &=& \der_{\hat{v} } \Psi_{\hat{v} \hat{w}} - \Psi ( \der_{\hat{v} } \hat{\frac{\pa}{\pa v}}, \hat{\frac{\pa}{\pa w}} )  -  \Psi (\hat{\frac{\pa}{\pa v}}, \der_{\hat{v} } \hat{\frac{\pa}{\pa w}} )  \\
&=& - \frac{1}{2} {\der}^{\hat{w} }  \Psi_{\hat{v} \hat{w}}  - \frac{\mu}{2 r } \Psi_{\hat{v} \hat{w}} + \frac{\mu}{2 r }  \Psi_{\hat{v} \hat{w}}   \\
&=& \frac{1}{2} [ {\der}^{\hat{e_{a}} }  \Psi_{\hat{v} \hat{e_{a}} }+ {\der}^{\hat{e_{b}} }  \Psi_{\hat{v} \hat{e_{b}} } ]\\
&=&   \frac{1}{2} [ \der_{\hat{e_{a}} } \Psi_ { \hat{v} \hat{e_{a}} } + \der_{\hat{e_{b}} } \Psi_{ \hat{v} \hat{e_{b}} } ] \\
&=&  \frac{1}{2} [ \Lie_{\hat{e_{a}} }  \Psi_ { \hat{v} \hat{e_{a}} } - \Psi ( \der_{\hat{e_{a}} } \hat{v}, \hat{e_{a}} ) -  \Psi ( \hat{v}, \der_{\hat{e_{a}} } \hat{e_{a}} ) + \Lie_{\hat{e_{b}} }  \Psi_{ \hat{v} \hat{e_{b}} } - \Psi ( \der_{\hat{e_{b}} } \hat{v}, \hat{e_{b}} ) -  \Psi ( \hat{v}, \der_{\hat{e_{b}} } \hat{e_{b}} ) ]\\
&=&  \frac{1}{2} [ \Lie_{\hat{e_{a}} }  \Psi_ { \hat{v} \hat{e_{a}} }  -  \Psi ( \hat{v}, \der_{\hat{e_{a}} } \hat{e_{a}} ) + \Lie_{\hat{e_{b}} }  \Psi_{ \hat{v} \hat{e_{b}} }  -  \Psi ( \hat{v}, \der_{\hat{e_{b}} } \hat{e_{b}} )  ]
\eeaa

(by using the field equations) \\

Thus,

\beaa
&& \int_{ \overline{v} = v - 1 }^{ \overline{v} = v } \int_{\S^{2}}  | \Lie_{v} \Psi_{\hat{v} \hat{w}} (v, w, \om) |^{2} d\sigma^{2} d\overline{v} \\
& \lesssim&    \int_{ \overline{v} = v - 1 }^{ \overline{v} = v } \int_{\S^{2}} | - \der_{\hat{e_{a}} }\Psi_ { \hat{v} \hat{e_{a}} } - \der_{\hat{e_{b}} } \Psi_{ \hat{v} \hat{e_{b}} } |^{2} d\sigma^{2} d\overline{v}  \\
&\lesssim&    \int_{ \overline{v} = v - 1 }^{ \overline{v} = v } \int_{\S^{2}} ( |   \der_{ \hat{e_{a}} } \Psi_{\hat{v} \hat{e_{a}}  } |^{2}  +  | \der_{\hat{e_{b}} }\Psi_{ \hat{v} \hat{e_{b}} } |^{2} ) d\sigma^{2} d\overline{v} \\
&\lesssim&   \int_{ \overline{v} = v - 1 }^{ \overline{v} = v } \int_{\S^{2}}  ( |   \rLie \Psi_{\hat{v} \hat{e_{a}} } |^{2}  +  | \rLie \Psi_{ \hat{v} \hat{e_{b}} } |^{2} +   |  \Psi_{ \hat{v} \hat{w} }  |^{2} +   | \Psi_{ \hat{v}\hat{e_{a}} }  |^{2} +   |  \Psi_{ \hat{v}\hat{e_{b}} }  |^{2} ) d\sigma^{2} d\overline{v}   \\
&\lesssim&   \int_{ \overline{v} = v - 1 }^{ \overline{v} = v } \int_{\S^{2}}  ( |  r \rLie \Psi_{\hat{v} \hat{e_{a}} } |^{2}  +  | r \rLie \Psi_{ \hat{v} \hat{e_{b}} } |^{2} +   |  \Psi_{ \hat{v} \hat{w} }  |^{2} +   | \Psi_{ \hat{v}\hat{e_{a}} }  |^{2} +   |  \Psi_{ \hat{v}\hat{e_{b}} }  |^{2} ) r^{2} d\sigma^{2} d\overline{v}   \\
\eeaa 
(for $r \geq 2m > 0$)\\

Recall that,
\beaa
\notag
&& - F_{\Psi}^{(H)} ( w  ) ( v-1 \leq  \overline{v} \leq v ) \\
&=& - F_{\Psi}^{(H)} ( w  ) ( v-1 \leq  \overline{v} \leq v )  - F_{\Psi}^{(H_{2})} ( w  ) ( v-1 \leq  \overline{v} \leq v )  \\
\notag
&=& \int_{ \overline{v} = v - 1 }^{ \overline{v} = v } \int_{\S^{2}}   2  h(r^{*}) [  \frac{1}{r^{2}} | \Psi_{v\th} |^{2} + \frac{1}{r^{2}\sin^{2}\th } | \Psi_{v\phi} |^{2}  \\
\notag
&& +   \frac{1}{(1-\mu)^{2}} | \Psi_{v w} |^{2} + \frac{1}{4 r^{4} \sin^{2}\th} | \Psi_{\th \phi} |^{2}  ] r^{2} d\sigma^{2} d\overline{v} \\
&=& \int_{ \overline{v} = v - 1 }^{ \overline{v} = v } \int_{\S^{2}}   2  h(r^{*}) [  | \Psi_{v e_{1}} |^{2} + | \Psi_{v e_{2}} |^{2}   +   \frac{1}{(1-\mu)^{2}} | \Psi_{v w} |^{2} + \frac{1}{4 } | \Psi_{e_{1} e_{2}} |^{2}  ] r^{2} d\sigma^{2} d\overline{v}
\eeaa
(by computing \eqref{firstexpressionforthefluxofHlonwequalconstant} using the orthonormal basis $e_{a}$, $a \in \{1, 2 \}$ instead of $\frac{\pa}{\pa \hat{ \th}}$ and $\frac{\pa}{\pa \hat{ \phi}}$ which are singular at $\th = 0, \pi$).

Consequently,

\beaa
 \int_{ \overline{v} = v - 1 }^{ \overline{v} = v } \int_{\S^{2}}  | \Lie_{v} \Psi_{\hat{v} \hat{w}} (v, w, \om) |^{2} d\sigma^{2} d\overline{v}  &\lesssim&     - F_{\Psi, r{\rLie} \Psi }^{(H)} ( w  ) ( v-1 \leq  \overline{v} \leq v )  \\
\eeaa

Therefore, we have,
\beaa
|F_{\hat{v} \hat{w}} (v, w, \om) | &\lesssim&    \int_{ \overline{v} = v - 1 }^{ \overline{v} = v } \int_{\S^{2}} ( | F_{ \hat{v}\hat{w} } |^{2} + |  \Lie_{v}   F_{ \hat{v}\hat{w} } |^{2}  + | \rLie F_{ \hat{v}\hat{w} } |^{2}  \\
&& +  | \rLie \Lie_{v} F_{ \hat{v}\hat{w} } |^{2}  +  | (\rLie)^{2}   F_{ \hat{v}\hat{w} } |^{2}  + | (\rLie)^{2}  \Lie_{v}  F_{ \hat{v}\hat{w} } |^{2}  ) d\sigma^{2} d\overline{v} \\
&\lesssim&   - F_{F, r{\rLie} F, r^{2}(\rLie)^{2}  F , r^{3}(\rLie)^{3}  F }^{(H)}  ( w  ) ( v-1 \leq  \overline{v} \leq v )   
\eeaa

From \eqref{estimate6Hwequalconstantflux}, we proved that,
\beaa
 - F_{\Psi}^{(H)}  ( w  ) ( v-1 \leq  \overline{v} \leq v )  &\lesssim &  \frac{ [ |E_{\Psi}^{  (\frac{\pa}{\pa t} )} | + E_{\Psi}^{ \# (\frac{\pa}{\pa t} )} ( t= t_{0} ) + E_{\Psi}^{M} ] }{ v_{+}^{2} }  
\eeaa

Thus,
\beaa
&& - F_{F, r{\rLie} F, r^{2}(\rLie)^{2}   F , r^{3}(\rLie)^{3}  F }^{(H)} ( w  ) ( v-1 \leq  \overline{v} \leq v ) \\
&\lesssim & \frac{ E_{F, r{\rLie} F, r^{2}(\rLie)^{2}  F , r^{3}(\rLie)^{3}  F }^{  (\frac{\pa}{\pa t} )} ( t= t_{0} )  + E_{F, r{\rLie} F, r^{2}(\rLie)^{2}  F , r^{3}(\rLie)^{3}  F }^{ \# (\frac{\pa}{\pa t} )} ( t= t_{0} )}{ v_{+}^{2}  }\\
&& + \frac{ E_{F, r{\rLie} F, r^{2}(\rLie)^{2}  F , r^{3}(\rLie)^{3}   F }^{M} ( t= t_{0} ) }{ v_{+}^{2}  }  \\
\eeaa

\beaa
&& E_{F, r{\rLie} F, r^{2}(\rLie)^{2}  F , r^{3}(\rLie)^{3}  F }^{  (\frac{\pa}{\pa t} )} ( t= t_{0} )  + E_{F, r{\rLie} F, r^{2}(\rLie)^{2}  F , r^{3}(\rLie)^{3}  F }^{ \# (\frac{\pa}{\pa t} )} ( t= t_{0} ) \\
&& + E_{F, r{\rLie} F, r^{2}(\rLie)^{2}  F , r^{3}(\rLie)^{3}   F }^{M} ( t= t_{0} ) \\
&=&  \sum_{j=0}^{3}  E_{r^{j}(\rLie)^{j} F }^{ (\frac{\pa}{\pa t} )} ( t= t_{0} ) + \sum_{j=0}^{3}  E_{r^{j}(\rLie)^{j} F }^{ \# (\frac{\pa}{\pa t} )} ( t= t_{0} ) +  \sum_{j=0}^{3}  E_{ r^{j}(\rLie)^{j} F }^{M}  \\
\eeaa

Recall that,

\beaa
E^{M}_{F} &=& \sum_{j=0}^{3} E_{ r^{j}(\rLie)^{j} F  }^{(\frac{\pa}{\pa t})} (t=t_{0})    + \sum_{j=0}^{2}  E_{ r^{j}(\rLie)^{j} F  }^{(K)} (t=t_{0}) \\
\eeaa

Thus,

\beaa
  \sum_{j=0}^{3}   E_{ r^{j}(\rLie)^{j} F }^{(M)}     &=&  E_{F, r{\rLie} F, r^{2}(\rLie)^{2} F , r^{3}(\rLie)^{3}  F, r^{4}(\rLie)^{4}  F, r^{5}(\rLie)^{5}  F }^{(\frac{\pa}{\pa t} + K)} (t=t_{0}) + E_{r^{6}(\rLie)^{6}  F }^{(\frac{\pa}{\pa t} )} (t=t_{0}) \\
&\les&  \sum_{j=1}^{6} E_{ r^{j}(\rLie)^{j} F  }^{(\frac{\pa}{\pa t} )} (t=t_{0}) +  \sum_{j=1}^{5} E_{ r^{j}(\rLie)^{j} F  }^{(K)} (t=t_{0}) \\
\eeaa

We get,

\beaa
&& |F_{\hat{v} \hat{w}} (v, w, \om) |^{2} \\
&\lesssim & \frac{  \sum_{j=0}^{3}   E_{r^{j}(\rLie)^{j} F }^{ \# (\frac{\pa}{\pa t} )} ( t= t_{0} )   + \sum_{j=0}^{6}  E_{ r^{j}(\rLie)^{j} F  }^{(\frac{\pa}{\pa t} )} (t=t_{0}) +  \sum_{j=0}^{5}  E_{ r^{j}(\rLie)^{j} F  }^{(K)} (t=t_{0})    }{ v_{+}  }  
\eeaa

Finally we get,

\beaa
 |F_{\hat{v} \hat{w}} (v, w, \om) | &\lesssim & \frac{ E_{1}  }{ v_{+}  }  
\eeaa

where $E_{1}$ is defined as follows,
\beaa
E_{1} &= & [ \sum_{j = 0}^{6}    E_{   r^{j}(\rLie)^{j}   F }^{(\frac{\pa}{\pa t}) } (t=t_{0})  + \sum_{j = 0}^{5}    E_{   r^{j}(\rLie)^{j}   F }^{( K ) } (t=t_{0})  +  \sum_{j=1}^{3}   E_{r^{j}(\rLie)^{j} F }^{ \# (\frac{\pa}{\pa t} )} ( t= t_{0} ) ]^{\frac{1}{2}}  \\
\eeaa

Concerning the component $F_{e_{1} e_{2} }$, similarly, we have the Sobolev inequality, 

\beaa
|F_{ e_{1} e_{2} } (v, w, \om ) |^{2}   
&\lesssim&   \int_{ \overline{v} = v - 1 }^{ \overline{v} = v } \int_{\S^{2}} ( | F_{ e_{1} e_{2} } |^{2} + |  \Lie_{v}   F_{ e_{1} e_{2} } |^{2}  + | \rLie F_{ e_{1} e_{2} } |^{2}  \\
&& +  | \rLie \Lie_{v} F_{ e_{1} e_{2} } |^{2}  +  | (\rLie)^{2}   F_{ e_{1} e_{2} } |^{2}  + | (\rLie)^{2}  \Lie_{v}  F_{ e_{1} e_{2} } |^{2}  ) d\sigma^{2} d\overline{v}
\eeaa

\beaa
\Lie_{\hat{v} } \Psi_{e_{1} e_{2} }  &=& \der_{\hat{v} } \Psi_{e_{1} e_{2} } - \Psi ( \der_{\hat{v} } e_{1}, e_{2} ) - \Psi ( e_{1}, \der_{\hat{v} } e_{2} ) \\
&=& - \der_{e_{1} } \Psi_{e_{2} \hat{v}} - \der_{e_{2} } \Psi_{ \hat{v} e_{1}  } \\
&=&  \der_{e_{1} } \Psi_ { \hat{v} e_{2} } - \der_{e_{2} } \Psi_{ \hat{v} e_{1} } \\
&=&  \Lie_{e_{1} } \Psi_ { \hat{v} e_{2} } - \Psi ( \der_{e_{1} } \hat{v}, e_{2} )  - \Psi ( \hat{v}, \der_{e_{1} } e_{2} ) - \Lie_{e_{2} } \Psi_{ \hat{v} e_{1} } +   \Psi ( \der_{e_{2} }  \hat{v}, e_{1} ) +  \Psi ( \hat{v},  \der_{e_{2} } e_{1} )  
\eeaa

Therefore,
\beaa
&& \int_{ \overline{v} = v - 1 }^{ \overline{v} = v } \int_{\S^{2}}  | \Lie_{v} \Psi_{e_{1} e_{2} } (v, w, \om) |^{2} d\sigma^{2} d\overline{v}\\ 
&\lesssim&   \int_{ \overline{v} = v - 1 }^{ \overline{v} = v }  \int_{\S^{2}} ( |   \Lie_{ e_{1} }  \Psi_{\hat{v} e_{2}  } |^{2}  +  | \Lie_{e_{2} } \Psi_{ \hat{v} e_{1} } |^{2} +  |  \Psi_{e_{1} e_{2}  } |^{2} +  |   \Psi_{\hat{v} e_{2}  } |^{2} ) d\sigma^{2} d\overline{v}   \\
&\lesssim&   \int_{ \overline{v} = v - 1 }^{ \overline{v} = v }  \int_{\S^{2}} ( |   \rLie \Psi_{\hat{v} e_{1} } |^{2}  +  | \rLie \Psi_{ \hat{v} e_{2} } |^{2} +   |  \Psi_{e_{1} e_{2}  } |^{2} +  |   \Psi_{\hat{v} e_{2}  } |^{2}  ) d\sigma^{2} d\overline{v}   \\
&\lesssim& - F_{\Psi, r{\rLie} \Psi  }^{(H)} ( w  ) ( v-1 \leq \overline{v} \leq v) 
\eeaa 

(using what we already proved, and the fact that $r$ is bounded in the region $0< 2m \leq r \leq R$).

Consequently, we have,

\beaa
|F_{e_{1} e_{2}} (v, w, \om) |^{2} &\lesssim&   - F_{F, r{\rLie} F, r^{2}(\rLie)^{2}  F , r^{3}(\rLie)^{3}  F }^{(H)} ( w  ) ( v-1 \leq \overline{v} \leq v)  \\
\eeaa

We proved that,
\beaa
&& - F_{F, r{\rLie} F, r^{2}(\rLie)^{2}  F , r^{3}(\rLie)^{3}  F }^{(H)} ( w  ) ( v-1 \leq \overline{v} \leq v) \\
  &\lesssim & \frac{ E_{F, r{\rLie} F, r^{2}(\rLie)^{2}  F , r^{3}(\rLie)^{3}  F }^{ \# (\frac{\pa}{\pa t} )} ( t= t_{0} ) + E_{F, r{\rLie} F, r^{2}(\rLie)^{2}  F , r^{3}(\rLie)^{3}  F }^{M}  }{ v_{+}^{2}  }  
\eeaa

Thus,
\beaa
 |F_{e_{1} e_{2}} (v, w_{0}, \om) | | &\lesssim & \frac{ E_{1} }{ v_{+} }  \\
\eeaa

\subsection{ Decay for $F_{\hat{v} e_{1} }$  and $F_{\hat{v} e_{2} }$ }\

We have the Sobolev inequality, 
\beaa
|F_{ \hat{v}\hat{\th} } (v, w, \om ) |^{2}   
&\lesssim&   \int_{ \overline{v} = v - 1 }^{ \overline{v} = v } \int_{\S^{2}} ( | F_{ \hat{v}\hat{\th} } |^{2} + |  \Lie_{v}   F_{ \hat{v}\hat{\th} } |^{2}  + | \rLie F_{ \hat{v}\hat{\th} } |^{2}  \\
&& +  | \rLie \Lie_{v} F_{ \hat{v}\hat{\th} } |^{2}  +  | (\rLie)^{2}   F_{ \hat{v}\hat{\th} } |^{2}  + | (\rLie)^{2}  \Lie_{v}  F_{ \hat{v}\hat{\th} } |^{2}  ) d\sigma^{2} d\overline{v}
\eeaa

On one hand, we can compute,
\beaa
\Lie_{\hat{v} } \Psi_{\hat{v} \hat{\th} } &=& \der_{\hat{v} } \Psi_{\hat{v} \hat{\th} }  +  \Psi ( \der_{\hat{v} } \hat{v}, \hat{\th} )  +  \Psi ( \hat{v}, \der_{\hat{v} } \hat{\th} )   \\
&=&  \der_{\hat{v} } \Psi_{\hat{v} \hat{\th} }  + \frac{\mu}{2 r   } \Psi_{\hat{v} \hat{\th} }     
\eeaa

\beaa
\der_{\hat{v} } \Psi_{\hat{v} \hat{\th} } &=&  \frac{1}{2}  \der_{\hat{v} } ( \Psi_{t \hat{\th} } +  \Psi_{r^{*} \hat{\th} }  ) = \frac{1}{2}  \der_{ ( \frac{1}{2} \frac{\pa }{\pa t} + \frac{1}{2} \frac{\pa }{\pa r^{*}} )  } ( \Psi_{t \hat{\th} } +  \Psi_{r^{*} \hat{ \th}}  ) \\
&=& \frac{1}{4}  \der_{t} ( \Psi_{t \hat{\th}} + \Psi_{r^{*} \hat{ \th}}  ) + \frac{1}{4}  \der_{r^{*}} ( \Psi_{t \hat{\th}} +  \Psi_{r^{*} \hat{ \th}}  ) \\
&=& \frac{1}{4}  \der_{t } ( \Psi_{t \hat{\th} } +  \Psi_{r^{*} \hat{\th}}  )  
+ \frac{1}{4}  ( - \der_{t }  \Psi_{ \hat{\th} r^{*} } - \der_{ \hat{\th} }  \Psi_{r^{*} t } + \der_{t } \Psi_{ t \hat{\th} }  - (1-\mu) \der_{\hat{\phi} } \Psi_{ \hat{\phi} \hat{\th} }   )  \\
&=& \frac{1}{2} \der_{t } ( \Psi_{t \hat{\th} } +  \Psi_{r^{*} \hat{\th}}  )  - \frac{1}{4}  (   \der_{ \hat{\th} }  \Psi_{r^{*} t }   + (1-\mu) \der_{\hat{\phi} } \Psi_{ \hat{\phi} \hat{\th} }   ) \\
&=&  \der_{t }  \Psi_{\hat{v} \hat{\th} }   - \frac{1}{4}  (1-\mu)  (   \der_{ \hat{\th} }  \Psi_{\hat{r^{*}} \hat{t} }   +   \der_{\hat{\phi} } \Psi_{ \hat{\phi} \hat{\th} }   ) 
\eeaa
(where we used the field equations and the Bianchi identities).

Thus,

\beaa
\der_{\hat{v} } \Psi_{\hat{v} \hat{\th} } &=&  \Lie_{t }  \Psi_{\hat{v} \hat{\th} } -  \Psi (\der_{t } \hat{v}, \hat{\th} ) -  \Psi (\hat{v}, \der_{t } \hat{\th} ) \\
&& + \frac{ (1-\mu)}{4} [ - \Lie_{ \hat{\th} }  \Psi_{\hat{r^{*}} \hat{t} }  + \Psi ( \der_{ \hat{\th} }  \hat{r^{*}}, \hat{t} ) +  \Psi (\hat{r^{*}},  \der_{ \hat{\th} } \hat{t} )  \\
&& -  \Lie_{\hat{\phi} } \Psi_{ \hat{\phi} \hat{\th} }  +     \Psi(\der_{\hat{\phi} } \hat{\phi}, \hat{\th} )  +  \Psi ( \hat{\phi}, \der_{\hat{\phi} } \hat{\th} )  ]  \\
&=&  \Lie_{t }  \Psi_{\hat{v} \hat{\th} } -  \frac{\mu  }{2 r  } \Psi_{\hat{v} \hat{\th} }  \\
&& + \frac{1}{4} (1-\mu) [ - \Lie_{ \hat{\th} }  \Psi_{\hat{r^{*}} \hat{t} }  +  \frac{\sqrt{1-\mu} }{r}  \Psi_{\hat{\th} \hat{t}}   -  \Lie_{\hat{\phi} } \Psi_{ \hat{\phi} \hat{\th} }  -   \frac{\sqrt{1-\mu} }{r}  \Psi_{\hat{r^{*}} \hat{\th}} ]\\
&=&  \Lie_{t }  \Psi_{\hat{v} \hat{\th} }  - \frac{1}{4} (1-\mu) \Lie_{ \hat{\th} }  \Psi_{\hat{r^{*}} \hat{t} } - \frac{1}{4} (1-\mu)  \Lie_{\hat{\phi} } \Psi_{ \hat{\phi} \hat{\th} }  \\
&& - \frac{1}{2}  ( \frac{\mu}{r  }    +  \frac{(1-\mu) }{r} )  \Psi_{\hat{v} \hat{\th}} \\
&=&  \Lie_{t }  \Psi_{\hat{v} \hat{\th} }   - \frac{(1-\mu)}{2}     \Lie_{ \hat{\th} }  \Psi_{\hat{v} \hat{w} }   - \frac{(1-\mu)}{4}  \Lie_{\hat{\phi} } \Psi_{ \hat{\phi} \hat{\th} }   \\
&& - \frac{1}{2}  ( \frac{\mu}{r  }    +  \frac{(1-\mu) }{r} )  \Psi_{\hat{v} \hat{\th}}
\eeaa

Hence,
\beaa
\Lie_{\hat{v} } \Psi_{\hat{v} \hat{\th} } &=&   \der_{\hat{v} } \Psi_{\hat{v} \hat{\th} }  + \frac{\mu}{2 r   } \Psi_{\hat{v} \hat{\th} }    \\
&=&  \Lie_{t }  \Psi_{\hat{v} \hat{\th} }   - \frac{(1-\mu)}{2}     \Lie_{ \hat{\th} }  \Psi_{\hat{v} \hat{w} }   - \frac{(1-\mu)}{4}  \Lie_{\hat{\phi} } \Psi_{ \hat{\phi} \hat{\th} }    - \frac{(1-\mu)}{2 r}   \Psi_{\hat{v} \hat{\th}}
\eeaa

and similarly, we obtain,

\beaa
\Lie_{\hat{v} } \Psi_{\hat{v} \hat{\phi} } 
&=&  \Lie_{t }  \Psi_{\hat{v} \hat{\phi} }   - \frac{(1-\mu)}{2}     \Lie_{ \hat{\phi} }  \Psi_{\hat{v} \hat{w} }   - \frac{(1-\mu)}{4}  \Lie_{\hat{\th} } \Psi_{ \hat{\th} \hat{\phi} }   - \frac{(1-\mu)}{2 r}   \Psi_{\hat{v} \hat{\phi}}
\eeaa

Thus, for $r \geq 2m > 0$,

\beaa
| \Lie_{v } \Psi_{\hat{v} \hat{\th} } |^{2} &=& | \Lie_{t }  \Psi_{\hat{v} \hat{\th} }   - \frac{(1-\mu)}{2}     \Lie_{ \hat{\th} }  \Psi_{\hat{v} \hat{w} }   - \frac{(1-\mu)}{4}  \Lie_{\hat{\phi} } \Psi_{ \hat{\phi} \hat{\th} }    - \frac{(1-\mu)}{2 r}   \Psi_{\hat{v} \hat{\th}} |^{2} \\
&\les&   | \Lie_{t }  \Psi_{\hat{v} \hat{\th} } |^{2}  + |  \Lie_{ \hat{\th} }  \Psi_{\hat{v} \hat{w} } |^{2}  + |   \Lie_{\hat{\phi} } \Psi_{ \hat{\phi} \hat{\th} }  |^{2} + |\Psi_{\hat{v} \hat{\th}} |^{2}
\eeaa
(by using $a.b \les a^{2} + b^{2}$)\\

We get,
\beaa
| \Lie_{v } \Psi_{\hat{v} \hat{\th} } |^{2} &\les&   | \Lie_{t }  \Psi_{\hat{v} \hat{\th} } |^{2}  + |  \rLie  \Psi_{\hat{v} \hat{w} } |^{2}  + |   \rLie  \Psi_{ \hat{\th} \hat{\phi} }  |^{2} + |   \Psi_{\hat{v} \hat{\th} } |^{2}\\
&\les&   | \Lie_{t }  \Psi_{\hat{v} \hat{\th} } |^{2}  + |  r\rLie  \Psi_{\hat{v} \hat{w} } |^{2}  + |   r\rLie  \Psi_{ \hat{\th} \hat{\phi} }  |^{2} + |   \Psi_{\hat{v} \hat{\th} } |^{2}\\
\eeaa
(in the region $0 < 2m \leq r \leq R$).\\

Therefore,

\beaa
&& |F_{ \hat{v}\hat{\th} } (v, w, \om ) |^{2} \\  
&\lesssim&   \int_{ \overline{v} = v - 1 }^{ \overline{v} = v } \int_{\S^{2}} ( | F_{ \hat{v}\hat{\th} } |^{2} + |  \Lie_{v}   F_{ \hat{v}\hat{\th} } |^{2}  + | \rLie F_{ \hat{v}\hat{\th} } |^{2}  \\
&& +  | \rLie \Lie_{v} F_{ \hat{v}\hat{\th} } |^{2}  +  | (\rLie)^{2}   F_{ \hat{v}\hat{\th} } |^{2}  + | (\rLie)^{2}  \Lie_{v}  F_{ \hat{v}\hat{\th} } |^{2}  ) d\sigma^{2} d\overline{v} \\
&\lesssim&   \int_{ \overline{v} = v - 1 }^{ \overline{v} = v } \int_{\S^{2}} ( | F_{ \hat{v}\hat{\th} } |^{2} + |  \Lie_{v}   F_{ \hat{v}\hat{\th} } |^{2}  + | r{\rLie} F_{ \hat{v}\hat{\th} } |^{2}  \\
&& +  | r{\rLie} \Lie_{v} F_{ \hat{v}\hat{\th} } |^{2}  +  | r^{2}(\rLie)^{2}   F_{ \hat{v}\hat{\th} } |^{2}  + | r^{2}(\rLie)^{2}  \Lie_{v}  F_{ \hat{v}\hat{\th} } |^{2}  ) d\sigma^{2} d\overline{v} \\
&\les& - F_{F, r{\rLie} F, r^{2}(\rLie)^{2} F, \Lie_{t } F, r{\rLie} \Lie_{t } F, r^{2}(\rLie)^{2} \Lie_{t } F, r^{3}(\rLie)^{3}   F }^{(H)} ( w  ) ( v-1 \leq \overline{v} \leq v ) \\
&\les& \sum_{j=0}^{2} \sum_{i=0}^{1} - F_{ r^{j}(\rLie)^{j} (\Lie_{t})^{i} F}^{(H)}  ( w  ) ( v-1 \leq \overline{v} \leq v )  - F_{ r^{3}(\rLie)^{3}  F }^{(H)} ( w  ) ( v-1 \leq \overline{v} \leq v ) \\
\eeaa

From estimate \eqref{estimate6Hwequalconstantflux},

\beaa
 - F_{\Psi}^{(H)}  ( w  ) ( v-1 \leq \overline{v} \leq v )  &\lesssim &  \frac{ [ |E_{\Psi}^{  (\frac{\pa}{\pa t} )} | + E_{\Psi}^{ \# (\frac{\pa}{\pa t} )} ( t= t_{0} ) + E_{\Psi}^{M} ] }{ v_{+}^{2}  }  \\
\eeaa

Recall that,

\beaa
E^{M}_{\Psi} &=& \sum_{j=0}^{3} E_{ r^{j}(\rLie)^{j} \Psi  }^{(\frac{\pa}{\pa t})} (t=t_{0})    + \sum_{j=0}^{2}  E_{ r^{j}(\rLie)^{j} \Psi  }^{(K)} (t=t_{0}) 
\eeaa

Therefore,

\beaa
| F_{ \hat{v}\hat{\th} } (v, w, \om ) | &\lesssim &  \frac{ E_{2}  }{  v_{+} }  
\eeaa

where,

\beaa
E_{2} &= & [ \sum_{j = 0}^{5}  \sum_{i = 0}^{1}   E_{   r^{j}(\rLie)^{j}  (\Lie_{t})^{i} F }^{(\frac{\pa}{\pa t}) } (t=t_{0})    +  E_{   r^{6}(\rLie)^{6}  F }^{(\frac{\pa}{\pa t}) } (t=t_{0}) \\
&& + \sum_{j = 0}^{4}  \sum_{i = 0}^{1}   E_{   r^{j}(\rLie)^{j}  (\Lie_{t})^{i}  F }^{( K ) } (t=t_{0})  +  E_{     r^{5}(\rLie)^{5}F }^{( K ) } (t=t_{0}) \\
&& + \sum_{j=1}^{2}  \sum_{i=0}^{1}    E_{ r^{j}(\rLie)^{j} (\Lie_{t})^{i} F }^{ \# (\frac{\pa}{\pa t} )} ( t= t_{0} )   +    E_{r^{3}(\rLie)^{3} F }^{ \# (\frac{\pa}{\pa t} )} ( t= t_{0} )  ]^{\frac{1}{2}} \\
&=& [ E_{F}^{2} + \sum_{j=1}^{2}  \sum_{i=0}^{1}    E_{ r^{j}(\rLie)^{j} (\Lie_{t})^{i} F }^{ \# (\frac{\pa}{\pa t} )} ( t= t_{0} )   +    E_{r^{3}(\rLie)^{3} F }^{ \# (\frac{\pa}{\pa t} )} ( t= t_{0} )  ]^{\frac{1}{2}} 
\eeaa

and similarly,

\beaa
| F_{ \hat{v}\hat{\phi} } (v, w, \om ) | &\lesssim &  \frac{ E_{2}  }{ v_{+}   }  \\
\eeaa

\subsection{ Decay for $\sqrt{1 - \frac{2m}{r}} F_{\hat{w} e_{1}}$ and $\sqrt{1 - \frac{2m}{r}}F_{\hat{w} e_{2} }$ }\

We have the Sobolev inequality, 

\beaa
&& | \sqrt{(1-\mu)} F_{ \hat{w}\hat{\th} } (v, w, \om ) |^{2}   \\
&\lesssim&   \int_{ \overline{w} = w_{0} }^{ \overline{w} = \infty } \int_{\S^{2}} ( | \sqrt{(1-\mu)} F_{ \hat{w}\hat{\th} } |^{2} + |  \Lie_{w} (  \sqrt{(1-\mu)} F_{ \hat{w}\hat{\th} } ) |^{2}  + | \rLie ( \sqrt{(1-\mu)} F_{ \hat{w}\hat{\th} } ) |^{2}  \\
&& +  | \rLie \Lie_{w} ( \sqrt{(1-\mu)} F_{ \hat{w}\hat{\th} } ) |^{2}  +  | (\rLie)^{2}  (\sqrt{(1-\mu)} F_{ \hat{w}\hat{\th} } ) |^{2} \\
&& + | (\rLie)^{2}  \Lie_{w}  (\sqrt{(1-\mu)} F_{ \hat{w} \hat{\th} } ) |^{2}  ) d\sigma^{2} d\overline{w}
\eeaa

We have,
\beaa
\Lie_{\hat{w} } \Psi_{\hat{w} \hat{\th} } &=& \der_{\hat{w} } \Psi_{\hat{w} \hat{\th} } + \Psi ( \der_{\hat{w} } \hat{w}, \hat{\th} ) + \Psi (\hat{w}, \der_{\hat{w} } \hat{\th} ) \\
&=& \der_{\hat{w} } \Psi_{\hat{w} \hat{\th} } \\
&=&  \frac{1}{(1-\mu)^{2}} \der_{w } \Psi_{w \hat{\th} }
\eeaa

Computing,

\beaa
\der_{w } \Psi_{w \hat{\th} } &=&  \frac{1}{2}  \der_{w } ( \Psi_{t \hat{\th} } -  \Psi_{r^{*} \hat{\th} }  ) = \frac{1}{2}  \der_{ ( \frac{1}{2} \frac{\pa }{\pa t} - \frac{1}{2} \frac{\pa }{\pa r^{*}} )  } ( \Psi_{t \hat{\th} } -  \Psi_{r^{*} \hat{ \th}}  ) \\
&=& \frac{1}{4}  \der_{t} ( \Psi_{t \hat{\th}} - \Psi_{r^{*} \hat{ \th}}  ) - \frac{1}{4}  \der_{r^{*}} ( \Psi_{t \hat{\th}} -  \Psi_{r^{*} \hat{ \th}}  ) \\
&=& \frac{1}{4}  \der_{t} ( \Psi_{t \hat{\th}} - \Psi_{r^{*} \hat{ \th}}  ) 
- \frac{1}{4}  ( - \der_{ t }  \Psi_{ \hat{\th} r^{*} } - \der_{ \hat{\th} }  \Psi_{r^{*} t } - \der_{t } \Psi_{ t \hat{\th} }  + (1-\mu) \der_{\hat{\phi} } \Psi_{ \hat{\phi} \hat{\th} }   )  \\
&=& \frac{1}{4}  \der_{t} ( \Psi_{t \hat{\th}} - \Psi_{r^{*} \hat{ \th}}  ) + \frac{1}{4}  (  - \der_{ t }  \Psi_{ r^{*} \hat{\th}  } + \der_{t } \Psi_{t \hat{\th} }  + \der_{ \hat{\th} }  \Psi_{r^{*} t }   -  (1-\mu) \der_{\hat{\phi} } \Psi_{ \hat{\phi} \hat{\th} }   ) \\
&=& \frac{1}{2} \der_{t } ( \Psi_{t \hat{\th} } -  \Psi_{r^{*} \hat{\th}}  )  + \frac{(1-\mu)}{4}  (  \der_{ \hat{\th} }  \Psi_{\hat{r^{*}} \hat{t} }   -   \der_{\hat{\phi} } \Psi_{ \hat{\phi} \hat{\th} }   ) 
\eeaa
(using the field equations and the Bianchi identities).

Thus,
\beaa
\der_{w } \Psi_{w \hat{\th} } &=& \frac{1}{2} ( \der_{t }  \Psi_{t \hat{\th} } -  \der_{t } \Psi_{r^{*} \hat{\th}}  )  + \frac{(1-\mu)}{4}  (  \der_{ \hat{\th} }  \Psi_{\hat{r^{*}} \hat{t} }   -  \der_{\hat{\phi} } \Psi_{ \hat{\phi} \hat{\th} }   ) \\
&=& \frac{1}{2} ( \Lie_{t }  \Psi_{t \hat{\th} } - \frac{\mu}{2 r} \Psi_{r^{*} \hat{\th} }   -  \Lie_{t } \Psi_{r^{*} \hat{\th}}    + \frac{\mu}{2 r  } \Psi_{t \hat{\th}}  ) \\
&&  + \frac{(1-\mu)}{4}  (  \Lie_{ \hat{\th} }  \Psi_{\hat{r^{*}} \hat{t} } - \frac{\sqrt{1-\mu} }{r} \Psi_{\hat{\th} \hat{t} }   -  \Lie_{\hat{\phi} } \Psi_{ \hat{\phi} \hat{\th} }  -  \frac{\sqrt{1-\mu}}{r} \Psi_{ \hat{r^{*}} \hat{\th} }   ) \\
&=&  \Lie_{t }  \Psi_{w \hat{\th} }   + \frac{(1-\mu)}{4}  (   \Lie_{ \hat{\th} }  \Psi_{\hat{r^{*}} \hat{t} }   -  \Lie_{\hat{\phi} } \Psi_{ \hat{\phi} \hat{\th} }   )  +    \frac{\mu}{ 2 r } \Psi_{\hat{w} \hat{\th}}  +   \frac{(1-\mu)}{2r} \Psi_{ w \hat{\th} }    \\
&=&  \Lie_{t }  \Psi_{w \hat{\th} }   + \frac{(1-\mu)}{2}     \Lie_{ \hat{\th} }  \Psi_{\hat{v} \hat{w} }   - \frac{(1-\mu)}{4}  \Lie_{\hat{\phi} } \Psi_{ \hat{\phi} \hat{\th} }   +   \Psi_{w \hat{\th}}    
\eeaa

and similarly, we obtain,

\beaa
\der_{w } \Psi_{w \hat{\phi} } 
&=&  \Lie_{t }  \Psi_{w \hat{\phi} }   + \frac{(1-\mu)}{2}     \Lie_{ \hat{\phi} }  \Psi_{\hat{v} \hat{w} }   - \frac{(1-\mu)}{4}  \Lie_{\hat{\th} } \Psi_{ \hat{\th} \hat{\phi} }  +   \Psi_{ w \hat{\phi} }   \\
\eeaa

Therefore,
\beaa
\Lie_{w } \Psi_{\hat{w} \hat{\th} } &=&  (1-\mu) \Lie_{\hat{w} } \Psi_{\hat{w} \hat{\th} } = \frac{1}{(1-\mu)} \der_{w } \Psi_{w \hat{\th} } \\
&=& \Lie_{t }  \Psi_{\hat{w} \hat{\th} }   + \frac{1}{2}     \Lie_{ \hat{\th} }  \Psi_{\hat{v} \hat{w} }   - \frac{1}{4}  \Lie_{\hat{\phi} } \Psi_{ \hat{\phi} \hat{\th} }   +   \Psi_{\hat{w} \hat{\th}}     
\eeaa

Hence,

\beaa
| \Lie_{w } \Psi_{\hat{w} \hat{\th} } |^{2} &\les&   | \Lie_{t }  \Psi_{\hat{w} \hat{\th} } |^{2}  + |  \Lie_{ \hat{\th} }  \Psi_{\hat{v} \hat{w} } |^{2}  + |   \Lie_{\hat{\phi} } \Psi_{ \hat{\phi} \hat{\th} }  |^{2} + | \Psi_{\hat{w} \hat{\th}} |^{2}  \\
\eeaa
(by using $a.b \les a^{2} + b^{2}$)\\

We get,
\beaa
| \Lie_{w } \Psi_{\hat{w} \hat{\th} } |^{2}
&\les&   | \Lie_{t }  \Psi_{\hat{w} \hat{\th} } |^{2}  + |  \rLie  \Psi_{\hat{v} \hat{w} } |^{2}  + |   \rLie  \Psi_{ \hat{\th} \hat{\phi} }  |^{2} + | \Psi_{\hat{w} \hat{\th}} |^{2}\\
&\les&   | \Lie_{t }  \Psi_{\hat{w} \hat{\th} } |^{2}  + |  r{\rLie}  \Psi_{\hat{v} \hat{w} } |^{2}  + |   r{\rLie}  \Psi_{ \hat{\th} \hat{\phi} }  |^{2} + | \Psi_{\hat{w} \hat{\th}} |^{2} \\
\eeaa
(in the region $0 < 2m \leq r \leq R$).\\

We also have,

\beaa
\frac{\pa}{\pa w} \sqrt{(1-\mu)} &=& ( \frac{1}{2} \frac{\pa }{\pa t} - \frac{1}{2} \frac{\pa }{\pa r^{*}} ) \sqrt{(1-\mu)} = - \frac{(1-\mu)}{2} \frac{\pa }{\pa r} \sqrt{(1-\mu)} =  - \frac{(1-\mu)}{2} \frac{\mu}{2 r \sqrt{(1-\mu)} } \\
&=&  - \frac{\mu  }{4 r  }  \sqrt{(1-\mu)}
\eeaa

Thus,

\beaa
&& | \sqrt{(1-\mu)} F_{ \hat{w}\hat{\th} } (v, w, \om ) |^{2}   \\
&\lesssim&   \int_{ \overline{w} = w_{0} }^{ \overline{w} = \infty } \int_{\S^{2}} ( | \sqrt{(1-\mu)} F_{ \hat{w}\hat{\th} } |^{2} + |  \Lie_{w} (  \sqrt{(1-\mu)} F_{ \hat{w}\hat{\th} } ) |^{2}  + | \rLie ( \sqrt{(1-\mu)} F_{ \hat{w}\hat{\th} } ) |^{2}  \\
&& +  | \rLie \Lie_{w} ( \sqrt{(1-\mu)} F_{ \hat{w}\hat{\th} } ) |^{2}  +  | (\rLie)^{2}  (\sqrt{(1-\mu)} F_{ \hat{w}\hat{\th} } ) |^{2} \\
&& + | (\rLie)^{2}  \Lie_{w}  (\sqrt{(1-\mu)} F_{ \hat{w} \hat{\th} } ) |^{2}  ) d\sigma^{2} d\overline{w}\\
&\lesssim&   \int_{ \overline{w} = w_{0} }^{ \overline{w} = \infty } \int_{\S^{2}} (1-\mu) ( | F_{ \hat{w}\hat{\th} } |^{2} + |  \Lie_{w}   F_{ \hat{w}\hat{\th} } |^{2}  + | r{\rLie} F_{ \hat{w}\hat{\th} } |^{2}  \\
&& +  | r{\rLie} \Lie_{w} F_{ \hat{w}\hat{\th} } |^{2}  +  | r^{2}(\rLie)^{2}  F_{ \hat{w}\hat{\th} } |^{2}  + | r^{2}(\rLie)^{2}  \Lie_{w}  F_{ \hat{w} \hat{\th} } |^{2}  ) d\sigma^{2} d\overline{w}\\
&\les& - F_{F, r{\rLie} F, r^{2}(\rLie)^{2} F, \Lie_{t } F, r{\rLie} \Lie_{t } F, r^{2}(\rLie)^{2} \Lie_{t } F, r^{3}(\rLie)^{3}   F }^{(H)} ( v  ) ( w_{0} \leq w \leq \infty ) \\
&\les& \sum_{j=0}^{2} \sum_{i=0}^{1} - F_{ r^{j}(\rLie)^{j} (\Lie_{t})^{i} F}^{(H)} ( v  ) ( w_{0} \leq w \leq \infty )  - F_{ r^{3}(\rLie)^{3}  F }^{(H)} ( v  ) ( w_{0} \leq w \leq \infty ) \\
\eeaa

From estimate 6,

\beaa
 - F_{\Psi}^{(H)} ( v  ) ( w_{0} \leq w \leq \infty )  &\lesssim &  \frac{ [ |E_{\Psi}^{  (\frac{\pa}{\pa t} )} | + E_{\Psi}^{ \# (\frac{\pa}{\pa t} )} ( t= t_{0} ) + E_{\Psi}^{M} ] }{ v_{+}^{2}  }  \\
\eeaa

As
\beaa
E^{M}_{\Psi} &=& \sum_{j=0}^{3} E_{ r^{j}(\rLie)^{j} \Psi  }^{(\frac{\pa}{\pa t})} (t=t_{0})    + \sum_{j=0}^{2}  E_{ r^{j}(\rLie)^{j} \Psi  }^{(K)} (t=t_{0}) 
\eeaa

we have,

\beaa
|\sqrt{(1-\mu)} F_{ \hat{w}\hat{\th} } (v, w, \om ) | &\lesssim &  \frac{ E_{2}  }{  v_{+} }  
\eeaa

and similarly,

\beaa
| \sqrt{(1-\mu)}  F_{ \hat{w}\hat{\phi} } (v, w, \om ) | &\lesssim &  \frac{ E_{2}  }{ v_{+}   }  \\
\eeaa

\end{proof}

\section{Appendix: The Schwarzschild Space-Time and Black Holes}\

General Relativity postulates that space-time is a 4-dimensional Lorentzian manifold $(M, \g)$ that satisfies the Einstein equations,

\[R_{\mu\nu} - \frac{1}{2} \g_{\mu\nu}R = 8\pi T_{\mu\nu}\]

where $T_{\mu\nu}$ is a symmetric 2-tensor on M that is the stress-energy-momentum tensor of matter. In vacuum $T_{\mu\nu} = 0$, thus the Einstein vacuum equations are $R_{\mu\nu} - \frac{1}{2} \g_{\mu\nu}R = 0$. In vacuum, this yields to $R_{\mu\nu} = \frac{1}{2} \g_{\mu\nu}R$ and since $R = {R^{i}}_{i}$ by definition, we get $R = \g^{ij} R_{ij} = \frac{1}{2} \g^{ij} \g_{ij}R = \frac{4}{2}R = 2 R$. This means that in vacuum $R = 0$ and therefore the Einstein vacuum equations can be written as the following:
\bea
R_{\mu\nu} = 0 \label{Einsteinvacuumequations}
\eea

The simplest solution of the Einstein vacuum equations is the 4-dimensional Minkowski metric, it represents a flat space-time. The first black hole solution of the Einstein vacuum equations was discovered by Karl Schwarzschild about a month after the publication of Einstein's theory of General Relativity. However, it took nearly 50 years from then for it to be fully understood as a black hole space-time. When it was first discovered, it was to represent the gravitational field outside a spherical, uncharged, non-rotating star with mass $m$ and was written under the form
\bea
ds^{2} = -(1 - \frac{2m}{r})dt^{2} + \frac{1}{(1 - \frac{2m}{r})}dr^{2} + r^{2}(d\th^{2} + sin^{2}\th d\phi^{2}) \label{eq:Schw}
\eea

It turned out that this solution could be extended, as solutions to the Einstein vacuum equations \eqref{Einsteinvacuumequations}, to describe the gravitational field inside the star, created by the mass of the star, and thought of as being there in vacuum, i.e. without the matter inside the star. The extended Schwarzschild solution is what became to be known as a black hole space-time. It is also good to note that according to Birkhoff's theorem, any spherically symmetric solution of the Einstein vacuum equations is locally isometric to the Schwarzschild solution. In this sense the Schwarzschild solution is unique although it can have a different form in a different system of coordinates.

\subsection{The extended Schwarzschild solution}\

To obtain the extended Schwarzschild solution, as explained in [HE], first let,
\bea
r^{*} = \int \frac{dr}{(1 - \frac{2m}{r})} = r + 2m\log(r - 2m)
\eea
We have,
\beaa
dr^{*} &=& dr + \frac{2m}{r - 2m}dr =  (1 + \frac{1}{\frac{r}{2m} - 1} ) dr =  (\frac{\frac{r}{2m}}{\frac{r}{2m} - 1}) dr\\
&=& (\frac{1}{1 - \frac{2m}{r}} ) dr
\eeaa

Thus, the Schwarzschild space-time in the exterior \eqref{eq:Schw} can be also written as:
\bea
ds^{2} = - (1 - \frac{2m}{r})dt^{2} + (1 - \frac{2m}{r})dr^{*2} + r^{2}d\sigma^{2} \label{extendedSchwarzschildwithrstar}
\eea
where $d\sigma^{2}$ is the usual metric on the unit sphere.\

Let,
\bea
v &=& t + r^{*} \label{v} \\
w &=& t - r^{*} \label{w}
\eea
We have
\beaa
t &=& \frac{v + w}{2} \\
dt &=& \frac{dv + dw}{2} \\
(dt)^{2} &=& \frac{dv^2}{4} + \frac{dw^2}{4} + \frac{dvdw}{2}  = \frac{dv^2}{4} + \frac{dw^2}{4} + \frac{dv \bigotimes dw}{4} + \frac{dw \bigotimes dv}{4} \\
\eeaa
Injecting in \eqref{extendedSchwarzschildwithrstar}, we get
\begin{eqnarray*}
ds^{2} &=& - (1 - \frac{2m}{r})[ \frac{dv^2}{4} + \frac{dw^2}{4} + \frac{dvdw}{2}] + (1 - \frac{2m}{r}) (dr^{*})^{2} + r^{2}d\sigma^{2}
\end{eqnarray*}
We also have
\beaa
 r^{*} &=& \frac{v - w}{2} \\
(dr^{*})^2 &=& (\frac{dv - dw}{2})^2 =  \frac{dv^2}{4} + \frac{dw^2}{4} - \frac{dvdw}{2} 
\eeaa

Therefore,
\beaa
ds^{2} &=&  - (1 - \frac{2m}{r})[ \frac{dv^2}{4} + \frac{d\om^2}{4} + \frac{dvdw}{2}] + (1 - \frac{2m}{r})^{-1}(1 - \frac{2m}{r})^{2} [\frac{dv^2}{4} + \frac{d\om^2}{4} - \frac{dvdw}{2}] + r^{2}d\sigma^{2}
\eeaa
Thus,
\bea
\notag
ds^{2} &=&  - (1 - \frac{2m}{r})dv dw + r^{2}d\sigma^{2} \\
&=&  - \frac{(1 - \frac{2m}{r})}{2} dv\bigotimes dw - \frac{(1 - \frac{2m}{r})}{2} dw \bigotimes dv + r^{2}d\sigma^{2}
\eea

\subsubsection{Kruskal coordinates}\

Let $v^{'} = v^{'}(v)$,   $w^{'} = w^{'}(w)$ where $v^{'}, w^{'}$ are arbitrary $C^{1}$ functions.\
$$dv = \frac{dv}{dv^{'}}dv^{'}$$ $$d\om = \frac{d\om}{d\om^{'}}d\om^{'}$$
\bea
ds^{2} =  - (1 - \frac{2m}{r})  \frac{dv}{dv^{'}} \frac{d\om}{d\om^{'}} dv^{'} d\om^{'} + r^{2}d\sigma^{2} \label{newformofthemetricwithdifferentiableprimefunctions}
\eea
Define,
\bea
x^{'} &=& \frac{v^{'} - w^{'}}{2} \\
t^{'} &=& \frac{v^{'} + w^{'}}{2} \label{Schwarzschildtime}
\eea
We get $$v^{'} = t^{'} + x^{'}$$ $$w^{'} = t^{'} - x^{'}$$
and $$dv^{'} = dt^{'} + dx^{'}$$
$$d\om^{'} = dt^{'} - dx^{'}$$

Thus,
\beaa
dv^{'}d\om^{'} &=& (dt^{'} + dx^{'})(dt^{'} - dx^{'}) = (dt^{'})^2 - dt^{'}dx^{'} + dx^{'}dt^{'} - (dx^{'})^{2} = (dt^{'})^{2} - (dx^{'})^{2} \\
&=& - ( - (dt^{'})^{2} + (dx^{'})^{2})
\eeaa
Let,
$$F^{2} = - ( - (1-\frac{2m}{r})\frac{dv}{dv^{'}}\frac{dw}{dw^{'}}) = (1-\frac{2m}{r}) \frac{dv}{dv^{'}} \frac{dw}{dw^{'}}$$
From \eqref{newformofthemetricwithdifferentiableprimefunctions}, we have
\beaa
(ds^{'})^{2} = F^{2}(t^{'}, x^{'}) ( - (dt^{'})^{2} + (dx^{'})^{2}) + r^{2}d\sigma^{2}
\eeaa

Kruskal's choice is:
\bea
v^{'} &=& \mbox{exp} (\frac{v}{4m}) \\
w^{'} &=& - \mbox{exp}(-\frac{w}{4m})
\eea
\beaa
t^{'} &=& \frac{1}{2}(v^{'} + w^{'}) \\
(t^{'})^{2} &=& \frac{1}{4}[(v^{'})^{2} + (w^{'})^{2} + 2v^{'}w^{'}] = \frac{1}{4}( \e(\frac{v}{2m}) + \e(- \frac{w}{2m}) + -2\e(\frac{v-w}{4m})) \\
x^{'} &=& \frac{1}{2}(v^{'} - w^{'}) \\
(x^{'})^{2} &=& \frac{1}{4}( \e(\frac{v}{2m}) + \e(- \frac{w}{2m}) + 2\e(\frac{v-w}{4m}))
\eeaa
Computing,
$$(t^{'})^{2} - (x^{'})^{2} = - \frac{2}{4}( \e(\frac{v-w}{2m}) + \e(\frac{v-w}{2m}) ) = - \frac{2}{2} \e(\frac{v-w}{4m}) = - \e(\frac{v-w}{4m}) $$

$$v - w = 2r^{*} = 2r + 4m\log(r - 2m)$$
Thus,
\bea
(t^{'})^{2} - (x^{'})^{2} = - \e(\frac{r}{2m}).(r- 2m)  \label{relationx't'andr}
\eea

Computing,
\beaa
\frac{dv^{'}}{dv} &=& \frac{1}{4m}\e(\frac{v}{4m}), \frac{dw^{'}}{dw} = \frac{1}{4m} \e(- \frac{w}{4m}) \\
\frac{dv}{dv^{'}} &=& 4m.\e(- \frac{v}{4m}), \frac{dw}{dw^{'}} = 4m.\e(\frac{w}{4m}) \\
\frac{dv}{dv^{'}} \frac{dw}{dw^{'}} &=& 16.m^{2} \e(\frac{w - v}{4m})
\eeaa
Therefore,
\begin{eqnarray*}
F^{2} &=& (1 - \frac{2m}{r}) 16.m^{2} \e(-\frac{2r^{*}}{4m})\\
&=& \frac{(r - 2m)}{r} 16.m^{2} \e(-\frac{2}{4m} (r + 2m\log(r - 2m)))\\
&=& \frac{(r - 2m)}{r} 16.m^{2} \e(- \frac{2r}{4m}) \e(- \log(r - 2m)) \\
& =& \frac{1}{r}.16.m^{2} \e(- \frac{r}{2m}) \\
&=& \frac{16.m^{2}}{r}. \e(-\frac{r}{2m}) 
\end{eqnarray*}

Finally, we obtain,
\bea
ds^{2} = \frac{16m^{2}}{r} \e(-\frac{r}{2m}) ( - (dt^{'})^{2} + (dx^{'})^{2} ) + r^{2}(t^{'}, x^{'}) d\sigma^{2}
\eea

\subsection{The Penrose diagram}\

The Penrose diagram, see [HE], is constructed by taking,
\bea
v^{''} &=& \arctan(\frac{v^{'}}{2m}) \\
w^{''} &=& \arctan(\frac{w^{'}}{2m}) \\
 -\pi &< & v^{''} + w^{''} < \pi\\
 -\frac{\pi}{2} &<& v^{''} < \frac{\pi}{2}\\
-\frac{\pi}{2} & <& w^{''} < \frac{\pi}{2}
\eea

\subsection{The compatible symmetric connection}   \label{The compatible symmetric connection} \

Computing the Christoffel symbols for the Schwarzschild metric, we have
\beaa
\Ga_{k l}^{i} = \frac{1}{2} \g^{im} (\frac{\pa \g_{mk}}{\pa x^{l}} + \frac{\pa \g_{ml}}{\pa x^{k}} - \frac{\pa \g_{kl}}{\pa x^{m}} )
\eeaa
Hence,
\beaa
\Ga_{\th \th}^{i} &=& \frac{1}{2} \g^{im} (\frac{\pa \g_{m\th}}{\pa x^{\th}} + \frac{\pa \g_{m\th}}{\pa x^{\th}} - \frac{\pa \g_{\th\th}}{\pa x^{m}} ) 
\eeaa
We get
\beaa
\Ga_{\th \th}^{\th} &=& \frac{1}{2} \g^{\th\th} ( 2 \frac{\pa \g_{\th\th}}{\pa x^{\th}} - \frac{\pa \g_{\th\th}}{\pa x^{\th}} ) \\
&=& 0 \\
\Ga_{\th \th}^{r} &=& \frac{1}{2} \g^{rr} ( - \frac{\pa \g_{\th\th}}{\pa x^{r}} ) \\
&=& \frac{1}{2} (1-\mu)  ( - \frac{\pa r^{2} }{\pa r } ) \\
&=& \frac{1}{2} (1-\mu)  ( - 2 r  ) \\
&=& - (1-\mu) r  \\
\Ga_{\th \th}^{t}&=& \Ga_{\th \th}^{\phi}     = 0
\eeaa
Consequently,
\bea
\notag
\der_{\th} \frac{\pa}{\pa \th} &=& - (1-\mu) r \frac{\pa}{\pa r}  \\
&=& - r \frac{\pa}{\pa r^{*}}  
\eea
We have,
\beaa
\Ga_{\phi \phi}^{i} = \frac{1}{2} \g^{im} ( 2 \frac{\pa \g_{m\phi}}{\pa x^{\phi}} - \frac{\pa \g_{\phi\phi}}{\pa x^{m}} )
\eeaa
therefore
\beaa
\Ga_{\phi \phi}^{r} &=& \frac{1}{2} \g^{rr} ( 2 \frac{\pa \g_{m\phi}}{\pa x^{\phi}} - \frac{\pa \g_{\phi\phi}}{\pa x^{r}} )  \\
&=&  \frac{1}{2} \g^{rr} ( - \frac{\pa \g_{\phi\phi}}{\pa x^{r}} ) \\
&=&  \frac{1}{2} (1-\mu) ( - 2 r \sin^{2}(\th) ) \\
&=& - r (1-\mu) \sin^{2}(\th) \\
\Ga_{\phi \phi}^{\th} &=& \frac{1}{2} \g^{\th\th} ( - \frac{\pa \g_{\phi\phi}}{\pa x^{\th}} ) \\
&=&  \frac{1}{2 r^{2}} \g^{im} ( - r^{2} \frac{\pa \sin^{2}(\th) }{\pa \th } ) \\
&=&  \frac{1}{2} (-2 \sin(\th) \cos(\th)) \\
&=& - \sin(\th) \cos(\th)
\eeaa
We obtain
\bea
\notag
\der_{\phi} \frac{\pa}{\pa \phi} &=& - (1-\mu) r \sin^{2}(\th) \frac{\pa}{\pa r} - \sin(\th) \cos(\th) \frac{\pa}{\pa \th}  \\
&=&  - r \sin^{2}(\th) \frac{\pa}{\pa r^{*}} - \sin(\th) \cos(\th) \frac{\pa}{\pa \th}
\eea
We also have
\beaa
\Ga_{\th \phi}^{i} &=& \frac{1}{2} \g^{im} (\frac{\pa \g_{m \th}}{\pa x^{\phi}} + \frac{\pa \g_{m \phi}}{\pa x^{\th}} - \frac{\pa \g_{\th \phi}}{\pa x^{m}} ) \\
&=& \frac{1}{2} \g^{im} (  \frac{\pa \g_{m \phi}}{\pa x^{\th}}  ) \\
\eeaa
hence
\beaa
\Ga_{\th \phi}^{\phi} &=& \frac{1}{2} \g^{\phi\phi} (  \frac{\pa \g_{\phi \phi}}{\pa x^{\th}}  ) \\
&=&  \frac{1}{2 r^{2} \sin^{2}(\th) } (  \frac{\pa r^{2} \sin^{2}(\th) }{\pa \th}  )  \\
&=&  \frac{1}{2 \sin^{2}(\th) } (  2 \sin(\th) cos(\th)  ) \\
&=& \frac{\cos(\th)}{\sin(\th)}  
\eeaa
We get
\bea
\der_{\th} \frac{\pa}{\pa \phi} &=&   \frac{\cos(\th)}{\sin(\th)} \frac{\pa}{\pa \phi}
\eea
We also compute
\beaa
\Ga_{\th t}^{i} &=& \frac{1}{2} \g^{im} (\frac{\pa \g_{m \th}}{\pa x^{t}} + \frac{\pa \g_{m t}}{\pa x^{\th}} - \frac{\pa \g_{\th t}}{\pa x^{m}} ) \\
&=& 0
\eeaa
from which we derive
\bea
\der_{\th} \frac{\pa}{\pa t} = 0
\eea

Similarly,
\beaa
\Ga_{\phi t}^{i} &=& 0 
\eeaa
gives
\bea
\der_{\phi} \frac{\pa}{\pa t} = 0
\eea
Computing
\beaa
\Ga_{tt}^{i} &=& \frac{1}{2} \g^{im} (\frac{\pa \g_{mt}}{\pa x^{t}} + \frac{\pa \g_{mt}}{\pa x^{t}} - \frac{\pa \g_{tt}}{\pa x^{m}} ) \\
&=&  \frac{1}{2} \g^{im} (  - \frac{\pa \g_{tt}}{\pa x^{m}} ) 
\eeaa
\beaa
\Ga_{tt}^{r} &=& \frac{1}{2} \g^{rr} (  - \frac{\pa \g_{tt}}{\pa x^{r}} )  \\
&=&    \frac{1}{2} (1-\mu) (   \frac{\pa (1-\mu) }{\pa r } ) \\
&=&    \frac{\mu (1-\mu) }{2r } 
\eeaa
we derive
\bea
\notag
\der_{t} \frac{\pa}{\pa t} &=& \frac{\mu (1-\mu) }{2r } \frac{\pa}{\pa r} \\
&=& \frac{\mu  }{2r } \frac{\pa}{\pa r^{*}} 
\eea
Computing
\beaa
\Ga_{r^{*} r^{*}}^{i} &=& \frac{1}{2} \g^{im} (\frac{\pa \g_{mr^{*}}}{\pa x^{r^{*}}} + \frac{\pa \g_{mr^{*}}}{\pa x^{r^{*}}} - \frac{\pa \g_{r^{*}r^{*}}}{\pa x^{m}} ) \\
\Ga_{r^{*} r^{*}}^{r^{*}} &=& \frac{1}{2} \g^{r^{*}r^{*}} ( 2 \frac{\pa \g_{r^{*}r^{*}}}{\pa x^{r^{*}}}  - \frac{\pa \g_{r^{*}r^{*}}}{\pa x^{r^{*}}} )  \\
&=& \frac{1}{2 (1-\mu) } ( 2 \frac{\pa (1-\mu) }{\pa r^{*} }  - \frac{\pa (1-\mu) }{\pa r^{*} } ) \\
&=& \frac{1}{ 2(1-\mu) } (  \frac{\pa (1-\mu) }{\pa r^{*} }   ) \\
&=& \frac{1}{ 2 } (  \frac{\pa (1-\mu) }{\pa r }   ) \\
&=& \frac{\mu}{ 2 r}  
\eeaa
we obtain
\bea
\der_{r^{*}} \frac{\pa  }{\pa r^{*} } &=& \frac{\mu}{ 2 r}   \frac{\pa  }{\pa r^{*} }   
\eea
Computing
\beaa
\Ga_{t r^{*}}^{i} &=& \frac{1}{2} \g^{im} (\frac{\pa \g_{mt}}{\pa x^{r^{*}}} + \frac{\pa \g_{mr^{*}} }{\pa x^{t}} - \frac{\pa \g_{tr^{*}}}{\pa x^{m}} ) \\
&=&  \frac{1}{2} \g^{it} (\frac{\pa \g_{tt}}{\pa x^{r^{*}}} )  \\
\Ga_{t r^{*}}^{t}  &=&  \frac{1}{2} \g^{tt} (\frac{\pa \g_{tt}}{\pa x^{r^{*}}} )  \\
&=& \frac{1}{2(1-\mu)}  (\frac{\pa (1-\mu) }{\pa r^{*} } )  \\
&=& \frac{1}{2 }  (\frac{\pa (1-\mu) }{\pa r } )  \\
&=& \frac{\mu}{2 r}  
\eeaa
we get
\bea
\der_{t} \frac{\pa  }{\pa r^{*} } &=& \frac{\mu}{ 2 r}   \frac{\pa  }{\pa t }   
\eea
We have
\beaa
\Ga_{\th r^{*}}^{i} &=& \frac{1}{2} \g^{im} (\frac{\pa \g_{m\th}}{\pa x^{r^{*}}} + \frac{\pa \g_{mr^{*}}}{\pa x^{\th}} - \frac{\pa \g_{\th r^{*}}}{\pa x^{m}} ) \\
&=& \frac{1}{2} \g^{im} (\frac{\pa \g_{m\th}}{\pa r^{*}} ) \\
\Ga_{\th r^{*}}^{\th} &=& \frac{1}{2} \g^{\th\th} (\frac{\pa \g_{\th\th}}{\pa r^{*}} ) \\
&=&  \frac{1}{2 r^{2}} (\frac{\pa r^{2} }{\pa r^{*}} ) \\
&=&  \frac{(1-\mu)}{2 r^{2}} (\frac{\pa r^{2} }{\pa r} ) \\
&=& \frac{(1-\mu)}{2 r^{2}} (2 r ) \\
&=& \frac{(1-\mu)}{ r }  
\eeaa
therefore,
\bea
\der_{\th} \frac{\pa}{\pa r^{*} }  &=& \frac{(1-\mu)}{ r }   \frac{\pa}{\pa \th }
\eea
Computing
\beaa
\Ga_{\phi r^{*}}^{i} &=& \frac{1}{2} \g^{im} (\frac{\pa \g_{m\phi}}{\pa x^{r^{*}}} + \frac{\pa \g_{mr^{*}}}{\pa x^{\phi}} - \frac{\pa \g_{\phi r^{*}}}{\pa x^{m}} ) \\
&=& \frac{1}{2} \g^{im} (\frac{\pa \g_{m\phi}}{\pa r^{*}} ) \\
\Ga_{\phi r^{*}}^{\phi} &=& \frac{1}{2} \g^{\phi\phi} (\frac{\pa \g_{\phi\phi}}{\pa r^{*}} ) \\
&=&  \frac{1}{2 r^{2} \sin^{2}(\th) } (\frac{\pa r^{2} \sin^{2}(\th)  }{\pa r^{*}} ) \\
&=&  \frac{(1-\mu)}{2 r^{2}} (\frac{\pa r^{2} }{\pa r} ) \\
&=& \frac{(1-\mu)}{2 r^{2}} (2 r ) \\
&=& \frac{(1-\mu)}{ r }  
\eeaa

we obtain

\bea
\der_{\phi} \frac{\pa}{\pa r^{*} }  &=& \frac{(1-\mu)}{ r }   \frac{\pa}{\pa \phi }
\eea

By letting,
\bea
\notag
\hat{\frac{\pa}{\pa t }}  &=&  \frac{1}{\sqrt{(1-\mu)}} \frac{\pa}{\pa t}\\
\hat{\frac{\pa}{\pa r^{*}}}  &=&  \frac{1}{\sqrt{(1-\mu)}}  \frac{\pa}{\pa r^{*}}
\eea

\bea
\notag
\hat{\frac{\pa}{\pa \th}}  &=&  \frac{1}{r} \frac{\pa}{\pa \th} \\
\hat{\frac{\pa}{\pa \phi}}  &=& \frac{1}{r \sin \th }  \frac{\pa}{\pa \phi}
\eea

we can compute,

\bea
\notag
\der_{\hat{\th}} \hat{\frac{\pa}{\pa \th}} &=& \frac{1}{r} \der_{\th} \hat{\frac{\pa}{\pa \th}} = \frac{1}{r^{2}} \der_{\th} \frac{\pa}{\pa \th} = - \frac{1}{ r} \frac{\pa}{\pa r^{*}} \\
&=&   - \frac{\sqrt{(1-\mu)}}{ r } \hat{\frac{\pa}{\pa r^{*}}}
\eea

\bea
\notag
\der_{\hat{\phi}} \hat{\frac{\pa}{\pa \phi}} &=& \frac{1}{r^{2} \sin^{2}(\th) } \der_{\phi} \frac{\pa}{\pa \phi} \\
&=&  - \frac{ \sqrt{(1-\mu)} }{r    }  \hat{\frac{\pa}{\pa r^{*} }} - \frac{\cos(\th) }{r \sin(\th) }\hat{\frac{\pa}{\pa \th }}
\eea

\bea
\notag
\der_{\hat{\th}}\hat{\frac{\pa}{\pa \phi }} &=& \frac{1}{r } \der_{\th}\hat{\frac{\pa}{\pa \phi }} = \frac{1}{r^{2} } \der_{\th} (\frac{1}{\sin(\th) } \frac{\pa}{\pa \phi }  ) \\
\notag
&=& - \frac{\cos(\th) }{r^{2} \sin^{2}(\th) }   \frac{\pa}{\pa \phi }  +  \frac{1}{r^{2} \sin(\th) } \der_{\th}  \frac{\pa}{\pa \phi }  \\
\notag
&=& - \frac{\cos(\th) }{r^{2} \sin^{2}(\th) }   \frac{\pa}{\pa \phi }   + \frac{\cos(\th)}{r^{2}\sin^{2}(\th)} \frac{\pa}{\pa \phi} \\
&=& 0
\eea

\bea
\notag
\der_{\hat{\phi}}\hat{\frac{\pa}{\pa \th }} &=& \frac{1}{r \sin(\th)  } \der_{\phi}\hat{\frac{\pa}{\pa \th }} = \frac{1}{r^{2} \sin(\th)  } \der_{\phi}  \frac{\pa}{\pa \th }   \\
\notag
&=& \frac{1}{r^{2} \sin(\th)  }  \frac{\cos(\th)}{\sin(\th)} \frac{\pa}{\pa \phi}  \\
&=& \frac{1}{r   }  \frac{\cos(\th)}{\sin(\th)} \hat{\frac{\pa}{\pa \phi}}
\eea

\bea
\notag
\der_{\hat{\th}}\hat{\frac{\pa}{\pa t }} &=& \frac{1}{r\sqrt{(1-\mu)} }  \der_{\th} \frac{\pa}{\pa t }  \\
&=& 0
\eea

\bea
\notag
\der_{\hat{\phi} }\hat{\frac{\pa}{\pa t }} &=& \frac{1}{r \sin(\th) \sqrt{(1-\mu)} }  \der_{\phi } \frac{\pa}{\pa t } \\
&=& 0
\eea

\bea
\notag
\der_{\hat{t}} \hat{\frac{\pa}{\pa t }} &=& \frac{1 }{(1-\mu) } \der_{t} \frac{\pa}{\pa t} \\
&=& \frac{\mu  }{2r \sqrt{(1-\mu)}  } \hat{\frac{\pa}{\pa r^{*} }} 
\eea

\bea
\notag
\der_{\hat{r^{*}}} \hat{\frac{\pa}{\pa r^{*}}}  &=&  \frac{1}{\sqrt{(1-\mu)} } \der_{r^{*}} \hat{\frac{\pa}{\pa r^{*}}}  = \frac{1}{\sqrt{(1-\mu)} } \der_{r^{*}} (\frac{1}{\sqrt{(1-\mu)} }   \frac{\pa }{\pa r^{*} } ) \\
\notag
&=& \sqrt{(1-\mu)}  \der_{r} (\frac{1}{\sqrt{(1-\mu)} })   \frac{\pa }{\pa r^{*} } + \frac{1}{(1-\mu) } \der_{r^{*}}  \frac{\pa }{\pa r^{*} }   \\  
\notag
&=& \sqrt{(1-\mu)}  \frac{(-\mu)}{2r(1-\mu)^{\frac{3}{2} }}   \frac{\pa }{\pa r^{*} } + \frac{1}{(1-\mu) } \der_{r^{*}}  \frac{\pa }{\pa r^{*} }   \\ 
\notag
&=&  \frac{-\mu}{2r(1-\mu)}   \frac{\pa }{\pa r^{*} } + \frac{1}{(1-\mu) } \frac{\mu}{ 2 r}   \frac{\pa  }{\pa r^{*} }   \\ 
&=&  0
\eea

\bea
\notag
\der_{\hat{t}} \hat{\frac{\pa}{\pa r^{*}}}  &=& \frac{1}{(1-\mu)} \der_{t} \frac{\pa  }{\pa r^{*} } \\
&=& \frac{\mu}{ 2 r \sqrt{(1-\mu)} }   \hat{\frac{\pa  }{\pa t}}   
\eea

\bea
\notag
\der_{\hat{r^{*}}} \hat{\frac{\pa  }{\pa t}}  &=& \frac{1}{\sqrt{(1-\mu)}}   \der_{r^{*}} ( \frac{1}{\sqrt{(1-\mu)}} \frac{\pa  }{\pa t } )\\
\notag
&=&   \frac{1}{\sqrt{(1-\mu)}}   \der_{r^{*}} ( \frac{1}{\sqrt{(1-\mu)}} ) \frac{\pa  }{\pa t }  + \frac{1}{(1-\mu)} \der_{t} \frac{\pa  }{\pa r^{*} }  \\
\notag
&=&   \sqrt{(1-\mu)}   \frac{(-\mu)}{2r(1-\mu)^{\frac{3}{2} }} \frac{\pa  }{\pa t }  + \frac{1}{(1-\mu)} \frac{\mu}{ 2 r}   \frac{\pa  }{\pa t }  \\
\notag
&=&  \frac{-\mu}{2r(1-\mu)}   \frac{\pa }{\pa t } + \frac{1}{(1-\mu) } \frac{\mu}{ 2 r}   \frac{\pa  }{\pa t }   \\ 
&=&  0
\eea

\bea
\notag
\der_{\hat{\th}} \hat{\frac{\pa}{\pa r^{*} }}  &=& \frac{1}{r \sqrt{(1-\mu)} } \der_{\th} \frac{\pa}{\pa r^{*} } \\
&=& \frac{\sqrt{(1-\mu)}}{ r  }  \hat{\frac{\pa}{\pa \th }}
\eea

\bea
\notag
\der_{\hat{r^{*}}}\hat{\frac{\pa}{\pa \th }}  &=& \frac{1}{ \sqrt{(1-\mu)} } \der_{r^{*}} (\frac{1}{r} \frac{\pa}{\pa \th } ) \\
\notag
&=& \frac{1}{ \sqrt{(1-\mu)} } \der_{r^{*}} (\frac{1}{r} ) \frac{\pa}{\pa \th }  +  \frac{1}{r \sqrt{(1-\mu)} } \der_{\th} \frac{\pa}{\pa r^{*} } \\
\notag
&=& \sqrt{(1-\mu)}  \der_{r} (\frac{1}{r} ) \frac{\pa}{\pa \th }  + \frac{\sqrt{(1-\mu)}}{ r  }  \hat{\frac{\pa}{\pa \th }} \\
\notag
&=& \frac{-\sqrt{(1-\mu)} }{r^{2}}  \frac{\pa}{\pa \th }  + \frac{\sqrt{(1-\mu)}}{ r^{2}  }   \frac{\pa}{\pa \th } \\
&=& 0
\eea

\bea
\notag
\der_{\hat{\phi}} \hat{\frac{\pa}{\pa r^{*} }}   &=& \frac{1}{r \sin(\th) \sqrt{(1-\mu)}  } \der_{\phi} \frac{\pa}{\pa r^{*} }  \\
&=&  \frac{\sqrt{(1-\mu)}}{ r  }  \hat{\frac{\pa}{\pa \phi }}
\eea

\bea
\notag
\der_{\hat{r^{*}}}\hat{\frac{\pa}{\pa \phi }}   &=& \frac{1}{ \sqrt{(1-\mu)}  } \der_{r^{*}} ( \frac{1}{r\sin(\th)} \frac{\pa}{\pa \phi } )  \\
\notag
&=&  \sqrt{(1-\mu)}   \der_{r} ( \frac{1}{r\sin(\th)}) \frac{\pa}{\pa \phi }    + \frac{1}{ r\sin(\th) \sqrt{(1-\mu)}  } \der_{\phi}  \frac{\pa}{\pa r^{*} } \\
\notag
&=& \frac{ -\sqrt{(1-\mu)} }{ r^{2}\sin(\th)}  \frac{\pa}{\pa \phi }    + \frac{1}{ r\sin(\th) \sqrt{(1-\mu)}  } \der_{\phi}  \frac{\pa}{\pa r^{*} } \\
\notag
&=&  \frac{ -\sqrt{(1-\mu)} }{ r^{2}\sin(\th)}  \frac{\pa}{\pa \phi }   +   \frac{\sqrt{(1-\mu)}}{ r^{2} \sin(\th)  }   \frac{\pa}{\pa \phi } \\
&=& 0
\eea

\subsection{The deformation tensor}\

We start by evaluating
$$\der_{\frac{\partial}{\partial v}} \frac{\partial}{\partial w} = \Ga_{v w}^{i} e_{i}$$
We have,
$$\Ga_{k l}^{i} = \frac{1}{2} \g^{im} (\frac{\pa \g_{mk}}{\pa x^{l}} + \frac{\pa \g_{ml}}{\pa x^{k}} - \frac{\pa \g_{kl}}{\pa x^{m}} )$$
Computing,
$$\Ga_{v w}^{i} = \frac{1}{2} \g^{im} (\frac{\pa \g_{m v}}{\pa w} + \frac{\pa \g_{m w}}{\pa v} - \frac{\pa \g_{v w}}{\pa x^{m}} )$$
we have,
\begin{eqnarray*}
\Ga_{v w}^{v} &=& \frac{1}{2} \g^{vm} (\frac{\pa \g_{w v}}{\pa w} + 0 - \frac{\pa \g_{v w}}{\pa w} ) \\
&=& 0 \\
\Ga_{v w}^{w} &=& \frac{1}{2} \g^{w v} (\frac{\pa \g_{v v}}{\pa w} + \frac{\pa \g_{v w}}{\pa v} - \frac{\pa \g_{v w}}{\pa v} ) \\
&=& \frac{1}{2} \g^{w v}  (0 + \frac{\pa \g_{v w}}{\pa v} - \frac{\pa \g_{v w}}{\pa v} )\\
&=& 0 \\
\Ga_{v w}^{\th} &=&  0 \\
\Ga_{v w}^{\phi} &=& 0
\eeaa
Thus,
$$ \der_{v} \frac{\pa}{\pa w} = 0 $$
also means,
$$ \der^{w} \frac{\pa}{\pa w} = 0 $$
Now, we want to comupte
\begin{eqnarray*}
\Ga_{w w}^{i} &=& \frac{1}{2} \g^{im} (\frac{\pa \g_{m w}}{\pa w} + \frac{\pa \g_{m w}}{\pa w} - \frac{\pa \g_{w w}}{\pa x^{m}} ) \\
&=& \g^{im} \frac{\pa \g_{m w}}{\pa w}\\
\Ga_{v v}^{i} &=& \g^{im} \frac{\pa \g_{m v}}{\pa v}
\end{eqnarray*}
We have,
\beaa
\Ga_{w w}^{v} &=& \g^{v w} \frac{\pa \g_{w w}}{\pa w}\\
&=& 0 \\
\eeaa
and
\beaa
\Ga_{w w}^{w} &=& \g^{v w} \frac{\pa \g_{v w}}{\pa w}
\eeaa
We have,
$$\g_{v w} = - \frac{(1 - \frac{2m}{r} )}{2}$$
and
$$r^{*} = \frac{v - w}{2} $$
Computing,
\beaa
\frac{\pa r}{\pa w} &=& \frac{\pa r}{\pa r^{*}} \frac{\pa r^{*}}{\pa w} = \frac{-1}{2} (1 - \mu)
\eeaa
where we define $\mu = \frac{2m}{r} .$
Computing,
\beaa
\frac{\pa r}{\pa v} &=& \frac{\pa r}{\pa r^{*}} \frac{\pa r^{*}}{\pa v} = \frac{1}{2} (1 - \mu)
\eeaa

\beaa
\frac{\pa \g_{v w} }{\pa r} &=& - \frac{1}{2}\frac{\pa (1 - \frac{2m}{r} )}{\pa r} \\
&=& - \frac{m}{r^{2}}
\eeaa

Thus,
\beaa
\Ga_{w w}^{w} &=& \g^{v w} \frac{\pa \g_{v w}}{\pa w} =  - \frac{2}{1 - \mu} \frac{\pa \g_{v w}}{\pa r} \frac{\pa r}{\pa w}\\
&=& - \frac{1}{1 - \mu} \frac{-2m}{r^{2}} (-\frac{1}{2}(1-\mu)) = - \frac{m(1-\mu)}{r^{2}(1-\mu)}\\
&=& - \frac{m}{r^{2}} \\
\eeaa
We also have
\beaa
\Ga_{w w}^{\th} &=& \Ga_{w w}^{\phi}  = \Ga_{w w}^{v}  = 0 
\eeaa
hence,
\beaa
\der_{w} \frac{\pa}{\pa w} &=& - \frac{m}{r^{2}} \frac{\pa}{\pa w}
\eeaa
On the other hand,
\beaa
\Ga_{v v}^{v} &=& \g^{w v} (\frac{\pa \g_{w v}}{\pa v}) = - \frac{2}{1 - \mu} \frac{\pa \g_{v w}}{\pa r} \frac{\pa r}{\pa v}\\
&=& \frac{2m}{r^{2}(1-\mu)} \frac{\pa r}{\pa v} = \frac{m ( 1 - \mu)}{r^{2}(1-\mu)} \\
&=& \frac{m}{r^{2}}
\eeaa
We also have,
$$\Ga_{v v}^{w} = \g^{w v} (\frac{\pa \g_{vv}}{\pa v}) = 0$$
and,
\beaa
\Ga_{v v}^{\th} &=& \Ga_{v v}^{\phi}    = 0
\eeaa
from which 
\beaa
\der_{v} \frac{\pa}{\pa v} &=&  \frac{m}{r^{2}} \frac{\pa}{\pa v}
\eeaa
Computing,
\beaa
\Ga_{\th v}^{i} &=& \frac{1}{2} \g^{im} (\frac{\pa \g_{m \th}}{\pa x^{v}} + \frac{\pa \g_{m v}}{\pa \th} - \frac{\pa \g_{\th v}}{\pa x^{m}} )\\
&=& \frac{1}{2} \g^{im} \frac{\pa \g_{m \th}}{\pa x^{v}} \\
\Ga_{\th v}^{\th} &=&  \frac{1}{2} \g^{\th\th} \frac{\pa \g_{\th \th}}{\pa v} = \frac{1}{2r^{2}} \frac{\pa (r^{2} ) }{\pa r} \frac{\pa r }{\pa v} = \frac{r}{2r^{2}} (1-\mu)\\
 &=& \frac{(1-\mu)}{2r}
\eeaa
thus,
$$\der_{\th} \frac{\pa}{\pa v} = \frac{(1 - \mu)}{2r} \frac{\pa}{\pa \th} $$
And,
\beaa
\Ga_{\phi v}^{i} &=& \frac{1}{2} \g^{im} (\frac{\pa \g_{m \phi}}{\pa x^{v}} + \frac{\pa \g_{m v}}{\pa \phi} - \frac{\pa \g_{\phi v}}{\pa x^{m}} )\\
&=& \frac{1}{2} \g^{im} \frac{\pa \g_{m \phi}}{\pa v} \\
\Ga_{\phi v}^{\phi} &=&  \frac{1}{2} \g^{\phi\phi} \frac{\pa \g_{\phi \phi}}{\pa v} = \frac{1}{2r^{2} \sin^{2} \th} \frac{\pa (r^{2} \sin^{2} \th ) }{\pa r} \frac{\pa r }{\pa v} = \frac{2r}{2r^{2}} \frac{1}{2}(1-\mu)\\
 &=& \frac{(1-\mu)}{2r}
\eeaa

thus,
$$\der_{\phi} \frac{\pa}{\pa v} = \frac{(1 - \mu)}{2r} \frac{\pa}{\pa \phi} $$

Also,
\beaa
\Ga_{\phi w}^{i} &=& \frac{1}{2} \g^{im} (\frac{\pa \g_{m \phi}}{\pa x^{w}} + \frac{\pa \g_{m w}}{\pa \phi} - \frac{\pa \g_{\phi w}}{\pa x^{m}} )\\
&=& \frac{1}{2} \g^{im} \frac{\pa \g_{m \phi}}{\pa w} \\
\Ga_{\phi w}^{\phi} &=&  \frac{1}{2} \g^{\phi\phi} \frac{\pa \g_{\phi \phi}}{\pa w} = \frac{1}{2r^{2} \sin^{2} \th} \frac{\pa (r^{2} \sin^{2} \th ) }{\pa r} \frac{\pa r }{\pa w} = - \frac{r}{2r^{2}} (1-\mu)\\
 &=& - \frac{(1-\mu)}{2r}
\eeaa

thus,
$$\der_{\phi} \frac{\pa}{\pa w} = - \frac{(1 - \mu)}{2r} \frac{\pa}{\pa \phi} $$
We also have,
$$\der_{\th} \frac{\pa}{\pa w} = - \frac{(1 - \mu)}{2r} \frac{\pa}{\pa \th} $$

Therefore, in conclusion, we have:

\bea
\der_{w} \frac{\pa}{\pa w} &=& - \frac{m}{r^{2}} \frac{\pa}{\pa w}\\
\der_{v} \frac{\pa}{\pa v} &=&  \frac{m}{r^{2}} \frac{\pa}{\pa v}\\
\der_{v} \frac{\pa}{\pa w} &=& \der_{w} \frac{\pa}{\pa v} = 0\\
\der_{\th} \frac{\pa}{\pa v} &=& \frac{(1 - \mu)}{2r} \frac{\pa}{\pa \th} \\
\der_{\th} \frac{\pa}{\pa w} &=&  \frac{- (1 - \mu)}{2r} \frac{\pa}{\pa \th} \\
\der_{\phi} \frac{\pa}{\pa v} &=& \frac{(1 - \mu)}{2r} \frac{\pa}{\pa \phi} \\
\der_{\phi} \frac{\pa}{\pa w} &=&  \frac{ - (1 - \mu)}{2r} \frac{\pa}{\pa \phi} 
\eea
Now, let

\bea
V = V^{w}(v, w) \frac{\pa}{\pa w} + V^{v} (v, w) \frac{\pa}{\pa v}
\eea

We use the notation $V^{w} (v, w) = V^{w}$, and $ V^{v}(v, w) = V^{v}$.

Computing,
\beaa
\der_{v} V &=& ( \pa_{v} V^{w} ) \frac{\pa}{\pa w} + V^{w} ( \der_{v} \frac{\pa}{\pa w} ) + ( \pa_{v} V^{v} ) \frac{\pa}{\pa v} + V^{v} ( \der_{v} \frac{\pa}{\pa v} ) \\
&=& ( \pa_{v} V^{w} ) \frac{\pa}{\pa w} + ( \pa_{v} V^{v} ) \frac{\pa}{\pa v} + V^{v}  \frac{m}{r^{2}} \frac{\pa}{\pa v}\\
\der_{w} V &=& ( \pa_{w} V^{w} ) \frac{\pa}{\pa w} + V^{w} ( \der_{w} \frac{\pa}{\pa w} ) + ( \pa_{w} V^{v} ) \frac{\pa}{\pa v} + V^{v} ( \der_{w} \frac{\pa}{\pa v} ) \\
&=& ( \pa_{w} V^{w} ) \frac{\pa}{\pa w} - V^{w}  \frac{m}{r^{2}} \frac{\pa}{\pa w} + ( \pa_{w} V^{v} ) \frac{\pa}{\pa v}\\ 
\der_{\th} V &=& ( \pa_{\th} V^{w} ) \frac{\pa}{\pa w} + V^{w} ( \der_{\th} \frac{\pa}{\pa w} ) + ( \pa_{\th} V^{v} ) \frac{\pa}{\pa v} + V^{v} ( \der_{\th} \frac{\pa}{\pa v} ) \\
&=& V^{v} \frac{(1-\mu)}{2r}  \frac{\pa}{\pa \th} - V^{w} \frac{(1-\mu)}{2r}  \frac{\pa}{\pa \th}\\
&=& \frac{(1-\mu)}{2r} ( V^{v} - V^{w} )  \frac{\pa}{\pa \th} \\
\der_{\phi} V &=&  V^{w} ( \der_{\phi} \frac{\pa}{\pa w} ) + V^{v} ( \der_{\phi} \frac{\pa}{\pa v} ) \\
&=& V^{v} \frac{(1-\mu)}{2r}  \frac{\pa}{\pa \phi} - V^{w} \frac{(1-\mu)}{2r}  \frac{\pa}{\pa \phi}\\
&=& \frac{(1-\mu)}{2r} ( V^{v} - V^{w} )  \frac{\pa}{\pa \phi} 
\eeaa

Therefore,
\bea
\der_{v} V &=& ( \pa_{v} V^{w} ) \frac{\pa}{\pa w} + ( \pa_{v} V^{v}  +  \frac{m}{r^{2}} V^{v} ) \frac{\pa}{\pa v}\\
\der_{w} V &=& ( \pa_{w} V^{v} ) \frac{\pa}{\pa v} + ( \pa_{w} V^{w} -  \frac{m}{r^{2}} V^{w} ) \frac{\pa}{\pa w}\\ 
\der_{\th} V &=& \frac{(1-\mu)}{2r} ( V^{v} - V^{w} )  \frac{\pa}{\pa \th}\\
\der_{\phi} V&=& \frac{(1-\mu)}{2r} ( V^{v} - V^{w} )  \frac{\pa}{\pa \phi}
\eea

The deformation tensor is, $$\pi^{\a\b} (V) = \frac{1}{2} ( \der^{\a}V^{\b} + \der^{\b}V^{\a})$$
Computing,
\bea
\notag
\pi^{w w} (V) &=& \der^{w}V^{w} = \g^{v\om}\der_{v}V^{w}\\
&=& \frac{-2}{(1 - \mu)} \pa_{v}V^{w}\\
\notag
\pi^{vv} (V) &=& \der^{v}V^{v} = \g^{v\om}\der_{w}V^{v}\\
&=& \frac{-2}{(1 - \mu)} \pa_{w}V^{v} \\
\notag
\pi^{v w} (V) &=& \frac{1}{2} ( \der^{v}V^{w} + \der^{w}V^{v})\\
\notag
&=& \frac{1}{2} ( \g^{v\om} \der_{w}V^{w} + \g^{w v}\der_{v}V^{v} )\\
&=& \frac{-1}{(1-\mu)} [    \pa_{v} V^{v}  + \pa_{w} V^{w}   +      \frac{m}{r^{2}} ( V^{v}  -   V^{w} )   ] \\
&=& \pi^{w v}(V) \\
\notag
\pi^{\th\th} (V) &=& \der^{\th} V^{\th} = \g^{\th\th}\der_{\th}V^{\th}\\
&=& \frac{(1 - \mu)}{2r^{3}}(V^{v} - V^{w}) \\
\notag
\pi^{\phi\phi} (V)&=& \der^{\phi} V^{\phi} = \g^{\phi\phi}\der_{\phi}V^{\phi}\\
&=& \frac{(1 - \mu)}{2r^{3}\sin^{2} \th}(V^{v} - V^{w}) \\
\notag
\pi^{\th\phi} (V) &=& \frac{1}{2} ( \der^{\th}V^{\phi} + \der^{\phi}V^{\th})\\
&=& 0 = \pi^{\phi\th} \\
\notag
\pi^{\th v} (V) &=& \frac{1}{2} ( \der^{\th}V^{v} + \der^{v}V^{\th})\\
&=& 0 = \pi^{v\th} (V) \\
\pi^{\phi v} (V) &=& 0 = \pi^{v\phi} (V) \\
\pi^{\th w} (V)  &=& 0 = \pi^{w\th} (V) \\
\pi^{\phi w} (V) &=& 0 = \pi^{w\phi} (V)
\eea

%\clearpage

\end{document}